\documentclass[12pt]{amsart}
\usepackage{amsmath}
\usepackage{amssymb}
\usepackage{bm}
\usepackage{esint}
\usepackage{graphicx}
\usepackage{verbatim}
\usepackage{enumitem}
\usepackage[normalem]{ulem}
\usepackage{stackengine}
\usepackage{algorithm}
\usepackage{algpseudocode}
\stackMath
\setlist[enumerate,1]{leftmargin=*}
\setlist[itemize,1]{leftmargin=*}
\usepackage{mathrsfs}
\usepackage[dvipsnames]{color}
\usepackage{todonotes}
\usepackage{float}
\usepackage[pdftex,pdfstartview=FitH,pdfborderstyle={/S/B/W 1},%
colorlinks=true, linkcolor=blue, urlcolor=blue, citecolor=blue,%
pagebackref=true]{hyperref}
\usepackage[nobysame,alphabetic,initials,msc-links]{amsrefs}
\usepackage{imakeidx}
\usepackage{tikz}
\usetikzlibrary{decorations.markings}

\newtheorem{Lemma}{Lemma}[section]
\newtheorem{Theorem}[Lemma]{Theorem}
\newtheorem{Proposition}[Lemma]{Proposition}
\newtheorem{Corollary}[Lemma]{Corollary}

\newtheorem{Claim}[Lemma]{Claim}

\newtheorem*{lem:longHammingSep}{Lemma  \ref{lem:longHammingSep}}
\newtheorem*{prop:cauchySeq}{Proposition  \ref{prop:cauchySeq}}
\newtheorem*{cor:sl2ergodic}{Corollary  \ref{cor:sl2ergodic}}

\theoremstyle{definition}

\newtheorem{Definition}[Lemma]{Definition}

\newtheorem{Remark}[Lemma]{Remark}

\numberwithin{equation}{section}

\newcommand {\C} {{\mathbb C}}

\newcommand {\N} {{\mathbb N}}

\newcommand {\R} {{\mathbb R}}
\newcommand {\G} {{\mathbb G}}

\newcommand {\Z} {{\mathbb Z}}
\newcommand {\absolute}[1] {\left| {#1} \right|}
\newcommand {\norm}[1] {\left\| {#1} \right\|}
\newcommand {\innproduct}[1] {\left\langle {#1} \right\rangle}

\DeclareMathOperator {\SL} {SL}

\DeclareMathOperator {\Kak} {Kak}
\DeclareMathOperator {\id} {id}
\DeclareMathOperator {\ad} {ad}
\DeclareMathOperator {\Ad} {Ad}
\DeclareMathOperator {\Bow} {Bow}

\DeclareMathOperator {\inj} {inj}
\newcommand {\Btr}{B^{\operatorname{Tr}}}

\newcommand {\bu} {\mathbf{u}}
\newcommand{\bou} {\hat{\mathbf{u}}}
\newcommand {\ba} {\mathbf{a}}
\newcommand {\bx} {\mathbf{x}}
\newcommand {\bw} {\mathbf{w}}
\newcommand {\by} {\mathbf{y}}
\newcommand {\tr} {\mathrm{tr}}

\DeclareMathSymbol{\shortminus}{\mathbin}{AMSa}{"39}

\newcommand {\ou} {\widehat u}
\newcommand {\oua} {{\widehat u}^{\,(1)}}
\newcommand {\oub} {{\widehat u}^{\,(2)}}
\newcommand {\OU} {\widehat{\mathbf{U}}}

\DeclareMathOperator {\neta} {\eta}

\newcounter{c}

\newcounter{C}
\newcommand{\Cnew}{%
\refstepcounter{C}%
\ensuremath{C_{\theC}}}
\newcommand{\Cold}[1]{\ensuremath{C_{\ref{#1}}}}

\newcounter{e}
\newcommand{\enew}{%
\refstepcounter{e}%
\ensuremath{\epsilon_{\thee}}}
\newcommand{\eold}[1]{\ensuremath{\epsilon_{\ref{#1}}}}

\newcounter{de}
\newcommand{\denew}{%
\refstepcounter{de}%
\ensuremath{\delta_{\thede}}}
\newcommand{\deold}[1]{\ensuremath{\delta_{\ref{#1}}}}

\newcounter{d}

\newcounter{R}
\newcommand{\Rnew}{%
\refstepcounter{R}%
\ensuremath{R_{\theR}}}
\newcommand{\Rold}[1]{\ensuremath{R_{\ref{#1}}}}

\newcounter{r}

\newcounter{K}
\newcommand{\Knew}{%
\refstepcounter{K}%
\ensuremath{{\mathcal K}_{\theK}}}
\newcommand{\Kold}[1]{\ensuremath{{\mathcal K}_{\ref{#1}}}}

\newcounter{O}
\newcommand{\Onew}{%
\refstepcounter{O}%
\ensuremath{O_{\theO}}}
\newcommand{\Oold}[1]{\ensuremath{O_{\ref{#1}}}}

\newcounter{Z}
\newcommand{\Znew}{%
\refstepcounter{Z}%
\ensuremath{{\mathcal Z}_{\theZ}}}
\newcommand{\Zold}[1]{\ensuremath{{\mathcal Z}_{\ref{#1}}}}

\newcounter{k}

\newcounter{w}
\newcommand{\wnew}{%
\refstepcounter{w}%
\ensuremath{w_{\thew}}}
\newcommand{\wold}[1]{\ensuremath{w_{\ref{#1}}}}

\newcounter{ps}
\newcommand{\psnew}{%
\refstepcounter{ps}%
\ensuremath{\psi_{\theps}}}
\newcommand{\psold}[1]{\ensuremath{\psi_{\ref{#1}}}}

\newcounter{th}

\newcounter{n}

\makeindex[title=Index of notation,columns=2]

\begin{document}
\title{Time change rigidity for unipotent flows}
\author{Elon Lindenstrauss}
\thanks{E.L. acknowledges support by ERC 2020 grant HomDyn (grant no.\ 833423)}
\address{E.L.: School of Mathematics, Institute for Advanced Study, 1 Einstein Drive
Princeton, New Jersey, 08540, USA \newline\textrm{\emph{and}}\newline Einstein Institute of Mathematics, Edmond J. Safra Campus,
The Hebrew University of Jerusalem,
Givat Ram, Jerusalem, 9190401, Israel
}
\email{elonl@ias.edu}
\author{Daren Wei}
\thanks{D.W. acknowledges support by ERC 2020 grant HomDyn (grant no.\ 833423), NUS startup grants A-0009806-00-00 and A-0009806-01-00}
\address{D.W.: Department of Mathematics, National University of Singapore, 119076, Singapore}
\email{darenwei@nus.edu.sg}
\date{\today}

\dedicatory{Dedicated to Benjamin Weiss with great admiration}

\begin{abstract}
We prove a dichotomy regarding the behavior of one-parameter unipotent flows on quotients of semisimple lie groups under time change. We show that if $u^{(1)}_t$ acting on $\mathbf{G}_{1}/\Gamma_1$ is such a flow it satisfies exactly one of the following:
\begin{enumerate}
    \item The flow is loosely Kronecker, and hence measurably isomorphic after an appropriate time change to any other loosely Kronecker system.
    \item The flow exhibits the following rigid behavior: if the one-parameter unipotent flow $u^{(1)} _ t$ on $\mathbf{G}_1/\Gamma_1$ is measurably isomorphic after time change to another  such flow $u^{(2)} _ t$ on $\mathbf{G}_2/\Gamma _ 2$, then $\mathbf{G}_1/\Gamma_1 $ is isomorphic to $\mathbf{G}_2/ \Gamma_2$ with the isomorphism taking $u^{(1)}_t$ to $u^{(2)}_t$ and moreover the time change is cohomologous to a trivial one up to a renormalization. 
\end{enumerate}

\end{abstract}

\maketitle

\setcounter{tocdepth}{1}
\tableofcontents
\section{Introduction}

The main subject of this paper is the notion of \textbf{monotone equivalence}\footnote{We shall also use interchangeably the term Kakutani equivalence.} between ergodic probability measure preserving flows. The study of this natural equivalence relation has a long and rich history, starting with the pioneering works of Kakutani in the 1940s. Works of Feldman, Katok, Ornstein, Rudolph, Weiss and others, motivated by Ornstein's work on isomorphism theory of Bernoulli systems, brought new life to the subject in the 1970s; we mention in particular  \cites{feldman1976new,
katok1977monotone,
ornstein1982equivalence}.

We say that two flows $(X_1,\mathcal{B}_1,\mu_1,u_t^{(1)})$
and $(X_2,\mathcal{B}_2,\mu_2,u_t^{(2)})$
are monotone equivalent if there is a 1-1 invertible measurable map $\psi$ (one that takes
measurable sets to measurable sets and sets of measure 0 to sets of measure
0 but does not necessarily preserve measure) and
$\psi$ maps orbits of $ u _ t ^ {(1)}$ in $X _ 1$ onto
orbits of $ u _ t ^ {(2)}$ in $X_2$
in an order preserving way. We shall see below an equivalent 
way to say the same thing using time change maps.

From the point of view of the theory of monotone equivalence, the simplest flows are the collection of \textbf{loosely Kronecker flows} (which are also known as the class of loosely Bernoulli systems of zero entropy (cf.~\cites{feldman1976new,ornstein1982equivalence}) or standard flows (cf.~\cite{katok1977monotone})). This is a fairly large class of systems, including the linear flow on a 2 (or higher) dimensional torus in an irrational direction as well as any system induced from an interval exchange map using the flow under a function construction; all of these systems are monotone equivalent to each other though they are certainly not isomorphic as measure preserving flows.

\medskip

We now restricted our attention to a special kind of dynamical systems: unipotent flows on homogeneous spaces. We take for the space a quotient of the identity component of a real linear algebraic group $\mathbf  G$ by a lattice $\Gamma$, equipped with the standard $\mathbf  G$-invariant measure $m$ (and Borel $\sigma$-algebra $\mathcal B$, and consider the action of a one parameter subgroup $u_t=\exp (t \mathbf u)$ of $\mathbf{G}$ with $\mathbf u$ a nilpotent element of the Lie algebra of $\mathbf G$ (we consider $\mathbf G$ as embedded in some $\SL_n(\R)$, hence we may consider $\mathbf u$ as a $n\times n$ real matrix). Such a group $u_t$ will be said to be a \textbf{one parameter unipotent group} and the system $(\mathbf  G/\Gamma, \mathcal B, m, u_t)$ a \textbf{unipotent flow}. We will restrict ourselves even further to the case of $\mathbf G$ semisimple.

\medskip

A simple case of this setup is when $\mathbf G = \SL _ 2 (\R)$ and $u _ t = \left(\begin{smallmatrix}
1 & t \\
0 & 1
\end{smallmatrix}\right)$, in which case the action of $u_t$ on $\mathbf G / \Gamma$ can be identified with a double cover of the horocycle flow on a hyperbolic surface of constant negative curvature. By a mild abuse of terminology, this type of unipotent flow is often called a horocycle flow.

In the late 1970s, Ratner proved two remarkable theorems. Firstly, in her paper \cite{ratner1978horocycle}
Ratner showed that horocycle flows are loosely Kronecker. Shortly thereafter,  Ratner showed in \cite{ratner1979cartesian} that the product of two horocycle flows, that is the space $X=\SL_2(\R)/\Gamma\times\SL_2(\R)/\Gamma$ equipped with the unipotent flow $u_t=\left(\left(\begin{smallmatrix}
    1 & t \\
    0 & 1
\end{smallmatrix}\right),\left(\begin{smallmatrix}
    1 & t\\
    0 & 1
\end{smallmatrix}\right)\right)$ and the uniform measure $m$, is \textbf{not} loosely Kronecker.\footnote{It was already known at the time that a product of two loosely Kronecker systems may fail to be loosely Kronecker, but this was considered an exotic property. That such a natural system as the horocycle flow has this property was a big surprise.
}
Somewhat later, by defining an invariant for dynamical system that remains unchanged under monotone equivalence, Ratner \cite{ratner1981some} proved that a product of $k$-copies of horocycle flows is not Kakutani equivalent to a product of $\ell$-copies of horocycle flows if $k\neq \ell$.

Ratner's invariant was calculated in the generality we consider in this paper by Kanigowski, Vinhage and the second named author in \cite{kanigowski2021kakutani}. It turns out that Ratner's invariant depends only on the group $\mathbf{G}$ and the one parameter unipotent subgroup $u_t$ but \textbf{not} on the lattice $\Gamma<\mathbf{G}$. In particular in \cite{kanigowski2021kakutani} the authors completely classify for which systems of the type we consider here is Ratner's invariant equal to zero, and show that this holds if and only if $(\mathbf G / \Gamma, \mathcal B, m, u _ t)$ is loosely Kronecker.

It follows that all unipotent flows with Ratner invariant zero are monotone equivalent to each other; in particular (as shown by Ratner in \cite{ratner1978horocycle}) for any two lattices $\Gamma _ 1, \Gamma _ 2 < \SL _ 2 (\R)$ the horocycle flow $ u _ t$ on $\SL _ 2 (\R) / \Gamma _ 1$ is monotone equivalence to the flow of $ u _ t$ on $\SL _ 2 (\R) / \Gamma _ 2$. In contrast to this, Ratner proved in \cite{ratner1982rigidity} that these systems are in general not isomorphic to each other as measure preserving flows: if $(\SL _ 2 (\R) / \Gamma _ 1, \mathcal B, m,  u _ t)$ is isomorphic to $(\SL _ 2 (\R) / \Gamma _ 1, \mathcal B, m,  u _ t)$ then there is an element $c$ in the centralizer of the group $\left\{  u _ t: t \in \R \right\}$ so that $\Gamma _ 1 = c ^{-1} \Gamma _ 2 c $ and $\psi (g \Gamma _ 1) = g c \Gamma _ 2$. Note that this is an instance of rigidity: a hypothetical measurable isomorphism between two such systems can only be of a very nice and algebraic type.
More generally, for~$i=1,2$, let $\mathbf{G}_i$ be real semisimple linear algebraic groups, $\Gamma_i< \mathbf{G}_i$ lattices (we assume that $\mathbf{G}_i$ acts faithfully on $\mathbf{G}_i/\Gamma_i$, otherwise replace $\mathbf{G}_i$ by a quotient group). Let $\mathcal{B}_i$ be the Borel $\sigma$-algebra for $\mathbf{G}_i/\Gamma_i$, $m_i$ the normalized Haar measure on~$\mathbf{G}_i/\Gamma_i$, and $u_t^{(i)}=\exp(t\bu_i)$ a one parameter unipotent subgroup of $\mathbf G _ i$. We assume (this is again a mild assumption that is easy to reduce to) that $u_t^{(i)}$ acts ergodically on $\mathbf G _ i / \Gamma _ i$. Under these conditions Ratner and Witte Morris showed in~\cites{ratner1983horocycle,witte1985rigidity,witte1987zero}, in increasing levels of generality, that if $\psi$ is a \emph{measurable} isomorphism between the two systems $(\mathbf{G}_1/\Gamma_1,\mathcal{B}_1,m_1,u_t^{(1)})$ and $(\mathbf{G}_2/\Gamma_2,\mathcal{B}_2,m_2,u_t^{(2)})$ then there is an isomorphism of algebraic groups $\Psi:\mathbf G _ 1 \to \mathbf G_2$ and an element $c \in \mathbf G_2$ so that
\begin{gather*}
\psi(g \Gamma_1)=c \Psi(g)\Gamma_2,\quad \text{$m_1$-a.e. $g \Gamma_1 \in \mathbf{G}_1/\Gamma_1$},\\
\Gamma_2=\Psi(\Gamma_1),\qquad u_t^{(2)}=c \Psi(u_t^{(1)})c^{-1}
.\end{gather*}
The general form of isomorphism rigidity of unipotent flows was deduced by Ratner from her landmark measure classification results in \cite{ratner1990measure}*{Corollary 6}; for more details and background  cf.\ eg.\ Witte Morris book \cite{Morris05Ratner}.

For $u_t=\left(\left(\begin{smallmatrix}
1 & t \\
0 & 1
\end{smallmatrix}\right),\left(\begin{smallmatrix}
1 & t\\
0 & 1
\end{smallmatrix}\right)\right)$, the action of $u_t$ on $\SL (2, \R) \times \SL (2, \R) / \Gamma$ for any lattice $\Gamma < \SL (2, \R) \times \SL (2, \R)$, whether irreducible or a product $\Gamma _1 \times \Gamma _2$ of two $\SL_2(\R)$-lattices,  all share the same Ratner invariant. Are they all monotone equivalent?

This natural question appears in Ratner's contribution to the 1994 ICM proceedings  \cite{ratner1995interactions}*{p.~179}. More recently it was raised by Gerber and Kunde in their paper \cite{gerber2021anti}, where they obtained an ``anti-classification'' result for monotone equivalence, showing that the monotone equivalence relation is not Borel, which in particular guarantees that one cannot classify measure preserving flows up to monotone equivalence by using any single or countably many invariants. In their paper, they ask about the special case of unipotent flows and Ratner's invariant. In this paper we answer the question, showing that unless they are loosely Kronecker, unipotent flows are \textbf{monotone equivalence rigid}, in the sense that monotone equivalence implies isomorphism.
Since isomorphism rigidity holds for unipotent flows, it follows that in the non loosely Kronecker case monotone equivalence in fact implies algebraic isomorphism.

 Before stating our main result, it will be useful to define more carefully the notions of monotone equivalence and time change. For both~$i=1$ and $i=2$ 
 let $(X_i,\mathcal{B}_i,\mu_i,u_t^{(i)})$ be ergodic measure preserving flows on standard Borel probability spaces. We say that they are \textbf{monotone equivalent}, if there exists a $1$-$1$ and onto measurable map $\psi:X'_1\to X'_2$ with $\mu_i(X'_i)=1$ for $i=1,2$ so that $\mu_2$ and $\psi_*\mu_1$ are equivalent and so that if~$x, u_t^{(1)}.x\in X'_1$ with $t>0$, then $\psi(u_t^{(1)}.x)=u_{\tau}^{(2)}.\psi(x)$ for some $\tau>0$. The term Kakutani equivalent is also commonly used to describe this equivalence relation, though sometimes this term is used to denote a closely related equivalence condition for $\Z$-actions.
Such a map $\psi$ will be said to be a \textbf{monotone} (or Kakutani) \textbf{equivalence}.  

Monotone equivalence can be characterized through time changes. Given an ergodic measure preserving flow $(X,\mathcal{B},\mu,u_t)$ on a standard Borel probability space and a function $\alpha\in L_+^1(X,\mathcal{B},\mu)$, we define the corresponding \textbf{time change} of $u_t$ as the new flow $u_{\alpha,t}$ determined as follows:
\[
u_{\alpha,t}(x)=u_{\rho(x,t)}(x),\qquad \text{$\rho$ and $t$ satisfy $\int_0^{\rho(x,t)}\alpha(u_s.x)ds=t$.}
\]
It is clear that the new flow $u_{\alpha,t}$ is monotonely equivalent to the original flow $u_t$. Moreover, $u_{\alpha,t}$ also preserves a measure $\mu^{\alpha}$ which is equivalent to $\mu$ and satisfies $d\mu^{\alpha}=\alpha d\mu/(\int\alpha d\mu)$. It can be shown that  $(X_1,\mathcal{B}_1,\mu_1,u_t^{(1)})$ is monotone equivalent to $(X_2,\mathcal{B}_2,\mu_2,u_t^{(2)})$ iff there is an $\alpha\in L_+^1(X_1,\mathcal{B}_1,\mu_1)$ such that the time changed system $(X_1,\mathcal{B}_1,\mu_1^{\alpha},u_{\alpha,t}^{(1)})$ is measurably isomorphic to $(X_2,\mathcal{B}_2,\mu_2,u_t^{(2)})$.

In the case that $\int\alpha d\mu=1$, we say $\psi$ is an \textbf{even monotone (Kakutani) equivalence}, which in particular implies that
\[
\psi(u_t^{(1)}.x)=u_{t+o_x(t)}^{(2)}.\psi(x),\qquad \text{a.s. as }t\to\infty.
\]
Note that if we omit the condition that $\alpha$ is integrable\footnote{To make sense, we then need to allow a time change to both flows, and of course we can no longer normalize the resulting measures.} the situation changes dramatically and becomes less interesting, as was shown by Ornstein and Weiss in \cite{OrnsteinWeiss84}.

Let $\psi:(X_1,\mathcal{B}_1,\mu_1,u_t^{(1)})\to(X_2,\mathcal{B}_2,\mu_2,u_t^{(2)})$ be a monotone equivalence. Then we may associate to $\psi$ a cocycle $\tau(x,t)$ on $(X_1,\mathcal{B}_1,\mu_1,u_t^{(1)})$ so that the following holds:
\[\psi(u_t^{(1)}.x)=u_{\tau(x,t)}^{(2)}.\psi(x).\] 
Two cocycles $\tau$ and $\tau'$ on $X_1$ are said to be \textbf{cohomologous} if there exists a measurable (but not necessarily integrable) $w:X_1\to\R$ such that
\[
\tau'(x,t)=\tau(x,t)+w(u_t^{(1)}.x)-w(x).
\]
If $\psi$ and $\psi'$ are two monotone equivalences between $(X_1,\mathcal{B}_1,\mu_1,u_t^{(1)})$ and $(X_2,\mathcal{B}_2,\mu_2,u_t^{(2)})$ we say that they are cohomologous if the corresponding cocycles are.

\medskip

A special feature of unipotent flows on quotients of semisimple groups that is an immediate corollary of the Jacobson-Morozov theorem \cite{Knapp96Lie}*{Thm. 10.3} is that any monotone equivalence between these flows can be easily modified by translating it using a diagonal element normalizing $u^{(2)}_t$ to be an even monotone equivalence. Indeed, if $u^{(2)}_t = \exp(t\bu_2)$, we can complete $\bu_2$ to an $\mathfrak{sl}_2$ triplet $\bu_2,\ba_2,\bou_2$ and apply $\exp(t\ba_2)$ to $\psi$ to fix the average speed of the equivalence.

 This observation allows us to state our main theorem only for even equivalences:
\begin{Theorem}\label{thm:main}
For $i=1,2$, let $\mathbf{G}_i$ be the identity component of a real semisimple linear algebraic group without compact factors, $\Gamma_i<\mathbf{G}_i$ a lattice, \ $m_i$ the probability measure on $\mathbf{G}_i/\Gamma_i$ induced by Haar measure on $\mathbf{G}_i$, \ $u_t^{(i)}$ a one-parameter unipotent subgroup of $\mathbf{G}_i$, and $\mathfrak{g}_i$ the Lie algebra of $\mathbf{G}_i$. Assume for $i=1,2$, the group $u_t^{(i)}$ acts ergodically on $(\mathbf{G}_i/\Gamma_i,m_i)$ and that $\mathbf{G}_i$ acts faithfully on $\mathbf{G}_i/\Gamma_i$. Then $(\mathbf{G}_1/\Gamma_1,m_1,u_t^{(1)})$ and $(\mathbf{G}_2/\Gamma_2,m_2,u_t^{(2)})$ are monotone equivalent if and only if one of the following holds
\begin{enumerate}[label=\textup{(A\arabic*)}]
    \item \label{i:LB} for both $i=1,2$, we have that $\mathfrak{g}_i=\mathfrak{sl}_2(\R)\oplus \mathfrak{g}'_i$ and the generator of $u_t^{(i)}$ is of the form 
    \[                     \begin{pmatrix}
                       0 & 1 \\
                       0 & 0 
                     \end{pmatrix}
                  \times 0 \in \mathfrak{g}_i;\]
   \item There exist an isomorphism $\phi:\mathbf{G}_1\to \mathbf{G}_2$ and $c\in \mathbf{G}_2$ such that $\phi(\Gamma_1)=\Gamma_2$ and $\phi(u^{(1)}_t) = c u^{(2)}_t c^{-1}$. Moreover, any even Kakutani equivalence between these systems is cohomologous to an actual isomorphism.
\end{enumerate}
\end{Theorem}

While there is no monotone equivalence rigidity for horocycle flows as they are all loosely Kronecker, Ratner in \cites{ratner1986rigidity,ratner1987rigid} showed that if the time change function $\alpha$ satisfies some regularity assumptions, e.g. $C^1$ or H\"{o}lder along certain directions, then a monotone equivalence obtained this way must arise from an isomorphism --- i.e.\ we have monotone equivalence rigidity for H\"{o}lder time changes. 

Recently, there have been several works extending Ratner's results  monotone equivalence rigidity for H\"{o}lder time changes. One motivation for these works is that time changes of unipotent flows form an interesting class of non homogeneous parabolic flows with interesting properties. For example, Marcus
proved in \cite{marcus1977ergodic} that smooth time-changes are mixing, and the rate
of mixing was studied by Forni and Ulcigrai in \cite{forni2012time}.

By extending the methods of \cite{ratner1986rigidity}, Tang \cite{tang2020new} obtained monotone equivalence rigidity for H\"{o}lder time changes for one-parameter unipotent subgroup actions on $\operatorname{SO}(n,1)$ with cocompact lattice. Following Tang's work, Artigiani, Flaminio and Ravotti \cite{afr22} extended Ratner's result to the general situation, under similar regularity assumptions for time change, and an additional non-convergence assumption on graph joinings.  In the case that the unipotent flow is loosely Kronecker, \cite{afr22} give a rigidity result similar to \cite{ratner1986rigidity} under a regularity assumption for the time change, but without the non-convergence assumption on graph joinings.

\subsection*{Acknowledgement}
The authors would like to thank Benjamin Weiss for encouragements and helpful discussions. D.W. would also like to thank Svetlana Katok for her encouragements. We thank and Adam Kanigowski, Akshay Venkatesh, Kurt Vinhage and Andreas Wieser for helpful discussions and comments. In particular, suggestions by Akshay Venkatesh allowed us to give a much cleaner (and shorter) proof of the results in Appendix \ref{sec:sl2ergodic}.

\smallskip

This paper is dedicated to Benjamin Weiss, with admiration and gratitude. Weiss is a true giant in ergodic theory, and his wide ranging and deep contribution to the subject during his ongoing long and productive career is an inspiration to us.

\section{Overview of the argument}
Our starting point is a key definition introduced by Feldman to study Kakutani equivalence --- the $\bar f$-metric, or more precisely a closely related notion for flows we call $(\delta, \epsilon, R)$-two-sided matching. If $u _ t$ is a flow on some space $X$, we say that $x,x' \in X$ are \emph{$(\delta, \epsilon, R)$-two-sides} matchable if after restricting to a subset of $[-R,R]$ of density $\geq (1-\epsilon)$, one can find a continuous map $h$ with derivative $\in (1-\epsilon, 1+ \epsilon)$ so that $u_t x$ and $u_{h(t)}.x'$ are within $\delta$ of each other for every $t$ in this set (see Definition~\ref{def:twoSidesMatching}).

The importance of this notion to study monotone equivalence stems from the fact that if we are given two flows
$u _ t ^ {(i)}$ on two spaces $X _ i$ (for $i=1,2$), and if $\psi: X _ 1 \to X _ 2$ is an even Kakutani equivalence, then on a compact set $K \subset X_1$ of large measure if $x,y \in X_1$ are $(\delta, \epsilon, R)$-two-sides matchable then $\psi(x),\psi(y)$ are $(\delta',20\epsilon,R)$-two sides matchable ($\delta'$ can be as small as we want if $\delta$ is chosen appropriately).

The orbits of two nearby points $x,y$ in either $X_1$ or $X_2$ initially do not diverge much so can be $(\delta, \epsilon, R)$-matched; we think of $\epsilon$ and $\delta$ as small but fixed and $R$ as large (measuring the `quality' of the match). Consider for $\mathbf G= \mathbf G_1$ or $\mathbf G_2$
\[
\operatorname{Bow}(R,\delta)=\{g\in \mathbf{G},d_\mathbf{G}(u_{t}gu_{-t},e)<\delta\textup{ for }\forall t\in[-R,R]\}
\]
Let $\mathfrak h = \mathfrak h_i$ ($i=1$ or $2$) be the sub-Lie algebra of $\operatorname{Lie} \mathbf{G_i}$ spanned by an $\mathfrak{sl}_2$-triplet containing $\mathbf u_i$ (say $\bu_i$, $\ba_i$ and $\bou_i$), and finally let $a_t=a_t^{(i)} = \exp (t\ba_i)$ and $\widehat{u}_s=\widehat{u}_s^{(i)} =\exp (s\bou_i)$.
If $y \in \operatorname{Bow}(R,\delta).x$ then $x$ and $y$ can be $(\delta, \epsilon, R)$ matched for any $\epsilon>0$ even without any time change. Moreover this is true also for the lifts of $x$ and $y$ to~$\mathbf G$. Allowing time change allows to accommodate deviation between $x$ and $y$ in the  $\exp(\mathfrak{h})$ direction significantly beyond what is contained in $\operatorname{Bow}(R,\delta)$. This allows us to define a larger subset of $\mathbf G$ we call Kakutani-Bowen balls $\Kak(R,\epsilon)$ so that if we have two nearby point $x,y$ with $y \in \Kak(R,\delta).x$ the two points can be lifted to points in $\mathbf G$ that are $(10\delta, \delta, R)$-matched. 
In  $\mathbf G$, this is essentially an if and only if condition: up to minor details in the exact choice of parameters, two nearby points $g,g' \in \mathbf G$ are $(\delta, \epsilon, R)$-matched iff $g' \in \Kak(\delta',R).g$. But in $\mathbf G/\Gamma$ two points may well be matchable even if the lifts to $\mathbf G$ are not.
The definition of Kakutani-Bowen balls is taken from \cite{kanigowski2021kakutani}, and in essence is used already in \cite{ratner1979cartesian}.

Our main lemma (Lemma~\ref{lem:longHammingNew}) shows\footnote{For technical reasons Lemma~\ref{lem:longHammingNew} requires that two point be $(\delta, \epsilon, T)$-two sided matchable for every $T$ between some fixed $R_0$ and $R$; we ignore this issue in this description.} that when $u _ t ^ {(2)}$ acting on $\mathbf G _ 2 / \Gamma _ 2$ is not loosely Kronecker, there is a compact subset $\Kold{036Klh1} \subset \mathbf G _ 2 / \Gamma _ 2$ of arbitrarily large measure so that if $x \in \mathbf G _ 2 / \Gamma _ 2$ and $y \in \Kold{036Klh1}$ are $(\delta, \epsilon, R)$-two sided matchable for sufficiently small $\delta, \epsilon$ then after shifting $y$ by a suitable element of controlled size in $\mathbf{U}_2$, say $u_{t_2}^{(2)}$, the points $x$ and $u_{t_2}^{(2)}y$ satisfy that
$x \in \Kak(R,C\delta).u_{t_2}^{(2)}y$. This lemma is essentially a sharper form of \cite{kanigowski2021kakutani}*{Theorem~6.1}. Combining this with the fact that (when restricted to a set of large measure) Kakutani equivalences map $(\delta', \epsilon', R)$-matchable points to $(\delta, \epsilon, R)$-matchable points one obtains (cf.~Lemma~\ref{lem:main}) that there is a compact set $\Kold{035Klm1} \subset \mathbf G _ 1 / \Gamma _ 1$ of large measure so that for all large enough $R$, fixed but small suitably chosen $\epsilon, \delta>0$,
\begin{equation}\label{eq: from Lemma 8.1}
\psi \left(\Kak\left(R,\epsilon,y\right) \cap \Kold{035Klm1}\right) \subset \Kak\left(R,\delta,\psi(y)\right).
\end{equation}
It is important that both $\Kold{035Klm1},\Kold{036Klh1}$ do not depend on $R$.

We next observe that for $t$ very small, the points $x, a _ t ^ {( 1 )} x$ are $(\delta, \epsilon, R)$-two-sided matchable for \emph{every} $R > 0$. It follows that if $x, a _ t ^ {(1 )}x$ are both in $\Kold{035Klm1}$, then $\psi (a _ t ^ {(1 )}.x) \in u_{t(R)}^{(2)}\Kak\left(R,\delta\right)x$ for all $R$, with $t(R)$ possibly dependent on $R$ but is uniformly bounded. This is a strong algebraic restriction that implies that there is some small $f(x) \in R$ and $c(x)$ in the centralizer $C_{\mathbf G_2}(\mathbf U_2)$ of $\mathbf U_2$ such that
\begin{equation}\label{eq:psi of a in intro}
\psi(a_{t}^{(1)}.x)=c(x) a_{f(x)}^{(2)}.\psi(x)
\end{equation}
(cf. Lemma~\ref{lem:Geolm1}).

This is still not precise enough for our purposes. For tiny $s$, if one chooses appropriately $h(t)$, then $u^{(1)}_t x$ and $u^{(1)}_{h(t)}\widehat u_s^{(1)}x$ diverge in the direction of $a^{(1)}_{\bullet}$. We use this to show that the $c(x)$ in \eqref{eq:psi of a in intro} has no component in the centralizer of $\mathfrak h$ --- i.e. $c(x)$ is in the intersection of $C_{\mathbf G_2}(\mathbf U_2)$ with the subgroup $\mathbf{G}_2^+<\mathbf G_2$ of those elements $g$ satisfying
\[
a^{(2)}_t g a^{(2)}_{-t} \to 1\qquad\text{as $t \to -\infty$}. 
\]
Essentially this is because if we can get $x$ and $a_t^{(1)}.x$ by adjusting time on trajectories of extremely nearby points the same must hold for $\psi(x)$ and $\psi(a_t^{(1)}.x)$.

The fact that $\psi$ is a monotone equivalence for the action of $u^{(i)}_t$ implies that $f(x)=f(u^{(1)}_r.x)$ and that the image of $c(x)$ in $C_{\mathbf G_2}(\mathbf U_2)/\mathbf U_2$ equals to that of $c(u^{(1)}_r.x)$. 
By ergodicity of the $u^{(1)}_r$ flow this means that $f(x)$ and the image of $c(x)$ in $C_{\mathbf G_2}(\mathbf U_2)/\mathbf U_2$ is constant almost everywhere. If $\psi$ is an even Kakutani equivalence it is not hard to show that $f(x)=t$.
We conclude (cf. Proposition~\ref{prop:CompatibleGeoMain}) that there exists $c\in C_{\mathbf{G}_2}(\mathbf U_2) \cap \mathbf{G}_2^+$ such that if we set $\tilde{\psi}=c\psi$ then for~$m_1$-a.e.$~x\in \mathbf{G}_1/\Gamma_1$
\begin{equation}\label{eq:psi of a in intro 2}
\tilde{\psi}(a_1^{(1)}.x)\in a_1^{(2)}\mathbf{U}_2.\tilde{\psi}(x).
\end{equation}
We may as well assume $\psi$ itself satisfies \eqref{eq:psi of a in intro 2}.

\medskip

Now comes the last, and perhaps most delicate part of the argument --- a renormalization procedure. We consider the sequence of maps
    \[
    \psi_n(x)=a_{-n}^{(2)}\psi(a_n^{(1)}.x).
    \]
    By \eqref{eq:psi of a in intro 2}, $\psi_n(x) \in \mathbf{U}_2.\psi(x)$. If $\psi$ is an even Kakutani equivalence, and if it were known that $\psi_n(x)$ converges to some measurable
    \[
    \phi:\mathbf{G}_1/\Gamma_1\to \mathbf{G}_2/\Gamma_2
    \]
    then it is not hard to show that $\phi$ will necessarily be a \emph{measure preserving isomorphism} between the two systems $(\mathbf{G}_1/\Gamma_1,\mathcal{B}_1,m_1,u_t^{(1)})$ and $(\mathbf{G}_2/\Gamma_2,\mathcal{B}_2,m_2,u_t^{(2)})$ which were completely classified by Ratner. 

\medskip

    Define $t_x\in\R$ so that 
    \[
    \psi(a_1^{(1)}.x)=a_1^{(2)}.u_{t_x}^{(2)}.\psi(x);
    \]
    this can be done by \eqref{eq:psi of a in intro 2}.
    Iterating one gets that for any $n$
    \[
    \psi(a_n^{(1)}.x)=a_n^{(2)}.u_{t_{n,x}}^{(2)}.\psi(x)
    \]
    with $t_{n,x}=\sum_{j=0}^{n-1} e^{-2j} \,t_{a^{(1)}_j\!.x}$. Thus if we could show the $t_{n,x}$ tend to a limit as $n \to \infty$ we would be done. If $t_{x}$ was bounded, then this would be obvious, but unfortunately the only thing we know about $t_x$ is that it is measurable.

    The renormalization strategy was successfully employed by Ratner in \cite{ratner1986rigidity}, where she proved time change rigidity for monotone equivalences $\psi$ satisfies some regularity assumptions. Note that in this case time change rigidity does \emph{not} hold without this regularity assumption; not surprisingly, her argument why $\psi_n$ converge to a limit heavily used the regularity of the time change function. 

    \medskip
    
    We give a different argument, based on an $L^2$-ergodic theory like result we prove in Appendix~\ref{sec:sl2ergodic}. While we cannot guarantee for a given point $x$ that $a^{(1)}_j\!.x$ lies in a good set (say, the set where $t_\bullet$ is bounded) for all $j$, we can guarantee that \emph{for \textbf{every} $j$ there is a point $x'$ close to~$x$} so that $a^{(1)}_j\!.x'$ is in a good set. Thus we obtain some control on the $\psi_n$-value of the aggregate of all points in a small neighbourhood of $x$, and when $a^{(1)}_j\!.x$ lies in a good set this says something about $\psi_n(a^{(1)}_j\!.x)$.
    
    This allows us to show that for almost every $x$, the sequence $\psi_n(x)$ converges to some limit $\psi(x)$ along a subsequence of \emph{full density}, which is sufficient for our purposes. This part of the argument is given in \S\ref{sec:renormalization}.

\subsection{Reader's guide}
The structure of our paper is as follows:
\begin{itemize}[label={---}]
    \item In \S\ref{sec:notationsPre}, we discuss some preliminaries: two sides matching, homogeneous spaces, Kakutani-Bowen balls and some basic matching lemmas.
    \item In \S\ref{sec:convertKakutani}, we prove some easy initial initial reductions, in particular showing we may assume the given monotone equivalence is an even Kakutani equivalence that is well behaved in the flow direction.
    \item In \S\ref{sec:polynomialDivergence}, we analyze the polynomial way in which one-parameter unipotent flows divergence. A combinatorial lemma related to orbit divergence that will be used in \S\ref{sec:mainLemma} is also provided.
    \item In \S\ref{sec:MesureEstimatesExcept}, we use Remez-type estimates (cf.~Theorem~\ref{thm:brudnyiLocal}) and Chevalley represenations (cf.~Theorem~\ref{thm:Chevalley}) to control measures of exceptional sets.
    \item In \S\ref{sec:normalcoreNBHD}, we study neighborhoods of the normal core of a special real semisimple algebraic group $\mathbf{L}$ (cf.~\eqref{eq:defN}). The main result of this section is Lemma~\ref{lem:HNint}, which will help us in \S\ref{sec:longHammingSep} and \S\ref{sec:renormalization} to verify the assumptions of Corollary~\ref{cor:globalRepEst} which will be needed to complete the estimates for the measure of the exceptional sets.
    \item In \S\ref{sec:mainLemma} and \S\ref{sec:longHammingSep}, we prove what we call the Main Lemma (Lemma~\ref{lem:main}), that gives us the initial stepping stone to taming arbitrary Kakutani equivalences. By using the combinatorial Lemma~\ref{lem:longHammingComb} and Algorithm~\ref{alg:cap}, we reduce the Main Lemma to Lemma~\ref{lem:longHammingSep}, which forms the second part of the proof of Lemma~\ref{lem:main}. 
    The proof of Lemma~\ref{lem:longHammingSep} relies on using Chevalley representations, Corollary~\ref{cor:globalRepEst} and Lemma~\ref{lem:HNint}.
    \item In \S\ref{sec:comgeo}, we study the compatibility of the diagonalizable subgroups $a_t^{(i)}=\exp(t\ba_i)$ that correspond to the acting unipotent one parameter groups with the given even Kakutani equivalence. 
    The main tool used in this section is the main lemma (Lemma~\ref{lem:main}).
    \item In \S\ref{sec:renormalization}, we study the limit behavior of $\psi_n(x)=a_{-n}^{(2)}\psi(a_n^{(1)}.x)$ as $n$ goes to infinity. 
    The existence of this limit is the second key ingredient in our proof, and in order to do that we shall need to rely on a new ergodic theorem for $\SL_2(\R)$-actions (Corollary~\ref{cor:sl2ergodic}).
    \item In \S\ref{sec:proofofMain}, we complete proof by collecting all our previous lemmas and applying Ratner's and Witte's isomorphic rigidity theorems for unipotent flows.
    \item In Appendix~\S\ref{sec:appSmoothM}, we provide for completeness a proof of Lemma~\ref{lem:smoothMatching}, which enables us to use smooth matching function when defining two sided matchings. This is a modification of an argument from \cite{kanigowskiPreHighkakutani}
    \item In Appendix~\S\ref{sec:appGoodTime}, we prove Lemma~\ref{lem:goodTimeChange}, which is a general result about Kakutani equivalences, stating that by replacing a given even Kakutani equivalence by  a different one in the same cohomology class, we can have a time change function that has almost constant derivative and is smooth in the flow direction.
    \item In Appendix~\S\ref{sec:sl2ergodic}, we prove a result on unitary representations of $\SL_2(\R)$ with a spectral gap that implies a kind of pointwise ergodic theorem needed for our proof in \S\ref{sec:renormalization}. This new ergodic theorem seems to be of independent interest. 
    \end{itemize}

    \medskip

    The following graph explains the relation between these sections.
    \begin{figure}[H]
 \centering
 \scalebox{0.6}
 {
 \begin{tikzpicture}[scale=5]
	 \tikzstyle{vertex}=[circle,minimum size=20pt,inner sep=0pt]
	 \tikzstyle{selected vertex} = [vertex, fill=red!24]
	 \tikzstyle{edge1} = [draw,line width=5pt,-,red!50]
     \tikzstyle{edge2} = [draw,line width=5pt,-,green!50]
     \tikzstyle{edge3} = [draw,line width=5pt,-,blue!50]
     \tikzstyle{edge4} = [draw,line width=5pt,-,brown!50]

	 \tikzstyle{edge} = [draw,thick,-,black]
      \node[vertex] (v00) at (0,0.5) {\S\ref{sec:notationsPre}};
      \node[vertex] (v10) at (0.05,0.5) {};
      
	 \node[vertex] (v01) at (0.75,1) {\S\ref{sec:convertKakutani}};
      \node[vertex] (v11) at (0.7,1) {};
      \node[vertex] (v21) at (0.8,1) {};
      
	 \node[vertex] (v02) at (0.75,0.66) {\S\ref{sec:polynomialDivergence}};
      \node[vertex] (v12) at (0.7,0.66) {};
      \node[vertex] (v22) at (0.8,0.66) {};
      
	 \node[vertex] (v03) at (0.75,0.33) {\S\ref{sec:MesureEstimatesExcept}};
      \node[vertex] (v13) at (0.7,0.33) {};
      \node[vertex] (v23) at (0.8,0.33) {};
      
      \node[vertex] (v04) at (0.75,0) {\S\ref{sec:normalcoreNBHD}};
      \node[vertex] (v14) at (0.7,0.0) {};
      \node[vertex] (v24) at (0.8,0) {};
      
	 \node[vertex] (v05) at (1.5,0.75) {\S\ref{sec:mainLemma}};
      \node[vertex] (v15) at (1.45,0.75) {};
      \node[vertex] (v25) at (1.55,0.75) {};
      \node[vertex] (v05a) at (1.48,0.75) {};
      \node[vertex] (v05b) at (1.52,0.75) {};
      
      \node[vertex] (v06) at (1.5,0.2) {\S\ref{sec:longHammingSep}};
      \node[vertex] (v16) at (1.45,0.2) {};
      \node[vertex] (v26) at (1.55,0.2) {};
      \node[vertex] (v06a) at (1.48,0.2) {};
      \node[vertex] (v06b) at (1.52,0.2) {};
      
      \node[vertex] (v07) at (2.0,0.4) {\S\ref{sec:comgeo}};
      \node[vertex] (v17) at (1.95,0.4) {};
      \node[vertex] (v27) at (2.05,0.4) {};
      
      \node[vertex] (v08) at (2.5,0.5) {\S\ref{sec:renormalization}};
      \node[vertex] (v18) at (2.45,0.5) {};
      \node[vertex] (v28) at (2.55,0.5) {};
      
      \node[vertex] (v09) at (3.1,0.5) {\S\ref{sec:proofofMain}};
      \node[vertex] (v19) at (3.05,0.5) {};

      \node[vertex] (v010) at (0,0.0) {\S\ref{sec:appSmoothM}};

      \node[vertex] (v011) at (0.2,1) {\S\ref{sec:appGoodTime}};
      \node[vertex] (v211) at (0.25,1) {};

      \node[vertex] (v012) at (2.5,1) {\S\ref{sec:sl2ergodic}};

	 \draw[-stealth] (v10)--(v11);
      \draw[-stealth] (v10)--(v12);
      \draw[-stealth] (v10)--(v13);
      \draw[-stealth] (v10)--(v14);
      \draw[-stealth] (v21)--(v15);
      \draw[-stealth] (v22)--(v15);
      \draw[-stealth] (v23)--(v16);
      \draw[-stealth] (v24)--(v16);
      %\draw[stealth-] (v06a)--(v05a);
      \draw[-stealth] (v06b)--(v05b);
      \draw[-stealth] (v25)--(v17);
      \draw[-stealth] (v25)--(v18);
      \draw[-stealth] (v27)--(v18);
      \draw[-stealth] (v28)--(v19);
      \draw[-stealth] (v010)--(v00);
      \draw[-stealth] (v211)--(v11);
      \draw[-stealth] (v012)--(v08);

 \end{tikzpicture}
 }
\caption{Structure of paper}\label{fig:StructureMap}
\end{figure}
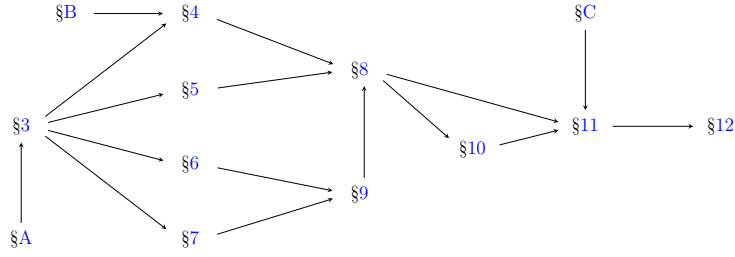

\subsection{Global and local notations}
In order to make our paper more readable, we adopt the following notation policy to define \emph{global} and \emph{local} notations:
\begin{itemize}
    \item The global notations are uniquely defined in the whole paper. 
    \item The local notations are only uniquely defined in each subsubsection (e.g. \S2.2.1) or proof. 
\end{itemize}

\medskip

Here is a table for global and local notations, where $*\in\N$.
\begin{table}[H]
    \centering
    \begin{tabular}{c|c|c}
    \hline
                       & Global notations  & Local notations \\
    \hline
    Full measure sets &  $\mathcal{Z}_*$  & $Z_*$ \\
    \hline
    Open sets & $O_*$ & \\
    \hline
    Compact sets & $\mathcal{K}_*$ & $K_*$ \\
    \hline
    Constants & $C_*$, $c_*$, $w_*$, $R_*$  & $\kappa_*$, $N_*$ \\
    \hline
    Small constants & $\epsilon_*$, $\delta_*$ & $\eta$\\
    \hline
    \end{tabular}
\end{table}

\section{Notations and preliminaries}\label{sec:notationsPre}
\subsection{\texorpdfstring{$(\delta,\epsilon,R)$}--two sides matching}\label{sec:LKmatching}
We recall one basic definition in Kakutani equivalence: $(\epsilon,\eta,R)$-two sides matching. For more details and further background we refer the reader to \cites{feldman1976new,katok1977monotone,ornstein1982equivalence,ratner1981some}.

Let $l(\cdot)$ denotes the Lebesgue measure on $\R$ and consider for $R>1$ the orbit segment
\[I_R(x)=\{T_sx:s\in[-R,R]\}.\]

\begin{Definition}[cf.~{\cite{ratner1981some}*{Def.~1}}]\label{def:twoSidesMatching}
Let $T_t$ be a flow on $(X,\mathcal{B},\mu)$ and $d$ a metric on $X$. Suppose $\delta,\epsilon\in(0,1)$ and $R>1$. We say $x,y\in X$ are \textbf{$(\delta,\epsilon,R)$-two sides matchable} if there exist a closed subset $A=A_{(x,y)}\subset[-R,R]$ with $l(A)>(1-\epsilon)2R$, and an increasing absolutely continuous map $h=h_{(x,y)}:A \to A'=A'_{(x,y)}\subset[-R,R]$ with $l(A')>(1-\epsilon)2R$, such that
\begin{enumerate}
    \item $d(T_{h(t)}x,T_ty)<\delta$ for all $t\in A$;
    \item $h(0)=0$ and $0\in A$;
    \item $\absolute{h'(t)-1}<\epsilon$ for all $t\in A$.
\end{enumerate}
We call $h$ an \textbf{$(\delta,\epsilon,R)$-two sided matching} from $I_R(x)$ to $I_R(y)$. 
\end{Definition}

The notion of $(\delta,\epsilon,R)$-two sided matching is an adaptation to flows of the $\bar{f}$-distance used by J.~Feldman\cite{feldman1976new} and (under a different name) by A.~Katok\cite{katok1977monotone} in their pioneering works.

\medskip

In fact, we can modify the matching function of any two sides matching to a $C^{\infty}$-matching function based on the following lemma, which is a modification of Proposition 4.3.2 in \cite{kanigowskiPreHighkakutani}. 

\begin{Lemma}[Smooth matching function lemma]\label{lem:smoothMatching}
There is an absolute constant $C$ so that the following holds: Let $T_t$ be a flow on $(X,\mathcal{B},\mu,d)$, and let $\epsilon,\delta\in(0,10^{-10}C^{-2})$ and $R>100$. Then for any $x,y\in X$ that are $(\delta,\epsilon,R)$-two sides matchable, there exists an $C^{\infty}$ function $\tilde{h}:[-R,R]\to[-R,R]$ and a closed subset $\tilde{A}\subset[-R,R]$ with 
\[
l(\tilde{A})>(1-C\sqrt{\epsilon})2R
\]
satisfying
\begin{enumerate}
    \item $d(T_{\tilde{h}(t)x},T_{t}y)<2\delta$ for all $t\in \tilde{A}$;
    \item $0\in\tilde{A}$ and $\tilde{h}(0)=0$;
    \item $\absolute{\tilde{h}'(t)-1}<C\sqrt{\epsilon}$ for all $t\in[-R,R]$.
\end{enumerate}
\end{Lemma}

For completeness, we provide a proof of this lemma in Appendix \ref{sec:appSmoothM}.

\medskip

As a result of this lemma, we assume from now on that all our $(\delta,\epsilon,R)$-two sided matchings are $C^1$-smooth.

\subsection{Preliminaries of homogeneous space} 
In this section we recall some standard notations and facts about homogeneous spaces. 

\subsubsection{Algebraic groups}
In this paper, an \textbf{affine variety} defined over $\R$ is a subset $Z$ of $\C^n$ that defined by the vanishing of polynomial equations with real coefficients. A variety $\G\subset\SL_N$ defined over $\R$ is a linear algebraic subgroup defined over $\R$ if $\G(\C)\subset\SL_N(\C)$ is a subgroup. For a $d$-dimensional vector space $V$ over $\C$, we use $\SL(V)$ to denote the group of special linear automorphisms of $V$. A real algebraic representation $\rho:\G\to\SL(V)$ is a homomorphism such that the matrix elements $\rho(g)_{ij}$ are polynomials in the $g_{ij}$ with real  coefficients for $1\leq i,j\leq d$. A \textbf{real algebraic group} is the set of real points of $\G$, i.e. the set defined by $\G(\R)=\G(\C)\cap\SL_N(\R)$. Throughout the paper we denote by $\mathbf{G}$ the identity component (in the Archimedean topology) of $\G(\R)$. We recall the following:
\begin{Lemma}\label{lem:centerDiscrete}
Let $\mathbf{G}$ be the identity component of a real semisimple algebraic group. Then $C(\mathbf{G})$ is finite. Moreover, any discrete normal subgroup of $\mathbf{G}$ is contained in $C(\mathbf{G})$ and is also finite.
\end{Lemma}

Another way to define  linear algebraic group is using Chevalley's theorem.

\begin{Theorem}\label{thm:Chevalley}
If $\mathbb{H}$ is a linear algebraic subgroup of $\G$ defined over $\R$, then there exists a real algebraic representation $\rho:\G\to\SL(V)$, and a vector $v\in V\cap\R^{\dim V}$ such that $\mathbb{H}(\R)=\{g\in \G(\R):\rho(g)v\in\R v\}$. Moreover, if $\mathbb{H(\C)}$ is generated by unipotent element, then 
\[
\mathbb{H}(\R)=\{g\in \G(\R):\rho(g).v=v\}.
\]
\end{Theorem}

For more details, see \cite{MargulisDiscrete1991}*{\S1, p.17}.

\subsubsection{Metric, measure and adjoint action}\label{sec:basis}
Let $\mathbf{G}$ be the identity component of a real semisimple algebraic group with Lie algebra $\mathfrak{g}$. Fix an inner product $<\cdot,\cdot>$ on $\mathfrak{g}$. This inner product induces a right invariant Riemannian metric $d_\mathbf{G}$. If $\Gamma$ is a lattice in $\mathbf{G}$, the right invariant Riemannian metric on $\mathbf{G}$ induces a Riemannian metric on $\mathbf{G}/\Gamma$, which we will denote by $d_{\mathbf{G}/\Gamma}$. The corresponding Riemannian volume on $\mathbf{G}$ gives a right Haar measure $\tilde{m}$ on $\mathbf{G}$, and also determines a probability measure on $m$ homogeneous space $\mathbf{G}/\Gamma$. Throughout this paper, we use~$\norm{\cdot}$ to denote the Hilbert Schmidt norm\footnote{A.k.a. the Frobenius norm.} of $g\in \mathbf{G}$ by viewing $g$ as an element of $\SL_d(\R)$ for some $d\in\N$. The following is standard:

\begin{Lemma}\label{lem:matrixNorm}
Let $\mathbf{G}<\operatorname{SL}_N(\R)$ be the identity component of a real semisimple algebraic group, then there exist $\denew\label{007mnC1}\in(0,1)$\index{$\deold{007mnC1}$, Lemma~\ref{lem:matrixNorm}} and $\Cnew\label{007mnC2}>1$\index{$\Cold{007mnC2}$, Lemma~\ref{lem:matrixNorm}} such that for any $g\in \mathbf{G}$, $d_\mathbf{G}(g,e)<R$ implies that $\norm{g}<\Cold{007mnC2}e^{R/\deold{007mnC1}}$.
\end{Lemma}

We leave the proof to the reader.

\medskip

Let $\mathbf{L}$ be a subgroup of $\mathbf{G}$, $\mathfrak{l}$ a subalgebra of $\mathfrak{g}$ and $\delta>0$, we define 
\begin{equation}\label{eq:subgroupBall}
 B_{\delta}^{\mathbf{L}}=\{l\in L:d_G(l,e)<\delta\},\ \ \ \  B_{\delta}^{\mathfrak{l}}=\{x\in\mathfrak{l}:d_{\mathfrak{g}}(x,0)<\delta\},   
\end{equation}
where $d_{\mathfrak{g}}$ is the Euclidean metric on $\mathfrak{g}$ corresponding to our chosen inner product.

Let $C_g$ denote the conjugation map  $C_g(h)=g^{-1}hg$ for $g,h\in \mathbf{G}$. For $X,Y\in\mathfrak{g}$, we let $\operatorname{Ad}_g(Y)=g^{-1}Yg$ denote the derivative of $C_g$ at the identity $e$ and $\operatorname{ad}_X(Y)=-[X,Y]$. Recall that 
\begin{equation}\label{eq:adjoint}
\begin{aligned}
C_{\exp(X)}(\exp(Y))&=\exp(\operatorname{Ad}_{\exp(X)}Y),\\
\exp(\operatorname{ad}_X)&=\operatorname{Ad}_{\exp(X)}
\end{aligned}
\end{equation}
with $\exp$ denoting the exponential map on $\mathfrak g$ and $\operatorname{Hom}(\mathfrak g)$ as appropriate.

 Let  $\mathfrak{g}=\oplus_{i=1}^n\mathfrak{e}_i$ be a vector space decomposition of $\mathfrak{g}$, where $\mathfrak{e}_i$ need not be subalgebras or commute with each other. Then by the implicit function theorem
\begin{Lemma}\label{lem:groupDecom}
There exists $\enew\label{006ecom1}>0$\index{$\eold{006ecom1}$, Lemma~\ref{lem:groupDecom}} 
such that if $g\in \mathbf{G}$ and  $d_{\mathbf{G}}(g,e)<\eold{006ecom1}$, then there exist unique $X_i\in\mathfrak{e}_i$ close to $0$ such that 
\[
g=\exp(A_1)\exp(A_2)\ldots\exp(X_n).
\]
\end{Lemma}

\noindent
Cf.~[Lemma 3.8 in \cite{kanigowski2021kakutani}] for details.

\medskip

\noindent
Notice $[\mathbf{x},\mathbf{x}]=[\mathbf{y},\mathbf{y}]=0$, then by Baker--Campbell--Hausdorff formula 
\begin{Lemma}\label{lem:comExp}
There exist $\enew\label{008eComExp}\in(0,1)$\index{$\eold{008eComExp}$, Lemma~\ref{lem:comExp}} and $\Cnew\label{008CComExp}>1$\index{$\Cold{008CComExp}$, Lemma~\ref{lem:comExp}}  such that for any $\mathbf{x},\mathbf{y}\in\mathfrak{g}_2$ satisfying  $\norm{\mathbf{x}},\norm{\mathbf{y}}<\eold{008eComExp}$
\[
\norm{\log(\exp(\mathbf{x})\exp(\mathbf{y}))-(\mathbf{x}+\mathbf{y})}<\Cold{008CComExp}\norm{\mathbf{x}}\norm{\mathbf{y}}.
\]
\end{Lemma}

\subsubsection{$\mathfrak{sl}_2$-triplet and chain basis}\label{sec:sl2basis}
Given a nonzero nilpotent element $\bu\in\mathfrak{g}$, we obtain from the Jacobson-Morozov theorem \cite{jacobson1979lie} that there exists a homomorphism $\varphi:\mathfrak{sl}_2(\R)\to\mathfrak{g}$ such that 
\[
\varphi\left(\begin{pmatrix}
    0 & 1 \\
    0 & 0 \\
\end{pmatrix}
  \right)=\bu, \ \ \varphi\left(\begin{pmatrix}
    1 & 0 \\
    0 & -1 \\
  \end{pmatrix}\right)=\ba, \ \ \varphi\left(\begin{pmatrix}
    0 & 0 \\
    1 & 0 \\
  \end{pmatrix}\right)=\bou,
\]
where $\ba\in\mathfrak{g}$ is a $\mathbb{R}$-diagonalizable element and $\bou\in\mathfrak{g}$ is a nilpotent element.

The following one-parameter subgroups will be used extensively in this paper:
\begin{itemize}
    \item $\mathbf{U}=\{u_t\}_{t\in \R}$ --- the one-parameter \textbf{unipotent} subgroup induced by $\varphi$ with $u_t=\exp(t\bu)$;
    \item $\mathbf{A}=\{a_t\}_{t \in \R}$  --- the corresponding one-parameter \textbf{diagonalizable} subgroup $a_t=\exp(t\ba)$; 
    \item $\OU=\{\ou_t\}_{t\in \R}$ --- the \textbf{opposite} one-parameter unipotent subgroup with $\ou_t=\exp(t\bou)$.
\end{itemize}
Let $\mathfrak{h}=\operatorname{span}_{\R}\{\bu,\ba,\bou\}$, $\mathbf{H}$ the connected subgroup of $\mathbf{G}$ (in the Hausdorff topology) generated by $\exp(\mathfrak{h})$ and $m_\mathbf{H}$ the Haar measure on $\mathbf{H}$. Note that the homomorphism $\varphi$ from $\mathfrak{sl}_2(\R)$ to $\mathfrak h$ induces a morphism of algebraic groups $\tilde\varphi$ from $\SL_2$ to the algebraic group $\mathbb G$ so that $\mathbf H$ is the image of $\SL_2(\R)$ under $\tilde\varphi$.

Recall that $\mathbf{G}$ is a semisimple Lie group with Lie algebra $\mathfrak{g}$. There exist simple Lie algebras $\mathfrak{g}_i$ for $i=1,\ldots,\ell$ such that $\mathfrak{g}=\oplus_{i=1}^{\ell}\mathfrak{g}_i$. Let $p_i:\mathfrak{g}\to\mathfrak{g}_i$ be the canonical projection, then \[p_i\circ\ad:\mathfrak{sl}_2(\R)\to\operatorname{End}(\mathfrak{g}_i)\] with $X\mapsto p_i\circ\ad_{\varphi(X)}$ gives a representation of $\mathfrak{sl}_2(\R)$. 

Since $\mathfrak{sl}_2(\R)$ is semisimple, $p_i\circ\ad$ splits into sum of finite dimensional irreducible representations, which are determined up to isomorphism by their dimension. Recall that the irreducible $d+1$-dimensional representation of $\mathfrak{sl}_2(\R)$, \ $(\pi_d,V_{d})$ has a basis $\bx^0,\ldots,\bx^{d}$, so that
\begin{equation}\label{eq:sl2relation}
    \begin{aligned}
    \pi_d(\bu)\bx^k&=\bx^{k-1},\\
    \pi_d(\ba)\bx^k&=(2k-d)\bx^k,\\
    \pi_d(\bou)\bx^k&=a_{d,k}\bx^{k+1},
    \end{aligned}
\end{equation}
for all $0\leq k\leq d+1$, where $a_{d,k}$ are nonzero constants; for notational convenience we adopt the convention that $\bx^{-1}=\bx^{d+2}=0$. Since $p_i\circ\ad$ is a direct sum of irreducible representations of $\mathfrak{sl}_2(\R)$, we obtain a basis of $\mathfrak{g}$ by collecting all such bases for $\mathfrak{g}_i$ together\footnote{Pick one $i_0\in\N$ such that $p_{i_0}(\bu)\neq0$, then we replace the basis for $\operatorname{span}_{\R}\{p_{i_0}(\bu),p_{i_0}(\ba),p_{i_0}(\bou)\}$ by $\bou$, $\bx$, $\bu$.}  and call it as the \textbf{chain basis for $\mathfrak{g}$ with respect to $\mathfrak{h}$}:
\begin{equation}\label{eq:chainbasis}
\{\bu,\ba,\bou,\bx^{0,1}\ldots,\bx^{q_1,1},\ldots,\bx^{0,n},\ldots,\bx^{q_n,n}\}.
\end{equation}
The sequence of numbers $(q_1,\ldots,q_n)$ is denoted the \textbf{chain structure indices of $\mathfrak{g}$ with respect to $\mathfrak{h}$}. 

\medskip

Given a subset $W \subset \mathfrak{g}$, we define $C_{\mathfrak{g}}(W)$ as the common centralizer of all $\bw \in W$, i.e. 
\begin{equation*}
C_{\mathfrak{g}}(W)=\{\by\in\mathfrak{g}:\operatorname{ad}_{\bw}(\by)=0 \ \forall\bw \in W\}.
\end{equation*}
In particular, we obtain that $C_{\mathfrak{g}}(\bu)$ is generated by the elements \[\bu,\bx^{0,1},\ldots,\bx^{0,n}\] of the chain basis \eqref{eq:chainbasis}. The complement to $\bu\R$ in $C_{\mathfrak{g}}(\bu)$ is defined as
\begin{equation}\label{eq:wSpace}
    \mathfrak{w}=\bigoplus_{j=1}^{n}\bx^{0,j}\R.
\end{equation}
In general $\mathfrak{w}$ is not a Lie algebra.

\medskip

Thorough the language of chain basis, we decompose Lie algebra $\mathfrak{g}$ as the direct sum of the following three Lie subalgebras:
    \begin{alignat*}{4}
        &\mathfrak{g}^+&&=&&\bu\R&&\oplus\{\bigoplus\bx^{i,j}\R:[\ba,\bx^{i,j}]=\lambda\bx^{i,j}\text{ for some }\lambda>0\},\\
        &\mathfrak{g}^0&&=&&\ba\R&&\oplus\{\bigoplus\bx^{i,j}\R:[\ba,\bx^{i,j}]=0\},\\
        &\mathfrak{g}^\shortminus&&=&&\bou\R&&\oplus\{\bigoplus\bx^{i,j}\R:[\ba,\bx^{i,j}]=\lambda\bx^{i,j}\text{ for some }\lambda<0\}.
    \end{alignat*}
Then define
\begin{equation}\label{eq:G+-0}
    \begin{aligned}
        \mathbf{G}^\shortminus=\exp(\mathfrak{g}^\shortminus),\quad \mathbf{G}^0=\exp(\mathfrak{g}^0),\quad \mathbf{G}^+=\exp(\mathfrak{g}^+).
    \end{aligned}
\end{equation}

\subsubsection{Injective radius and matching function}\label{sec:fDomain}

For any $x\in \mathbf{G}$, the injectivity radius of $x$ w.r.t. a lattice $\Gamma<\mathbf{G}$ is denoted as $\operatorname{inj}(x)$:
\[
\operatorname{inj}(x)=\sup\{r\geq0:B_\mathbf{G}(x,r)\cap B_\mathbf{G}(x,r)\gamma=\emptyset\text{ for all }\Gamma\ni\gamma\neq e\}.
\]
Suppose $K\subset \mathbf{G}$, then we define:
\[
\operatorname{inj}(K)=\inf_{x\in K}\operatorname{inj}(x)=\inf_{\gamma\in\Gamma\setminus\{e\}}\inf_{y\in K}d_\mathbf{G}(y\gamma y^{-1},e).
\]

\medskip

The following lemma will be used to study orbits of nearby point in $\mathbf{G}/\Gamma$ under $\mathbf{U}$. To find the best match between these two orbits one needs to apply a different element of $\mathbf{U}$ to the two points:

\begin{Lemma}\label{lem:matchingFunction}
Suppose
\[
w=\ou_{s}a_{p}\in \mathbf{G}.
\]
Then
\begin{equation}\label{eq: 2by2matrices calulation}
u_{\phi(t)}wu_{-t}=\ou_{s_t}a_{p_t},
\end{equation}
with 
\begin{equation*}
\begin{aligned}
\phi(t)&=\frac{te^{p}}{e^{-p}-tse^{p}},\\
s_t&=se^{p}(e^{-p}-se^{p}t),\\
p_t&=-\log(e^{-p}-se^{p}t),
\end{aligned} 
\end{equation*}
whenever the denominator in the expression for $\phi$ does not vanish.

There exists $\enew\label{014EMa}>0$\index{$\eold{014EMa}$, Lemma~\ref{lem:matchingFunction}} 
such that if
$\absolute{p}<\eold{014EMa}$, then for every $\absolute{t}\leq\eold{014EMa}\absolute{s}^{-1}$, we have that $
\absolute{s_t}<2\absolute{s}$ and $\absolute{p_t}<2(\absolute{p}+\absolute{t}\absolute{s})$.
Moreover, if for $\epsilon\in(0,\eold{014EMa}]$ we have $\absolute{p}<\epsilon$, then for $\absolute{t}\in[0,\epsilon\absolute{s}^{-1}]$
\[
\absolute{\phi'(t)-1}<\sqrt{\epsilon}.
\]
\end{Lemma}

This important observation is used in Ratner's work, and can be viewed as a special case of what Ratner calls the H-property (see \cite{ratner1983horocycle}*{Def.~1}). Equation~\eqref{eq: 2by2matrices calulation} can be directly verified by multiplying $2\times2$ matrices, and the subsequent inequalities are straightforward. Cf.~\cite{kanigowski2021kakutani}*{Lemma 5.3} for more details.

\subsubsection{Kakutani-Bowen balls}\label{sec:Kakutanibowen}
Let $\bu$ be a nilpotent element of $\mathfrak{g}$, then there exists a chain basis of $\mathfrak{g}$ by \S\ref{sec:sl2basis}:
\[
\{\bu,\ba,\bou,\bx^{0,1},\ldots,\bx^{q_1,1},\ldots,\bx^{0,n},\ldots,\bx^{q_n,n}\},
\]
where $\{\bu,\ba,\bou\}$ is $\mathfrak{sl}_2$-triplet and $\bx^{i,j}$ satisfying \eqref{eq:sl2relation}. Recall that $\bu,\ba,\bou$ span the Lie algebra $\mathfrak{h}$ of $\mathbf{H}$. Define
\begin{equation}\label{eq:Ic}
I_c=\{1\leq j\leq n:q_j=0\},\qquad \mathbf{Z}=C_{\mathbf{G}}(\mathbf{H});
\end{equation}
then the Lie algebra of Lie group $\mathbf{Z}$ can be characterized  as
\[
\mathfrak{z} = \bigoplus_{j \in I_c}\bx^{0,j} \R.
\]
For simplicity, we define $\mathfrak{l}$ as the direct sum of $\mathfrak{h}$ and $\mathfrak{z}$,
\[
\mathfrak{l}=\mathfrak{h}\oplus\mathfrak{z}.
\]

\medskip

By Lemma~\ref{lem:groupDecom}, for $g\in G$ sufficiently close to $e$
\begin{equation}\label{eq:gDecom}
\begin{aligned}
    g=g^{\mathfrak{h}}g^{\mathfrak{z}}g^{\tr},
\end{aligned}
\end{equation}
where $g^{\mathfrak{h}}$, $g^{\mathfrak{z}}$ and $g^{\tr}$ are defined as following:
\begin{equation}\label{eq:gDecom1}
    \begin{gathered}
    g^{\mathfrak{h}}=u_{\vartheta_{\bu}(g)}\ou_{\vartheta_{\bou}(g)}a_{\vartheta_{\ba}(g)},\\
    g^{\mathfrak{z}}=\exp(\sum_{j\in I_c}\vartheta_{0,j}(g)\bx^{0,j}),\qquad 
     g^{\tr}=\exp(\sum_{j\notin I_c}\sum_{i=0}^{q_j}\vartheta_{i,j}(g)\bx^{i,j}).
    \end{gathered}   
\end{equation}
We also introduce $B_{r_1,r_2}^{\mathfrak l \oplus \tr}(x)$ for $x\in \mathbf{G}/\Gamma$ and $r_1\geq r_2>0$ as following:
\begin{equation}\label{eq:bch}
    \begin{aligned}
        B_{r_1,r_2}^{\mathfrak l \oplus \tr}(x)=\{y\in \mathbf{G}/\Gamma:y=g.x, \ \  g^{\mathfrak{h}},g^{\mathfrak{z}}\in B_{r_1}^{\mathbf{G}},\ \ g^{\tr}\in B_{r_2}^{\mathbf{G}}\}.
    \end{aligned}
\end{equation}

The following definition will be essential for our analysis of Kakutani equivalent map:
\begin{Definition}[Kakutani-Bowen balls]\label{def:Kakball}
For $\epsilon>0$ and $R>1$, let 
\[
\operatorname{Bow}(R,\epsilon)=\{g\in \mathbf{G},d_\mathbf{G}(u_{t}gu_{-t},e)<\epsilon\textup{ for }\forall t\in[-R,R]\}
\]
be the Bowen ball of $e\in \mathbf{G}$ for $u_t$. Suppose $x,y\in \mathbf{G}/\Gamma$, we say $x\in\Kak(R,\epsilon,y)$ if and only if $x$ can be written as $gy$ for $g\in \mathbf{G}$ satisfying the following:
\begin{enumerate}[label=(\roman*)]
    \item $\absolute{\vartheta_{\bou}(g)}<\frac{\epsilon}{R}$,
    \item $\absolute{\vartheta_{\ba}(g)}<\epsilon$,
    \item $\absolute{\vartheta_{\bu}(g)}<\epsilon$,
    \item\label{item:Bowen} $g^{\mathfrak{z}}g^{\tr}\in\operatorname{Bow}(R,\epsilon)$,
\end{enumerate}
where $\vartheta_{\bou}$, $\vartheta_{\ba}$, $\vartheta_{\bu}$ and $g^{\mathfrak{z}}g^{\tr}$ are defined in \eqref{eq:gDecom}. Moreover, for $\tilde{x},\tilde{y}\in \mathbf{G}$, we say $\tilde{x}\in\Kak(R,\epsilon,\tilde{y})$ if and only if $\tilde{x}=g\tilde{y}$ with $g$ satisfying $(1),(2),(3),(4)$ above.

\end{Definition}

Definition \ref{def:Kakball} together with Lemma~\ref{lem:matchingFunction} and the right invariance of~$d_G$ imply for any $g\in\Kak(R,\epsilon)$
\[
\text{$d_\mathbf{G}(u_{\phi_g(t)}.g,u_{t})<10\epsilon$ for every $\absolute{t}\leq R$,}
\]
where $\phi_g(t)$ is the \textbf{best matching function for $g$} defined as
\begin{equation}\label{eq:bestMatchingFun}
\phi_g(t)=te^{\vartheta_{\ba}(g)}/(e^{-\vartheta_{\ba}(g)}-t\vartheta_{\bou}(g)e^{\vartheta_{\ba}(g)}).
\end{equation}

\noindent
The significance of Definition \ref{def:Kakball} is explained by the following propositions:

\begin{Proposition}\label{prop:KakutaniBallStayClose}
There exists constant $\denew\label{012deKakStayClose}>0$\index{$\deold{012deKakStayClose}$, Proposition~\ref{prop:KakutaniBallStayClose}} such that the following holds. Let $\tilde{y},\tilde{x}=g\tilde{y}\in\mathbf{G}$ be two points with $g\in\Kak(R,\delta)$ for $R\geq1$ and $\delta\in(0,\deold{012deKakStayClose}]$, then for every $\absolute{t}\leq R$
\[
d_{\mathbf{G}}(u_{\phi_g(t)}\tilde{x},u_t\tilde{y})\leq10\delta,
\]
where $\phi_g(t)$ is the best matching function defined in \eqref{eq:bestMatchingFun}.
\end{Proposition}
\begin{proof}[Proof of Proposition~\ref{prop:KakutaniBallStayClose}]
Let $\deold{012deKakStayClose}=\eold{014EMa}$, where $\eold{014EMa}$ is from Lemma~\ref{lem:matchingFunction}.

By right invariance of $d_{\mathbf{G}}$ and the decomposition in  \eqref{eq:gDecom}, for $\absolute{t}\leq R$
\begin{equation*}
    \begin{aligned}
        d_{\mathbf{G}}(u_{\phi_g(t)}\tilde{x},u_t\tilde{y})=&d_{\mathbf{G}}(u_{\phi_g(t)}gu_{-t},e)=d_{\mathbf{G}}(u_{\phi_g(t)}g^{\mathfrak{h}}g^{\mathfrak{z}}g^{\tr}u_{-t},e)\\
        =&d_{\mathbf{G}}(u_{\phi_g(t)}g^{\mathfrak{h}}u_{-t}u_tg^{\mathfrak{z}}g^{\tr}u_{-t},e)\\
        =&d_{\mathbf{G}}(u_{\phi_g(t)}g^{\mathfrak{h}}u_{-t},(u_tg^{\mathfrak{z}}g^{\tr}u_{-t})^{-1})\\
        \leq&d_{\mathbf{G}}(u_{\phi_g(t)}g^{\mathfrak{h}}u_{-t}, e)+d_{\mathbf{G}}(e,(u_tg^{\mathfrak{z}}g^{\tr}u_{-t})^{-1})\\
        =&d_{\mathbf{G}}(u_{\phi_g(t)}g^{\mathfrak{h}}u_{-t}, e)+d_{\mathbf{G}}(u_tg^{\mathfrak{z}}g^{\tr}u_{-t},e).
    \end{aligned}
\end{equation*}
We estimate the two terms in the last line of the above displayed equation as follows:
\begin{itemize}
    \item By Lemma~\ref{lem:matchingFunction}, the first term is bounded by $9\delta$. 
    \item By item~\ref{item:Bowen} of Definition~\ref{def:Kakball}, the second term is bounded by $\delta$. 
\end{itemize}
Together, these imply the estimate given by Proposition~\ref{prop:KakutaniBallStayClose}.
\end{proof}

\medskip

\noindent
The opposite of Proposition~\ref{prop:KakutaniBallStayClose} also holds, here $\eold{014EMa}$ is from Lemma~\ref{lem:matchingFunction}:
\begin{Proposition}\label{prop:matchKak}
There exists a constant $\Cnew\label{012matchKak}>1$\index{$\Cold{012matchKak}$, Proposition~\ref{prop:matchKak}} such that the following holds. Let $\tilde{y},\tilde{x}=g\tilde y\in \mathbf{G}$ be two points with $d_\mathbf{G}(\tilde{x},\tilde{y})<\epsilon$, \  $\epsilon\in(0,\eold{014EMa})$, and let $[a,b]\subset[0,\epsilon\absolute{\vartheta_{\bou}(g)^{-1}}]$ be an interval. Assume
\begin{equation}\label{eq:intervalClose}
    d_\mathbf{G}(u_{\phi_g(t)}\tilde{x},u_{t}\tilde{y})<\epsilon\qquad\text{for all $t\in[a,b]$}
\end{equation}
where $\phi_g(t)$ is the best matching function defined in \eqref{eq:bestMatchingFun} for $g$. Then for every $t\in[a,b]$
\[
u_{\phi_g(t)}\tilde{x}\in\Kak(b-a,\Cold{012matchKak}\epsilon)u_{t}\tilde{y}.
\]
\end{Proposition}
\begin{proof}[Proof of Proposition~\ref{prop:matchKak}]
Since $[a,b]\subset[0,\epsilon\absolute{\vartheta_{\bou}(g)^{-1}}]$, Lemma~\ref{lem:matchingFunction} gives for every $t\in[a,b]$:
\begin{equation*}
    u_{\phi_g(t)}\ou_{\vartheta_{\bou}(g)}a_{\vartheta_{\ba}(g)}u_{-t}
    =\ou_{\vartheta_{\bou,t}(g)}a_{\vartheta_{\ba,t}(g)},\nonumber
\end{equation*}
where 
\[
\absolute{\vartheta_{\bou,t}(g)}<2\absolute{\vartheta_{\bou}(g)}, \ \ \ \ \absolute{\vartheta_{\ba,t}(g)}<4\epsilon.
\]
This and inequality \eqref{eq:intervalClose} imply that for every $t\in[a,b]$
\begin{equation*}
    \begin{aligned}
        d_\mathbf{G}(u_{t}g^{\mathfrak{z}}g^{\tr},u_{t})<\kappa_1\epsilon,
    \end{aligned}
\end{equation*}
where $g^{\mathfrak{z}}g^{\tr}$ is defined in \eqref{eq:gDecom1} and $\kappa_1>0$. Together with the definition of Kakutani-Bowen ball and polynomial divergence, we complete the proof.
\end{proof}

\subsubsection{Auxiliary lemmas for Kakutani-Bowen balls}
The following are some controlling lemmas related to Kakutani-Bowen balls.
\begin{Lemma}\label{lem:diangonalConjugateKak}
For any $R>1,\delta>0$, if $g\in\Kak(R,\delta)$ and $p=\frac{1}{2}\log R$, then for 
\[
\bar{g}=a_{-p}ga_{p},
\]
its coordinates (cf.~\eqref{eq:gDecom} and \eqref{eq:gDecom1}) satisfy $g\in B^{\mathfrak l \oplus \tr}_{4\delta,\:\delta\dim\mathfrak{g}/R^{1/2}}$, i.e.
\begin{equation*}
\begin{alignedat}{3}
\bar{g}^{\mathfrak{h}},\bar{g}^{\mathfrak{z}}\in B_{4\delta}^{\mathbf{G}},\quad 
\bar{g}^{\tr}\in B_{\delta\dim\mathfrak{g}/R^{1/2}}^{\mathbf{G}}.
\end{alignedat}
\end{equation*}
\end{Lemma}
\begin{proof}[Proof of Lemma~\ref{lem:diangonalConjugateKak}]
Since $g\in\Kak(R,\delta)$, we can write 
\begin{equation}\label{eq:gKakDecom}
g=u_{\vartheta_{\bu}(g)}\ou_{\vartheta_{\bou}(g)}a_{\vartheta_{\ba}(g)}g^{\mathfrak{z}}g^{\tr},
\end{equation}
where $\vartheta_{\bu}(g)$, $\vartheta_{\bou}(g)$, $\vartheta_{\ba}(g)$, $g^{\mathfrak{z}}$ and $ g^{\tr}$ satisfying
\begin{equation}\label{eq:middleFormConjuagte}
    \begin{alignedat}{3}
        &\absolute{\vartheta_{\bu}(g)},\absolute{\vartheta_{\ba}(g)}\leq\delta,\quad &&\absolute{\vartheta_{\bou}(g)}\leq\frac{\delta}{R},\quad 
        g^{\mathfrak{z}}g^{\tr}\in\Bow(R,\delta).
    \end{alignedat}
\end{equation}

We obtain from $\mathfrak{sl}_2$-relation, \eqref{eq:gKakDecom} and \eqref{eq:middleFormConjuagte} 
\begin{equation}\label{eq:sl2ConjugateSize}
    a_{-p}g^{\mathfrak{h}}a_p\in B_{3\delta}^{\mathbf{H}}.
\end{equation}
Recall that $g^{\mathfrak{z}}$ commutes $\mathbf{H}$, then
\begin{equation}\label{eq:centerConjugateSize}
a_{-p}g^{\mathfrak{z}}a_{p}=g^{\mathfrak{z}}.    
\end{equation}

On the one hand,  $g^{\tr}=\exp(\sum_{j\notin I_c}\sum_{i=0}^{q_j}\vartheta_{i,j}(g)\bx^{i,j})$. This and   $g^{\mathfrak{z}}g^{\tr}\in\Bow(R,\delta)$ imply that
$\absolute{\vartheta_{i,j}(g)}\leq\frac{\delta}{R^{i}}$. On the other hand, 
\[
\quad a_{-p}g^{\tr}a_{p}=\exp(\sum_{j\notin I_c}\sum_{i=0}^{q_j}e^{(2i-q_j)p}\vartheta_{i,j}(g)\bx^{i,j}).
\]
This and the estimate on $\vartheta_{i,j}(g)$ give
\begin{equation*}
\absolute{\vartheta_{i,j}(\bar{g})}=\absolute{e^{(2i-q_j)p}\vartheta_{i,j}(g)}\leq \frac{\delta}{R^{\frac{1}{2}q_j}}.   
\end{equation*}
We complete the proof by combining above with  \eqref{eq:sl2ConjugateSize} and \eqref{eq:centerConjugateSize}.
\end{proof}

\begin{Lemma}\label{lem:zlConvertDirect}%\todo{needs fixing. done.}
There exist $\wnew\label{014wExpBound},\denew\label{014dCommuteBound},\Rnew\label{014RCommuteBound}>0$\index{$\wold{014wExpBound}$, Lemma~\ref{lem:zlConvertDirect}}\index{$\deold{014dCommuteBound}$, Lemma~\ref{lem:zlConvertDirect}}\index{$\Rold{014RCommuteBound}$, Lemma~\ref{lem:zlConvertDirect}} depending on $\mathbf{G}$ such that for every $\delta\in(0,\deold{014dCommuteBound})$, $R\geq\Rold{014RCommuteBound}$, $\absolute{r}\leq R^{\wold{014wExpBound}}$ the following holds. Suppose $g\in\mathbf{G}$ is of the form
\[
g=g^{\mathfrak{h}}g^{\mathfrak{z}}g^{\tr}
\]
with $g^{\mathfrak{h}}\in B_{\delta}^{\mathbf{H}}$, $g^{\mathfrak{z}}\in B_{\delta}^{\mathbf{Z}}$ and  $g^{\tr}\in B_{\delta/R}^{\mathbf{G}}$. 

Then we can write $u_{r}.g$ as
\[
u_{r}g = z_1g^{\mathfrak{z}}u_{r}g^{\mathfrak{h}}
\]
with $z_1\in B_{\delta/R^{0.99}}^{\mathbf{G}}$. Similarly
\[
u_{r}g^{-1} = z_2(g^{\mathfrak{z}})^{-1}u_{r}(g^{\mathfrak{h}})^{-1}
\]
with $z_2\in B_{\delta/R^{0.99}}^{\mathbf{G}}$.
\end{Lemma}
\begin{proof}[Proof of Lemma~\ref{lem:zlConvertDirect}]
Since the proof of two equations are quite similar, we only prove the first equation.

Notice that $g^{\mathfrak{z}}\in B_{\delta}^{\mathbf{Z}}$ implies that $g^{\mathfrak{z}}$ commutes with $\mathbf{H}$ and thus
\[
u_rg=u_{r}g^{\mathfrak{h}}g^{\mathfrak{z}}g^{\tr}= g^{\mathfrak{z}}u_{r}g^{\mathfrak{h}}g^{\tr}.
\]

By picking $\Rold{014RCommuteBound}$ sufficiently large,  $\wold{014wExpBound}$ and $\deold{014dCommuteBound}$ sufficiently small, \eqref{eq:sl2relation} and Lemma~\ref{lem:comExp} imply 
\[
\tilde{g}=g^{\mathfrak{h}}g^{\tr}(g^{\mathfrak{h}})^{-1}=\exp(\sum_{j\notin I_c}\sum_{i=0}^{q_j}\vartheta_{i,j}(\tilde{g})\bx^{i,j}),\qquad \absolute{\vartheta_{i,j}(\tilde{g})}\leq\delta/R^{0.999i}.
\]

By chain basis construction in \S\ref{sec:sl2basis} and $\wold{014wExpBound}$ is sufficiently small
\begin{equation*}
    \begin{aligned}
      u_{r}\tilde{g}u_{-r}=\exp(\sum_{j\notin I_c}\sum_{i=0}^{q_j}(\sum_{k=0}^{q_j-i}\frac{r^k\vartheta_{i+k,j}(\tilde{g})}{k!})\bx^{i,j})\in B^{\mathbf{G}}_{\delta/R^{0.998}}.
    \end{aligned}
\end{equation*}

Applying Lemma~\ref{lem:comExp} again, 
\[
 g^{\mathfrak{z}}u_{r}\tilde{g}u_{-r}(g^{\mathfrak{z}})^{-1}\in B_{\delta/R^{0.99}}^{\mathbf{G}}.
\]
By picking $z_1=g^{\mathfrak{z}}u_{r}\tilde{g}u_{-r}( g^{\mathfrak{z}})^{-1}$, we finish the proof.
\end{proof}

With the help of chain basis and $I_c$, it is possible to classify when the action of $u_t$ on $\mathbf{G}/\Gamma$ is a loosely Kronecker system:
\begin{Lemma}[\cite{kanigowski2021kakutani}]\label{lem:chainLK}
Let $\mathfrak{l}$ be Lie algebra generated by $\bu,\bou,\ba$ and $\bx^{0,j}$, where $j\in I_c$ as in \eqref{eq:Ic}. Suppose $u_t$ a one-parameter unipotent subgroup acting ergodically on $\mathbf{G}/\Gamma$ induced by $\bu\in\mathfrak{g}$, then $u_t$ is loosely Kronecker if and only if $\mathfrak{l}=\mathfrak{g}$.
\end{Lemma}
\begin{proof}[Proof of Lemma~\ref{lem:chainLK}]
If $\mathfrak{l}=\mathfrak{g}$, then we have $\sum_{j=1}^n\frac{q_j(q_j+1)}{2}=3$ and \cite{kanigowski2021kakutani}*{Theorem 1.1} gives that the $(\mathbf{G}/\Gamma,m,u_t)$ is loosely Kronecker. 

If the $(\mathbf{G}/\Gamma,m,u_t)$ is loosely Kronecker, then \cite{kanigowski2021kakutani}*{Corollary 1.2} gives that $\bu=\left(\begin{smallmatrix}
     0 & 1 \\
   0 & 0 \\
\end{smallmatrix}\right)\oplus\id$ and $\mathfrak{g}$ is isomorphic to $\mathfrak{sl}_2(\R)\oplus\mathfrak{g}'$, which implies that $\mathfrak{l}=\mathfrak{g}$.
\end{proof}

\subsubsection{An \texorpdfstring{$\SL_2(\mathbb{R})$}{SL2(R)}-ergodic theorem}\label{sec:sl2Ergodiccorollary}
\begin{Corollary}\label{cor:sl2ergodic}
Let $\mathbf{G}_1,\Gamma_1,m_1, \mathbf{H}_1,a^{(1)}_t$ be as in~\S\ref{subsec:assupmtions}.
Given $\epsilon\in(0,1)$, then for any $f\in L^2(\mathbf{G}_1/\Gamma_1,m_1)$ and $m_1$-a.e. $x\in \mathbf{G}_1/\Gamma_1$
\[
\lim_{n\to+\infty}\frac{1}{m_{\mathbf{H}_1}(B_{\epsilon}^{\mathbf{H}_1,\norm{\cdot}})}\int_{B_{\epsilon}^{\mathbf{H}_1,\norm{\cdot}}}f(a^{(1)}_{n}h.x)dm_{\mathbf{H}_1}(h)=\int_{\mathbf{G}_1/\Gamma_1}fdm_1.
\]
\end{Corollary}
\begin{proof}[Proof of Corollary~\ref{cor:sl2ergodic}]
Let $\mathcal{H}_{\rho}=L_0^2(\mathbf{G}_1/\Gamma_1,m_1)$, be the space of square integrable functions with mean zero, equipped with the action $\rho$ given by $\Bigl(\rho(g)(f)\Bigr)(x)=f(g^{-1}x)$.
Then $\tilde{f}=f-\int_{\mathbf{G}_1/\Gamma_1}fdm \in \mathcal{H}_{\rho}$. Note that by e.g.\ \cites{kleinbock1999log,KelSar09} if  $\mathbf{G}_1$ is semisimple without compact factors (as we assume) then the representation $\rho$ on $\mathcal{H}_{\rho}$ has a spectral gap.

Apply Theorem~\ref{thm:NewNewsl2ergodic} with function $\chi:\SL_2(\R)\to \R$ satisfying 
\[
\chi(h)=\frac{1}{m_{\mathbf{H}_1}(B_{\epsilon}^{\mathbf{H_1},\norm{\cdot}})}\chi_{B_{\epsilon}^{\mathbf{H_1},\norm{\cdot}}}(h)\text{ for }h\in\mathbf{H}_1.
\]
We conclude that for every $k\in\N$
\begin{equation}\label{eq:normalConvergence}
\sum_{n=1}^{+\infty}\norm{\int_{\mathbf{H}_1}\chi(h)\tilde{f}(a^{(1)}_{n}h.x)dm_{\mathbf{H}_1}(h)}^{2}_{2}<+\infty.
\end{equation}
Let $g_{n}(x)$ be defined by 
\[
g_{n}(x)=\int_{\mathbf{H}_1}\chi(h)\tilde{f}(a^{(1)}_{n}h.x)dm_{\mathbf{H}_1}(h).
\]
Then by \eqref{eq:normalConvergence} 
\[
\sum_{n=1}^{+\infty}\norm{g_{n}(x)}^{2}_{2}<+\infty.
\]
It follows that by the monotone convergence theorem
\[
\int_{\mathbf{G}_1/\Gamma_1}\sum_{n=1}^{+\infty}\absolute{g_{n}(x)}^{2}dm_1=\sum_{n=1}^{+\infty}\norm{g_{n}(x)}^2_{2}<\infty,
\]
which implies that $\sum_{n=1}^{+\infty}\absolute{g_{n}(x)}^{2}<\infty$ for $m_1$-a.e. $x\in \mathbf{G}_1/\Gamma_1$. Thus $\lim_{n\to+\infty}g_{n}(x)=0$ for $m_1$-a.e. $x\in \mathbf{G}_1/\Gamma_1$.
\end{proof}

\section{Kakutani equivalence in homogeneous spaces}\label{sec:convertKakutani}
The purpose of this section is to show that without loss of generality, we can assume that a given an arbitrary Kakutani equivalence between two actions of one-parameter unipotent subgroups on quotients of a semisimple algebraic group is an even Kakutani equivalence with a time change function that behaves nicely in the flow direction\footnote{Note however that as highlighted in the introduction, it is \textbf{not} permissible to assume the time change behaves nicely in any transverse direction.}.

Generally speaking, even Kakutani equivalence is more special than Kakutani equivalence since the former one requires additional condition on time change function. But for the one-parameter unipotent subgroups we consider, they essentially coincide due to the following lemma:
\begin{Lemma}\label{lem:evenKakutani}
For $i=1,2$, let $\mathbf{G}_i$ be a real semisimple algebraic group, $\Gamma_i$ a lattice in $\mathbf{G}_i$, with $m_i$ the probability measure on $\mathbf{G}_i/\Gamma_i$ induced by the Haar measure on $\mathbf{G}_i$, and let  $u_t^{(i)}$ be a one-parameter unipotent subgroup acting ergodically on $(\mathbf{G}_i/\Gamma_i,m_i)$. If $\psi$ is a Kakutani equivalence between the corresponding actions of $u_t^{(1)}$ and $u_t^{(2)}$, then there is a unique $s\in\R$ such that $a_s^{(2)}.\psi$ is an even Kakutani equivalence, where (for $i=1,2$)  $a_s^{(i)}$ is a diagonalizable group arising from a $\mathfrak{sl}_2$-triplet for the generator of $u_t^{(i)}$ as in \S\ref{sec:sl2basis}.
\end{Lemma}
\begin{proof}[Proof of Lemma~\ref{lem:evenKakutani}]
Since $u_t^{(1)}$ is Kakutani equivalent to $u_t^{(2)}$, there exist an $\alpha\in L_1^+(\mathbf{G}_1/\Gamma_1,m_1)$ and an invertible map $\psi:\mathbf{G}_1/\Gamma_1\to \mathbf{G}_2/\Gamma_2$ such that $\psi_*m_1^{\alpha}=m_2$ and for every $t\in\R$ and $m_1$-a.e. $x\in \mathbf{G}_1/\Gamma_1$
\begin{equation}\label{eq:geodesicEven}
    \psi(u_{\rho(x,t)}^{(1)}.x)=u_t^{(2)}.\psi(x),
\end{equation}
where $\rho(x,t)$ satisfies 
\begin{equation}\label{eq:timeChange}
\int_0^{\rho(x,t)}\alpha(u_p^{(1)}.x)dp=t.
\end{equation}

Since $\int_{\mathbf{G}_1/\Gamma_1}\alpha dm_1>0$, there exists $s\in\R$ such that
\[
e^{2s}\int_{\mathbf{G}_1/\Gamma_1}\alpha dm_1=1.
\]
Applying $a_s^{(2)}$ from left on both sides of \eqref{eq:geodesicEven}
\begin{equation}\label{eq:newTimeChange}
a_s^{(2)}.\psi(u_{\rho(x,t)}^{(1)}.x)=u_{e^{2s}t}^{(2)} a_s^{(2)}.\psi(x).
\end{equation}
Let $\bar{\alpha}=e^{2s}\alpha$ and $\bar{\rho}(x,t)$ satisfy $\int_0^{\bar{\rho}(x,t)}\bar{\alpha}(u^{(1)}_p.x)dp=t$. This, together with \eqref{eq:timeChange}, guarantees that $\bar{\rho}(x,e^{2s}t)=\rho(x,t)$, and thus \eqref{eq:newTimeChange} implies that $u_{\bar{\alpha},t}^{(1)}$ and $u_t^{(2)}$ are isomorphic through $a_s^{(2)}.\psi$. By our choice of~$s$, \[\int_{\mathbf{G}_1/\Gamma_1}\bar{\alpha}dm_1=1.\] 
Note also that $a_s^{(2)}$ preserves $m_2$ and $m_1^{\bar{\alpha}}=m_1^{\alpha}$. This implies that $a_s^{(2)}\circ\psi$ is indeed an even Kakutani equivalence.
\end{proof}

\begin{Definition}\label{def:controlKakutani}
Let $T_t$ and $S_t$ be two ergodic flows acting on $(X,\mathcal{B},\mu)$ and $(Y,\mathcal{C},\nu)$ respectively. For any $\epsilon\in(0,10^{-3})$, we say $\psi$ is an \textbf{$\epsilon$-well-behaved Kakutani equivalence} between $T_t$ and $S_t$ if the following holds:
\begin{enumerate}
    \item $\psi$ is an even Kakutani equivalence between $T_t$ and $S_t$ with time change function $\alpha$;
    \item $\operatorname{esssup}\absolute{\alpha-1}<\epsilon$;
    \item there exists a full measure set $Z\subset X$ such that for every $x\in Z$, $\alpha(T_tx)$ is a $C^{\infty}$ function in $t$ and $\absolute{\alpha(T_tx)-1}<\epsilon$ for every $t\in\R$.
\end{enumerate}
\end{Definition}

The following lemma shows that any even Kakutani equivalence is cohomologous to a well-behaved Kakutani equivalence.  

\begin{Lemma}[Improving behaviour of even Kakutani equivalences]\label{lem:goodTimeChange}
Let $T_t$ and $S_t$ be two ergodic flows acting on $(X,\mathcal{B},\mu)$ and $(Y,\mathcal{C},\nu)$ respectively. If $\epsilon\in(0,10^{-3})$ and $\psi$ is an even Kakutani equivalence between $T_t$ and $S_t$, then there exists an $\epsilon$-well-behaved Kakutani equivalence $\tilde{\psi}$ that is cohomologous to $\psi$.
\end{Lemma}
The proof of this lemma follows from combing {\cite{katok1977monotone}*{Proposition 2.3}} and {\cite{ornstein1982equivalence}*{Theorem 1.4}}.
We provide a proof in Appendix \ref{sec:appGoodTime} for completeness.

\medskip

If we have a ``good'' time change function as in Lemma~\ref{lem:goodTimeChange}, the following lemma shows that two sided matchings will be preserved under the corresponding Kakutani equivalence map $\psi$.

Recall that $\psi\circ T_{\rho(x,t)}(x)=S_t\circ\psi(x)$ for every $t\in\R$ and $\mu$-a.e. $x\in X$, where $\rho(x,t)$ is determined by 
\begin{equation}\label{eq:rhoalpha}
\int_0^{\rho(x,t)}\alpha(T_sx)ds=t.
\end{equation}
Since $\psi$ is an $\epsilon$-well-behaved Kakutani equivalence, there exists a full measure set $\Znew\label{017ZfullKak}\subset X$\index{$\Zold{017ZfullKak}$} such that for every $x\in \Zold{017ZfullKak}$,
\begin{enumerate}[label=\textup{(\roman*)}]
    \item\label{item:alphaClose1}  $\absolute{\alpha(T_tx)-1}<\epsilon$ for every $t\in\R$ and  $\alpha(T_tx)$ is $C^{
    \infty}$ in $t$;
    \item there exists $\tau(x,t)$ such that for every $t,s\in\R$ 
    \begin{equation}\label{eq:rhotau}
        \rho(T_sx,\tau(T_sx,t))=t.
    \end{equation}
\end{enumerate} 

\begin{Lemma}\label{lem:matchingPrese}
Let $T_t$ and $S_t$ be two ergodic flows acting on $(X,\mathcal{B},\mu,d_1)$ and $(Y,\mathcal{C},\nu,d_2)$ respectively. Suppose for some $\epsilon\in(0,10^{-3})$, $\psi$ is an $\epsilon$-well-behaved Kakutani equivalence between $T_t$ and $S_t$. Then for any $\eta\in(0,\epsilon)$, there exist a compact set $\Knew\label{017KmP1}\subset X$\index{$\Kold{017KmP1}$, Lemma~\ref{lem:matchingPrese}} 
with $\mu(\Kold{017KmP1})>1-6\eta$ and $\Rnew\label{017RmP1}>1$\index{$\Rold{017RmP1}$, Lemma~\ref{lem:matchingPrese}} such that the following holds: 
\begin{itemize}
    \item for any $\delta\in(0,10^{-3}\epsilon)$, there exists $\enew\label{017emP1}(\delta)\in(0,\delta)$\index{$\eold{017emP1}(\cdot)$, Lemma~\ref{lem:matchingPrese}} such that if 
    \[
    x,y\in \Kold{017KmP1}, \qquad R\geq\Rold{017RmP1}, \qquad \absolute{t_0}\leq\epsilon R,
    \]
    and if $x$ and $T_{t_0}y$ are $(\eold{017emP1}(\delta),\epsilon,R)$-two sides matchable with a $C^1$-matching function $h$, then $\psi(x)$ and $\psi(T_{t_0}y)$ are $(\delta,30\epsilon,R)$-two sides matchable for the $C^1$-matching function $h_1$ defined by
    \[
    h_1(t)=\tau\left(x,h(\rho(T_{t_0}y,t))\right).
    \]
\end{itemize}
Moreover, given $\delta>0$ and $x,y\in \Kold{017KmP1}$ satisfying $d_1(x,y)<\eold{017emP1}(\delta)$, then
\[
d_2(\psi(x),\psi(y))<\delta.
\]
\end{Lemma}

\begin{proof}[Proof of Lemma~\ref{lem:matchingPrese}]
By the Lusin theorem, for any given $\eta\in(0,\epsilon)$, there exists a compact subset $K_1\subset X$ with $\mu(K)>1-\eta$ such that $\psi$ is uniformly continuous on $K_1$. 

Since $\psi$ is uniformly continuous on $K_1$, for any fixed  $\delta\in(0,10^{-3}\epsilon)$, there exists $\eold{017emP1}(\delta)\in(0,\delta)$ such that 
\begin{equation}\label{eq:lusinConstant}
\text{if $x,y\in K_1$ and $d_1(x,y)<\eold{017emP1}(\delta)$, then $d_2(\psi(x),\psi(y))<\delta$.}
\end{equation}

By the Pointwise Ergodic Theorem,  there exist $N_1>1$ and a compact set $K_2\subset X$ with $\mu(K_2)>1-4\eta$ such that for every $x\in K_2$ and every $R\geq N_1$
\begin{equation}\label{eq:goodTime}
\begin{aligned}
    &\frac{1}{2R}\int_{-R}^R\chi_{\Zold{017ZfullKak}\cap K_1}(T_sx)ds>1-5\eta.
    \end{aligned}
\end{equation}

\medskip

Let $\Kold{017KmP1}=\Zold{017ZfullKak}\cap K_1\cap K_2$ and $\Rold{017RmP1}=10N_1$; in particular $\mu(\Kold{017KmP1})>1-5\eta$. Moreover, the choices of $\Zold{017ZfullKak}$, $K_1$ and $K_2$ imply that $\Kold{017KmP1}$ is a compact subset of $X$ independent of the choice of $\delta$; the choice of $N_1$ implies that $\Rold{017RmP1}$ is independent of the choice of $\delta$.

Assume that $x,T_{t_0}y\in \Kold{017KmP1}$ are $(\eold{017emP1}(\delta),\epsilon,R)$-two sides matchable for some $R\geq \Rold{017RmP1}$. Then there exist a $C^1$-function $h:[-R,R]\to[-R,R]$ and $A\subset[-R,R]$ with $l(A)>(1-\epsilon)2R$ such that
\begin{enumerate}[label=\textup{(\alph*)}]
    \item\label{item:xyclose} $d_1(T_{h(t)}x,T_{t+t_0}y)<\eold{017emP1}(\delta)$ for all $t\in A$,
    \item $h(0)=0$ and $0\in A$,
    \item\label{item:h1deri} $\absolute{h'(t)-1}<\epsilon$ for all $t\in [-R,R]$.
\end{enumerate}
\medskip

Let $h_1$ be defined as
\begin{equation}\label{eq:newMatching}
h_1(t)=\tau(x,h(\rho(T_{t_0}y,t))).
\end{equation}
Since $x,y\in \Zold{017ZfullKak}$, we obtain from \eqref{eq:rhoalpha}, \ref{item:alphaClose1} and \eqref{eq:rhotau} that for $s\in\R$
\begin{equation}\label{eq:rhotauDeri}
    \rho'(T_{t_0}y,s)=(\alpha(T_{\rho(T_{t_0}y,s)+t_0}y))^{-1}\text{ and }  \tau'(x,s)=\alpha(T_{s}x).
\end{equation}
These in particular imply that $h_1$ is a $C^{1}$-function in $t$. It also follows from chain rule that
\begin{equation*}
    \begin{aligned}
    h'_1(t)=&\tau'\left(x,h(\rho(T_{t_0}y,t))\right)\cdot h'(\rho(T_{t_0}y,t))\cdot \rho'(T_{t_0}y,t)\\
    =&h'(\rho(T_{t_0}y,t))\frac{\alpha(T_{h(\rho(T_{t_0}y,t))}x)}{\alpha(T_{\rho(T_{t_0}y,t)+t_0}y)}.
    \end{aligned}
\end{equation*}
This together with $x,y\in \Zold{017ZfullKak}$, \ref{item:alphaClose1} and \ref{item:h1deri} gives for every $t\in[-R,R]$
\begin{equation}\label{eq:h1deri}
    \absolute{h'_1(t)-1}<5\epsilon.
\end{equation}

\medskip

Let $A_0$ and $A_1$ be defined as
\begin{gather}
 A_0=A\cap\{t\in[-R,R]:T_{h(t)}x,T_{t+t_0}y\in K\cap \Zold{017ZfullKak}\}, \nonumber\\
 A_1=\{t\in[-R,R]:\rho(T_{t_0}y,t)\in A_0\}.\label{eq:defA1}
\end{gather}
Combining $\eta\in(0,\epsilon)$, $y\in \Zold{017ZfullKak}$,  \ref{item:alphaClose1}, \eqref{eq:goodTime} and  \eqref{eq:rhotauDeri}, then
\begin{equation}\label{eq: l-bar-A}
l(A_1)>(1-21\epsilon)2R.
\end{equation}
Moreover, combining \eqref{eq:lusinConstant}, \ref{item:xyclose} and \eqref{eq:newMatching},  we have for every $t\in A_1$
\begin{equation}\label{eq:newClose}
    d_2(S_{h_1(t)}\psi(x),S_{t}\psi(T_{t_0}y))<\delta.
\end{equation}

\medskip

Finally, since $\rho(T_{t_0}y,0)=0$, $\tau(x,0)=0$ and $h(0)=0$, we obtain from \eqref{eq:newMatching} that
\begin{equation}\label{eq:h1zero}
    h_1(0)=0.
\end{equation}

\medskip

Combining \eqref{eq:h1deri}, \eqref{eq: l-bar-A}, \eqref{eq:newClose} and \eqref{eq:h1zero}, we obtain that $\psi(x)$ and $\psi(T_{t_0}y)$ are $(\delta,30\epsilon,R)$-two sides matchable with the $C^1$-matching function defined as in \eqref{eq:newMatching}. This finishes the proof of Lemma~\ref{lem:matchingPrese}.

\end{proof}

\section{Some local estimates regarding divergence of one-parameter unipotent subgroups}\label{sec:polynomialDivergence}

A key feature of unipotent one parameter groups is that the entries of the elements of the group are given by polynomials. For this reason, properties of polynomials play a key role in the study of unipotent flows. In particular, we mention the following elementary fact about polynomials:

\begin{Lemma}\label{lem:polCoe}
Let $p(t)=\sum_{k=0}^da_kt^k$ be a polynomial of degree $d$. For every $\epsilon>0$, there exists $\Cnew\label{0191DPol}=\Cold{0191DPol}(d)\geq1$ \index{$\Cold{0191DPol}$, Lemma~\ref{lem:polCoe}}
such that if $\absolute{p(t)}\leq\epsilon$ for all $t\in[0,T]$, then $\absolute{a_k}<\Cold{0191DPol}T^{-k}\epsilon$ for all $k=0,\ldots,d$. Conversely, if $\absolute{a_k}\leq \Cold{0191DPol}^{-1}T^{-k}\epsilon$ for all $k$, then $\absolute{p(t)}<\epsilon$ for all $t\in[0,T]$.
\end{Lemma}

The proof is left to the reader. Cf.~ also \cite{kanigowski2019slow}*{Lemma 2.7}.

\medskip

The following lemma gives a quantitative description of matching time of two points in group $\mathbf{G}$:

\begin{Lemma}\label{lem:numConnectedCom}
There exist constants $\Cnew\label{0192Cnum}>0$\index{$\Cold{0192Cnum}$, Lemma~\ref{lem:numConnectedCom}} and such that for any $\epsilon,\delta\in(0,\Cold{0192Cnum}^{-1}\eold{014EMa})$ the following holds.
Suppose that $\tilde{x},\tilde{y}\in \mathbf{G}$ are two points with $d_\mathbf{G}(\tilde{x},\tilde{y})<\delta$, and let $h:\R\to\R$ be a $C^{1}$-function such that $h(0)=0$ and $\absolute{h'(t)-1}<\epsilon$. Then we can cover the set  
\[
I_{\tilde{x},\tilde{y}}=\{t\geq 0:d_{\mathbf{G}}(u_{h(t)}.\tilde{x},u_{t}.\tilde{y})<4\delta\} 
,\] 
by at most $3(\dim\mathfrak{g})^2$ disjoint closed intervals $[b_i,d_i]$, \ $i=1$, \dots, $k$, so that on each interval $[b_i,d_i]$
\begin{equation*}
 d_\mathbf{G}(u_{\phi_g(t)}.\tilde{x},u_{t}.\tilde{y})<\Cold{0192Cnum}\delta\qquad \forall t\in [b_i,d_i],
\end{equation*}
where $\phi_g(t)$ is the best matching function defined in \eqref{eq:bestMatchingFun} for $g$ satisfying $\tilde{x}=g.\tilde{y}$.

\end{Lemma}

\begin{proof}[Proof of Lemma~\ref{lem:numConnectedCom}]
Let $I_{1,\delta}, I_{2,\delta}$ be defined by
\begin{equation*}
    \begin{aligned}
    I_{1,\delta} &= \left \{t\geq 0: d_{\mathbf{G}}(u_{h(t)}.g^{\mathfrak{h}}.u_{-t},e)<\delta\right\},\\
    I_{2,\delta} &= \left\{t \geq 0 : d_{\mathbf{G}}\Bigl(\exp\Bigl(\Ad_{u_{-t}}\Bigl(\sum_{j=1}^n\sum_{i=0}^{q_j}\vartheta_{i,j}(g)\bx^{i,j}\Bigr)\Bigr),e\Bigr)<\delta\right\},
    \end{aligned}
\end{equation*}
in the notations of \eqref{eq:gDecom}, in particular
\[g^{\mathfrak{h}}=u_{\vartheta_{\bu}(g)}.\ou_{\vartheta_{\bou}(g)}.a_{\vartheta_{\ba}(g)}.\]
The decomposition \eqref{eq:gDecom} implies that there exists  $\kappa_1>0$ such that 
\begin{equation*}\label{eq:firstIncl}
I_{x,y}\subset (I_{1,\kappa_1\delta}\cap I_{2,\kappa_1\delta}). 
\end{equation*}

\medskip

Explicitly writing out the $2 \times 2$ matrix $u_{h(t)}.g^{\mathfrak{h}}.u_{-t}$ in terms of $\vartheta_{\bu}(g)$, $\vartheta_{\bou}(g)$ and $\vartheta_{\ba}(g)$, this implies that every $t \in I_{1,\kappa_1\delta}$ 
\begin{equation}\label{eq:sl2Parts}
    \begin{aligned}
    &\absolute{e^{\vartheta_{\ba}(g)} (\vartheta_{\bou}(g) (\vartheta_{\bu}(g)+h(t))+1)-1}<\kappa_2\delta,\\
    &\absolute{e^{-\vartheta_{\ba}(g)} (\vartheta_{\bu}(g)+h(t))-e^{\vartheta_{\ba}(g)} t (\vartheta_{\bou}(g) (\vartheta_{\bu}(g)+h(t))+1)}<\kappa_2\delta,\\
    &\absolute{\vartheta_{\bou}(g) e^{\vartheta_{\ba}(g)}}<\kappa_2\delta,\\
    &\absolute{e^{-\vartheta_{\ba}(g)}-\vartheta_{\bou}(g) e^{\vartheta_{\ba}(g)} t -1}<\kappa_2\delta
    \end{aligned}
\end{equation}
for a suitably chose constant $\kappa_2>0$.
The last equality of \eqref{eq:sl2Parts} together with $d_G(x,y)<\delta$ show that 
\begin{equation*}
    I_{1,\kappa_1\delta}\subset\left[-\frac{\kappa_3\delta}{\absolute{\vartheta_{\bou}(g)}},\frac{\kappa_3\delta}{\absolute{\vartheta_{\bou}(g)}}\right]
\end{equation*} 
for some constant $\kappa_3>0$.

\medskip

For every $i,j$ the coefficient of $\bx^{i,j}$ in 
\[
\Ad_{u_{-t}}\left(\sum_{j=1}^n\sum_{i=0}^{q_j}\vartheta_{i,j}(g)\bx^{i,j}\right)
\] 
is a polynomial  in $t$ of degree less equal than $q_j$, say $p_{i,j}(t)$.
Then for any $\delta'>0$, the solution set of $\absolute{p_{i,j}(t)}=\delta'$ consists of at most $2q_j$ points. Therefore the solutions to
\[\{t \geq 0: \absolute{p_{i,j}(t)}\leq\delta' \text{ for all $1\leq j\leq n$ and $0\leq i\leq q_j$}\}\]
is given by a collection of intervals with at most $2\sum_{j=1}^{n} q_j^2$ endpoints, hence has at most $3\dim (\mathfrak g )^2$ connected components. Thus there exist $k' \leq 3 \dim(\mathfrak g)^2$ disjoint closed intervals $[b_l,d_l]$, $l=1,\ldots,k'$ such that 
\[
\left(I_{2,\kappa_1\delta}\bigcap\left[-\frac{\kappa_3\delta}{\absolute{\vartheta_{\bou}(g)}},\frac{\kappa_3\delta}{\absolute{\vartheta_{\bou}(g)}}\right]\right)\subset\bigcup_{l=1}^{k'}J_l\subset \left[-\frac{\kappa_3\delta}{\absolute{\vartheta_{\bou}(g)}},\frac{\kappa_3\delta}{\absolute{\vartheta_{\bou}(g)}}\right]
\]
and (for suitably chosen $\kappa_4>0$)
\begin{equation}\label{eq:sl2TControl}
    \absolute{p_{i,j}(t)}<\kappa_4\delta, \qquad\text{for all $ t\in J_l$, \  $1\leq j\leq n$, \ $0\leq i\leq q_j$.}
\end{equation}

\medskip

Let $\phi_g(t)$ be the best matching function defined in \eqref{eq:bestMatchingFun} for $g$ satisfying $\tilde{x}=g.\tilde{y}$, then $\absolute{\vartheta_{\bou}(g)}\leq\frac{\kappa_3\delta}{d_k}$, $\absolute{\vartheta_{\ba}(g)}<\delta$, \eqref{eq:sl2TControl} and Lemma~\ref{lem:matchingFunction} give that for some constant $\kappa_5>0$
\[
d_\mathbf{G}(u_{\phi_g(t)}.\tilde{x},u_{t}.\tilde{y})<\kappa_5\delta \qquad  \forall t\in[b_{\ell},d_{\ell}].
\] 
This completes the proof.

\end{proof}

The following lemma is a combinatorial lemma which plays a similar role in our analysis to Ratner's~\cite{ratner1979cartesian}*{Lemma~2}. For two disjoint intervals $J_1, J_2$ let  $d(J_1,J_2)$ denote the distance between them.
\begin{Lemma}\label{lem:longHammingComb}
Let $C,w,n_{\max}>0$ and $\epsilon\in(0,10^{-3})$ be constants. Then there is $\Rnew\label{021RlC1}>0$\index{$\Rold{021RlC1}$, Lemma~\ref{lem:longHammingComb}} depending on $C,w,n_{\max}$ so that if $I$ is an interval of length $R \geq \Rold{021RlC1}$ the following holds. Suppose \[\mathcal{P}=\{J_1,\ldots,J_{n_0},J_1^{(1)},\ldots,J_{n_1}^{(1)},\ldots,J_{1}^{(m)},\ldots,J_{n_m}^{(m)}\}\]
is a partition of $I$ into intervals with disjoint interiors, $n_i \leq n_{\max}$ for all $i\geq 1$ and let
\[
\mathcal{P}_b=\{J_1^{(1)},\ldots,J_{n_1}^{(1)},\ldots,J_{1}^{(m)},\ldots,J_{n_m}^{(m)}\}.
\]
Suppose further that
\begin{enumerate}
    \item $\sum_{J\in\mathcal{P}_b}l(J)\geq(1-\epsilon)R$;
    \item for any $J_i^{(p)},J_j^{(q)}\in\mathcal{P}_b$ and $p\neq q$
    \[
    d(J_i^{(p)},J_j^{(q)})\geq\max\left(\Rold{021RlC1}, C\min\left(l(J_i^{(p)}),l(J_j^{(q)})\right)^{1+w}\right).
    \]
\end{enumerate}

\noindent 
Then there exists $J_{i_0}^{j_0}\in\mathcal{P}_b$ such that
\begin{equation*}
    \begin{aligned}
    l(J_{i_0}^{j_0})\geq\frac{1}{32n_{\max}}R \qquad\text{and}\qquad \sum_{i=1}^{n_{j_0}}l(J_{i}^{j_0})\geq\frac{3}{4}R.
    \end{aligned}
\end{equation*}
\end{Lemma}

\begin{proof}[Proof of Lemma~\ref{lem:longHammingComb}]
Let $J_{\max}^{(k)}$ be an interval from the collection 
\[
\{J_{1}^{(k)},\ldots,J_{n_k}^{(k)}\}
\]
of maximal length,  $1\leq k\leq m$. Then assumption (1) implies that
\begin{equation}\label{eq:largePortion}
\sum_{k=1}^ml(J_{\max}^{(k)})\geq\frac{1-\epsilon}{n_{\max}}R.
\end{equation}
Let $\mathcal{P}_{b,\max}=\{J_{\max}^{(1)},\ldots,J_{\max}^{(m)}\}$.
By changing the indices we may assume that these intervals are indexed from left to right.

Fix $\Rold{021RlC1}$ so that
\begin{gather}
\Rold{021RlC1}>\max\left(16n_{\max},\frac{(2^{200}n^2_{\max})^{1+1/w}}{C^{1/w}}\right),\label{eq: R2 def}\\
    \sum_{n\geq\log(\Rold{021RlC1}/16n_{\max})}\frac{2^{n+2}}{C2^{n(1+w)}}<(16n_{\max})^{-1}. \label{eq: R2 def mod}
\end{gather}

Let 
\[
B=\{J\in\mathcal{P}_{b,\max}:l(J)< (16n_{\max})^{-1}\Rold{021RlC1}\}.
\]
Assumption (2) of the lemma implies in particular that successive intervals in $\mathcal{P}_{b,\max}$ are separated by at least $R_2$, hence
\[
(\operatorname{Card}(B)-1)\Rold{021RlC1}+\sum_{J\in B}l(J)\leq R,
\]
which implies that 
\[
\frac{\sum_{J\in B}l(J)}{R}\leq \frac{\sum_{J\in B}l(J)}{(\operatorname{Card}(B)-1)\Rold{021RlC1}}\leq\frac{\operatorname{Card}(B)(16n_{\max})^{-1}}{\operatorname{Card}(B)-1}.
\] 
If $\operatorname{Card}(B)\geq2$, we have $\frac{\sum_{J\in B}l(J)}{R}\leq (8n_{\max})^{-1}$;
 otherwise $\frac{\sum_{J\in B}l(J)}{R}\leq\frac{(16n_{\max})^{-1}\Rold{021RlC1}}{R}\leq(16n_{\max})^{-1}$ by the choice of $R$.
Either way,
\begin{equation}\label{eq:baseBlock}
    \sum_{J\in B}l(J)\leq(8n_{\max})^{-1}R.
\end{equation}

For any $n\geq \log\frac{\Rold{021RlC1}}{16n_{\max}}$, we define 
\[
A_n=\{J\in\mathcal{P}_{b,\max}:2^n\leq l(J)< 2^{n+1}\}.
\]
If $\operatorname{Card}(A_n)\geq2$, then assumption (2) implies that
\[
\operatorname{Card}(A_n)2^n+(\operatorname{Card}(A_n)-1)C\cdot2^{n(1+w)}\leq R,
\]
which gives that 
\begin{equation}\label{eq:smallBlock}
\sum_{J\in A_n}l(J)\leq \operatorname{Card}(A_n)\cdot 2^{n+1}\leq\frac{2^{n+2}}{C2^{n(1+w)}}R.
\end{equation}

Let 
\[\Delta=\{i\geq\log (\Rold{021RlC1}/16n_{\max}):\operatorname{Card}(A_i)=1\}
\]
and for $i \in \Delta$ denote the unique element of $A_i$ by $J_{A_i}$. 
Let
\[
\Phi = \{i\geq\log (\Rold{021RlC1}/16n_{\max}):\operatorname{Card}(A_i)\geq 2\}.
\]
Combining \eqref{eq:largePortion}, \eqref{eq: R2 def mod},  \eqref{eq:baseBlock} and \eqref{eq:smallBlock} we obtain that
\begin{equation}\label{eq:longFamily}
\begin{aligned}
\sum_{n\in\Delta}l(J_{A_n}) &\geq\frac{1-\epsilon}{n_{\max}}R-\sum_{J\in B}l(J)-\sum_{n \in \Phi}\sum_{J\in A_n}l(J)\\
&\geq\frac{1-\epsilon}{n_{\max}}R-(8n_{\max})^{-1}R\ -\!\!\!\!\!\sum_{n\geq\log(\Rold{021RlC1}/16n_{\max})}\frac{2^{n+2}}{C2^{n(1+w)}}R\\
&\geq\frac{1}{8n_{\max}}R.
\end{aligned}
\end{equation}
Define $i_0=\max(i\in\Delta)$; then we have $4l(J_{A_{i_0}})\geq\sum_{n\in\Delta}l(J_{A_n})$ by definition of $A_n$ and $\Delta$. In view of \eqref{eq:longFamily} this implies that
\begin{equation}\label{eq:lhct1}
l(J_{A_{i_0}})\geq\frac{1}{32n_{\max}}R,
\end{equation}
which proves the first assertion of the lemma.

\medskip

Let $i\in\Delta$ and $i\neq i_0$, then assumption (2) implies that
\begin{equation}\label{eq:lhci}
l(J_{A_{i_0}})+l(J_{A_i})+C(l(J_{A_i}))^{(1+w)}\leq R.
\end{equation}
Since the definition of $\Delta$ and $A_n$ imply that 
\begin{equation}\label{eq:lhcd}
2^{i_0-i-1}l(J_{A_i})\leq l(J_{A_{i_0}})\leq 2^{i_0-i+1}l(A_i),    
\end{equation} 
we obtain from \eqref{eq:lhci} that
\begin{equation*}
    l(J_{A_{i_0}})+2^{-i_0+i-1}l(J_{A_{i_0}})+C(2^{-i_0+i-1}l(J_{A_{i_0}}))^{(1+w)}\leq R,
\end{equation*}
which together with \eqref{eq:lhct1} give 
\[
C(2^{-i_0+i-1}\frac{1}{32n_{\max}}R)^{(1+w)}\leq R,
\] 
and thus if $R\geq R_2$, by \eqref{eq: R2 def}
\begin{equation}\label{eq:lhci1}
2^{i_0-i-1}\geq 2^{100}n_{\max}.
\end{equation}
Then \eqref{eq:lhcd} and \eqref{eq:lhci1} give that for any $i\in\Delta$ and $i\neq i_0$, we have
$
    l(J_{A_i})\leq 2^{-100}n_{\max}^{-1}R
$, which together with the definition of $\Delta$ and $A_n$ give
\begin{equation}\label{eq:lhcc3}
    \sum_{i\in\Delta,i\neq i_0}l(J_{A_i})\leq 2^{-50}n_{\max}^{-1}R.
\end{equation}
Then we have from \eqref{eq: R2 def mod},  \eqref{eq:baseBlock}, \eqref{eq:smallBlock} and \eqref{eq:lhcc3} that
\begin{equation}\label{eq:lhcs1}
\begin{aligned}
    \sum_{\substack{J\in\mathcal{P}_{b,\max}\\J\neq J_{A_{i_0}}}}l(J)&\leq (8n_{\max})^{-1}R+2^{-50}n_{\max}^{-1}R+\sum_{n\geq\log(\Rold{021RlC1}/16n_{\max})}\frac{2^{n+2}}{C2^{n(1+w)}}R\\
    &\leq\frac{7}{32n_{\max}}R.
\end{aligned}
\end{equation}

Suppose $J_{A_{i_0}}=J_{l_0}^{j_0}$ for some $1\leq j_0\leq m$ and $1\leq l_0\leq n_{j_0}$, then \eqref{eq:lhcs1} and the definition of $\mathcal{P}_{b,\max}$ give that
\[
    \sum_{j\neq j_0}\sum_{i=1}^{n_j}l(J_{i}^j)\leq\frac{7}{32}R, 
\]
which together with $\sum_{J\in\mathcal{P}_b}l(J)\geq(1-\epsilon)R$ and $\epsilon\in(0,10^{-3})$ give
$
\sum_{i=1}^{n_{j_0}}l(J_i^{j_0})\geq\frac{3}{4}R
$. This completes the proof.
\end{proof}

\section{Measure estimate of exceptional returns}\label{sec:MesureEstimatesExcept}

The purpose of this section is to establish some preliminary estimates that will help us to obtain Lemma~\ref{lem:longHamming}, which roughly states that if two points $x,y\in \mathbf{G}/\Gamma$ are $(\epsilon,\delta,R)$-two sides matchable, then there are lifts $\tilde x$ and $\tilde y$ of $x$ and $y$ respectively so that for some $t_1,t_2 \in [0,R]$, the point $u_{t_1}.\tilde x$ is in $\Kak(\epsilon',R).u_{t_2}\tilde y$ for some $\epsilon'$ that does not depend on $R$.
Our approach uses a combination of Brundnyi inequality on algebraic groups and Lojasiewicz inequality for multivariable polynomials.

\medskip

Let $\mathcal{P}_{k,n}(\R)$ be the space of the real polynomials $p\in\R[x_1,\ldots,x_n]$ of degree at most $k$ and $V\subset\R^n$ a real algebraic variety of pure dimension $m$ ($1\leq m\leq n-1$). We use the following Remez-type estimate from \cite{brudnyi1999local}, which is the extension of the classical Brudnyi-Ganzburg inequality \cite{brudnyui1973extremal}. One may compare it with the $(C,\alpha)$-property in~\cite{kleinbock1998flows}.
\begin{Theorem}[Theorem 1.3 in \cite{brudnyi1999local}]\label{thm:brudnyiLocal}
For every regular point $x\in V$ there is an open neighborhood $\mathcal{N}=\mathcal{N}(V)$ of $x$ such that
\[
\sup_B\absolute{p}\leq\left(\frac{d\cdot\lambda_V(B)}{\lambda_V(\mathcal{W})}\right)^{\alpha k\deg(V)}\sup_{\mathcal{W}}\absolute{p}
\]
for every ball\footnote{With respect to metric induced from Euclidean metric on $\R^n$.} $B\subset \mathcal{N}$, measurable subset $\mathcal{W}\subset B$ and polynomial $p\in\mathcal{P}_{k,n}(\R)$. Here $\lambda_V$ denotes the induced Lebesgue measure in $V$, $d=d(m)$, $\deg(V)$ is the degree of manifold $V$ and $\alpha$ is an absolute constant.
\end{Theorem}

Since the measure and metric in the above theorem are not our desired measure and metric, we need the following lemma to compare them in a quantitative way:
\begin{Lemma}\label{lem:modBru}
Let $\mathbf{G}$ be the identity component of a real semisimple algebraic group $\mathbb{G}(\R)$ (considered as a subgroup of some $\SL_N(\R)$). Let $\tilde{m}$ denote the Haar measure on $\mathbf{G}$, and let $d_G$ denote some right invariant Riemannian metric on $\mathbf{G}$. Then for every $z\in \mathbf{G}$ and $r < \enew\label{024demB1}$\index{$\eold{024demB1}$, Lemma~\ref{lem:modBru}},  every measurable subset\footnote{Note that here the balls are with respect to the (right) invariant metric $d_G$ on~$\mathbf{G}$.} $\mathcal{W}\subset B_r^\mathbf{G}(z)$ and any polynomial $p\in\mathcal{P}_{k,N^2}(\R)$ 
\[
\sup_{B_r^\mathbf{G}(z)}\absolute{p}\leq \Cold{024CmB2}\left(\frac{\tilde{m}(B_r^\mathbf{G}(z))}{\tilde{m}(\mathcal{W})}\right)^{\deold{024CmB4}}\sup_{\mathcal{W}}\absolute{p},
\]
where $\eold{024demB1},\Cnew\label{024CmB2},\denew\label{024CmB4}>0$\index{$\Cold{024CmB2}$, Lemma~\ref{lem:modBru}} \index{$\deold{024CmB4}$, Lemma~\ref{lem:modBru}} depend only on $\mathbf{G}$ and $k$.

\end{Lemma}
\begin{proof}[Proof of Lemma~\ref{lem:modBru}]
As $\mathbf{G}$ is the identity component of a real semisimple algebraic group $\mathbb{G}(\R)$, every point of $\mathbf{G}$ is a regular point and $\mathbb{G}(\R)$ has pure dimension $m$, where $1\leq m\leq N^2-1$.

In order to fit in the framework of Theorem~\ref{thm:brudnyiLocal}, let $\lambda_G$ be the induced Lebesgue measure in $\mathbf{G}$ (typically this is \emph{not} the Haar measure) and $\hat{d}_\mathbf{G}$ the induced distance in $\mathbf{G}$ from Euclidean metric on $\R^{N^2}$. Denote by $\hat{B}_r^\mathbf{G}(z)$ and $B_r^\mathbf{G}(z)$ the balls of radius $r$ around $z\in \mathbf{G}$ under the distance functions $\hat{d}_\mathbf{G}$ and $d_\mathbf{G}$ respectively. Moreover, let $\mathcal{N}$ be a neighborhood as in Theorem~\ref{thm:brudnyiLocal} for $V=\mathbb{G}(\R)$ and $x=\operatorname{id}$. Then there exist $r_0\in(0,1)$ and $\kappa_0>1$ depending only on $\mathbf{G}$ such that for every $r\in(0,r_0]$
\begin{gather}\label{eq:br1n}
    \hat{B}_{r}^\mathbf{G}\subset\mathcal{N}, \text{ and }   
    B_{r/\kappa_0}^{\mathbf{G}}\subset \hat{B}_{r}^{\mathbf{G}}\subset B_{r\kappa_0}^\mathbf{G}.
\end{gather} 

\medskip

Notice that the right invariance of $d_\mathbf{G}$ implies that  $B_{r}^\mathbf{G}(z).z^{-1}=B_r^\mathbf{G}$ for every $z\in \mathbf{G}$ and $r>0$. This together with \eqref{eq:br1n} gives for every $r\in(0,r_0/\kappa_0)$
\begin{equation}\label{eq:rightTrasBall}
    B_{r}^\mathbf{G}(z).z^{-1}\subset \hat{B}_{r\kappa_0}^\mathbf{G}\subset\mathcal{N}.
\end{equation}
In particular, for every measurable subset $\mathcal{W}\subset B_r^{\mathbf{G}}(z)$,
\[
\mathcal{W}.z^{-1}\subset\hat{B}_{r\kappa_0}^\mathbf{G}\subset\mathcal{N}.
\]

Suppose $p\in\mathcal{P}_{k,N^2}(\R)$, define $p_{z}(x)=p(xz)$ for $x\in \mathbf{G}$ and thus $p_{z}\in\mathcal{P}_{k,N^2}(\R)$. Apply Theorem~\ref{thm:brudnyiLocal} for $V=\mathbb{G}(\R)$,  $x=\operatorname{id}$, $p=p_{z}$, $B=\hat{B}_{r\kappa_0}^\mathbf{G}$ and $\mathcal{W}=\mathcal{W}.z^{-1}$, where $\mathcal{W}\subset B_r^\mathbf{G}(z)$, then for every $r\in(0,r_0/\kappa_0]$
\begin{equation}\label{eq:idLocalBrudnyi}
\sup_{\hat{B}_{r\kappa_0}^\mathbf{G}}\absolute{p_{z}}\leq\left(\frac{d\cdot\lambda_\mathbf{G}(\hat{B}_{r\kappa_0}^\mathbf{G})}{\lambda_\mathbf{G}(\mathcal{W}.z^{-1})}\right)^{\alpha k\deg(\mathbb{G}(\R))}\sup_{\mathcal{W}.z^{-1}}\absolute{p_{z}}.
\end{equation}
Recall that \eqref{eq:rightTrasBall} and definition of $p_{z}$ give that
\begin{equation}\label{eq:supsame}
\sup_{B_{r}^\mathbf{G}(z)}\absolute{p}\leq\sup_{\hat{B}_{r\kappa_0}^\mathbf{G}}\absolute{p_{z}} \text{ and }\sup_{\mathcal{W}.z^{-1}}\absolute{p_{z}}=\sup_{\mathcal{W}}\absolute{p}.
\end{equation}
Moreover, the local equivalence of metric and \eqref{eq:br1n} imply that  for some $\kappa_1,\kappa_2>1$ depend only on $\mathbf{G}$
\begin{gather}
\lambda_\mathbf{G}(\hat{B}_{r\kappa_0}^{\mathbf{G}})\leq\kappa_1\tilde{m}(\hat{B}_{r\kappa_0}^{\mathbf{G}})\leq\kappa_1\tilde{m}(B_{r\kappa_0^2}^\mathbf{G})\leq\kappa_2\tilde{m}(B_{r}^\mathbf{G}),\label{eq:ballEquiv}\\
\lambda_\mathbf{G}(\mathcal{W}.z^{-1})\geq\kappa_2^{-1}\tilde{m}(\mathcal{W}.z^{-1})=\kappa_2^{-1}\tilde{m}(\mathcal{W})\label{eq:meaEquiv}.
\end{gather}

We complete the proof of Lemma~\ref{lem:modBru} by combining  \eqref{eq:idLocalBrudnyi}, \eqref{eq:supsame}, \eqref{eq:ballEquiv} and  \eqref{eq:meaEquiv}.

\end{proof}

The following lemma helps us to compare matrix norm and right invariance metric quantitatively. 
\begin{Lemma}\label{cor:metricRel}
Let $\mathbf{G}<\operatorname{SL}_N(\R)$ be a connected Lie group equipped with a right invariant metric. Then there exists $\Cnew\label{0261Cl1}>0$\index{$\Cold{0261Cl1}$, Lemma~\ref{cor:metricRel}} and $\denew\label{0261dl1}\in(0,1)$\index{$\deold{0261dl1}$, Lemma~\ref{cor:metricRel}} so that the following holds. Let $z_1,z_2\in \mathbf{G}$, and set $R(z_2) = \max(1,\norm{z_2})^{1/\deold{0261dl1}}$. Then if 
\[
\norm{z_1-z_2}<\Cold{0261Cl1}^{-1}R(z_2)^{-1},
\]
then 
\[
\Cold{0261Cl1}^{-1}R(z_2)^{-1}\norm{z_1-z_2}  < d_{\mathbf{G}}(z_1,z_2)<\Cold{0261Cl1}R(z_2)\norm{z_1-z_2}.
\]
\end{Lemma}
\begin{proof}[Proof of Lemma~\ref{cor:metricRel}]
Denote by $\hat{B}_r^\mathbf{G}(z)$ and $B_r^\mathbf{G}(z)$ as the balls on $\mathbf{G}$ of radius $r$ around $z\in \mathbf{G}$ under the distance functions $\norm{\cdot}$ and the right invariance metric $d_{\mathbf{G}}$ respectively. Then there exist $r_0\in(0,1)$ and $\kappa_0>1$ such that for every $r\in(0,r_0]$
\begin{gather}\label{eq:localSame}
    B_{r/\kappa_0}^\mathbf{G}\subset \hat{B}_{r}^\mathbf{G}\subset B_{r\kappa_0}^\mathbf{G}.
\end{gather} 

By the sub--multiplicativity of Hilbert--Schmidt norm and Cayley--Hamilton theorem\footnote{This in particular guarantees that $\norm{z^{-1}}\leq\zeta_1\max(1,\norm{z})^{\zeta_2}$ for some $\zeta_1,\zeta_2>0$ depending only on $\mathbb{G}(\R)$.}, there exist $\kappa_1>1$ such that for every $z\in\operatorname{SL}_N(\R)$ $r\in(0,r_0/L_1]$ with $L_1=\kappa_1\max(1,\norm{z})^{\kappa_1}$
\begin{equation*}
\hat{B}_{r/L_1}^\mathbf{G}(z)\subset \hat{B}_{r}^\mathbf{G}.z\subset\hat{B}^\mathbf{G}_{rL_1}(z).
\end{equation*} 
This together with \eqref{eq:localSame} guarantees that there exists $\kappa_2>1$ such that if $r\in(0,r_0/L_2]$ for $L_2=\kappa_2\max(1,\norm{z})^{\kappa_2}$ 
%every $r\in(0,r_1\kappa_4^{-1}r_{\tilde{z}}^{-\kappa_5}]$:
\begin{gather}\label{eq:ballVolEq}
    B_{r/L_2}^\mathbf{G}(z)\subset \hat{B}_{r}^\mathbf{G}(z)\subset B_{rL_2}^\mathbf{G}(z).
\end{gather}

Suppose that $\deold{0261dl1}$ is sufficiently small, then 
\begin{equation*}\label{eq:cm1}
    \begin{aligned}
    L_2\leq r_0\deold{0261dl1}^{-1}\max(1,\norm{z_2})^{1/\deold{0261dl1}},
    \end{aligned}
\end{equation*}
where $r_0$ is defined in \eqref{eq:localSame} and $L_2$ is defined in \eqref{eq:ballVolEq} (with $z=z_2$). This together with $\norm{z_1-z_2}<\deold{0261dl1}\max(1,\norm{z_2})^{-1/\deold{0261dl1}}$  and \eqref{eq:ballVolEq} completes the proof.
\end{proof}

\begin{Corollary}\label{cor:slgxbelonging}
Let $\mathbf{G}$ be the identity component of $\mathbb{G}(\R)$. There exist constants $\Cnew\label{0262CMovingNHD}, \denew\label{0262deMovingNHD}>0$\index{$\Cold{0262CMovingNHD}$, Corollary~\ref{cor:slgxbelonging}}\index{$\deold{0262deMovingNHD}$, Corollary~\ref{cor:slgxbelonging}} such that for any $\epsilon\in(0,10^{-3})$,  $\delta\in(0,\deold{0262deMovingNHD}]$, $z\in \mathbf{G}$ satisfying $\norm{z}\leq\epsilon^{-1}$ and $x\in\mathbb{G}(\R)$ satisfying $\norm{z-x}\leq \frac{\Cold{0262CMovingNHD}\epsilon^{1/\delta}}{2}$, we have
\[
x\in \mathbf{G}.
\]
\end{Corollary}
\begin{proof}
    Since $\mathbf{G}$ is the identity component of $\mathbb{G}(\R)$, there exists $\kappa_1\in(0,1)$ such that if $x\in\mathbb{G}(\R)$ and $d_{\operatorname{SL}_N(\R)}(x,e)<\kappa_1$, then $x\in \mathbf{G}$. Let $\Cold{0261Cl1}$ and $\deold{0261dl1}$ be defined in Lemma~\ref{cor:metricRel} for $\operatorname{SL}_N(\R)$, we pick
\begin{equation*}
\Cold{0262CMovingNHD}=\kappa_1/(2\Cold{0261Cl1}),\qquad \delta\in(0,\deold{0261dl1}).
\end{equation*}

Suppose that $z\in \mathbf{G}$ satisfying $\norm{z}\leq\epsilon^{-1}$ and $x\in\mathbb{G}(\R)$ satisfying $\norm{z-x}\leq \frac{\Cold{0262CMovingNHD}\epsilon^{1/\delta}}{2}$. By Lemma~\ref{cor:metricRel}, 
\begin{equation*}
\begin{aligned}
d_{\SL_N(\R)}(e,xz^{-1})&=d_{\SL_N(\R)}(z,x)\\
&\leq\Cold{0261Cl1}\max(1,\norm{z})^{1/\deold{0261dl1}}\frac{\Cold{0262CMovingNHD}\epsilon^{1/\delta}}{2}\leq\frac{\kappa_1}{4}\epsilon^{1/\delta-1/\deold{0261dl1}}\leq\frac{\kappa_1}{4}.
\end{aligned}
\end{equation*}
As a result, we obtain that $xz^{-1}\in \mathbf{G}$, this together with $z\in \mathbf{G}$ gives that $x\in \mathbf{G}$. We finish the proof by picking $\deold{0262deMovingNHD}=\deold{0261dl1}$.
\end{proof}

Notice that by passing to the identity component in Zariski topology, we can always assume that $\mathbb{G}(\R)$ is Zariski irreducible.
\begin{Lemma}\label{lem:discreteBase}
Let $\mathbf{G}<\operatorname{SL}_N(\R)$ be the identity component of  a Zariski irreducible real semisimple algebraic group $\mathbb{G}(\R)$ and $\mathbb{L}(\R)$ a real algebraic subgroup of $\mathbb{G}(\R)$. Then for every $\delta>0$, there exist $x_1,\ldots,x_n\in B_{\delta}^\mathbf{G}$ such that 
\[
\textstyle{\bigcap_{i=1}^nx_i^{-1}\mathbb{L}(\R)x_i=\bigcap_{x\in \mathbb G(\R)}x^{-1}\mathbb{L}(\R)x.}
\]
\end{Lemma}
\begin{proof}[Proof of Lemma~\ref{lem:discreteBase}]
Since $\mathbb G$ is connected, for every $\eta>0$ we have that $B_{\eta}^\mathbf{G}$ is Zariski dense in $\mathbb G$.
It follows that $\bigcap_{x\in B_{\eta}^\mathbf{G}}x^{-1}\mathbb{L}x$ is a normal subgroup of $\mathbb G$,
hence
\[
\textstyle{\bigcap_{x\in B_{\eta}^\mathbf{G}}x^{-1}\mathbb{L}(\R)x=\bigcap_{x\in \mathbb G(\R)}x^{-1}\mathbb{L}(\R)x.}
\]
By the descending chain condition, there is a finite set of elements $x_1,\dots,x_n$ so that 
\[
\textstyle{\bigcap_{i=1}^nx_i^{-1}\mathbb{L}(\R)x_i=\bigcap_{x\in \mathbb G(\R)}x^{-1}\mathbb{L}(\R)x.}
\]
\end{proof}

The following lemma is essential for Lemma~\ref{lem:longHammingSep} (which in turn is the main ingredient in the proof of Lemma~\ref{lem:longHamming} mentioned earlier in this section). The key ingredients of the proof are Brudnyi inequality, algebraic representation theory and properties of semi-algebraic sets.

\begin{Lemma}\label{lem:algebraicRepEst}
Let $\mathbf{G}<\operatorname{SL}_N(\R)$ be the identity component of a real semisimple algebraic group $\mathbb{G}(\R)$, and let $\tilde{m}$ be the Haar measure on $\mathbf{G}$.
Let $\rho:\mathbb{G}(\R)\to\SL(V)$ be a real algebraic representation of $\mathbb{G}(\R)$ on a real vector space $V$, equipped with a vector $v_{\mathbf{L}}\in V$, and define 
\[
\textstyle{\mathbb{L}(\R)=\{g\in \mathbb{G}(\R):\rho(g)v_{\mathbf{L}}=v_{\mathbf{L}}\}.}
\] 
Let $\mathbf{N}$ be the normal core of $\mathbf{L}=\mathbb{L}(\R)\cap \mathbf{G}$, i.e. 
\[
\textstyle{\mathbf{N}=\bigcap_{g \in \mathbf{G}}\, g \mathbf{L} g^{-1}.}
\]

For any $\delta>0$, there exist $\denew\label{027dcoeSmallnew},\denew\label{027dexpSmallnew}\in(0,1)$\index{$\deold{027dcoeSmallnew}$, Lemma~\ref{lem:algebraicRepEst}}\index{$\deold{027dexpSmallnew}$, Lemma~\ref{lem:algebraicRepEst}} and  $\Cnew\label{027CcoeSmall},\Cnew\label{027ClocalCoe}>0$\index{$\Cold{027CcoeSmall}$, Lemma~\ref{lem:algebraicRepEst}}\index{$\Cold{027ClocalCoe}$, Lemma~\ref{lem:algebraicRepEst}} such that for every $\epsilon\in(0,10^{-3})$, if~$z,\gamma \in \mathbf{G}$ satisfy
\[
\norm{\gamma},\norm{z} \leq \epsilon^{-1} \qquad\text{and}\qquad\min_{x\in \mathbf{N}}\norm{x-\gamma}\geq \Cold{027CcoeSmall}\epsilon^{1/\delta},
 \]
then
\begin{multline}
\tilde{m}\left(\{g\in B_{\deold{027dcoeSmallnew}}^G(z):\norm{\rho(\gamma g^{-1}).v_{\mathbf{L}}-\rho(g^{-1}).v_{\mathbf{L}}}_V<\epsilon^{1/\deold{027dexpSmallnew}}\}\right)  \\ \leq
 \Cold{027ClocalCoe}\epsilon^{\deold{027dexpSmallnew}}\tilde{m}(B_{\deold{027dcoeSmallnew}}^G(z)).
\end{multline}
\end{Lemma}

\begin{proof}[Proof of Lemma~\ref{lem:algebraicRepEst}]
Without loss of generality, we can assume $\delta$ is as small as we wish as the statement of Lemma~\ref{lem:algebraicRepEst} is less restrictive the smaller $\delta$ is. Apply Lemma~\ref{lem:modBru} to the polynomial map $f_{\gamma}:\mathbf{G}\mapsto\R$ defined by: 
\[
f_{\gamma}(g)=\norm{\rho(\gamma g^{-1}).v_{\mathbf{L}}-\rho(g^{-1}).v_{\mathbf{L}}}_V^2.
\]
Then for every $z\in \mathbf{G}$, every $\delta\in(0,\eold{024demB1}]$ and every measurable subset $\mathcal{W}\subset B_{\delta}^\mathbf{G}(z)$
\begin{equation}\label{eq:rhoCornew}
    \sup_{B_{\delta}^\mathbf{G}(z)}\absolute{f_{\gamma}}\leq \Cold{024CmB2}\left(\frac{\tilde{m}(B_{\delta}^\mathbf{G}(z))}{\tilde{m}(\mathcal{W})}\right)^{\deold{024CmB4}}\sup_{\mathcal{W}}\absolute{f_{\gamma}},
\end{equation}
where $\eold{024demB1},\deold{024CmB4},\Cold{024CmB2}>0$ are constants that depend only on $\mathbf{G},\rho,\mathbf{L}$. Let $\Cold{0262CMovingNHD}$ and $\deold{0262deMovingNHD}$ be defined as in Corollary~\ref{cor:slgxbelonging}. Then we define $\Cold{027CcoeSmall}$, $\deold{027dcoeSmallnew}$ and $\delta$ as
\[
\Cold{027CcoeSmall}=\Cold{0262CMovingNHD},\qquad 
\deold{027dcoeSmallnew}=\min(\eold{024demB1},\deold{0262deMovingNHD},1),\qquad \delta\in(0,\deold{027dcoeSmallnew}).
\]

\medskip

By Lemma~\ref{lem:discreteBase}, there exist $x_1,\ldots,x_n\in B_{\deold{027dcoeSmallnew}}^\mathbf{G}$ such that 
\begin{equation}\label{eq:NtwoForms}
\textstyle{\mathbb{N}(\R)=\bigcap_{i=1}^n x_i^{-1}\mathbb{L}(\R)x_i=\bigcap_{x\in \mathbb{G}(\R)}x^{-1}\mathbb{L}(\R)x.}
\end{equation}
Moreover, $\mathbb{N}(\R)$ is a normal subgroup of $\mathbb{G}(\R)$.

Denote $W=\bigoplus_{i=1}^nV$ and $p(z)=\oplus_{i=1}^n\rho(z^{-1}x_i^{-1}).v_{\mathbf{L}}$ for every $z\in\mathbb{G}(\R)$; let $\rho_W$ be the action of $\mathbb{G}(\R)$ on $W$. We claim that the kernel of map $$g\mapsto\rho_W(g).p(z)-p(z)$$ from $\mathbb{G}(\R)$ to $W$ equals to $\mathbb{N}(\R)$. Indeed:
\begin{enumerate}[label=\textup{(\alph*)}]
\item On the one hand, $\rho_W(g).p(z)=p(z)$ implies for $i=1,\ldots,n$, 
\[
\rho(x_izgz^{-1}x_i^{-1}). v_{\mathbf{L}}= v_{\mathbf{L}}
\]
 which gives that $g\in \mathbb{N}(\R)$. 
\item On the other hand, if $g\in \mathbb{N}(\R)$, then
\[
\rho(g) \rho(z^{-1}x_i^{-1}).v_{\mathbf{L}} \in \rho(z^{-1}x_i^{-1}\mathbb{L}(\R)).v_{\mathbf{L}}=\rho(z^{-1}x_i^{-1}).v_{\mathbf{L}}
\]
hence $\rho_W(g).p(z)=p(z)$.
\end{enumerate}

As $x_i\in \mathbf{G}$, we have from \eqref{eq:NtwoForms}
\begin{equation}\label{eq:NanotherDefnew}
\textstyle{\mathbf{G}\cap\mathbb{N}(\R)=\bigcap_{i=1}^nx_i^{-1}(\mathbb{L}(\R)\cap \mathbf{G})x_i=\mathbf{N}.}
\end{equation}

\medskip

Since $\mathbb{G}(\R)$ is a real semisimple algebraic group, there exists a real polynomial system $\phi_i:\R^{N^2}\to\R$, $i=1,\ldots,k$ such that 
\[
\mathbb{G}(\R)=\{g\in\R^{N^2}:\phi_i(g)=0,\ \ i=1,\ldots,k\}.
\]
Then for $z,\gamma\in\R^{N^2}$, define
\[
q(z,\gamma)=\norm{\rho(\gamma).p(z)-p(z)}_W^2+\sum_{i=1}^k\left(\phi_i^2(\gamma)+\phi_i^2(z)\right).
\]
If $q(z,\gamma)=0$, we have $z,\gamma\in\mathbb{G}(\R)$ and $\rho_W(\gamma).p(z)=p(z)$, which in particular implies that that the set of the real zeros of $q(z,\gamma)$ is
\[
Z(q)=\left\{(a,b)\in\mathbb{G}(\R)\times\mathbb{G}(\R):b\in \mathbb{N}(\R)\right\}.
\]

\medskip

By \cite{Hor58Division}*{Lemma~2}, there exist some constants $\kappa_2,\kappa_3>0$ and $\kappa_4\in\R$ depend only on $\mathbf{G},\rho,\mathbf{L}$ such that
\begin{multline}\label{eq:HorIneqtwo}
\norm{\rho_W(\gamma).p(z)-p(z)}_W^2+\sum_{i=1}^k\left(\phi_i^2(\gamma)+\phi_i^2(z)\right)\\
\geq \kappa_2(1+\norm{\gamma}^2+\norm{z}^2)^{-\kappa_4}s(z,\gamma),
\end{multline}
where 
\[
s(z,\gamma)=\min_{(a,b)\in Z(q)}\left(\norm{a-z}^2+\norm{b-\gamma}^2\right)^{\kappa_3}.
\]

Assume from now on that $z,\gamma\in \mathbf{G}$ and satisfy
\begin{equation*}
    \begin{aligned}
\norm{z},\norm{\gamma}\leq\epsilon^{-1},\qquad\min_{x\in \mathbf{N}}\norm{x-\gamma}\geq \Cold{027CcoeSmall}\epsilon^{1/\delta}.
\end{aligned}
\end{equation*}
Then we claim that 
\begin{equation}\label{eq:AlgRepEstTarger}
\norm{\rho_W(\gamma).p(z)-p(z)}_W^2\geq\eta_1\epsilon^{\kappa_5},
\end{equation}
where $\eta_1=\kappa_2(\Cold{027CcoeSmall}/2)^{2\kappa_3}$ and 
$
\kappa_5=\left\{
  \begin{array}{ll}
    2\kappa_3/\delta,& \hbox{if $\kappa_4\leq0$;} \\
    2(\kappa_3/\delta)+2\kappa_4, & \hbox{other.}
  \end{array}
\right.$

If \eqref{eq:AlgRepEstTarger} is not holding, then the assumptions and \eqref{eq:HorIneqtwo} imply that
\begin{equation}\label{eq:sLowerBound}
s(z,\gamma)\leq\left(\frac{\Cold{027CcoeSmall}\epsilon^{1/\delta}}{2}\right)^{2\kappa_3}.
\end{equation}
This together Corollary~\ref{cor:slgxbelonging} implies that $a,b\in G$. However, if $a,b\in \mathbf{G}$, equations~\eqref{eq:NanotherDefnew} and \eqref{eq:sLowerBound} imply that
\[
\min_{x\in \mathbf{N}}\norm{x-\gamma}\leq \frac{\Cold{027CcoeSmall}\epsilon^{1/\delta}}{2},
\]
which contradicts to our assumption. This proves \eqref{eq:AlgRepEstTarger}.

\medskip

By combining \eqref{eq:AlgRepEstTarger} with the definition of $p(z)$, there exists $i\in\{1,\ldots,n\}$ such that 
\[
\norm{\rho(\gamma z^{-1}x_i^{-1}).v_{\mathbf{L}}-\rho(z^{-1}x_i^{-1}).v_{\mathbf{L}}}^2_V>\frac{1}{n}\eta_1\epsilon^{\kappa_{5}}, 
\]
and thus 
\begin{equation}\label{eq:ballMaxnew}
    \sup_{g\in B_{\deold{027dcoeSmallnew}}^\mathbf{G}(z)}\absolute{f_{\gamma}(g)}>\frac{1}{n}\eta_1\epsilon^{\kappa_{5}}.
\end{equation}

Define 
\[
\mathcal{W}=\{g\in B_{\deold{027dcoeSmallnew}}^\mathbf{G}(z):f_{\gamma}(g)<\epsilon^{2\kappa_{5}}\},\]
then \eqref{eq:rhoCornew}, \eqref{eq:ballMaxnew}, $f_{\gamma}\geq0$ and $\norm{z}\leq\epsilon^{-1}$ imply that for every $\epsilon\in(0,10^{-4})$
\begin{equation*}
\tilde{m}(\{g\in B_{\deold{027dcoeSmallnew}}^{\mathbf{G}}(z):f_{\gamma}(g)<\epsilon^{2\kappa_{5}}\})\leq \kappa_6\epsilon^{\kappa_7}\tilde{m}(B_{\deold{027dcoeSmallnew}}^{\mathbf{G}}(z)),
\end{equation*}
where $\kappa_{6},\kappa_7>0$ depend only on $\mathbf{G},\rho,\mathbf{L}$. We finish the proof by defining $\deold{027dexpSmallnew}=\min(1/\kappa_5,\kappa_7,1)$ and $\Cold{027ClocalCoe}=\kappa_6$. 
\end{proof}

Recall by \cite{GorodnikNevoErgodic2010}*{Theorem~1.5}, if $\Gamma<\mathbf{G}$ is a lattice, then there exist constants $\denew\label{0301Clattice}\in(0,\deold{0261dl1})$\index{$\deold{0301Clattice}$, Equation~\eqref{eq:latticeCount}} and $\Rnew\label{030Rlattice}>1$\index{$\Rold{030Rlattice}$, Equation~\eqref{eq:latticeCount}} such that for $T\geq\Rold{030Rlattice}$
\begin{equation}\label{eq:latticeCount}
\absolute{\gamma\in\Gamma:\norm{\gamma}\leq T}\leq T^{1/\deold{0301Clattice}}.
\end{equation}
The following corollary is a combination of Lemma~\ref{lem:algebraicRepEst} and \eqref{eq:latticeCount}:

\begin{Corollary}\label{cor:globalRepEst}
Let $\delta>0$, $\mathbf{G}$, $\rho$, $\mathbf{L}$, $v_{\mathbf{L}}$ be as in Lemma~\ref{lem:algebraicRepEst}, with corresponding constants $\deold{027dexpSmallnew}$ and $\Cold{027CcoeSmall}$. Let $K\subset \mathbf{G}$ be a compact subset. Let $V_\epsilon$ be the set of $g \in K$ such that
\[
\norm{\rho(\gamma g^{-1}).v_{\mathbf{L}}-\rho(g^{-1}).v_{\mathbf{L}}}_V<\epsilon^{1/\deold{027dexpSmallnew}}
\]
for some $\gamma\in\Gamma$ satisfying $\min_{x\in \mathbf{N}}\norm{x-\gamma}\geq \Cold{027CcoeSmall}\epsilon^{1/\delta}$ and $\norm{\gamma}\leq\epsilon^{-\deold{027dexpSmallnew}\deold{0301Clattice}/2}$.
Then there exists $\denew\label{0302CglobalCom1}>0$\index{$\deold{0302CglobalCom1}$, Corollary~\ref{cor:globalRepEst}} such that if $\epsilon>0$ is small enough
\[
\tilde{m}(V_{\epsilon})\leq   \epsilon^{\deold{0302CglobalCom1}}.
\]
\end{Corollary}

\begin{proof}[Proof of Corollary~\ref{cor:globalRepEst}]
For any $z\in K$, 
we have that $\{B_{\deold{027dcoeSmallnew}}^{\mathbf{G}}(z)\}_{z\in K}$ is an open cover of $K$. Since $K$ is compact,  there exists a subcover $\{B_{\deold{027dcoeSmallnew}}^{\mathbf{G}}(z_i)\}_{i=1}^p$ such that 
\begin{equation}\label{eq:openCover}
    \begin{aligned}
    \textstyle{K\subset\bigcup_{i=1}^pB_{\deold{027dcoeSmallnew}}^{\mathbf{G}}(z_i).}
    \end{aligned}
\end{equation} 
We define $\kappa_0=\sum_{i=1}^p\tilde{m}(B_{\deold{027dcoeSmallnew}}^{\mathbf{G}}(z_i))$.

For $i=1,\ldots,p$, let $\Delta(z_{i},\gamma)$ be the set of $g\in B_{\deold{027dcoeSmallnew}}^{\mathbf{G}}(z_{i})$ such that
\begin{equation*}
    \norm{\rho(\gamma g).v_{\mathbf{L}}-\rho(g).v_{\mathbf{L}}}_V<\epsilon^{1/\deold{027dexpSmallnew}} 
\end{equation*}
for some $\gamma\in\Gamma$ satisfying $\min_{x\in \mathbf{N}}\norm{x-\gamma}\geq\Cold{027CcoeSmall}\epsilon^{1/\delta}$ and $\norm{\gamma}\leq\epsilon^{-\deold{027dexpSmallnew}\deold{0301Clattice}/2}$.

Since $\deold{027dexpSmallnew}\deold{0301Clattice}\in(0,1)$,
\begin{equation}\label{eq:inverseEst}
\begin{aligned}
\tilde{m}(V_{\epsilon})\leq& \sum_{\gamma\in\Gamma,\norm{\gamma}\leq\epsilon^{-\deold{027dexpSmallnew}\deold{0301Clattice}/2}}\sum_{i=1}^p\tilde{m}_2(\Delta(z_{i},\gamma))\\
\leq& \sum_{\gamma\in\Gamma,\norm{\gamma}\leq\epsilon^{-\deold{027dexpSmallnew}\deold{0301Clattice}/2}}\left(\Cold{027ClocalCoe}\epsilon^{\deold{027dexpSmallnew}}\sum_{i=1}^p\tilde{m}_2\left(B_{\deold{027dcoeSmallnew}}^{\mathbf{G}}(z_i)\right)\right)\\
\leq&\epsilon^{-\deold{027dexpSmallnew}/2}\Cold{027ClocalCoe}\epsilon^{\deold{027dexpSmallnew}}\sum_{i=1}^p\tilde{m}_2\left(B_{\deold{027dcoeSmallnew}}^{\mathbf{G}}(z_i)\right)\\
    \leq&\kappa_0 \Cold{027ClocalCoe}\epsilon^{\deold{027dexpSmallnew}/2},
\end{aligned}
\end{equation}
where the first inequality is due to \eqref{eq:openCover}; the second inequality is due to Lemma~\ref{lem:algebraicRepEst}; the third inequality is due to \eqref{eq:latticeCount} and the last inequality is due to the definition of $\kappa_0$.

Let $\deold{0302CglobalCom1}=\deold{027dexpSmallnew}/4$ and $\epsilon$ so small such that $\epsilon^{-\deold{027dexpSmallnew}/4}>\kappa_0 \Cold{027ClocalCoe}$, we finish the proof.
\end{proof}

\section{Normal core and its neighborhood}\label{sec:normalcoreNBHD}
In this section, we investigate neighborhoods around the normal core of a real semisimple algebraic group. This will be used to guarantee that the assumption on $\min_{g\in \mathbf{N}}\norm{\gamma-g}$ in Lemma~\ref{lem:algebraicRepEst} is satisfied in Lemma~\ref{lem:longHammingMEst}

\medskip

Recall that $\mathbf{G}$ is the identity component of a real semisimple algebraic group with Lie algebra $\mathfrak{g}=\oplus_i^{\ell}\mathfrak{g}_i$, where each  $\mathfrak{g}_i$ is a simple Lie algebra. Let $\operatorname{p}_i:\mathfrak{g}\to\mathfrak{g}_i$ be the canonical projection. For every $i=1,\ldots,\ell$, there exists a normal closed connected simple Lie subgroup $\mathbf{G}_i\lhd \mathbf{G}$ with Lie algebra $\mathfrak{g}_i$ such that 
\begin{equation}\label{eq:almostDirect}
\mathbf{G}=\mathbf{G}_1\ldots \mathbf{G}_{\ell},\ \ \ \ \text{ $\mathbf{G}_i$ commutes with $\mathbf{G}_j$ for $i\neq j$.}
\end{equation}

Let $\Ad:\mathbf{G}\to\operatorname{GL}(\mathfrak{g})$ be the adjoint representation, by replacing $\mathbf{G}$ and $\mathbf{G}_i$ with $\Ad(\mathbf{G})$ and $\Ad(\mathbf{G}_i)$, we obtain the following isomorphism
\begin{gather}
    \Ad(\mathbf{G}_1)\times\ldots\times\Ad(\mathbf{G}_{\ell})\to \Ad(\mathbf{G}_1)\ldots\Ad(\mathbf{G}_{\ell})\nonumber\\
    (g_1,\ldots,g_{\ell})\mapsto g_1\ldots g_{\ell}\label{eq:isomAd}.
\end{gather}
This in particular gives that the canonical projection 
\[
\pi_i:\Ad(\mathbf{G})\to\Ad(\mathbf{G}_i)
\]
is well defined. For every $z\in \Ad(\mathbf{G})$ and $i=1,\ldots,\ell$, we define
\begin{equation*}
    z_i=\pi_i(z).
\end{equation*}

Recall that $\mathbf{H}$ is a connected subgroup of $\mathbf{G}$ with Lie algebra $\mathfrak{h}=\operatorname{span}_{\R}\{\bu,\bou,\ba\}$. Then \eqref{eq:isomAd} in particular implies that $\Ad(\mathbf{H})$ is isomorphic to $\operatorname{PGL}^+_2(\R)$, and moreover for $ i=1,\dots,\ell$,
\begin{equation}\label{eq:AdHIsoPGL}
\text{$\pi_i|_{\Ad(\mathbf{H})}$ is either an isomorphism onto its image or is trivial.}
\end{equation}

\medskip

Let $\mathbf{N}\lhd \mathbf{G}$ be a closed normal Lie subgroup with Lie algebra $\mathfrak{n}$. Then there exists a decomposition 
\[
\mathfrak{n}=\textstyle{\bigoplus_{i=1}^{\ell}\mathfrak{n}_i},
\]
where $\mathfrak{n}_i$ is an ideal of $\mathfrak{g}_i$ for $i=1,\ldots,\ell$. Let $\mathbf{N}_i=\pi_i(\Ad(\mathbf{N}))$, then
\[
\mathbf{N}_i\lhd \Ad(\mathbf{G}_i)\qquad \operatorname{Lie}(\mathbf{N}_i)=\mathfrak{n}_i.
\]

The following lemma describes the structure of normal subgroups of~$\mathbf{G}$.

\begin{Lemma}\label{lem:normalSubgroup}
Let $\mathbf{N}$ be a closed normal subgroup of $\mathbf{G}$, then 
\[
\Ad(\mathbf{N})\leq \mathbf{N}_1\ldots \mathbf{N}_{\ell}.
\]
Moreover, $\mathbf{N}_i=\Ad(\mathbf{G}_i)$ or $\mathbf{N}_i=\{e\}$ for $i=0,\ldots,\ell$. 
\end{Lemma}
\begin{proof}[Proof of Lemma~\ref{lem:normalSubgroup}]
The definition of $\mathbf{N}_i$ and \eqref{eq:isomAd} imply
\[
\Ad(\mathbf{N})\leq \mathbf{N}_1\ldots \mathbf{N}_{\ell}.
\]

As $\Ad(\mathbf{G}_i)$ is a simple Lie group, then $\mathbf{N}_i$ either is a discrete normal subgroup of $\Ad(\mathbf{G}_i)$ or $\Ad(\mathbf{G}_i)$ itself. If $\mathbf{N}_i$ is a discrete normal subgroup of $\Ad(\mathbf{G}_i)$, then the connectedness of $\Ad(\mathbf{G}_i)$ in Hausdorff topology shows that $\mathbf{N}_i\leq C(\Ad(\mathbf{G}_i))=\{e\}$.

\end{proof}

Let
\[
\mathfrak{l}=\operatorname{span}_{\R}\{\bu,\bou,\ba\}\oplus\operatorname{span}_{\R}\{\bx^{0,j}:j\in I_c\},
\]
with $I_c=\{1\leq j\leq n:q_j=0\}$ and $\bx^{0,j}$ as defined in \S\ref{sec:sl2basis}. Equivalently, $\mathfrak{l}$ is the normalizer of $\mathfrak{h}$ in $\mathfrak{g}$.

\medskip

Let $\mathbb{L}(\R)=\{g \in \mathbb G(\R): \Ad (g)\mathfrak h=\mathfrak h\}$, then $\mathbb{L}(\R)$ is a real linear algebraic group, and its Lie algebra is $\mathfrak{l}$. We define $\mathbf{L}$ and $\mathbf{N}$ as follows:
\begin{equation}\label{eq:defN}
    \begin{aligned}
    \mathbf{L}=\mathbb{L}(\R)\cap \mathbf{G}, \ \ \ \ \mathbf{N}=\textstyle{\bigcap_{g\in \mathbf{G}}g\mathbf{L}g^{-1}}.
    \end{aligned}
\end{equation}
By Lemma~\ref{lem:chainLK} we may (and will) assume $\mathfrak{l}\neq\mathfrak g$ (hence $\mathbf{L} \lneq \mathbf{G}$). 

\medskip

Define $\bu_i=\operatorname{p}_i(\bu)$, $\ba_i=\operatorname{p}_i(\ba)$ and  $\bou_i=\operatorname{p}_i(\bou)$, \ $\operatorname{p}_i:\mathfrak{g}\to\mathfrak{g}_i$ as above. The following lemma describes the projection of $\mathfrak{h}$ under $p_i$:

\begin{Lemma}\label{lem:goodfactor}
There exists $i_0\in\{1,\ldots,\ell\}$ such that $\ba_{i_0}\neq 0$ and~$\mathbf{N}_{i_0}=\{e\}$, where $\mathbf{N}_{i_0}=\pi_{i_0}(\Ad(\mathbf{N}))$.
\end{Lemma}
\begin{proof}[Proof of Lemma~\ref{lem:goodfactor}]
Since $\mathfrak{l}\neq\mathfrak{g}$, there exists $j$ such that $q_j>0$ and $\bx^{0,j}\notin\mathfrak{l}$ (cf.~\eqref{eq:chainbasis}). This in particular implies that $\bx^{0,j}$ does not commute with $\ba$. Let $\mathfrak{g}_{i_0}$ be the simple Lie algebra containing $\bx^{0,j}$, then  $\bx^{0,j}$ does not commute with $\ba_{i_0}$, hence $\ba_{i_0}\neq 0$. 

Since $\bx^{0,j}\notin\mathfrak{l}$, we have $\bx^{0,j}\notin\mathfrak{n}_{i_0}$. This gives $\mathbf{N}_{i_0}\neq \Ad(G_{i_0})$ and thus Lemma~\ref{lem:normalSubgroup} shows that $\mathbf{N}_{i_0}=\{e\}$.
\end{proof}

The main result of this section is the following lemma:
\begin{Lemma}\label{lem:HNint}
Suppose that $\mathfrak{l}\neq\mathfrak{g}$, then there exist $\enew\label{033enbh}>0$\index{$\eold{033enbh}$, Lemma~\ref{lem:HNint}} and $\Cnew\label{033CnormRadius}>1$\index{$\Cold{033CnormRadius}$, Lemma~\ref{lem:HNint}} such that if $h\in \mathbf{H}$ satisfies
\[
hzp\in \mathbf{N}
\]
for some $z\in \mathbf{Z}$, $p\in B_{\epsilon}^{\mathbf{G}}$ and $\epsilon\in(0,\eold{033enbh})$, then $h$ is within distance $\Cold{033CnormRadius}\epsilon$ of an element in the center of $\mathbf{H}$.
\end{Lemma}
\begin{proof}[Proof of Lemma~\ref{lem:HNint}]
Let $\eold{006ecom1}$ be as in Lemma~\ref{lem:groupDecom}. Define 
\[
\eold{033enbh}=\frac{1}{100}\eold{006ecom1}.
\]

Since $p\in B_{\epsilon}^\mathbf{G}$ for some $\epsilon\in(0,\eold{033enbh})$ and $\Ad:\mathbf{G}\to\Ad(\mathbf{G})$ is a cover map, for every $p_i=\pi_i(\Ad(p))$ with  $i=1,\ldots,\ell$
\begin{equation}\label{eq:pdecom}
    \begin{aligned}
    p_i\in B_{2\epsilon}^{\Ad(\mathbf{G}_i)}.
    \end{aligned}
\end{equation}

If $h\in \mathbf{H}$ and $z\in \mathbf{Z}$, then for $i=1,\ldots,\ell$, let 
\begin{gather}\label{eq:zdecom}
    h_i=\pi_i(\Ad(h)), \qquad z_i=\pi_i(\Ad(z)).
\end{gather}
Moreover, $z_i\in C_{\Ad(\mathbf{G}_i)}(\pi_i(\Ad(\mathbf{H})))$.

\medskip

Since $\mathfrak{l}\neq\mathfrak{g}$, Lemma~\ref{lem:goodfactor} gives that there exists $1\leq i_0\leq \ell$ such that~$\mathbf{N}_{i_0}=\{e\}$ and $\ba_{i_0}\neq 0$. This together with $\Ad(hzp)\in \Ad(\mathbf{N})$,  \eqref{eq:pdecom}  and \eqref{eq:zdecom}, gives
\begin{equation*}
\begin{aligned}
h_{i_0}z_{i_0}p_{i_0}=e, \qquad z_{i_0}\in C_{\Ad(G_{i_0})}(\pi_{i_0}(\Ad(\mathbf{H}))).
\end{aligned}
\end{equation*}
In particular, we have for every $X\in\mathfrak{h}_{i_0}$
\[
\Ad_{p_{i_0}h_{i_0}}(X)=X.
\]
This together with $p_{i_0}\in B_{2\epsilon}^{\Ad(G_{i_0})}$ and $h_{i_0}\in \pi_{i_0}(\Ad(\mathbf{H}))$ gives  that the restriction of $\Ad_{h_{i_0}}$ to $\mathfrak{h}_{i_0}$ is within $O(\epsilon)$ of the identity.
Thus for a suitable (absolute) $\kappa_0>1$, \eqref{eq:AdHIsoPGL} imply that
\[
d_{\Ad(\mathbf{H})}(\Ad(h),e)<\kappa_0\epsilon.
\]
Combining with $\Ad:\mathbf{H}\to\Ad(\mathbf{H})$ is a cover map whose kernel equals to $C(\mathbf{H})$, the above inequality completes the proof.
\end{proof}

\section{Main Lemma}\label{sec:mainLemma}
In this section, we prove our main lemma, which roughly states that Kakutani-Bowen balls (cf.~Definition \ref{def:Kakball}) are preserved under ``good'' even Kakutani equivalences as in \S\ref{sec:convertKakutani}. An analogous statement (\cite{ratner1986rigidity}*{Lemma~4.1}) is key in Ratner's paper \cite{ratner1986rigidity}, which proves time change rigidity for the horocycle flow on quotients of $\SL_2(\R)$ under a H\"older regularity condition; in Ratner's paper the H\"older regularity condition is used in an essential way. Indeed, the horocycle flow on quotients of $\SL_2(\R)$ is loosely Kronecker \cite{ratner1978horocycle}, so without the strong H\"{o}lder regularity condition nothing of the sort can be said about the images of Kakutani-Bowen balls in this case.

\medskip

\subsection{Assumptions}\label{subsec:assupmtions} From now on, the following assumptions are fixed until the end of the paper:
\begin{enumerate}[label=\textup{(a\arabic*)}]
    \item $\mathbb{G}_i$ is a semisimple algebraic group defined over $\R$ for $i=1,2$;
    \item $\mathbf{G}_i$ is the identity component of the real points of $\mathbb{G}_i$ for $i=1,2$;
    \item $\mathfrak{g}_i=\operatorname{Lie}(\mathbf{G}_i)$ and $\Gamma_i$ is a lattice in $\mathbf{G}_i$ for $i=1,2$;
    \item $m_i$ is the probability measure on $\mathbf{G}_i/\Gamma_i$ induced by the Haar measure on $\mathbf{G}_i$ for $i=1,2$;
    \item $u_t^{(i)}=\exp(t \bu_i)$ is a one-parameter unipotent subgroup of $\mathbf{G}_i$ acting ergodically on $(\mathbf{G}_i/\Gamma_i,m_i)$ for $i=1,2$;
    \item  $(\mathbf{G}_1/\Gamma_1,m_1,u_t^{(1)})$ is Kakutani equivalent to $(\mathbf{G}_2/\Gamma_2,m_2,u_t^{(2)})$ with a Kakutani equivalence~$\psi$;

    \item\label{item:chainBasis} $\{\bu_k,\ba_k,\bou_k,\bx_{k}^{0,1},\ldots,\bx_{k}^{q_1^{(k)},1},\ldots,\bx_{k}^{0,n^{(k)}},\ldots,\bx_{k}^{q_n^{(k)},n^{(k)}}\}$ is a basis as in~\eqref{eq:chainbasis} for $\mathfrak{g}_k$ with respect to $\bu_k$ for $k=1,2$;
    \item $(\mathbf{G}_1/\Gamma_1,m_1,u_t^{(1)})$ is \textbf{not} loosely Kronecker.
\end{enumerate}
\medskip

\noindent
It is worth to point out that under our assumptions, for any $t\neq 0$, the element $a_t^{(i)}=\exp(t\ba_i)$ acts ergodically on $(\mathbf{G}_i/\Gamma_i,m_i)$
for $i=1,2$ (cf. \cite{MargulisDiscrete1991}*{\S~II}).

\medskip

\subsection{Preparation for main lemma}
\subsubsection{Modifications of Kakutani equivalence}\label{sec:modifiedKakutani}
Let $\psi$ be a Kakutani equivalence as above. While apriori $\psi$ has very little regularity, we now show that using the results of \S\ref{sec:convertKakutani} we may (after renormalization and replacing $\psi$ with a Kakutani equivalence cohomologous to it) assume $\psi$ is a ``good'' Kakutani equivalence.

\medskip

The first step is the renormalization. By Lemma~\ref{lem:evenKakutani}, there exists a unique $s_0\in\R$ such that $a_{s_0}^{(2)}.\psi$ is an even Kakutani equivalence between $(\mathbf{G}_1/\Gamma_1,m_1,u_t^{(1)})$ and $(\mathbf{G}_2/\Gamma_2,m_2,u_t^{(2)})$.

\medskip

Then we need to find a ``good'' even Kakutani equivalence in the same cohomology class as $a_{s_0}^{(2)}.\psi$. Applying Lemma~\ref{lem:goodTimeChange} for $(\mathbf{G}_1/\Gamma_1,m_1,u_t^{(1)})$, $(\mathbf{G}_2/\Gamma_2,m_2,u_t^{(2)})$ with $\epsilon=10^{-99}$, there is an $10^{-99}$-well-behaved Kakutani equivalence map 
\begin{equation}\label{eq:evenKakEqui}
\tilde\psi:(\mathbf{G}_1/\Gamma_1,m_1,u_t^{(1)})\mapsto(\mathbf{G}_1/\Gamma_2,m_2,u_t^{(2)})
\end{equation}
that is cohomologous to $a_{s_0}^{(2)}\circ\psi$ (cf.~Definition \ref{def:controlKakutani}). Replacing $\psi$ by $\tilde\psi$ if necessary we may as well assume that $\psi$ itself is an $10^{-99}$-well-behaved Kakutani equivalence.

\subsubsection{Time change good set}\label{sec:timeChagneGoodSet}
As $\psi$ is an $10^{-99}$-well-behaved Kakutani equivalence with time change function $\alpha$, there exists an $u_t^{(1)}$-invariant set $\Znew\label{035KtimeU}\subset \mathbf{G}_1/\Gamma_1$\index{$\Zold{035KtimeU}$, Equation~\eqref{eq:controlKakCon}} with $m_1(\Zold{035KtimeU})=1$ and $\tau:\Zold{035KtimeU}\times\R\to\R$ such that for every $x\in \Zold{035KtimeU}$, $\alpha(u_t^{(1)}.x)$ is a $C^{\infty}$ function in $t$. Moreover for every $t\in\R$:
\begin{equation}\label{eq:controlKakCon}
\begin{aligned}
&\int_{0}^{\rho(x,t)}\alpha(u_s^{(1)}.x)ds=t,\ \ \ \ 
        \operatorname{sup}_{x\in \Zold{035KtimeU}}\absolute{\alpha(x)-1}<10^{-99},\\
 &\psi(u_{\rho(x,t)}^{(1)}.x)=u_{t}^{(2)}.\psi(x), \ \ \ \ \psi(u_{t}^{(1)}.x)=u_{\tau(x,t)}^{(2)}.\psi(x),\\
 &\rho(x,\tau(x,t))=t, \ \ \ \ \ \ \ \ \ \ \ \ \tau(x,\rho(x,t))=t.\\
\end{aligned}
\end{equation}
Until the end of this section, unless stated otherwise, every subset of $\mathbf{G}_1/\Gamma_1$ we consider is assumed to be a subset of $\Zold{035KtimeU}$.

\subsection{Statement of the main lemma and a first reduction}
 
Following is the statement of our main lemma:
\begin{Lemma}[Main Lemma]\label{lem:main} 
For any $\eta\in(0,10^{-3})$, there exist $\denew\label{035delm1}>0$\index{$\deold{035delm1}$, Lemma~\ref{lem:main}}, $\Rnew\label{035Rlm1}>1$\index{$\Rold{035Rlm1}$, Lemma~\ref{lem:main}} and a compact set $\Knew\label{035Klm1}\subset \mathbf{G}_1/\Gamma_1$\index{$\Kold{035Klm1}$, Lemma~\ref{lem:main}} with $m_1(\Kold{035Klm1})>1-10\eta$ so that the following holds. Let $\delta\in(0,\deold{035delm1})$, then there exists $\enew\label{035eMainDel}(\delta)\in(0,\delta)$\index{$\eold{035eMainDel}(\cdot)$, Lemma~\ref{lem:main}} such that for $y\in\Kold{035Klm1}$, $R>\Rold{035Rlm1}$ 
\[
\psi \left(\Kak\left(R,\eold{035eMainDel}(\delta),y\right) \cap \Kold{035Klm1}\right) \subset \Kak\left(R,\delta,\psi(y)\right).
\]

More generally, if $R>\Rold{035Rlm1}$, \ $x,y\in\Kold{035Klm1}$ and $\absolute{t_1}\leq\deold{035delm1}R$ satisfy
\[
x\in\Kak\left(R,\eold{035eMainDel}(\delta),u_{t_1}^{(1)}.y\right),
\]
then there exists $\absolute{t_2}\leq \Rold{035Rlm1}$ such that
\[
\psi(x)\in \Kak\left(R,\delta,u_{t_2}^{(2)}.\psi(u_{t_1}^{(1)}.y)\right).
\]
We may assume that $\psi$ is uniformly continuous on $\Kold{035Klm1}$.

\end{Lemma}

The proof of Lemma~\ref{lem:main} follows from Lemma~\ref{lem:matchingPrese} and Lemma~\ref{lem:longHammingNew}, whose proof is postponed to \S\ref{sec:longHammingNew}.

\begin{Lemma}\label{lem:longHammingNew}
For any $\eta\in(0,10^{-3})$, there exist a compact subset $\Knew\label{036Klh1}\subset \mathbf{G}_2/\Gamma_2$\index{$\Kold{036Klh1}$, Lemma~\ref{lem:longHammingNew}} with $m_2(\Kold{036Klh1})\geq1-\eta$, $\deold{035delm1}>0$, and  $\Rnew\label{036Rlh1}>1$\index{$\Rold{036Rlh1}$, Lemma~\ref{lem:longHammingNew}} such that for any $\epsilon\in(0,10^{-40})$, any $\delta\in(0,\deold{035delm1})$ and $R\geq\Rold{036Rlh1}$ the following holds. Let $x \in \mathbf{G}_2/\Gamma_2$ and $y\in \Kold{036Klh1}$ be $(\delta,\epsilon,T)$-two sides matchable for every $T\in[\Rold{036Rlh1}, R]$ with the same matching function, then there exists $\absolute{t_2}\leq\Rold{036Rlh1}$ such that 
\[
x\in\Kak(R,\Cold{036Clm}\delta,u_{t_2}^{(2)}.y).
\]
Here $\Cnew\label{036Clm}>1$\index{$\Cold{036Clm}$, Lemma~\ref{lem:longHammingNew}} is a constant\footnote{The value of $\Cold{036Clm}$ will be given in the proof of Lemma~\ref{lem:longHammingNew}.} depending only on $\mathbf{G}_2$.
\end{Lemma}
\begin{Remark}
Lemma~\ref{lem:longHammingNew} is similar to \cite{kanigowski2021kakutani}*{Theorem~6.1}, but significantly more precise. The reader may want to compare the methods for obtaining these two results, which are similar in principle but quite different in detail. 
\end{Remark}

\begin{proof}[Proof of Lemma 
\ref{lem:main} assuming Lemma~\ref{lem:longHammingNew}]
Fix an $\eta\in(0,1)$, the selection of parameters are specified as following:
\begin{enumerate}[label=\textup{(\roman*)}]
    \item Let $\deold{035delm1}$, $\Kold{036Klh1}$, $\Rold{036Rlh1}$ and $\Cold{036Clm}$ be defined as in Lemma~\ref{lem:longHammingNew};
    \item 
    Let $\Kold{017KmP1}$, $\Rold{017RmP1}$ and $\eold{017emP1}(\cdot)$ be defined as Lemma~\ref{lem:matchingPrese} applied to $\psi$ with $\epsilon=10^{-99}$. Note that Lemma~\ref{lem:matchingPrese} also gives that $\psi$ is uniformly continuous on $\Kold{017KmP1}$,
    \item\label{item:compactReturnBound} Let $K_1\subset \mathbf{G}_2/\Gamma_2$ be a compact set with $m_2(K_1)\geq 1-10^{-20}\eta$ and $\kappa_1\in(0,10^{-10})$ such that the following holds
    \begin{itemize}
        \item For any $y\in K_1$, $x \in \mathbf{G}_2/\Gamma_2$ with $d_{\mathbf{G}_2/\Gamma_2}(x,y)<\kappa_1$, if $\tilde{x}$ and $\tilde{y}$ are lifts of $x$ and $y$ satisfying $d_{\mathbf{G}_2}(\tilde{x},\tilde{y})<\kappa_1$ and $s$, $t$ and $e \neq \gamma \in \Gamma_2$ satisfy
    \begin{equation*}
        d_{\mathbf{G}_2}(u_{t}^{(2)}.\tilde{x},u_{s}^{(2)}.\tilde{y}\gamma)<\kappa_1,\qquad u_s^{(2)}.y\in K_1,
    \end{equation*}
    then $\max(\absolute{t},\absolute{s})> 2\max(\Rold{017RmP1},\Rold{036Rlh1})$. 
    \end{itemize}
    
    \item Fix a $\delta\in(0,\min(\deold{035delm1},10{-200},\kappa_1))$, we define 
    \begin{gather*}
        \eold{035eMainDel}(\delta)=\eold{017emP1}\left((10\Cold{036Clm})^{-1}\delta\right)/(10^{100}\dim\mathfrak{g}_2)^3,\qquad 
        \Rold{035Rlm1}=\max(\Rold{017RmP1},\Rold{036Rlh1}), \\ \Kold{035Klm1}= \Kold{017KmP1}\cap\psi^{-1}(\Kold{036Klh1}\cap K_1).
    \end{gather*}
\end{enumerate}

\medskip

Having chosen these parameters, we can now begin the actual proof. For any $x,y\in \Kold{035Klm1}$, $R>\Rold{035Rlm1}$, and $\absolute{t_1}\leq\deold{035delm1}R$ satisfying 
\[
x\in\Kak(R,\eold{035eMainDel}(\delta),u_{t_1}^{(1)}.y),
\] 
we obtain from the local decomposition~\eqref{eq:gDecom} and Definition \ref{def:Kakball} that there exists $g\in \mathbf{G}_1$ such that 
$x=g.u_{t_1}^{(1)}.y$,
with
\[
g\in \Kak(R,\eold{035eMainDel}(\delta)).
\]

Let $\phi_g(t)$ be the best matching function defined in \eqref{eq:bestMatchingFun}, then by Proposition~\ref{prop:KakutaniBallStayClose}, for every $t\in[-R,R]$
\begin{equation}\label{eq:oldLongMatching}
d_{\mathbf{G}_1/\Gamma_1}(u_{\phi_g(t)}^{(1)}.x,u_{t+t_1}^{(1)}.y)<10\eold{035eMainDel}(\delta).
\end{equation}
Recall also that Lemma~\ref{lem:matchingFunction} gives that 
\[
\absolute{\phi_g'(t)-1}<\sqrt{\eold{035eMainDel}(\delta)}\qquad\text{for $t\in[-R,R]$}.
\]
This together with \eqref{eq:oldLongMatching} gives that  $x,u_{t_1}^{(1)}.y$ are $(20\eold{035eMainDel}(\delta),\sqrt{\eold{035eMainDel}(\delta)},T)$-two sides matchable for every $T\in[\Rold{036Rlh1},R]$ with the matching function~$\phi_g$ (note the matching function is independent of $T$). As $\sqrt{\eold{035eMainDel}(\delta)}<10^{-99}$, we can apply Lemma~\ref{lem:matchingPrese} and obtain that 
\begin{itemize}
    \item $\psi(x),\psi(u_{t_1}^{(1)}.y)$ are $(\frac{\delta}{10\Cold{036Clm}},10^{-97},T)$-two sides matchable for every $T\in[\Rold{036Rlh1},R]$ with the same matching function.
\end{itemize}
This together with $\psi(x)\in \Kold{036Klh1}$, $R>\Rold{035Rlm1}\geq\Rold{036Rlh1}$ and Lemma~\ref{lem:longHammingNew} gives that
\begin{equation}\label{eq:t1neq0}
    \psi(x)\in\Kak(R,\delta/10,u_{t_2}^{(2)}.\psi(u_{t_1}^{(1)}.y)),
\end{equation}
for some $\absolute{t_2}\leq\Rold{035Rlm1}$. 

If $t_1=0$, since $d_{\mathbf{G}_1/\Gamma_1}(x,y)\leq(\dim\mathfrak{g}_2)\eold{035eMainDel}(\delta)$ and $x,y\in \Kold{035Klm1}$,
\[
d_{\mathbf{G}_2/\Gamma_2}(\psi(x),\psi(y))\leq \delta/10.
\]
This together with choice of $K_1$ (cf.~\ref{item:compactReturnBound} on p.~\pageref{item:compactReturnBound}), equation \eqref{eq:t1neq0},  $\absolute{t_2}\leq \Rold{035Rlm1}$ and $t_1=0$ gives
\[
\absolute{t_2}\leq\delta/10
\]
hence
\[
\psi(x)\in\Kak(R,\delta,\psi(y)).
\]
\medskip

Finally recall that $\psi_*m_1^{\tilde{\alpha}}=m_2$, where $dm_1^{\tilde{\alpha}}=\tilde{\alpha}dm_1$ and $\tilde{\alpha}$ is the corresponding time change for $\psi$. Since Lemma~\ref{lem:goodTimeChange} guarantees that $\operatorname{esssup}\absolute{\tilde{\alpha}-1}<10^{-99}$
\begin{equation*}
\begin{aligned}
m_1\left(\psi^{-1}((\mathbf{G}_2/\Gamma_2)\setminus (\Kold{036Klh1}\cap K_1))\right) &<  2 m_1^{\tilde{\alpha}}\left(\psi^{-1}((\mathbf{G}_2/\Gamma_2)\setminus (\Kold{036Klh1}\cap K_1))\right)\\
&=  2m_2\left((\mathbf{G}_2/\Gamma_2)\setminus(\Kold{036Klh1}\cap K_1)\right)<3\eta.
\end{aligned}
\end{equation*}
This together with $m_1(\Kold{017KmP1})>1-6\cdot10^{-5}\eta$ gives that 
\[
m_1(\Kold{035Klm1})>1-6\cdot10^{-5}\eta-3\eta>1-10\eta,
\]
and thus completes the proof of Lemma~\ref{lem:main}.
\end{proof}

\subsection{Proof of Lemma~\ref{lem:longHammingNew}}\label{sec:longHammingNew}
\subsubsection{Preparation for the proof}
In order to prove Lemma~\ref{lem:longHammingNew}, we first prove a weaker result, namely Lemma~\ref{lem:longHamming} below, that roughly gives that if 
two ``good'' points of~$\mathbf{G}_2/\Gamma_2$ are $(\delta,\epsilon,R)$-two sides matchable for $R$ sufficiently large, then we can find orbit segments along the orbits of these two points of lengths $R/C$ for some $C>1$ depends only on $\mathbf{G}_2$ so that these orbits segments stay close in the universal cover.
For $\mathbf{G}_2=\SL_2(\R)\times \SL_2(\R)$, such a result was a key ingredient in Ratner's proof the product of horocycle flows is not loosely Kronecker in \cite{ratner1979cartesian}. Later Kanigowski, Vinhage and the second named author extended this to general $\mathbf{G}_2$ in \cite{kanigowski2021kakutani}, but in a weaker form that gives orbit segment that match in the universal cover of size $\gg_\eta R^{1-\eta}$ for any $\eta\in(0,1)$ (instead of $R/C$). However, for our purposed we need the stronger version. We prove Lemma~\ref{lem:longHamming} in \S\ref{sec:longHammingPre} and \S\ref{sec:longHamming}.

\begin{Lemma}\label{lem:longHamming}
For any $\eta\in(0,10^{-3})$, there exist a compact subset $\Kold{036Klh1}\subset~ \mathbf{G}_2/\Gamma_2$ with $m_2(\Kold{036Klh1})\geq1-\eta$ and $\deold{035delm1}>0$, $\Rold{036Rlh1}>1$ such that the following holds.  If $x \in \mathbf{G}_2/\Gamma_2$ and $y\in \Kold{036Klh1}$ are two points that are $(\delta,\epsilon,R)$-two sides matchable for $\delta\in(0,\deold{035delm1})$, $\epsilon \in (0,10^{-40})$, $R\geq\Rold{036Rlh1}$ and with a $C^1$-matching function~$h$,  then there exist lifts $\tilde{x}, \tilde{y} \in \mathbf{G}_2$ with $d_{\mathbf{G}_2}(\tilde{x},\tilde{y})<\delta$, $s_R\leq -R/2$, $L_R\geq3R/2$ and  $\gamma_R\in\Gamma_2$  such that 
\begin{equation*}
    \begin{aligned}
   u_{h(s_R)}^{(2)}.\tilde{x}&\in\Kak(L_R, \Cold{038Clh1}\delta)u_{s_R}^{(2)}\tilde{y}\gamma_R,\\
   u_{h(s_R+L_R)}^{(2)}.\tilde{x}&\in\Kak(L_R, \Cold{038Clh1}\delta)u_{s_R+L_R}^{(2)}\tilde{y}\gamma_R.
    \end{aligned}
\end{equation*}
Here $\Cnew\label{038Clh1}>0$\index{$\Cold{038Clh1}$, Lemma~\ref{lem:longHamming}} is a constant depending only on $\mathbf{G}_2$.

\end{Lemma}

\medskip

If two points are two sides matchable with a \textbf{same} matching function for many (or even better, infinitely many) time moments, then we obtain more delicate relations between these points as in the following corollary of Lemma~\ref{lem:longHamming}.  This corollary implies Lemma~\ref{lem:longHammingNew} and will be used in \S\ref{sec:comgeo}. The proof of Corollary~\ref{cor:inifinteMatching} assuming Lemma~\ref{lem:longHamming} is given in \S\ref{sec:infiniteMatching}. Here $\mathfrak{w}$ is the subspace of $\mathfrak g_2$ defined in~\eqref{eq:wSpace}. 

\begin{Corollary}\label{cor:inifinteMatching}
    Under the same assumptions as in Lemma~\ref{lem:longHamming}. If $x,y$ are $(\delta,\epsilon,R)$-two sides matchable for every $R$ satisfying $L_1\leq R\leq L_2$ with a fixed $C^1$-matching function $h$ and  $L_1\geq\Rold{036Rlh1}$, then there exists  $t_0\in\R$ and $\gamma\in\Gamma_2$ such that 
    \[\tilde{x}\in\Kak(L_2,\Cold{039infiniteMatch}\delta)u_{t_0}^{(2)}\tilde{y}\gamma \qquad \absolute{t_0}\leq0.01L_1.
    \]   
    Moreover, there are $\tilde s\leq -L_2/2$, $\tilde L\geq3L_2/2$ so that
    \begin{equation*}
    \begin{aligned}
        u^{(2)}_{h(\tilde s)}\tilde{x}\in&\Kak(\tilde L,\Cold{039infiniteMatch}\delta)u_{\tilde s}^{(2)}\tilde{y}\gamma,\\
u_{h(\tilde s + \tilde L)}^{(2)}\tilde{x}\in&\Kak(L, \Cold{039infiniteMatch}\delta)u_{\tilde s + \tilde L}^{(2)}\tilde{y}\gamma.
\end{aligned}
    \end{equation*}
    where $\tilde{x}$ and $\tilde{y}$ are lifts of $x$ and $y$ respectively, and $\Cnew\label{039infiniteMatch}>\Cold{038Clh1}$\index{$\Cold{039infiniteMatch}$, Corollary~\ref{cor:inifinteMatching}} is a constant depends only on $\mathbf{G}_2$.
    
   If $L_2=+\infty$, then there exist unique 
   \begin{equation*}
       \begin{aligned}
           f(x,y)&\in\R\text{ satisfying } \absolute{f(x,y)}\leq\Cold{039infiniteMatch}\delta,\\
           c(x,y)&\in B_{\Cold{039infiniteMatch}\delta}^{\mathbf{G}_2}\cap \exp(\mathfrak{w}),\\ 
           t_0(x,y)&\in[-L_1,L_1]
       \end{aligned}
   \end{equation*}
   such that
   \[
   y=c(x,y)a_{f(x,y)}^{(2)}u_{t_0(x,y)}^{(2)}.x.
   \]
\end{Corollary}

\subsubsection{Proof of Lemma~\ref{lem:longHammingNew} assuming Lemma~\ref{lem:longHamming}}
Let $L_1=\Rold{036Rlh1}$ and $L_2=R$ in Corollary~\ref{cor:inifinteMatching}, we complete the proof of Lemma~\ref{lem:longHammingNew}.\hfill\qed.

\subsection{Preparations for the proof of Lemma~\ref{lem:longHamming}}\label{sec:longHammingPre}
\subsubsection*{Outline and preparation} The proof of Lemma~\ref{lem:longHamming} depends on a combinatorial lemma, namely Lemma~\ref{lem:longHammingComb}, and a new lemma, Lemma~\ref{lem:longHammingSep} below. Section \ref{sec:longHammingSep} is devoted to the proof of Lemma~\ref{lem:longHammingSep}. The idea of the proof of of Lemma~\ref{lem:longHamming} is to construct a family of intervals satisfying the conditions in Lemma~\ref{lem:longHammingComb} and then obtain at least one long ``good'' interval whose size is comparable to $R$. We remark that our Lemma~\ref{lem:longHammingSep} is a generalization of Lemma~4 in \cite{ratner1979cartesian}. 

\subsubsection{Two relations}\label{sec:twoRelations}
Let $K\subset \mathbf{G}_2/\Gamma_2$ be a compact set, $x\in \mathbf{G}_2/\Gamma_2$, $y\in K$ and $\epsilon<~\operatorname{inj}(K)$. Suppose that for some $t,s\geq0$, we have 
\[
d_{\mathbf{G}_2/\Gamma_2}(x,y)<\epsilon, \ \ d_{\mathbf{G}_2/\Gamma_2}(u^{(2)}_t.x,u^{(2)}_s.y)<\epsilon, \text{ and }u^{(2)}_s.y\in K.
\]
Lift $x$ and $y$ to points $\tilde x$ and $\tilde y$ in $\mathbf{G}_2$, so that $d_{\mathbf{G}_2}(\tilde x,\tilde y)<\epsilon$. Then there exists a unique $\gamma\in\Gamma_2$ such that 
\[
d_{\mathbf{G}_2}(u^{(2)}_t.\tilde x,u^{(2)}_s.\tilde y\gamma)<\epsilon.
\]
Note that the conjugacy class $[\gamma]$ of $\gamma$ does not depend on the lifts $\tilde x,\tilde y$ we choose for $x$ and $y$ (as long as $d_{\mathbf{G}_2}(\tilde x,\tilde y)<\epsilon$).
If $x,y,s,t,[\gamma]$ satisfy the above we shall write
\[(x,y)\overset{[\gamma]}{\leadsto}(u^{(2)}_t.x,u^{(2)}_s.y);
\]
we also write 
\begin{itemize}
    \item $(x,y)\overset{e}{\leadsto}(u^{(2)}_t.x,u^{(2)}_s.y)$ for $(x,y)\overset{[e]}{\leadsto}(u^{(2)}_t.x,u^{(2)}_s.y)$,
    \item $(x,y)\overset{\neq e}{\leadsto}(u^{(2)}_t.x,u^{(2)}_s.y)$ for  $(x,y)\overset{[\gamma]}{\leadsto}(u^{(2)}_t.x,u^{(2)}_s.y)$ for $\gamma\neq e$.
\end{itemize}

\subsubsection{Statement of Lemma~\ref{lem:longHammingSep}}\label{first statement of lemma 8.1}
For convenience, we shall use in this section the notation $x_t=u_t^{(2)}.x$ and $y_t=u_t^{(2)}.y$ for $x,y\in \mathbf{G}_2/\Gamma_2$. 

\medskip

Let $\tilde{R}>0$ such that the compact set $\Knew\label{0401Kbase}\subset \mathbf{G}_2/\Gamma_2$\index{$\Kold{0401Kbase}$} defined as 
    \begin{equation*}
        \Kold{0401Kbase}=\{x\in \mathbf{G}_2/\Gamma_2:d_{\mathbf{G}_2/\Gamma_2}(x,e)\leq\tilde{R}\}
    \end{equation*} 
satisfy $m_2(\Kold{0401Kbase})>1-10^{-90}$. Denote 
\begin{equation*}
r_{\Kold{0401Kbase}}=\min\left(\operatorname{inj}(\Kold{0401Kbase}),10^{-10}\right).
\end{equation*} 
Let $\Onew\label{040KbaseNew}$\index{$\Oold{040KbaseNew}$} be the interior of $\Kold{0401Kbase}$. Since $\Oold{040KbaseNew}$ is measurable, the regularity of measure $m_2$ gives that there exists a compact set 
\begin{equation}\label{eq:kpreHat}
\Kold{0402KbaseNewNew}\subset\Oold{040KbaseNew}
\end{equation}
such that $m_2(\Knew\label{0402KbaseNewNew})>1-10^{-89}$\index{$\Kold{0402KbaseNewNew}$, Equation~\eqref{eq:kpreHat}}.

\begin{lem:longHammingSep}
Given an $\epsilon\in(0,10^{-40})$, there exist a compact set $\Knew\label{0403KlHS2}\subset \mathbf{G}_2/\Gamma_2$\index{$\Kold{0403KlHS2}$, Lemma~\ref{lem:longHammingSep}} with $m_2(\Kold{0403KlHS2})>1-\epsilon$ and constants $\Rnew\label{040RlHS1},\Cnew\label{040ClHS1},\wnew\label{040wlHS1}>0$\index{$\Rold{040RlHS1}$, Lemma~\ref{lem:longHammingSep}}\index{$\Cold{040ClHS1}$, Lemma~\ref{lem:longHammingSep}}\index{$\wold{040wlHS1}$, Lemma~\ref{lem:longHammingSep}} such that for any  sufficiently small $\delta >0$, if the following conditions hold
\begin{enumerate}
    \item $y\in \Kold{0403KlHS2}$ and  $x\in\Kak(R,\delta,y)$ for some $R\geq \Rold{040RlHS1}$, 
    \item $y_s\in \Oold{040KbaseNew}$ and  $x_{t}\in\Kak(R,\delta,y_s)$ for some $t,s\geq R$,
    \item $(x,y)\overset{\neq e}{\leadsto}(x_{t},y_s)$,
\end{enumerate}
then we have
\[
\max(t,s)> \Cold{040ClHS1}R^{1+\wold{040wlHS1}}.
\]
\end{lem:longHammingSep}

\noindent
This lemma will be used to show that the ``good'' intervals are well separated from one another. The proof of Lemma~\ref{lem:longHammingSep} will be given in~\S\ref{sec:longHammingSep}.

\subsection{Proof of Lemma~\ref{lem:longHamming} assuming Lemma~\ref{lem:longHammingSep}}\label{sec:longHamming}
In order to prove Lemma~\ref{lem:longHamming}, we need to construct a collection of families of ``good'' intervals on the time interval $[-R,R]$ of points $y$ in a suitable set of large measure, and then show that all the assumptions of Lemma~\ref{lem:longHammingComb} are satisfied. 

The first assumption of Lemma~\ref{lem:longHammingComb}, namely that for each point $y$ in the good set, the intervals we will construct cover most of $[-R,R]$, will follow from the construction algorithm of the intervals. 

In order to fulfill the second assumption of Lemma~\ref{lem:longHammingComb}, namely that the distance between two good intervals from different families are large, we use Lemma~\ref{lem:longHammingSep}. To verify that Lemma~\ref{lem:longHammingSep} is applicable, additional work is needed.

\subsubsection{Construction of constants and sets}\label{sec:constants and sets}
Fix an $\eta\in(0,10^{-3})$. The choices of constants and sets are specified as following:
\begin{enumerate}[label=\textup{(\roman*)}]
    \item Let $\Kold{0403KlHS2},\Rold{040RlHS1},\Cold{040ClHS1},\wold{040wlHS1}$ be defined as in Lemma~\ref{lem:longHammingSep} for $\epsilon=10^{-50}$.
    \item\label{item:choiceN1longHamming} Let $N_1\geq\Rold{040RlHS1}$ such that for every $R\geq N_1$, we have  
    \begin{equation*}
        \frac{1}{2}R^{1+\wold{040wlHS1}}\geq R.
    \end{equation*}
    \item Let $\Rold{021RlC1}$ be defined as in Lemma~\ref{lem:longHammingComb} for $w=\frac{\wold{040wlHS1}}{2}$, $C=\frac{\Cold{040ClHS1}}{2}$ and $n_{\max}=3(\dim\mathfrak{g}_2)^2$. Then we define   
    \[
    N_2=\max(2\Rold{021RlC1},10^{30}(\dim\mathfrak{g}_2)^2N_1).
    \]
    \item\label{eq:returnLar}  Let $K_1\subset \mathbf{G}_2/\Gamma_2$ be a compact set with $m_2(K_1)\geq 1-10^{-20}$ and $\kappa_1\in(0,10^{-10})$ such that the following holds
    \begin{itemize}
        \item For any $y\in K_1$, $x \in \mathbf{G}_2/\Gamma_2$ with $d_{\mathbf{G}_2/\Gamma_2}(x,y)<\kappa_1$ and $\tilde{x}$ and $\tilde{y}$ are lifts of $x$ and $y$ satisfying $d_{\mathbf{G}_2}(\tilde{x},\tilde{y})<\kappa_1$, if $s$, $t$ and $e \neq \gamma \in \Gamma_2$ satisfy
    \begin{equation*}
        d_{\mathbf{G}_2}(u_{t}^{(2)}.\tilde{x},u_{s}^{(2)}.\tilde{y}\gamma)<\kappa_1,\qquad u_s^{(2)}.y\in\Kold{0401Kbase},
    \end{equation*}
    then $\max(\absolute{t},\absolute{s})> 2N_2$. 
    \end{itemize}

    \item\label{item:ergodicGoodSety} Recall that $\Kold{0402KbaseNewNew}$ is defined in \eqref{eq:kpreHat}. Let 
\[K_2=\Kold{0402KbaseNewNew}\cap\Kold{0403KlHS2}\cap K_1.
    \]
    By the Pointwise Ergodic Theorem for $u_t^{(2)}$ and $\chi_{K_2}$, there exist a compact set $\Kold{036Klh1}\subset \mathbf{G}_2/\Gamma_2$ with $m_2(\Kold{036Klh1})>1-10^{-4}\eta$  and $N_3>0$ such that for every $y\in \Kold{036Klh1}$ and $R\geq N_3$
    \begin{equation*}
    \absolute{\{t\in[-R,R],y_t\in K_2\}}\geq(1-10^{-10})2R.
    \end{equation*}
    \item Finally let
    \begin{gather*}
    \deold{035delm1}=(10^{100}\dim\mathfrak{g}_2)^{-10}\Cold{038Clh1}^{-2}\min\left(\kappa_1,\eold{014EMa}, 10^{-40}\eta\right),\\
    \Rold{036Rlh1}=\max\left(N_2,N_3,10^{30}\right),
    \end{gather*}
    where $\eold{014EMa}$ is defined in Lemma~\ref{lem:matchingFunction}.
\end{enumerate}

\subsubsection{Preparation for the construction algorithm} Fix a $\delta\in(0,\deold{035delm1})$ and an $\epsilon\in(0,10^{-40})$. Suppose $y\in \Kold{036Klh1}$, $x\in \mathbf{G}_2/\Gamma_2$ and $(x,y)$ are $(\delta,\epsilon,R)$-two sides matchable with a $C^1$-matching function $h$ and matching set $A$, where $R\geq \Rold{036Rlh1}$. 
Let $J_y=\{t\in[-R,R]:y_t\in K_2\}$ and set $\tilde{A}={A}\cap J_y$.
It follows from item \ref{item:ergodicGoodSety} on p.~\pageref{item:ergodicGoodSety} that $l(\tilde{A})\geq(1-10^{-9})2R$.

\medskip

For any $s\in\tilde{A}$, let $\tilde{x}_{h(s)},\tilde{y}_s\in \mathbf{G}_2$ be lifts of $x_{h(s)}$ and $y_s$ such that $d_{\mathbf{G}_2}(\tilde{x}_{h(s)},\tilde{y}_s)<\delta$. Let 
\[
I_{\tilde{x}_{h(s)},\tilde{y}_s}=\{t\geq 0:d_{\mathbf{G}_2}(u_{h(t)}.\tilde{x},u_{t}.\tilde{y})<4\delta\}.
\]
By Lemma~\ref{lem:numConnectedCom}, for every $s\in\tilde{A}$, there exist $k(s)$, $b_i(s)$ and $d_i(s)$ satisfying
\begin{gather*}
    1\leq k(s)\leq3(\dim\mathfrak{g}_2)^2,\\
        b_{i}(s)\leq d_{i}(s)<b_{i+1}(s)\leq d_{i+1}(s) \text{ for }i=1,\ldots,k(s)-1,
\end{gather*}
such that
\begin{equation}\label{eq:ksDef}
I_{\tilde{x}_{h(s)},\tilde{y}_s}\subset\bigcup_{i=1}^{k(s)}[b_{i}(s),d_{i}(s)].
\end{equation}
We use the following algorithm,  to construct a family of subintervals of the time interval $[-R,R]$ covering $\tilde{A}$.

\subsubsection{Algorithm to construct a cover of $\tilde A$ by intervals}
Following are the steps of the algorithm.

\begin{algorithm}[H]
\caption{}\label{alg:cap}
\begin{algorithmic}
\State $m \gets 0$, $r_{-1}\gets-R-1$. \Comment{Initialization}

\While{$r_{m-1}< \sup\{t\in \tilde{A}\}$} \Comment{Repeat if $\tilde A$ is not yet covered}

\State $t_{m}\gets\inf\{t\in \tilde{A}:t>r_{m-1}\},$
\If{$m=0$}
\State $E_{m}\gets\emptyset$;
\Else
\State $E_{m}\gets\{i\in\{0,\ldots,m-1\}:(x_{h(t_i)},y_{t_i})\overset{e}{\leadsto}(x_{h(t_{m})},y_{t_{m}})\}.$
\EndIf

\\

\If{$E_{m}=\emptyset$}
\begin{alignat*}{3}
        r_{m}\gets&\sup\{t\in  \tilde{A}\cap&&[t_{m},t_{m}+d_{1}(t_{m})]: \\& &&d_{\mathbf{G}_2}(u_{h(t)-h(t_{m})}^{(2)}.\tilde{x}_{h(t_{m})},u_{t-t_{m}}^{(2)}.\tilde{y}_{t_{m}})\leq\delta\}.
\end{alignat*}\Comment{$d_{1}(t_{m})$ as in \eqref{eq:ksDef}}
\Else
    \State $i_m\gets\min(E_m)$.
    \State Find $q$ in the range $1 \dots k(t_{i_{m}})$ for which 
     \[
     t_{m}\in[t_{i_{m}}+b_{q}(t_{i_{m}}),t_{i_{m}}+d_{q}(t_{i_{m}})].
     \]
     \Comment{Such $q$ exists by \eqref{eq:ksDef}}
     \State Let
     \begin{alignat*}{3}
     r_{m}\gets&\sup\{t\in  \tilde{A}\cap&&[t_{i_{m}}+b_{q}(t_{i_{m}}),t_{i_{m}}+d_{q}(t_{i_{m}})]:\\
     & &&d_{\mathbf{G}_2}(u_{h(t)-h(t_{m})}^{(2)}.\tilde{x}_{h(t_{m})},u_{t-t_{m}}^{(2)}.\tilde{y}_{t_{m}})\leq\delta\}.
     \end{alignat*}
\EndIf
\State $B_{m}\gets[t_{m},r_{m}]$.
\State $m \gets m + 1$
\EndWhile
\State Output the cover $\{B_i:0\leq i \leq m\}$ of $\tilde A$.
\end{algorithmic}
\end{algorithm}

\subsubsection{Tidying the output of Algorithm~\ref{alg:cap}}
By the algorithm above, we obtain a family of pairwise disjoint intervals $\beta=\{B_0,\ldots, B_m\}$, and in our notations for every $B_i\in\beta$, $l(B_{i})=r_{i}-t_{i}$. Extending the definitions of $\overset{e}{\leadsto}$ and $\overset{\neq e}{\leadsto}$ in \S\ref{sec:twoRelations}, we say that for $B_i,B_j\in\beta$
\begin{enumerate}[label=\alph*)]
    \item $B_i\overset{e}{\leadsto}B_j$ if $(x_{h(t_i)},y_{t_i})\overset{e}{\leadsto}(x_{h(t_j)},y_{t_j})$;
    \item $B_i\overset{\neq e}{\leadsto}B_j$ if $(x_{h(t_i)},y_{t_i})\overset{\neq e}{\leadsto}(x_{h(t_j)},y_{t_j})$.
\end{enumerate}

\medskip

Since $k(s) \leq 3(\dim\mathfrak{g}_2)^2$ for any $s \in \tilde A$, the output of Algorithm~\ref{alg:cap} satisfies that for each $B_i\in\beta$, there are at most $3(\dim\mathfrak{g}_2)^2$ different intervals in $\beta$ which are $\overset{e}{\leadsto}$ with $B_i$. 

Define $\beta_1=\{B_i\in\beta:B_i\overset{e}{\leadsto} B_1\}$. Assume that $\beta_1,\ldots,\beta_k$ have already been constructed, let $B_{l_{k+1}}$ be the interval with the smallest index in $\beta\setminus(\cup_{j=1}^k\beta_j)$ and define
\[
\beta_{k+1}=\{B_i\in\beta\setminus(\cup_{j=1}^{k}\beta_j):B_i\overset{e}{\leadsto} B_{l_{k+1}}\}.
\]
This gives a decomposition of $\beta$
\begin{equation*}
    \beta=\cup_k\beta_k.
\end{equation*}
Then for every $B_i\in\beta_p,B_j\in\beta_q$, we have $B_i\overset{e}{\leadsto}B_j$ if and only if $p=q$.

\subsubsection{Constructing the partition for Lemma~\ref{lem:longHammingComb}}
 The main result of this subsection is the following claim:
\begin{Claim}\label{claim:assumption1longHammingComb}
    Let $\mathcal{I}=\{0\leq i\leq m:l(B_i)\geq N_1\}$, then
\begin{equation*}
\sum_{i\in \mathcal{I}}l(B_i)\geq(1-10^{-8})2R.
\end{equation*}
\end{Claim}
\begin{proof}[Proof of Claim~\ref{claim:assumption1longHammingComb}]
The output of Algorithm~\ref{alg:cap} satisfies that for every $t\in \tilde{A}$, there exists $B_i\in\beta$ such that $t\in [t_i,r_i]$. This together with $l(\tilde{A})\geq(1-10^{-9})2R$ gives 
\begin{equation}\label{eq:totalBlockLength}
    \sum_{i=0}^m l(B_i)\geq(1-10^{-9})2R.
\end{equation}
Let $\mathcal{I}^c=\{0,\ldots,m\}\setminus\mathcal{I}$, then we have following two cases:
\begin{enumerate}[label=\textup{(\alph*)}]
    \item If $\cup_{i\in\mathcal{I}^c}B_i$ can be covered by only one family, e.g. $\beta_k$, then by the assumptions in \S\ref{sec:constants and sets},
    $R\geq 10^{30}(\dim\mathfrak{g}_2)^2N_1$, so as 
    \[
    \operatorname{Card}(\beta_k)\leq 3(\dim\mathfrak{g}_2)^2
    \]
    and then
    \[\sum_{i\in \mathcal{I}^c}l(B_i)\leq 3(\dim\mathfrak{g}_2)^2N_1\leq10^{-30}R,\]
which together with \eqref{eq:totalBlockLength} give the claim.
    \item 
    Suppose that we need $n>1$ different families  $\beta_{i_1},\ldots,\beta_{i_n}$ to cover $\cup_{i\in\mathcal{I}^c}B_i$, then the output of Algorithm \ref{alg:cap}, $\operatorname{Card}(\beta_{i_p})\leq 3(\dim\mathfrak{g}_2)^2$ and the assumptions in~\S\ref{sec:constants and sets} give
    \begin{equation}\label{eq:goodParl1}
    \sum_{i\in \mathcal{I}^c}l(B_i)\leq\sum_{p=1}^n\sum_{\substack
{{B_q\in\beta_{i_p}}\\{l(B_q)\leq N_1}}}l(B_q)\leq 3n(\dim\mathfrak{g}_2)^2N_1\leq nN_2.
    \end{equation}

It follows from \ref{eq:returnLar} in \S\ref{sec:constants and sets} that if $B_i=[t_i,r_i]$ and $B_{i+1}=[t_{i+1},r_{i+1}]$ are two consecutive intervals from the collection $\beta$, then 
    \begin{equation}\label{eq:largeReturnGap}
    t_{i+1}-r_{i}\geq N_2, \qquad (r_{i},t_{i+1})\subset[-R,R]\setminus\tilde A.
    \end{equation}
    Since we need $n$ different families $\beta_{i_1},\ldots,\beta_{i_n}$ to cover $\cup_{i\in\mathcal{I}^c}B_i$, this in particular gives that $\beta$ has at least $n$ different families. Each time we have two consecutive intervals from $\beta$ that belong to different families \eqref{eq:largeReturnGap} applies; since we have $n$ different families this needs to happen at least $n-1$ times. This gives that 
    \begin{equation}\label{eq:goodParl2}
    (n-1)N_2\leq l([-R,R]\setminus\tilde A).
    \end{equation}
    
    Combining $l(\tilde{A})\geq(1-10^{-9})2R$ with \eqref{eq:goodParl1} and \eqref{eq:goodParl2}
    \[
    \sum_{i\in \mathcal{I}^c}l(B_i)\leq nN_2\leq\frac{n}{n-1}10^{-9}\cdot2R,
    \]
    which together with \eqref{eq:totalBlockLength} give the claim.

\end{enumerate}
   
\end{proof}

\medskip

Let $\tilde{\beta}=\{B_i\in\beta:i\in \mathcal{I}\}$, then for each $\beta_k$, write
\begin{equation}\label{eq:betaGood}
\tilde\beta_k=\tilde\beta\cap\beta_k=\left\{J_{1}^{(k)},\ldots,J_{n_k}^{(k)}\right\},
\end{equation}
where $J_{n_p}^{(k)}\in \tilde\beta$ (if $\tilde\beta_k$ is empty, write $n_k=0$). Then define $\mathcal P_b$ as
\begin{equation}\label{eq:beta}
    \mathcal{P}_b=\bigcup_{k : \tilde\beta_k\neq\emptyset}\left\{J_{1}^{(k)},\ldots,J_{n_k}^{(k)}\right\}
\end{equation}

Write $[-R,R]\setminus\cup_{B_i\in\tilde{\beta}}B_i$ as a union of disjoint intervals
\[
[-R,R]\setminus\cup_{B_i\in\tilde{\beta}}B_i=\bigcup_{i=1}^{n_0}J_{i}.
\]
We now define a partition $\mathcal P$ of $[-R,R]$ by
\begin{equation*}
    \mathcal P=\{J_1,\ldots,J_{n_0}\}\cup\mathcal{P}_b.
\end{equation*}

From now on, we will verify the assumptions of Lemma~\ref{lem:longHammingComb} for $\mathcal{P}$ and $\mathcal P_b$. Note that Claim~\ref{claim:assumption1longHammingComb} above verifies the \textbf{first assumption} of Lemma~\ref{lem:longHammingComb} for $\mathcal{P}_b$.

\subsubsection{Assumptions $(1)$ and $(3)$ of Lemma~\ref{lem:longHammingSep}:} If $\tilde \beta_k \neq \emptyset$, we write for any $1\leq p \leq n_k$ \ $J_p^{(k)}=\left[t_p^{k},r_p^{k}\right]$. We also write 
\[
t^{k} = t_1^{k}, \qquad \tau^k = h(t_1^{k})\quad  \text{ and }\quad r^{k}=r_1^{k},\qquad \sigma^k=h(r_1^{k}).
\]
Let $\tilde{x}_{\tau^k},\tilde{y}_{t^{k}}\in \mathbf{G}_2$ be lifts of $x_{\tau^k}, y_{t^{k}}$ such that 
 \[d_{\mathbf{G}_2}\bigl(\tilde{x}_{\tau^k},\tilde{y}_{t^{k}}\bigl)\leq\delta.
 \]
Let $g^{(k)}\in \mathbf{G}_2$ satisfy $\tilde{x}_{\tau^k}=g^{(k)}.\tilde{y}_{t^{k}}$ and $\phi_k$ the best matching function defined in \eqref{eq:bestMatchingFun} for~$g^{(k)}$.

The following claim, which will complete the verification of $(1)$ and $(3)$ of Lemma~\ref{lem:longHammingSep}.
\begin{Claim}\label{claim:assumption 1 and 3 of longHammingSep}
For $s=t_p^{k}$ and $r_p^{k}$, let $\tilde{x}_{h(s)}$, $\tilde{y}_{s}$ and $g\in\mathbf{G}_2$ be defined as
\begin{gather*}
    \tilde{x}_{h(s)}=u_{h(s)-\tau^k}^{(2)}.\tilde{x}_{\tau^k}, \quad \tilde{y}_{s}=u_{s-t^{k}}^{(2)}.\tilde{y}_{t^{k}}, \quad \tilde{x}_{h(t_p^{k})}=g.\tilde{y}_{t_p^{k}}.
\end{gather*}
Then there exists $C>0$ depends only on $\mathbf{G}_2$ such that
\begin{gather*}
\tilde{x}_{h(s)}\in\Kak\left(l\left(J_{p}^{(k)}\right),C\delta\right).\tilde{y}_{s}.
\end{gather*}
\end{Claim}
\begin{proof}[Proof of Claim~\ref{claim:assumption 1 and 3 of longHammingSep}]
As $J^{(k)}_p \in \beta$, by Proposition~\ref{prop:matchKak} and Lemma~\ref{lem:numConnectedCom}, the way the elements of $\beta$ were constructed in Algorithm \ref{alg:cap} implies that for both $s=t_p^{k}$ and $s=r_p^{k}$ we have
\begin{equation}\label{eq:expCon1}
\begin{aligned}
u_{\phi_{k}(s-t^{k})}^{(2)}.\tilde{x}_{\tau^k}&\in\Kak\left(l\left(J_{p}^{(k)}\right),\kappa\delta\right)u_{s-t^{k}}^{(2)}.\tilde{y}_{t^{k}}
\end{aligned}
\end{equation}
for some constant $\kappa$ that depends only on $\mathbf{G}_2$ (namely $2\Cold{012matchKak}\Cold{0192Cnum}$).

Recall Algorithm \ref{alg:cap} also gives that for $s=t_p^{k}$ and $r_p^{k}$
\begin{equation*}
    \begin{aligned}
    &d_{\mathbf{G}_2}\left(u_{h(s)-\tau^k}^{(2)}.\tilde{x}_{\tau^k},u_{s-t^{k}}^{(2)}.\tilde{y}_{t^{k}}\right)\leq\delta,
    \end{aligned}
\end{equation*}
which together with right invariance of $d_{\mathbf{G}_2}$ and \eqref{eq:expCon1} guarantee that for $s=t_p^{k}$ and $r_p^{k}$ 
\begin{equation}\label{eq:liftClose}
    \begin{aligned}
   &\absolute{h\bigl(s\bigr)-\tau^k-\phi_{k}\left(s-t^{k}\right)}\leq2\kappa\delta.  \end{aligned}
\end{equation} 
This together with \eqref{eq:expCon1} gives the Claim~\ref{claim:assumption 1 and 3 of longHammingSep}.

\end{proof}

\subsubsection{Assumption $(2)$ of Lemma~\ref{lem:longHammingComb}}
Recall that $l\left(J_p^{(k)}\right)\geq N_1\geq \Rold{040RlHS1}$ for every $J_p^{(k)}=\left[t_{p}^{k},r_{p}^{k}\right]\in\mathcal{P}_b$, and the choice of $\tilde{A}$ implies that 
\[
y_{t_{p}^{k}},y_{r_{p}^{k}}\in K_2\subset \Kold{0403KlHS2}.
\]
These in particular guarantee that $y_{t_{p}^{k}}$ and $y_{r_{p}^{k}}$ both stay in the good set where Lemma~\ref{lem:longHammingSep} will hold.

Suppose that
\[
J_{p}^{(k)}=\left[t_{p}^{k},r_{p}^{k}\right],\qquad J_{q}^{(\ell)}=\left[t_{q}^{\ell},r_{q}^{\ell}\right],\qquad r_{p}^{k}<t_{q}^{\ell},\quad k\neq\ell.
\]
Set
\[
\tau_{p}^{k}=h\left(t_{p}^{k}\right),\qquad \tau_{q}^{\ell}=h\left(t_{q}^{\ell}\right),\qquad \sigma_p^k=h\left(r_p^k\right) \qquad  \sigma_q^{\ell}=h\left(r_q^{\ell}\right).
\]
Then $y_{t_{p}^{k}},y_{r_{p}^{k}},y_{t_{q}^{\ell}},y_{r_{q}^{\ell}}\in\Kold{0403KlHS2}$ together with Claim~\ref{claim:assumption 1 and 3 of longHammingSep} guarantee that both pairs
\begin{equation*}
    \begin{aligned}
        \left(x_{\tau_{p}^{k}},y_{t_{p}^{k}}\right) \text{ and } \left(x_{\tau_{q}^{\ell}},y_{t_{q}^{\ell}}\right),\qquad 
        \left(x_{\sigma_{p}^{k}},y_{r_{p}^{k}}\right) \text{ and } \left(x_{\sigma_{q}^{\ell}},y_{r_{q}^{\ell}}\right)
    \end{aligned}
\end{equation*}
satisfy all conditions of Lemma~\ref{lem:longHammingSep} and thus
\begin{equation}\label{eq:expGapIni}
    \begin{aligned}
        &t_{q}^{\ell}-t_{p}^{k}\geq\Cold{040ClHS1}\min\left(\left[l\left(J_{p}^{(k)}\right)\right]^{1+\wold{040wlHS1}},[l(J_{q}^{(\ell)})]^{1+\wold{040wlHS1}}\right),\\
        &r_{q}^{\ell}-r_{p}^{k}\geq\Cold{040ClHS1}\min\left(\left[l\left(J_{p}^{(k)}\right)\right]^{1+\wold{040wlHS1}},\left[l\left(J_{q}^{(\ell)}\right)\right]^{1+\wold{040wlHS1}}\right).
    \end{aligned}
\end{equation}

Recall that
\begin{equation*}
    \begin{aligned}
        r_{p}^{k}-t_{p}^{k}=l\left(J_{p}^{(k)}\right),\qquad r_{q}^{\ell}-t_{q}^{\ell}= l\left(J_{q}^{(\ell)}\right).
    \end{aligned}
\end{equation*}
This together with \eqref{eq:expGapIni} and the choice of $N_1$ in item~\ref{item:choiceN1longHamming} on p.~\pageref{item:choiceN1longHamming} gives
\[
t_{q}^{\ell}-r_{p}^{k}> \frac{1}{2}\Cold{040ClHS1}\min\left(\left[l\left(J_{p}^{(k)}\right)\right]^{1+\wold{040wlHS1}},\left[l\left(J_{q}^{(\ell)}\right)\right]^{1+\wold{040wlHS1}}\right).
\]

Moreover, since $J_{p}^{(k)}\overset{\neq e}{\leadsto}J_{q}^{(\ell)}$ and $y_{r_{i}},y_{t_j}\in K_2$, we obtain from \ref{eq:returnLar}
\[
\max\left(\tau_{q}^{\ell}-\sigma_{p}^{k},t_{q}^{\ell}-r_{p}^{k}\right)> N_2\geq2\Rold{021RlC1}.
\]
Recall $\absolute{h'(t)-1}<10^{-20}$ for every $t\in[-R,R]$, then 
\[
t_{q}^{\ell}-r_{p}^{k}\geq\max\left(\Rold{021RlC1},0.5\Cold{040ClHS1}\min\left(\left[l\left(J_{p}^{(k)}\right)\right]^{1+0.5\wold{040wlHS1}},\left[l\left(J_{q}^{(\ell)}\right)\right]^{1+0.5\wold{040wlHS1}}\right)\right),
\]
which shows that the \textbf{second assumption} of Lemma~\ref{lem:longHammingComb} is satisfied.

\subsubsection{Application of Lemma~\ref{lem:longHammingComb}}

Applying Lemma~\ref{lem:longHammingComb} for $\mathcal{P}$ and $\mathcal{P}_b$ (cf.~\eqref{eq:beta}), we obtain that
\begin{gather}
\text{there exists $J_{p_0}^{(k_0)}\in\mathcal{P}_b$ such that $l\left(J_{p_0}^{(k_0)}\right)\geq\frac{R}{32(\dim\mathfrak{g}_2)^2}$,}\label{eq:longbb}\\
    \sum_{i=1}^{n_{k_0}}l(J_{i}^{k_0})\geq\frac{3}{2}R.\label{eq:3/4RLarge}
\end{gather}
Recall $\tilde{x}$ and $\tilde{y}$ are the lifts of $x$ and $y$ such that $d_{\mathbf{G}_2}(\tilde{x},\tilde{y})<\delta$. If
\[
J_{p_0}^{(k_0)}=\left[t_{p_0}^{k_0},r_{p_0}^{k_0}\right],
\] 
then the Algorithm \ref{alg:cap} and Claim~\ref{claim:assumption 1 and 3 of longHammingSep} give $\gamma_0\in\Gamma_2$ 
such that 
\begin{equation*}
    \begin{aligned}
    u_{h(t_{p_0}^{k_0})}^{(2)}.\tilde{x}=g_1 u_{t_{p_0}^{k_0}}^{(2)}.\tilde{y}\gamma_0,\qquad  u_{h(r_{p_0}^{k_0})}^{(2)}.\tilde{x}=g_2 u_{r_{p_0}^{k_0}}^{(2)}.\tilde{y}\gamma_0,
    \end{aligned}
\end{equation*}
where $g_1,g_2\in \Kak\left(l\left(J_{p_0}^{(k_0)}\right),\kappa\delta\right)$ for some  $\kappa>0$ depends only on $\mathbf{G}_2$. This together with Lemma~\ref{lem:matchingFunction}, Lemma~\ref{lem:numConnectedCom} and \eqref{eq:longbb} gives a constant $\kappa_2>0$ such that for every $\absolute{t}\leq R$
\begin{equation}\label{eq:RKakBall}
u_{\phi_{g_1}(t-t_{p_0}^{k_0})+h(t_{p_0}^{k_0})}^{(2)}.\tilde{x}\in\Kak\left(2R,\kappa_2\delta\right)u_t^{(2)}.\tilde{y}\gamma_0.
\end{equation}
where $\phi_{g_1}$ is the best matching function defined in \eqref{eq:bestMatchingFun}.

Let $\alpha_0=t_1^{k_0}$ and $\beta_0=r_{n_{k_0}}^{k_0}$, then \eqref{eq:3/4RLarge} imply 
\[
\absolute{\alpha_0-\beta_0}\geq3R/2
\]
and
\begin{gather}
d_{\mathbf{G}_2}\left(u_{h(\alpha_0)}^{(2)}.\tilde{x},u_{\alpha_0}^{(2)}.\tilde{y}\gamma_0\right)\leq\delta,\quad d_{\mathbf{G}_2}\left(u_{h(\beta_0)}^{(2)}.\tilde{x},u_{\beta_0}^{(2)}.\tilde{y}\gamma_0\right)\leq\delta.\label{eq:halphabeta}
\end{gather}
These together with \eqref{eq:RKakBall} give for some $\kappa_3>0$
\begin{equation*}
\absolute{\phi_{g_1}(\alpha_0-t_{p_0}^{k_0})+h(t_{p_0}^{k_0})-h(\alpha_0)}\leq\kappa_3\delta,
\end{equation*}
which shows that for some $\kappa_4>0$
\[
u_{h(\alpha_0)}^{(2)}.\tilde{x}=g_3.u_{\alpha_0}^{(2)}.\tilde{y}\gamma_0,\qquad g_3\in\Kak(2R,\kappa_4\delta).
\]
By definition of Kakutani-Bowen ball, the above equation give $\kappa_5>0$ such that for every $\absolute{t}\leq R$
\[
u_{\phi_{g_3}(t-\alpha_0)+h(\alpha_0)}^{(2)}.\tilde{x}\in\Kak(R,\kappa_5\delta)u_t^{(2)}.\tilde{y}\gamma_0,
\]
where $\phi_{g_3}$ is the best matching function defined in \eqref{eq:bestMatchingFun}. This together with \eqref{eq:halphabeta} gives
\begin{gather*}
    \absolute{\phi_{g_3}(\beta_0-\alpha_0)+h(\alpha_0)-h(\beta_0)}\leq2\kappa_5\delta
\end{gather*}
and thus finish the proof of Lemma~\ref{lem:longHamming}.\hfill\qed

\subsection{An infinite matching corollary}\label{sec:infiniteMatching}
\subsubsection{An auxiliary lemma}
The following lemma follows directly from polynomial divergence of $u_t^{(2)}$-flow, which will be used in the proof of Corollary~\ref{cor:inifinteMatching}.
\begin{Lemma}\label{lem:polyDivergence}
Let $\tilde x,\tilde y\in \mathbf{G}_2$, $\gamma_R\in\Gamma_2$, $\epsilon,\delta>0$ sufficiently small and $R \geq \Rold{036Rlh1}$. Let $h:\R\to\R$ be a $C^1$ function satisfying $\absolute{h'(t)-1}<\epsilon$ for all $\absolute{t}\leq R$.  Suppose that for some~$\absolute{s_R}\leq R$ and $L_R\geq R$, 
\begin{equation*}
    \begin{aligned}
   u_{h(s_R)}^{(2)}\tilde{x}\in \Kak(L_R, \Cold{038Clh1}\delta)u_{s_R}^{(2)}\tilde{y}\gamma_R,\\
    \end{aligned}
\end{equation*}
then there exists $\absolute{t_R}\leq0.01R$ such that
\[
\tilde{x}\in\Kak(R,\Cold{036Clm}\delta)u_{t_R}^{(2)}\tilde{y}\gamma_R.
\]
Here $\Cold{036Clm}>0$ is a constant depending only on $\mathbf{G}_2$. 
\end{Lemma}
\begin{proof}[Proof of Lemma~\ref{lem:polyDivergence}]

Let $\phi_{g_0}$ be the best matching function defined in~\eqref{eq:bestMatchingFun} for $g_0$. 
Since $L_R\geq R$, $g_0 \in\Kak(L_R,\Cold{038Clh1}\delta)$, Definition~\ref{def:Kakball} (the definition of Kakutani-Bowen balls), Lemma~\ref{lem:matchingFunction}, and choice of $\delta$,  we have that for $t\in[-10R,10R]$
\[
\absolute{\phi'_{g_0}(t)-1}<\sqrt{\kappa_1\delta},
\]
where $\kappa_1>0$ is large enough and depends only on $\mathbf{G}_2$.

Let $\bar{t}\in[-2R,2R]$ be such that $\phi_{g_0}(\bar{t})+h(s_R)=0$, then by assumption $\absolute{h'(t)-1}<\epsilon$, and in view of the above estimate on $\phi'_{g_0}(t)$ we have that 
\begin{equation}\label{eq:trsize}
\absolute{s_R+\bar{t}}<2(\epsilon+\sqrt{\kappa_1\delta})R.
\end{equation}

\medskip

Let $t_R=s_R+\bar{t}$ and $g_1$ satisfy 
\[
\tilde{x}=g_1.u_{t_R}^{(2)}.\tilde{y}\gamma_R.
\]
Then Lemma~\ref{lem:matchingFunction}, Lemma~\ref{lem:polCoe} and $g_0 \in\Kak(L_R,\Cold{038Clh1}\delta)$ give that for some constant $\Cold{036Clm}\geq \kappa_1$  depending only on $\mathbf{G}_2$
\begin{gather*}
    \absolute{\vartheta_{\ba_2}(g_1)}<\Cold{036Clm}\delta,  \qquad g_1^{\mathfrak{z}}g_1^{\tr}\in\Bow(R,\Cold{036Clm}\delta),\\
    \vartheta_{\bou_2}(g_1)=\vartheta_{\bou_2}(g_0)\left(1-\vartheta_{\bou_2}(g_0)e^{2\vartheta_{\ba_2}(g_0)}\bar{t} \right)\in\left[-\frac{\Cold{036Clm}\delta}{R},\frac{\Cold{036Clm}\delta}{R}\right],
    \end{gather*}
where $\vartheta_{\bou_2}$, $\vartheta_{\ba_2}$, $g_1^{\mathfrak{z}}$ and $g_1^{\tr}$ are defined in \eqref{eq:gDecom1}. Notice by definition of Kakutani-Bowen balls (cf.~Definition \ref{def:Kakball}), the above indeed gives
\begin{equation*}
    g_1\in\Kak(R,\Cold{036Clm}\delta),
\end{equation*} 
which together with \eqref{eq:trsize} complete the proof of Lemma~\ref{lem:polyDivergence}.
\end{proof}

\subsubsection{Proof of Corollary~\ref{cor:inifinteMatching}}
\begin{proof}[Proof of Corollary~\ref{cor:inifinteMatching}]
    Assume that $y\in\Kold{036Klh1}$, $x\in \mathbf{G}_2/\Gamma_2$ and $x,y$ are $(\delta,\epsilon,R)$-two sides matchable for every $R\in[L_1,L_2]$ with a same $C^1$-matching function. Let $\{\lambda_q\}_{q\in\N}$ be a sequence of real numbers satisfying 
\begin{gather}
    \lambda_0=L_1, \ \ 
    \lambda_{q_0}=L_2\text{ for some }q_0\in\N,\nonumber\\
    \lambda_{q}<\lambda_{q+1}<\frac{9}{8}\lambda_q\text{ for every }q\in\N,\ \
    \lambda_q\to+\infty\text{ as }q\to+\infty.\label{eq:lambdaRel}
\end{gather}
Let $J_{(q)}$ be the interval obtained as in~\eqref{eq:longbb} for $R=\lambda_q$ (in the notations of that equation it is $J_{p_0}^{(k_0)}$; we omit in our notation here these two indices, but make explicit the dependence of this interval on $q$). Let $\tilde \beta_{(q)}$ be the corresponding good family of intervals with $J \overset{e}{\leadsto} J_{(q)}$ for $J \in \tilde \beta_{(q)}$ as in \eqref{eq:betaGood}. Write $J_{(q)}$ as 
\[
J_{(q)}=\left[t^{(q)},s^{(q)}\right].
\] 
Notice that Lemma~\ref{lem:longHammingComb} gives that 
\begin{equation*}\label{eq:longBblocksize}
    \sum_{J\in\tilde{\beta}_{(q)}}l(J)\geq\frac{3}{4}\lambda_q.
\end{equation*}
This, together with \eqref{eq:lambdaRel} and our assumption that we have one matching function $h(\cdot)$ that works for $\lambda_0$, \dots, $\lambda_{q_0}$, gives that
\begin{equation}\label{eq:longBblock}
    J_{(q)}\overset{e}{\leadsto} J_{(q+1)}.
\end{equation}

Recall Lemma~\ref{lem:longHamming} gives that for some $\gamma_{q}\in\Gamma_2$ and $s_q\leq-\lambda_q/2$
\begin{equation}\label{eq:lambdaMkak}
\begin{aligned}
u^{(2)}_{h(s_q)}\tilde{x}\in&\Kak(\lambda_q,\Cold{038Clh1}\delta)u_{s_q}^{(2)}\tilde{y}\gamma_{q},\\
u_{h(s_q+\lambda_q)}^{(2)}\tilde{x}\in&\Kak(\lambda_q, \Cold{038Clh1}\delta)u_{s_q+\lambda_q}^{(2)}\tilde{y}\gamma_{q}.
\end{aligned}
\end{equation}
Equation \eqref{eq:longBblock} implies that there is a $\gamma \in \Gamma_2$ so that
\begin{equation*}
    \gamma_{0}=\dots=\gamma_{q_0}=\gamma.
\end{equation*}
Together with \eqref{eq:lambdaMkak} and Lemma~\ref{lem:polyDivergence}, this implies that for every $q=0,\ldots,q_0$, there exists $\absolute{t_q}\leq 0.01\lambda_q$ such that
\begin{equation}\label{eq:xyqKak}
\tilde{x}\in\Kak\left(\lambda_q,\Cold{038Clh1}\delta\right)u_{t_q}^{(2)}\tilde{y}\gamma.
\end{equation}
This in particular gives that there is $\kappa>1$ depends only on $\mathbf{G}_2$ such that for every $p,q\in\{0,\ldots,q_0\}$
\[
\absolute{t_{p}-t_q}\leq \kappa\delta.
\]
Combining  with \eqref{eq:lambdaMkak} and~\eqref{eq:xyqKak}, we obtain
    \begin{equation}\label{eq:finiteManyKak}
    \begin{aligned}
        \tilde{x}\in&\Kak(\lambda_{q_0},(\kappa+\Cold{038Clh1})\delta)u_{t_0}^{(2)}\tilde{y}\gamma,\\
u^{(2)}_{h(s_{q_0})}\tilde{x}\in&\Kak(\lambda_{q_0},(\kappa+\Cold{038Clh1})\delta)u_{s_{q_0}}^{(2)}\tilde{y}\gamma,\\
u_{h(s_{q_0}+\lambda_{q_0})}^{(2)}\tilde{x}\in&\Kak(\lambda_{q_0}, (\kappa+\Cold{038Clh1})\delta)u_{s_{q_0}+\lambda_{q_0}}^{(2)}\tilde{y}\gamma.
    \end{aligned}
   \end{equation}
This proves the first part of Corollary~\ref{cor:inifinteMatching}.

\medskip

If $L_2=+\infty$, then for a sequence satisfying \eqref{eq:lambdaRel}, the definition of Kakutani-Bowen ball (cf.~Definition \ref{def:Kakball}),  $\lim_{q\to+\infty}\lambda_{q}=+\infty$ and~\eqref{eq:finiteManyKak} guarantee that (passing to a subsequence if necessary) for
\begin{equation}\label{eq:fctCondition}
    \begin{aligned}
        f(x,y)&\in\R\text{ satisfying }\absolute{f(x,y)}\leq\kappa\Cold{038Clh1}\delta, \\
        c(x,y)&\in B_{\kappa\Cold{038Clh1}\delta}^{\mathbf{G}_2}\cap\exp(\mathfrak{w}),\\
        t_0(x,y)&\in[-L_1,L_1], \quad \text{where $t_0(x,y)$ is the $t_0$ in \eqref{eq:finiteManyKak},}
    \end{aligned}
\end{equation}
we have
\begin{equation}\label{eq:CorInfiniteMiddle}
    \tilde{y}=c(x,y)a_{f(x,y)}^{(2)}u_{t_0(x,y)}^{(2)}\tilde{x}\gamma.
\end{equation}

\medskip

If there exist $\tilde{f}(x,y)$, $\tilde{c}(x,y)$ and $\tilde{t}_0(x,y)$ also satisfying \eqref{eq:fctCondition} and 
\[
\tilde{y}=\tilde{c}(x,y)a_{\tilde{f}(x,y)}^{(2)}u_{\tilde{t}_0(x,y)}^{(2)}\tilde{x}\gamma.
\]
Combining this with with \eqref{eq:CorInfiniteMiddle}
\[
\tilde{c}(x,y)a_{\tilde{f}(x,y)}^{(2)}u_{\tilde{t}_0(x,y)}^{(2)}=c(x,y)a_{f(x,y)}^{(2)}u_{t_0(x,y)}^{(2)}.
\]
Applying $a_{-s}^{(2)}$ from left on the both sides of above equation for some $s$ sufficiently large, then Lemma~\ref{lem:groupDecom} and the definition of $C_{\mathfrak{g}_2}(\bu_2)$ show that $\tilde{c}(x,y)=c(x,y)$, $\tilde{f}(x,y)=f(x,y)$ and $\tilde{t}_0(x,y)=t_0(x,y)$. This completes the proof of \textbf{second part} of Corollary~\ref{cor:inifinteMatching}.
\end{proof}

\section{Large gap between good intervals}\label{sec:longHammingSep}
Under the same assumptions as the beginning of \S\ref{sec:mainLemma}, the main result of this section is the following large gap lemma. Recall from \S\ref{first statement of lemma 8.1} the notations $x_t=u_t^{(2)}.x$ and $y_t=u_t^{(2)}.y$ for $x,y\in \mathbf{G}_2/\Gamma_2$. Also recall from \S\ref{first statement of lemma 8.1} that $\Kold{0401Kbase}\subset \mathbf{G}_2/\Gamma_2$ is compact and $\Oold{040KbaseNew}$ is its interior.
\begin{Lemma}\label{lem:longHammingSep}
Given an $\epsilon\in(0,10^{-40})$, there exist a compact set $\Kold{0403KlHS2}\subset \mathbf{G}_2/\Gamma_2$ with $m_2(\Kold{0403KlHS2})>1-\epsilon$ and constants $\Rold{040RlHS1},\Cold{040ClHS1},\wold{040wlHS1}>0$ such that for any  sufficiently small $\delta >0$, \emph{if} the following conditions hold
\begin{enumerate}
    \item $y\in \Kold{0403KlHS2}$ and  $x\in\Kak(R,\delta,y)$ for some $R\geq \Rold{040RlHS1}$, 
    \item $y_s\in \Oold{040KbaseNew}$ and $x_{t}\in\Kak(R,\delta,y_s)$ for some $s,t\geq R$,
    \item $(x,y)\overset{\neq e}{\leadsto}(x_{t},y_s)$,
\end{enumerate}
then we have
\[
\max(t,s)> \Cold{040ClHS1}R^{1+\wold{040wlHS1}}.
\]
\end{Lemma}
\subsection*{Outline of the proof}
Our strategy to prove Lemma~\ref{lem:longHammingSep} is to estimate the measure of points that do not satisfy the conditions in Lemma~\ref{lem:longHammingSep} and show it is small. A similar result was established by Ratner \cite{ratner1979cartesian} for $\mathbf{G}_2=\SL(2,\R)\times\SL(2,\R)$ and $\Gamma$ a cocompact reducible lattice in~$ \mathbf{G}_2$. A main tool in estimating the measure of the relevant sets is Corollary~\ref{cor:globalRepEst}.

\subsection{Proof of Lemma~\ref{lem:longHammingSep}}
We give a conditional proof of Lemma~\ref{lem:longHammingSep} assuming Lemma~\ref{lem:longHammingMEst} in this section. The proof of Lemma~\ref{lem:longHammingMEst} is postponed to ~\S\ref{sec:setEstimates}. 

\subsubsection{Notations and preliminaries} \label{8.1.1 notations}
The following lemma provides us the good set needed for the proof of Lemma~\ref{lem:longHammingSep}. Proof of Lemma~\ref{lem:chainGoodSet} follows from the Pointwise Ergodic Theorem.

\begin{Lemma}\label{lem:chainGoodSet}
Fix an $\epsilon\in(0,10^{-40})$. Let $K\subset \mathbf{G}_2/\Gamma_2$ be a set with $m_2(K)>1-\frac{\epsilon}{2}$, then for every $c\in(0,1)$, there exist a compact set $W\subset \mathbf{G}_2/\Gamma_2$ with $m_2(W)>1-\frac{\epsilon}{2}$ and $T>1$ such that for every $x\in W$ and $R\geq T$
\[
\{t\in[R(1-c),R]\cap\Z : a_{-t}^{(2)}x\in K\}\neq\emptyset.
\]
\end{Lemma}

\noindent 
We leave the details to the reader. In our case $a_1^{(2)}$ acts ergodically on $ \mathbf{G}_2/\Gamma_2$, so if we want we can even make the measure of $W$ as close to~$1$ as we wish.

\medskip

Recall that $\Kold{0401Kbase}$ is a compact set chosen in \S\ref{first statement of lemma 8.1}. Fix an $\epsilon\in(0,10^{-40})$, we apply Lemma~\ref{lem:chainGoodSet} with $K=\Kold{0401Kbase}$ and $c>0$ a sufficiently small constant (which will be denoted by $\wold{040wlHS1}$) to be specified later, we define $\Knew\label{052KGeo3}$\index{$\Kold{052KGeo3}$, Equation~\eqref{eq:geoErg}} and $\Rnew\label{052RGeo}$\index{$\Rold{052RGeo}$, Equation~\eqref{eq:geoErg}} as
\begin{equation}\label{eq:geoErg}
    \Kold{052KGeo3}=W,\qquad \Rold{052RGeo}=T.
\end{equation}

\medskip

For any $w,R,\delta>0$, define the set $V(w,R,\delta)$ as follows:
\begin{equation*}
    \begin{aligned}
    V(w,R, \delta)&=\{y\in \Kold{052KGeo3}:\exists x\in\Kak(R,\delta,y), \ t,s \in [R, R^{1+w}] \\
    &\text{such that }y_s \in \Oold{040KbaseNew},  x_{t}\in\Kak(R,\delta,y_s),  (x,y)\overset{\neq e}{\leadsto}(x_{t},y_s)\},
    \end{aligned}
\end{equation*}
where $\Oold{040KbaseNew}$ is the interior of the compact set $\Kold{0401Kbase}$ on $\mathbf{G}_2/\Gamma_2$, cf.~\eqref{eq:kpreHat}. With these notations, we are able to state the Lemma~\ref{lem:longHammingMEst}, whose proof is provided in \S\ref{sec:setEstimates}.
\begin{Lemma}\label{lem:longHammingMEst}
There exist positive constants $\wold{040wlHS1}$, $\denew\label{0521deKakSmall}$\index{$\deold{0521deKakSmall}$, Lemma~\ref{lem:longHammingMEst}}, $\denew\label{0522deExpR}$\index{$\deold{0522deExpR}$, Lemma~\ref{lem:longHammingMEst}}, $\Rnew\label{052Rlhe1}$\index{$\Rold{052Rlhe1}$, Lemma~\ref{lem:longHammingMEst}} such that for every $R\geq\Rold{052Rlhe1}$
\[
m_2(V(\wold{040wlHS1},R,\deold{0521deKakSmall}))<R^{-\deold{0522deExpR}}.
\]
\end{Lemma}
\index{$\Rold{052Rlhe1}$, Lemma~\ref{lem:longHammingMEst}}%

\subsubsection{Proof of Lemma~\ref{lem:longHammingSep}}
\begin{proof}[Proof of Lemma~\ref{lem:longHammingSep} assuming Lemma~\ref{lem:longHammingMEst}]

Let $\epsilon\in(0,10^{-40})$ be fixed as above. Let $N_1\in\N$ be such that 
\[
2^{N_1}\geq\Rold{052Rlhe1}^{\deold{0522deExpR}} \text{ and }  \sum_{k=N_1}^{\infty}2^{-k}<\frac{\epsilon}{2}.\] 
Then we define 
\[
\Rold{040RlHS1}=2^{N_1/\deold{0522deExpR}}\text{ and }\Kold{0403KlHS2}=\bigcap_{k=N_1}^{\infty}\left(\Kold{052KGeo3}\setminus V(\wold{040wlHS1},2^{k/\deold{0522deExpR}},\deold{0521deKakSmall})\right).
\]
By Lemma~\ref{lem:longHammingMEst}
\begin{equation*}
\begin{aligned}
m_2(\Kold{0403KlHS2})\geq &m_2(\Kold{052KGeo3})-\sum_{k=N_1}^{\infty}m_2(V(\wold{040wlHS1},2^{k/\deold{0522deExpR}},\deold{0521deKakSmall}))\\
\geq&1-\frac{\epsilon}{2}-\sum_{k=N_1}^{\infty}2^{-k}
\geq1-\epsilon.
\end{aligned}
\end{equation*}
Moreover, since each $V(\wold{040wlHS1},2^{k/\deold{0522deExpR}},\deold{0521deKakSmall})$ is an open set and $\Kold{052KGeo3}$ is compact, we obtain that $\Kold{0403KlHS2}$ is a compact set. 

\medskip

Given $R\geq\Rold{040RlHS1}$, there exists an integer  $k_R\geq N_1$ such that 
\[
2^{k_R/\deold{0522deExpR}}\leq R<2^{(k_R+1)/\deold{0522deExpR}}.
\]
If $y\in\Kold{0403KlHS2}$, $x\in \mathbf{G}_2/\Gamma_2$, $\delta\in(0,\deold{0521deKakSmall}]$ and $s,t\geq R\geq2^{k_R/\deold{0522deExpR}}$ satisfy 
\[
x\in\Kak(R,\delta,y),\quad y_s\in \Oold{040KbaseNew},\quad x_t\in\Kak(R,\delta,y_s),\quad (x,y)\overset{\neq e}{\leadsto}(x_{t},y_s),
\]
then these together with the definition of $V(w,R,\delta)$ and 
\[
\Kold{0403KlHS2}\subset \Kold{052KGeo3}\setminus V(\wold{040wlHS1},2^{k_R/\deold{0522deExpR}},\deold{0521deKakSmall})
\]
give that
\begin{equation*}
    \begin{aligned}
    \max(s,t)>2^{(1+\wold{040wlHS1})k_R/\deold{0522deExpR}}= \Cold{040ClHS1}2^{(1+\wold{040wlHS1})(k_R+1)/\deold{0522deExpR}}>\Cold{040ClHS1}R^{1+\wold{040wlHS1}},
    \end{aligned}
\end{equation*}
where $\Cold{040ClHS1}=2^{-(1+\wold{040wlHS1})/\deold{0522deExpR}}$. This proves Lemma~\ref{lem:longHammingSep}.
\end{proof}

Now the only missing piece for Lemma~\ref{lem:longHammingSep} is the proof of Lemma~\ref{lem:longHammingMEst}, which will be provided in the next section.
\subsection{Proof of Lemma~\ref{lem:longHammingMEst}}\label{sec:setEstimates}

\subsubsection{Notations}\label{sec:constructionL2}
We first specify some notations that are used in the proof of Lemma~\ref{lem:longHammingMEst}.

\medskip

Let $\mathfrak{l}_2$ be the Lie algebra generated by $\bu_2,\bou_2,\ba_2$ and $\bx_2^{0,j}$ for $j\in I_c$, where $I_c$ is defined as in~\eqref{eq:Ic}.

 Since $(\mathbf{G}_1/\Gamma_1,m_1,u_t^{(1)})$ is Kakutani equivalent to $(\mathbf{G}_2/\Gamma_2,m_2,u_t^{(2)})$ and $(\mathbf{G}_1/\Gamma_1,m_1,u_t^{(1)})$ is not loosely Kronecker, Lemma~\ref{lem:chainLK} guarantees 
\begin{equation*}
\mathfrak{l}_2\neq\mathfrak{g}_2.
\end{equation*}

 Recall that $\mathfrak{h}_2=\operatorname{span}_{\R}\{\bu_2,\ba_2,\bou_2\}$ and let 
\[
\mathbb{L}_2(\R)=\{g \in \mathbb{G}_2(\R): \Ad (g)\mathfrak{h}_2=\mathfrak{h}_2\},
\]
then we have 
\[
\mathfrak{l}_2=\operatorname{Lie}(\mathbb{L}_2(\R))\subsetneq\mathfrak{g}_2.
\]
Define
    \begin{equation*}
        \mathbf{L}_2=\mathbb{L}_2(\R)\cap \mathbf{G}_2.
    \end{equation*} 
Then $\mathbf{L}_2$ is a subgroup of finite index in $\mathbb{L}_2(\R)$ since $\mathbf{G}_2$ is a subgroup of finite index in $\mathbb{G}_2(\R)$.
Let $\mathbf{N}_2$ be the \textbf{normal core} of $\mathbf{L}_2$, i.e. 
\begin{equation*}
\mathbf{N}_2=\bigcap_{g\in \mathbf{G}_2}g\mathbf{L}_2g^{-1}.
\end{equation*}

Applying Chevalley's Theorem (cf.\ Theorem~\ref{thm:Chevalley}) to $\mathbb{L}_2(\R)$ and $\mathbb{G}_2(\R)$, by restricting to~$\mathbf{G}_2$, there exist an algebraic representation over $\R$
\[
\rho:\mathbf{G}_2\to\SL(V)
\]
and an $\R$-vector 
\[
v_{\mathbf{L}_2}\in V\cap\R^{\dim V}
\]
such that 
    \begin{equation}\label{eq:vLdef}
        \mathbf{L}_2=\{g\in \mathbf{G}_2:\rho(g).v_{\mathbf{L}_2}=v_{\mathbf{L}_2}\}.
    \end{equation} 

By Lemma~\ref{lem:HNint}, there exist $\eold{033enbh}>0$ and $\Cold{033CnormRadius}>1$ such that if $h\in \mathbf{H}_2$ satisfies
\[
hzp\in \mathbf{N}_2
\]
for some $z\in \mathbf{Z}_2=C_{\mathbf{G}_2}(\mathbf{H}_2)$, $p\in B_{\epsilon}^{\mathbf{G}_2}$ and $\epsilon\in(0,\eold{033enbh})$, then $h$ is within distance $\Cold{033CnormRadius}\epsilon$ of an element in the center of $\mathbf{H}_2$ (a group of order $\leq 2$).

\subsubsection{Proof of Lemma~\ref{lem:longHammingMEst}}
We start the proof of Lemma~\ref{lem:longHammingMEst} with the following technical lemma, which aims to check the conditions in Corollary~\ref{cor:globalRepEst}. Let $\widetilde{\Kold{0401Kbase}}$ be a compact subset of $\mathbf{G}_2$ whose projection on $\mathbf{G}_2/\Gamma_2$ contains $\Kold{0401Kbase}$. The proof of Lemma~\ref{lem:representationConditions} is provided in \S\ref{sec:representationConditions}.
\begin{Lemma}\label{lem:representationConditions}
There exist $\wnew\label{054wconditionRep}$\index{$\wold{054wconditionRep}$, Lemma~\ref{lem:representationConditions}}, $\Rnew\label{054RconditionRep}$\index{$\Rold{054RconditionRep}$, Lemma~\ref{lem:representationConditions}} and $\denew\label{054deconditionRep}>0$\index{$\deold{054deconditionRep}$, Lemma~\ref{lem:representationConditions}} depending on $\mathbf{G}_2$ and $\Gamma_2$ such that for every $\delta\in(0,\deold{054deconditionRep}]$ and $R\geq\Rold{054RconditionRep}$ the following hold. 

If there are $e\neq\bar{\gamma}\in\Gamma_2$ and $\bar{y}\in \widetilde{\Kold{0401Kbase}}$ such that
\begin{equation}\label{eq:ugyrExpression}
u_{-\bar{s}}^{(2)}\bar{g}_2^{-1}u_{\bar{t}}^{(2)}\bar{g}_1=\bar{y}\bar{\gamma}\bar{y}^{-1}
\end{equation}
holds for some $\bar{s},\bar{t}\in[1,R^{\wold{054wconditionRep}}]$, $\bar{g}_i^{\mathfrak{h}}, \bar{g}_i^{\mathfrak{z}}\in B_{\delta}^{\mathbf{G}_2}$ and  $\bar{g}_i^{\tr}\in B_{R^{-0.4}}^{\mathbf{G}_2}$ with $i=1,2$. Then
\begin{equation*}
\begin{gathered}
\min_{g\in \mathbf{N}_2}\norm{\bar{\gamma}-g}\geq \Cold{027CcoeSmall}R^{-\deold{027dexpSmallnew}},\qquad \norm{\bar{\gamma}}\leq  R^{\deold{027dexpSmallnew}^2\deold{0301Clattice}/10},\\
\norm{\rho(\bar{\gamma}\bar{y}^{-1})v_{\mathbf{L}_2}-\rho(\bar{y}^{-1})v_{\mathbf{L}_2}}_V<R^{-0.2},
\end{gathered}
\end{equation*}
where $\deold{027dexpSmallnew}$ and $\Cold{027CcoeSmall}$ are constants from Corollary~\ref{cor:globalRepEst} with $\delta=1/5$, and $\deold{0301Clattice}$ is the constant from~\eqref{eq:latticeCount} related to the number of lattice elements up to a given norm.
\end{Lemma}
Now we proceed to the proof of Lemma~\ref{lem:longHammingMEst}.
\begin{proof}[Proof of Lemma~\ref{lem:longHammingMEst}]
Let $\wold{040wlHS1}=\wold{054wconditionRep}$, $\deold{0521deKakSmall}=10^{-9}\deold{054deconditionRep}$ and $\Rold{052Rlhe1}\geq\Rold{054RconditionRep}$ a positive constant to be determined later, where $\wold{054wconditionRep}$, $\deold{054deconditionRep}$ and  $\Rold{054RconditionRep}$ are from Lemma~\ref{lem:representationConditions}.

Suppose that $y\in\Kold{052KGeo3}$ (with $\Kold{052KGeo3}$ depending on $\wold{040wlHS1}$ as in \S\ref{8.1.1 notations}),  
$x\in\Kak(R,\deold{0521deKakSmall},y)$, $y_s\in\Oold{040KbaseNew}$ and $x_{t}\in\Kak(R,\deold{0521deKakSmall},y_s)$ with $s$ and $t$ satisfying  $R\leq s,t\leq R^{1+\wold{040wlHS1}}$ for some $R\geq\Rold{052Rlhe1}$, and $(x,y)\overset{\neq e}{\leadsto}(x_{t},y_s)$ (cf.~\S\ref{sec:twoRelations}).

\medskip

Let $\tilde{y}\in \mathbf{G}_2$ be a lift of $y$, then the above assumptions imply that there exist~$e\neq\gamma\in\Gamma_2$ and  $g_1,g_2\in\Kak(R,\deold{0521deKakSmall})$ satisfying 
\[
x=g_1.y,\qquad x_{t}=g_2.y_s,
\]
such that
\begin{equation}\label{eq:firstFun}
\begin{aligned}
u_{t}^{(2)}g_1.\tilde{y}=g_2u_{s}^{(2)}.\tilde{y}\gamma.
\end{aligned}    
\end{equation}

The choice of $\Kold{052KGeo3}$ (cf.~\eqref{eq:geoErg} and Lemma~\ref{lem:chainGoodSet}) implies that for every $R\geq \Rold{052RGeo}$ there exists $p_y\in\N$ such that
\begin{equation}\label{eq:geoErg1}
    \begin{aligned}
    \frac{1}{2}(1-\wold{040wlHS1})\log R\leq p_y\leq\frac{1}{2}\log R\text{ and }
    a_{-p_y}^{(2)}.y\in\Kold{0401Kbase}.
    \end{aligned}
\end{equation}
This gives $\gamma_1\in\Gamma_2$ such that 
\[
a_{-p_y}^{(2)}.\tilde{y}\gamma_1\in \widetilde{\Kold{0401Kbase}}.
\]
Let  $\bar{y}=a_{-p_y}^{(2)}.\tilde{y}\gamma_1$, then \eqref{eq:firstFun} can be rewritten as:
\begin{equation}\label{eq:expLarTar0}
   u_{-\bar{s}}^{(2)}\bar{g}_2^{-1}u_{\bar{t}}^{(2)}\bar{g}_1=\bar{y}\bar{\gamma}\bar{y}^{-1}, 
\end{equation}
where 
\begin{equation}\label{eq:gammastgi}
\begin{gathered}
\bar{\gamma}=\gamma_1^{-1}\gamma\gamma_1\in\Gamma_2, \qquad   \bar{s}=e^{-2p_y}s,\qquad \bar{t}=e^{-2p_y}t, \\
\bar{g}_i=a_{-p_y}^{(2)}g_ia_{p_y}^{(2)} \text{ for }i=1,2,
\end{gathered}
\end{equation}
so $1\leq \bar{s},\bar{t}\leq R^{2\wold{040wlHS1}}$. Also, by Lemma~\ref{lem:diangonalConjugateKak}, $\bar g _i $ satisfy
\[
\bar{g}_i^{\mathfrak{h}}, \bar{g}_i^{\mathfrak{z}}\in B_{4\deold{0521deKakSmall}}^{\mathbf{G}_2}, \qquad 
\bar{g}_i^{\tr}\in B_{R^{-0.4}}^{\mathbf{G}_2}.
\]

\medskip

\begin{itemize}
\item Let $V_{R^{-\deold{027dexpSmallnew}/5}}$ be the set of $\bar{z}\in \widetilde{\Kold{0401Kbase}}$ such that 
\[
\norm{\rho(\gamma'\bar{z}^{-1})v_{\mathbf{L}_2}-\rho(\bar{z}^{-1})v_{\mathbf{L}_2}}_V<R^{-0.2}
\]
for some $\gamma'\in\Gamma_2$ satisfying 
\begin{gather}
    \min_{g\in \mathbf{N}_2}\norm{\gamma'-g}\geq\Cold{027CcoeSmall}R^{-\deold{027dexpSmallnew}}\quad \text{ and }\quad
    \norm{\gamma'}\leq  R^{\deold{027dexpSmallnew}^2\deold{0301Clattice}/10},\nonumber
\end{gather}
where $\Cold{027CcoeSmall}$ and $\deold{027dexpSmallnew}$ are from Corollary~\ref{cor:globalRepEst} with $\delta=1/5$, and $\deold{0301Clattice}$ is as in~\eqref{eq:latticeCount}.
\end{itemize}
\smallskip
As $p_y\in[\frac{1}{2}(1-\wold{040wlHS1})\log R,\frac{1}{2}\log R]$, equations \eqref{eq:expLarTar0}, \eqref{eq:gammastgi} and Lemma~\ref{lem:representationConditions} give
\begin{equation}\label{eq:diagonalInclusion}
    V(\wold{040wlHS1},R,\deold{0521deKakSmall})\subset\pi_2\left(\bigcup_{n=\lfloor\frac{\log R}{2}(1-\wold{040wlHS1})\rfloor}^{\lfloor\frac{\log R}{2} \rfloor+1}a_{-n}^{(2)}\left(V_{R^{-\deold{027dexpSmallnew}/5}}\right)\right)
\end{equation}
where $\pi_2:\mathbf{G}_2\to\mathbf{G}_2/\Gamma_2$ is the natural projection. All together
\begin{equation*}
\begin{aligned}
m_2(V(\wold{040wlHS1},R,\deold{0521deKakSmall}))\stackrel{\eqref{eq:diagonalInclusion}}{\leq}\quad &\wold{040wlHS1}(\log R)\tilde{m}_2(V_{R^{-\deold{027dexpSmallnew}/5}})\\
\stackrel{\text{Corollary~\ref{cor:globalRepEst}}}{\leq}&\wold{040wlHS1}(\log R)  R^{-\deold{027dexpSmallnew}\deold{0302CglobalCom1}/5}\leq R^{-\deold{027dexpSmallnew}\deold{0302CglobalCom1}/10},
\end{aligned}
\end{equation*}
where the last two inequalities are due to $\Rold{052Rlhe1}\geq\max(\Rold{052RGeo},\Rold{054RconditionRep})$ can be chosen sufficiently large. Taking $\deold{0522deExpR}=\deold{027dexpSmallnew}\deold{0302CglobalCom1}/10$, this completes the proof of Lemma~\ref{lem:longHammingMEst}.
  
\end{proof}

\subsubsection{Proof of Lemma~\ref{lem:representationConditions}}\label{sec:representationConditions}
    
The first step of our proof is to convert $u_{-\bar{s}}^{(2)}.\bar{g}^{-1}_2.u_{\bar{t}}^{(2)}.\bar{g}_1$ to an expression in the form of $zl$, where $z\in B^{\mathbf{G}_2}_{1/R^c}$ for some $c>0$ and $l\in\mathbf{L}_2$. This indeed follows from the Lemma~\ref{lem:comExp} and Lemma~\ref{lem:zlConvertDirect}, which provide the commuting relations between $\mathfrak{sl}_2$-triples and other elements in the chain basis.

\begin{Lemma}\label{lem:convertLemma}
There exist $\wnew\label{056wConvert}$\index{$\wold{056wConvert}$, Lemma~\ref{lem:convertLemma}}, $\Rnew\label{056RConvert}$\index{$\Rold{056RConvert}$, Lemma~\ref{lem:convertLemma}}, $\denew\label{056deConvert}>0$\index{$\deold{056deConvert}$, Lemma~\ref{lem:convertLemma}} depending on $\mathbf{G}_2$ such that for every $\delta\in(0,\deold{056deConvert}]$, $R\geq\Rold{056RConvert}$, $\bar{s},\bar{t}\in[1,R^{\wold{056wConvert}}]$ and $\bar{g}_i$, $i=1,2$
\[
\bar{g}_i^{\mathfrak{h}},\bar{g}_i^{\mathfrak{z}}\in B_{\delta}^{\mathbf{G}_2}, \qquad 
\bar{g}_i^{\tr}\in B_{R^{-0.4}}^{\mathbf{G}_2},
\]
we can write $u_{-\bar{s}}^{(2)}\bar{g}^{-1}_2u
_{\bar{t}}^{(2)}\bar{g}_1$ in the form of
\begin{equation*}
\begin{aligned}
u_{-\bar{s}}^{(2)}\bar{g}^{-1}_2u
_{\bar{t}}^{(2)}\bar{g}_1&=z_1z_2u_{b_0}^{(2)}\oub_{b_1}u_{b_2}^{(2)}\oub_{b_3}a_{b_4}^{(2)}
\end{aligned}
\end{equation*}
with $z_1,z_2\in \mathbf{G}_2$ and $b_1$, $b_2$, $b_3$, $b_4$, $b_5\in\R$ satisfying
\begin{equation*}
\begin{gathered}
    z_1\in B^{\mathbf{G}_2}_{R^{-0.3}}, \qquad z_2\in B^{\mathbf{Z}_2}_{20\delta},\\
    -R^{2\wold{056wConvert}}\leq b_0\leq-\frac{1}{2},\quad  \frac{1}{2}\leq b_2\leq R^{2\wold{056wConvert}},\quad \absolute{b_1},\absolute{b_3},\absolute{b_4}<3\delta.\\
    \end{gathered}
\end{equation*}
\end{Lemma}

\noindent
Lemma~\ref{lem:convertLemma} is proved by iteratively using Lemma~\ref{lem:comExp} and Lemma~\ref{lem:zlConvertDirect} several times; we leave the details to the reader. 

\medskip

We now continue with the proof of Lemma~\ref{lem:representationConditions}. Suppose \eqref{eq:ugyrExpression} holds.
If $\wold{054wconditionRep}\leq\wold{056wConvert}/2$, $\Rold{054RconditionRep}\geq\Rold{056RConvert}$ and $\deold{054deconditionRep}\leq\deold{056deConvert}$, then 
\begin{equation*}
    \begin{alignedat}{3}
\rho(\bar{\gamma}\bar{y}^{-1}).v_{\mathbf{L}_2}&\quad\stackrel{\eqref{eq:ugyrExpression}}{=}&&\rho(\bar{y}^{-1}u_{-\bar{s}}^{(2)}\bar{g}^{-1}_2u_{\bar{t}}^{(2)}\bar{g}_1).v_{\mathbf{L}_2}\\
    &\stackrel{\text{Lemma~\ref{lem:convertLemma}}}{=}&&\rho(\bar{y}^{-1}z_1z_2u_{b_0}^{(2)}\oub_{b_1}u_{b_2}^{(2)}\oub_{b_3}a_{b_4}^{(2)}).v_{\mathbf{L}_2}.
 \end{alignedat}
\end{equation*}
As $z_2,u_{b_0}^{(2)},\oub_{b_1},u_{b_2}^{(2)},\oub_{b_3},a_{b_4}^{(2)}\in\mathbf{L_2}$ and $\rho(g)v_{\mathbf{L}_2}=v_{\mathbf{L}_2}$ for every $g\in~\mathbf{L}_2$,
\[
\rho(\bar{\gamma}\bar{y}^{-1}).v_{\mathbf{L}_2}=\rho(\bar{y}^{-1}z_1).v_{\mathbf{L}_2}.
\]
Since $z_1\in B^{\mathbf{G}_2}_{3(\dim\mathfrak{g}_2)^2\delta /R^{0.3}}$, enlarge $\Rold{054RconditionRep}$ if necessary, this together with that $\rho$ is a real algebraic representation gives for $R\geq\Rold{054RconditionRep}$,
\[
\rho(\bar{\gamma}\bar{y}^{-1}).v_{\mathbf{L}_2}=\rho(\bar{y}^{-1}).v_{\mathbf{L}_2}+\tilde{v},\]
where $\tilde{v}\in V\cap\R^{\dim V}$ satisfies $\norm{\tilde{v}}_V\leq R^{-0.2}$. 

\medskip

As for the norm of $\bar{\gamma}\in\Gamma_2$ satisfying assumptions of Lemma~\ref{lem:representationConditions}, shrinking $\wold{040wlHS1}$ if necessary, it follows from $
 \bar{\gamma}=\bar{y}^{-1}u_{-\bar{s}}^{(2)}\bar{g}_2^{-1}u_{\bar{t}}^{(2)}\bar{g}_1\bar{y}$, $\bar{y}\in\widetilde{\Kold{0401Kbase}}$, $\bar{s},\bar{t}\in[1,R^{2\wold{040wlHS1}}]$ and $\bar{g}_i\in B_{3(\dim\mathfrak{g}_2)^2\deold{0521deKakSmall}}^{\mathbf{G}_2}$ for $i=1,2$ that
\[
\norm{\bar{\gamma}}\leq R^{8\wold{040wlHS1}}\leq R^{\deold{027dexpSmallnew}^2\deold{0301Clattice}/10}.
\]

\medskip

The estimate of the lower bound for $\min_{g\in \mathbf{N}_2}\norm{\bar{\gamma}-g}$ directly follows from the following technical lemma, $\bar{y}\in\widetilde{\Kold{0401Kbase}}$, $\mathbf{N}_2\lhd\mathbf{G}_2$ and \eqref{eq:ugyrExpression}.

\begin{Lemma}\label{lem:minDis}
There exist $\wnew\label{057wDistanceMin},\Rnew\label{057RDistanceMin},\denew\label{057deDistanceMin}>0$\index{$\wold{057wDistanceMin}$, Lemma~\ref{lem:minDis}}\index{$\Rold{057RDistanceMin}$, Lemma~\ref{lem:minDis}}\index{$\deold{057deDistanceMin}$, Lemma~\ref{lem:minDis}} such that for every $\delta\in(0,\deold{057deDistanceMin}]$ and $R\geq\Rold{057RDistanceMin}$, if  $\bar{s},\bar{t}\in[1,R^{\wold{057wDistanceMin}}]$, $\bar{g}_i^{\mathfrak{h}}, \bar{g}_i^{\mathfrak{z}}\in B_{\delta}^{\mathbf{G}_2}$ and $\bar{g}_i^{\tr}\in B_{R^{-0.4}}^{\mathbf{G}_2}$ for $i=1,2$,
then 
\[
\min_{g\in \mathbf{N}_2}\norm{u_{-\bar{s}}^{(2)}\bar{g}^{-1}_2u_{\bar{t}}^{(2)}\bar{g}_1-g}\geq R^{-\deold{027dexpSmallnew}/2}.
\]
\end{Lemma}
\begin{proof}[Proof of Lemma~\ref{lem:minDis}]
Let $\wold{057wDistanceMin}\leq \wold{056wConvert}$, $\Rold{057RDistanceMin}\geq \Rold{056RConvert}$, $\deold{057deDistanceMin}\leq\deold{056deConvert}$ be constants to be determined later, where $\wold{056wConvert}$, $\Rold{056RConvert}$ and  $\deold{056deConvert}$ are from Lemma~\ref{lem:convertLemma}.

Assume by contradiction that
\begin{equation}\label{eq:contraTar}
    \min_{g\in \mathbf{N}_2}\norm{u_{-\bar{s}}^{(2)}\bar{g}^{-1}_2u_{\bar{t}}^{(2)}\bar{g}_1-g}<R^{-\deold{027dexpSmallnew}/2}.
\end{equation}
Enlarge $\Rold{057RDistanceMin}$ if necessary, there exists $g_3\in B_{\delta}^{\mathbf{G}_2}$ such that 
\[
g_3u_{-\bar{s}}^{(2)}\bar{g}^{-1}_2u_{\bar{t}}^{(2)}\bar{g}_1\in \mathbf{N}_2.
\]
Together with Lemma~\ref{lem:convertLemma} and $\mathbf{N}_2\lhd\mathbf{G}_2$, 
\[
u_{b_0}^{(2)}\oub_{b_1}u_{b_2}^{(2)}\oub_{b_3}a_{b_4}^{(2)}g_3z_1z_2\in\mathbf{N}_2,
\]
where $z_1,z_2,b_0,b_1,b_2,b_3,b_4$ are defined in Lemma~\ref{lem:convertLemma}.

Let $\deold{057deDistanceMin}\leq \eold{033enbh}/\dim\mathfrak{g}_2$,   Lemma~\ref{lem:HNint} gives $h_0\in C_{\mathbf{H}_2}(\mathbf{H}_2)$ satisfying
\begin{equation}\label{eq:hCentralCondition}
   u_{b_0}^{(2)}\oub_{b_1}u_{b_2}^{(2)}\oub_{b_3}a_{b_4}^{(2)}\in B_{100\Cold{033CnormRadius}\delta}^{\mathbf{H}_2}(h_0), 
\end{equation}
where 
$\Cold{033CnormRadius}>1$ is defined in Lemma~\ref{lem:HNint}.

\begin{enumerate}[label=(\Roman*)] 
    \item If $h_0=e$, this together with equation~\eqref{eq:ugyrExpression},  Lemma~\ref{lem:convertLemma} and inclusion~\eqref{eq:hCentralCondition} gives for some $\kappa_0$ depends only on $\mathbf{G}_2$ that
\begin{equation}\label{eq:NminDis}
 \bar{y}\bar{\gamma}\bar{y}^{-1}=u_{-\bar{s}}^{(2)}\bar{g}^{-1}_2u_{\bar{t}}^{(2)}\bar{g}_1\in B^{\mathbf{G}_2}_{\kappa_0\delta},\qquad \kappa_0>100\Cold{033CnormRadius}.
\end{equation}
Recall that $\bar y \in \widetilde{\Kold{0401Kbase}}$; thus, by picking $\deold{057deDistanceMin}\leq\inj(\widetilde{\Kold{0401Kbase}})/(10^2\kappa_0)$, we obtain from $\bar{\gamma}\neq e$ that \eqref{eq:NminDis} cannot hold and thus $h_0\neq e$.

\item If $h_0\neq e$, then \eqref{eq:hCentralCondition}, the bounds of $b_3$ and~$b_4$ in Lemma~\ref{lem:convertLemma}~give for some $\kappa_1>0$ depends only on $\mathbf{G}_2$ that
\begin{equation}\label{eq:hCentralCondition1}
   u_{b_0}^{(2)}\oub_{b_1}u_{b_2}^{(2)}\in B_{\kappa_1\delta}^{\mathbf{H}_2}(h_0).
\end{equation}
By the homomorphism $\tilde{\varphi}$ from \S\ref{sec:sl2basis}, we may write 
\[
\quad u_{b_0}^{(2)}=\tilde{\varphi}\left(\begin{pmatrix}
        1 & b_0 \\
        0 & 1  \\
    \end{pmatrix}\right),\ \ \oub_{b_1}=\tilde{\varphi}\left(\begin{pmatrix}
        1 & 0 \\
        b_1 & 1  \\
    \end{pmatrix}\right),\ \ u_{b_2}^{(2)}=\tilde{\varphi}\left(\begin{pmatrix}
        1 & b_2 \\
        0 & 1  \\
    \end{pmatrix}\right).
    \]
    Then \eqref{eq:hCentralCondition1} and $h_0\neq e$ imply that for some $\kappa_2>0$ depends only on $\mathbf{G}_2$ that
    \[
    \norm{\begin{pmatrix}
        1 & b_0 \\
        0 & 1  \\
    \end{pmatrix}\begin{pmatrix}
        1 & 0 \\
        b_1 & 1  \\
    \end{pmatrix}\begin{pmatrix}
         1 & b_2 \\
        0 & 1  \\
    \end{pmatrix}\begin{pmatrix}
        0 \\ 1
    \end{pmatrix}+\begin{pmatrix}
        0 \\ 1
    \end{pmatrix}}\leq \kappa_2\delta.
    \]
    In particular, above inequality gives for some $\kappa_3>0$ depends only on $\mathbf{G}_2$ that
    \[
\absolute{b_2-(1+\kappa_3\delta)b_0}<2\kappa_2\delta.
\]
However, by picking $\deold{057deDistanceMin}$ less than $10^{-6}\min(\kappa_2^{-1},\kappa_3^{-1})$, above inequality contradicts with $b_0\leq-\frac{1}{2}$, $b_2\geq\frac{1}{2}$ (cf.~Lemma~\ref{lem:convertLemma}) and thus \eqref{eq:hCentralCondition1} cannot hold.
\end{enumerate}

Since \eqref{eq:hCentralCondition} is not true for $h_0=e$ and $h_0\neq e$, we obtain that \eqref{eq:contraTar} cannot hold and this proves the Lemma~\ref{lem:minDis}. 
\end{proof}

Let $\wold{054wconditionRep}=\min(\wold{056wConvert}/2, \wold{057wDistanceMin})$, $\Rold{054RconditionRep}=\max(\Rold{056RConvert},\Rold{057RDistanceMin})$ and   $\deold{054deconditionRep}=\min(\deold{056deConvert},\deold{057deDistanceMin})$,  we finish the proof of Lemma~\ref{lem:representationConditions}.

\section{Compatibility of the diagonalizable subgroup}\label{sec:comgeo}
In this section, we consider how our modified even Kakutani equivalence $\psi$ defined \S\ref{sec:modifiedKakutani} behaves under the diagonalizable subgroup $a_t^{(1)}\!$. The main result of this section is Proposition~\ref{prop:CompatibleGeoMain}, which establishes that without loss of generality we may assume that for $m_1$-a.e.$~x\in \mathbf{G}_1/\Gamma_1$, the difference between $\psi_n(x)=a_{-n}^{(2)}.\psi(a_n^{(1)}.x)$ and $\psi(x)$ is a shift in the one-parameter unipotent subgroup $u_t^{(2)\!}$.

\medskip

Our proof of Proposition~\ref{prop:CompatibleGeoMain} consists of three steps. We first show (cf.~Proposition~\ref{prop:Geo}) that for sufficiently small $\delta$, and $x$ in a good set, the points $\psi(x)$ and $\psi_\delta(x)=a_{-\delta}^{(2)}.\psi(a_\delta^{(1)}.x)$ are in the same orbit of $C_{\mathbf{G}_2}(\mathbf{U}_{2})$. By using our Main Lemma (Lemma~\ref{lem:main}) carefully, we can refine Proposition~\ref{prop:Geo} by excluding  the centralizer of $\mathbf{H}_2$ from the difference between $\psi$ and $\psi_{\delta}(x)$ (cf.~Proposition~\ref{prop:GeoTri}), where $\mathbf{H}_2$ is the subgroup of $\mathbf{G}_2$, locally isomorphic to $\operatorname{SL}_2(\R)$, that was defined in \S\ref{sec:sl2basis}. Finally, we obtain the Proposition~\ref{prop:CompatibleGeoMain} by combining the first two steps.

\medskip

We note that one of these steps (Proposition~\ref{prop:GeoTri}) holds automatically in some cases, specifically if the centralizer of $\mathbf{H}_2$ is trivial, e.g. for 
\[
\mathbf{G}_2=\SL_2(\R)\times\SL_2(\R),\  u_t^{(2)}=\left(\begin{smallmatrix}
    1 & t \\
    0 & 1 \\
\end{smallmatrix}\right)\times\left(\begin{smallmatrix}
    1 & t \\
    0 & 1 \\
\end{smallmatrix}\right),\ \Gamma_2<\mathbf{G}_2 \text{ a lattice} .
\]
This makes the proof in this case somewhat simpler.

\medskip

We recall the definition chain basis for $\ad_{\bu_2}$ as in assumption \ref{item:chainBasis} at the beginning of \S\ref{sec:mainLemma}. Recall $\mathfrak v^+$ as
\begin{equation}\label{eq:Cgn}
    \mathfrak{v}^+=\bu_2\R\oplus\bigoplus_{j\notin I_c}\bx_2^{0,j}\R,
\end{equation}
where $I_c=\{1\leq j\leq n^{(2)}:q_j^{(2)}=0\}$. Note that $\mathfrak v^+$ is the Lie algebra of the unipotent subgroup $\mathbf{V}^+<\mathbf{G}_2$ defined by
\begin{equation}\label{eq:V+}
\mathbf{V}^+=C_{\mathbf{G}_2}(\mathbf U_2) \cap \mathbf{G}_2^+
\end{equation}
with $\mathbf{G}_2^+$ is defined in \eqref{eq:G+-0}. 

\medskip

\begin{Proposition}\label{prop:CompatibleGeoMain}
There exists $c\in\mathbf{V}^+$ such that if we set $\tilde{\psi}=c\psi$ then for~$m_1$-a.e.$~x\in \mathbf{G}_1/\Gamma_1$
\[
\tilde{\psi}(a_1^{(1)}.x)\in a_1^{(2)}\mathbf{U}_2.\tilde{\psi}(x).
\]
\end{Proposition}
The proof of Proposition~\ref{prop:CompatibleGeoMain} is given in~\S\ref{sec:CompatibleGeoMain}.

\begin{Remark}\label{rmk:iterateReason}
Applying this proposition iteratively, we obtain that for every $n \in \N$, it holds that $\tilde{\psi}_n(x)=a_{-n}^{(2)}.\tilde{\psi}(a_n^{(1)}.x)$ satisfies 
\[
\tilde{\psi}_n(x)\in \mathbf{U}_2.\tilde{\psi}(x),
\]
and then a key point is to show, as we do in \S\ref{sec:renormalization}, that in the appropriate sense $\tilde{\psi}_n$ converges to a limit as $n\to\infty$. In her paper~\cite{ratner1986rigidity} Ratner used a similar strategy to understand Kakutani isomorphisms between horocycle flows on $\SL_2(\R)/\Gamma_i$ for two lattices~$\Gamma_i$, \textbf{assuming the corresponding time change map satisfies a H\"older regularity condition}. We do not assume any such regularity of the time change map, hence need to employ a significantly more delicate argument than \cite{ratner1986rigidity} to establish the existence of a limiting object.
\end{Remark}

\subsection{Local compatibility}\label{sec:localCompatibility}
We start our analysis from the local compatibility of $a_{-\delta}^{(2)}.\psi(a_{\delta}^{(1)}.x)$. 

\begin{Proposition}\label{prop:Geo}
There exist positive constants $\denew\label{0601dePGeo2},\denew\label{0602dePGeo1} >0$\index{$\deold{0601dePGeo2}$, Proposition~\ref{prop:Geo}}\index{$\deold{0602dePGeo1}$, Proposition~\ref{prop:Geo}} such that for every $\delta\in(0,\deold{0602dePGeo1})$, 
we can find 
a full measure subset  $\Znew\label{060Kpn1}=\Zold{060Kpn1}(\delta)$\index{$\Zold{060Kpn1}$, Proposition~\ref{prop:Geo}},
 $c(\delta)\in B^{\mathbf{G}_2}_{\deold{0601dePGeo2}}\cap C_{\mathbf{G}_2}(\mathbf{U}_2)$,
and a measurable function $w(x,\delta)$
such that for every $x\in \Zold{060Kpn1}$
\[
\psi(a^{(1)}_{\delta}.x)= c(\delta)a_{\delta}^{(2)}u_{w(x,\delta)}^{(2)}.\psi(x). 
\]

\end{Proposition}

We first prove the following weak form of Lemma~\ref{prop:Geo}:

\begin{Lemma}\label{lem:Geolm1}
There exist positive constants  $\deold{0601dePGeo2},\deold{0602dePGeo1}>0$ such that for every $\delta\in(0,\deold{0602dePGeo1})$, we can find 
a full measure subset  $Z=Z(\delta)$,
 $c(\delta)\in B^{\mathbf{G}_2}_{\deold{0601dePGeo2}}\cap C_{\mathbf{G}_2}(\mathbf{U}_2)$,
a measurable function $w(x,\delta)$, and in addition $\absolute{f(\delta)}\leq\deold{0601dePGeo2}$,
so that for every $x\in Z$\begin{equation*}
\psi(a_{\delta}^{(1)}.x)= c(\delta)a_{f(\delta)}^{(2)}u_{w(x,\delta)}^{(2)}.\psi(x).
\end{equation*}

\end{Lemma}

\noindent
The remainder of subsection \S\ref{sec:localCompatibility} is dedicated to the proofs of Lemma~\ref{lem:Geolm1} and Proposition~\ref{prop:Geo}. The structure of this subsection is following:
\begin{enumerate}
    \item We first specify notations in~\S\ref{sec:geoConstants}.
    
    \item We prove Lemma~\ref{lem:Geolm1} in \S\ref{sec:Geolm1}, using the Main Lemma (Lemma~\ref{lem:main}) for small perturbations in the diagonal direction and then Corollary~\ref{cor:inifinteMatching} to eliminate error terms except in the centralizer of $u_t^{(1)}$.
    
    \item Finally we give the proof of Proposition~\ref{prop:Geo} in~\S\ref{sec:Geolm2} by showing that error terms in diagonal direction is in fact precisely equal to~$\delta$. The main idea is to use Lemma~\ref{lem:longHammingNew} for small perturbations in the opposite horocycle direction up to some certain time scale. 
\end{enumerate}

\subsubsection{Choices of constants and sets}\label{sec:geoConstants}
We specify the choices of constants and sets that will be used in~\S\ref{sec:Geolm1} and~\S\ref{sec:Geolm2}. The following is a table of the relevant notations from previous sections that will be used in the proof.
    \begin{table}[H]
    \centering
    \begin{tabular}{c|c}
    \hline
     Notation & Origin and definition\\
    \hline
      $\Kold{017KmP1}$, $\Rold{017RmP1}$, $\eold{017emP1}(\cdot)$  & Lemma~\ref{lem:matchingPrese} applied for $\psi$ and $\epsilon=10^{-99}$ \\
    \hline
     $\Zold{035KtimeU}$ & a full measure set defined in~\eqref{eq:controlKakCon} \\
    \hline
     $\Kold{035Klm1}$, $\Rold{035Rlm1}$, $\deold{035delm1}$, $\eold{035eMainDel}(\cdot)$  & Lemma~\ref{lem:main} with $\eta=10^{-8}$\\
     \hline
    $\deold{035delm1}$, $\Rold{036Rlh1}$, $\Kold{036Klh1}$,  $\Cold{039infiniteMatch}$ & Corollary~\ref{cor:inifinteMatching} (or Lemma~\ref{lem:longHamming}) with $\eta=10^{-8}$ \\
    \hline
    $\alpha$ &  the time change function corresponding \\
    & to $\psi$ as in \S\ref{sec:timeChagneGoodSet}\\
    \hline
    
    \end{tabular}
    \caption{\label{table:elementaryGeoCompatible}}
    \end{table}

\medskip

We will also make use of the following sets and constants.
\begin{enumerate}[label=\textup{(c:\roman*)}]   
    \item\label{item:constant1} Let $\denew\label{0611delog}>0$\index{$\deold{0611delog}$, \ref{item:constant1}} be a constant such that
    $\log:\mathbf{G}_2\to\mathfrak{g}_2$ is a diffeomorphism on~$B^{\mathbf{G}_2}_{\deold{0611delog}}$ and
    \[
    \log(B^{\mathbf{G}_2}_{\deold{0611delog}/10})\subset B^{\mathfrak{g}_2}_{\deold{0611delog}/5},\qquad  B^{\mathfrak{g}_2}_{\deold{0611delog}/2}\subset\log(B^{\mathbf{G}_2}_{\deold{0611delog}}).
    \]
    \item\label{item:KtempCpt1} Let $\Knew\label{061KtempCpt1}\subset \mathbf{G}_2/\Gamma_2$\index{$\Kold{061KtempCpt1}$, \ref{item:KtempCpt1}} be a compact subset satisfying 
    \[
    m_2(\Kold{061KtempCpt1})\geq 1-10^{-100}.
    \]
    \item\label{item:constant3} Define $\denew\label{0612dePGeo}$\index{$\deold{0612dePGeo}$, \ref{item:constant3}} and $\deold{0601dePGeo2}$  as 
    \[
    \deold{0612dePGeo}=\min(\kappa_1,\inj(\Kold{061KtempCpt1})^2,\deold{035delm1},\deold{0611delog})/(10^8\dim(\mathfrak{g}_2)^2\Cold{039infiniteMatch}^3),\qquad \deold{0601dePGeo2}=\Cold{039infiniteMatch}\deold{0612dePGeo}.
    \]
    \item\label{item:deltaFixed} Let $\deold{0602dePGeo1}=\min(\eold{017emP1}(\deold{0612dePGeo}),\eold{035eMainDel}(\deold{0612dePGeo}))/2$ and fix  
    \[
    \delta\in(0,\deold{0602dePGeo1}).
    \]
\end{enumerate}
Note that the choice of $\deold{0602dePGeo1}$ can not be made explicit since in Lemma~\ref{lem:matchingPrese} the function $\eold{017emP1}(\cdot)$ is determined using Lusin's theorem applied to the Kakutani equivalence $\psi$.

\medskip

We then specify the good sets used in our proof.
\begin{enumerate}[label=\textup{(s:\roman*)}]   
    \item\label{item:ergodicUnipotentGoodSetConstant} By the Pointwise Ergodic Theorem for flow $u_t^{(1)}$ and function $\chi_{\Kold{017KmP1}}$, there exist a compact set $K_1$ with $m_1(K_1)>1-10^{-9}$ and a positive constant \index{$\Rold{061Rwbound}$, \ref{item:ergodicUnipotentGoodSetConstant}}
    \[    \Rnew\label{061Rwbound}\geq2\max(\Rold{017RmP1},\Rold{036Rlh1},10^{1000})
    \]
    such that for every $x\in K_1$ and $R\geq \Rold{061Rwbound}$ 
    \[
    \absolute{\{t\in[-R,R]:u_t^{(1)}.x\in \Kold{017KmP1}\}}>(1-10^{-90})R.
    \]
    \item  Let $Z_1$ be the set of full measure on which the Pointwise Ergodic Theorem holds for the function $\alpha$ and the flow $u_t^{(1)}$. 
    \item  Let $Z_2$ be the set of full measure on which the Pointwise Ergodic Theorem holds for the function $\chi_{\Kold{061KtempCpt1}}$ and the flow $a_t^{(2)}$.
    \item Let $Z_3\subset \mathbf{G}_1/\Gamma_1$ be the set of full measure
    \begin{equation*}
        Z_3=\bigcap_{k=-\infty}^{\infty}u_k^{(1)}(\psi^{-1}(Z_2)).
    \end{equation*}

    \item\label{item:geoConstant6}
    Recall $\delta$ is fixed in \ref{item:deltaFixed}. Define $Z_4$, $K_2$ and $K_3$ by
    \begin{equation*}
        \begin{aligned}
        &Z_4= \Zold{035KtimeU}\cap Z_1\cap a_{-\delta}^{(1)}\left(\Zold{035KtimeU}\cap Z_1\right)\cap Z_3,\\
        &K_2=\Kold{017KmP1}\cap\psi^{-1}(\Kold{036Klh1})\cap\psi^{-1}(\Kold{061KtempCpt1})\cap K_1\cap Z_3,\\
        &K_3=K_2\cap a_{-\delta}^{(1)}(K_2).
        \end{aligned}
    \end{equation*}
     Then $m_1(K_2)>0.99$, $m_1(Z_4)=1$ and $Z_4$ is $u_1^{(1)}$-invariant. 

\item 
For every $x\in \mathbf{G}_1/\Gamma_1$, let $n_{K_3}(x)$ be
\[
n_{K_3}(x)=\min\{n\in\N:u_n^{(1)}.x\in K_3\}.
\]
By Poincar\'{e} recurrence and ergodicity of $u_t^{(1)}$, there exists a full measure $u_1^{(1)}$-invariant set 
\[
Z_5\subset \mathbf{G}_1/\Gamma_1
\]
such that $n_{K_3}(x)$ is finite for $x\in Z_5$.
\end{enumerate}

\subsubsection{Proof of Lemma~\ref{lem:Geolm1}}\label{sec:Geolm1}

Let $\mathfrak{w}<\mathfrak{g}_2$ be a complement to $\R\bu_2$ in $C_{\mathfrak{g}_2}(\bu_2)$ as in \eqref{eq:wSpace}. 
\begin{Claim}\label{claim:GeoFirstClaim}
Given $x\in Z_4\cap Z_5$, there exist  $w(x,\delta)\in\R$,   $\absolute{f(x,\delta)}\leq\deold{0601dePGeo2}$ and $c(x,\delta)\in B_{\deold{0601dePGeo2}}^{\mathbf{G}_2}\cap \exp(\mathfrak{w})$ such that 
\begin{equation*}
   \psi(a_{\delta}^{(1)}.x)=c(x,\delta)a_{f(x,\delta)}^{(2)}u_{w(x,\delta)}^{(2)}.\psi(x).
\end{equation*}
\end{Claim}
\begin{proof}[Proof of Claim~\ref{claim:GeoFirstClaim}]
For every $x\in Z_4\cap Z_5$, the definitions of $K_3$ and~$Z_5$ imply that
\[
u_{n_{K_3}(x)}^{(1)}.x,\quad a_{\delta}^{(1)}u_{n_{K_3}(x)}^{(1)}.x\in K_2.
\]
As $u_{n_{K_3}(x)}^{(1)}.x$ and $a_{\delta}^{(1)}u_{n_{K_3}(x)}^{(1)}.x$ are $(\deold{0602dePGeo1},10^{-99},R)$-two sides matchable for \emph{every} $R$, if we take $R\geq \Rold{061Rwbound}/2$ we obtain from the fact that $K_2\subset \Kold{017KmP1}$ and Lemma~\ref{lem:matchingPrese} the following:
\begin{itemize}
\item
    $\psi(u_{n_{K_3}(x)}^{(1)}.x)$ and $\psi(a_{\delta}^{(1)}u_{n_{K_3}(x)}^{(1)}.x)$ are $(\deold{0612dePGeo},10^{-97},R)$-two sides matchable for every $R\geq \Rold{061Rwbound}/2$, with a $C^1$-matching function $h(t)$ that does not depend on $R$.
\end{itemize} 

\medskip
\noindent
Then by Corollary~\ref{cor:inifinteMatching} and choice of $\deold{0601dePGeo2}=\Cold{039infiniteMatch}\deold{0612dePGeo}$ give that there exist 
\begin{align*}
c(x,\delta)&\in B_{\deold{0601dePGeo2}}^{\mathbf{G}_2}\cap \exp(\mathfrak{w})\\
f(x,\delta)&\in\R \quad\text{satisfying $\absolute{f(x,\delta)}\leq\deold{0601dePGeo2}$},\\
t_0(x,\delta)&\in[-\Rold{061Rwbound}/2,\Rold{061Rwbound}/2],
\end{align*}
such that
\begin{equation*}
   \psi(a_{\delta}^{(1)}u_{n_{K_3}(x)}^{(1)}.x)=c(x,\delta)a_{f(x,\delta)}^{(2)}u_{t_0(x,\delta)}^{(2)}.\psi(u_{n_{K_3}(x)}^{(1)}.x).
\end{equation*}
This together with \eqref{eq:controlKakCon} and $x,a_{\delta}^{(1)}.x\in Z_4\subset\Zold{035KtimeU}$ implies that 
\begin{equation*}
\psi(a_{\delta}^{(1)}x)= c(x,\delta)a_{f(x,\delta)}^{(2)}u^{(2)}_{w(x,\delta)}.\psi(x),
\end{equation*}
where 
\begin{equation}\label{eq:tux}
    w(x,\delta)=t_0(x,\delta)+\tau(x,n_{K_3}(x))-e^{-2f(x,\delta)}\tau(a_{\delta}^{(1)}. x,e^{2\delta}n_{K_3}(x)) 
\end{equation} 
and $\tau(x,\cdot)$ is defined in \eqref{eq:controlKakCon}. This completes the proof of Claim~\ref{claim:GeoFirstClaim}.
\end{proof}

\begin{Claim}\label{claim:GeoSecondClaim}
    The functions $c(x,\delta)$ and $f(x,\delta)$ defined in Claim~\ref{claim:GeoFirstClaim} are $u_1^{(1)}$-invariant.
\end{Claim}
\begin{proof}[Proof of Claim~\ref{claim:GeoSecondClaim}]
Suppose $x\in Z_4\cap Z_5$, note that $Z_4\cap Z_5$ is an $u_1^{(1)}$-invariant set and thus $u_1^{(1)}.x\in Z_4\cap Z_5$. Then apply  Claim~\ref{claim:GeoFirstClaim} for both $x$ and $u_1^{(1)}.x$
\begin{gather*}
\psi(a_{\delta}^{(1)}.x)= c(x,\delta)a_{f(x,\delta)}^{(2)}u^{(2)}_{w(x,\delta)}.\psi(x),\\
\psi(a_{\delta}^{(1)}u_1^{(1)}.x)=c(u_1^{(1)}.x,\delta)a_{f(u_1^{(1)}.x,\delta)}^{(2)}u^{(2)}_{w(u_1^{(1)}.x,\delta)}.\psi(u_1^{(1)}.x).
\end{gather*}
Together with $\psi(\mathbf{U}_1.x)=\mathbf{U}_2.\psi(x)$ for every $x\in Z_4\cap Z_5\subset \Zold{035KtimeU}$, then there exists $\bar{t}\in\R$ such that
\begin{equation*}
c(x,\delta)a_{f(x,\delta)}^{(2)}.\psi(x)=u_{\bar{t}}^{(2)}c(u_1^{(1)}.x,\delta)a_{f(u_1^{(1)}.x,\delta)}^{(2)}.\psi(x).
\end{equation*}
Let $\overline{\psi(x)}$ be a lift of $\psi(x)$ in $\mathbf{G}_2$, then the above equation imply there exist $\gamma\in\Gamma_2$ and $\tilde{t}\in\R$ satisfying
\begin{equation}\label{eq:geolsemMod}
\tilde{c}a_{\eta}^{(2)}u_{\tilde{t}}^{(2)}.\overline{\psi(x)}=\overline{\psi(x)}\gamma,
\end{equation}
where
\[
\tilde{c}=c(x,\delta)(c(u_1^{(1)}.x,\delta))^{-1},\qquad \eta=f(x,\delta)-f(u_1^{(1)}.x,\delta).
\]
It follows from Claim~\ref{claim:GeoFirstClaim} that
\[
\tilde{c}\in B_{2\deold{0601dePGeo2}}^{\mathbf{G}_2}\cap\exp(\mathfrak{w})\subset\exp(\mathfrak{g}_2^0\oplus\mathfrak{g}_2^{+})\qquad\text{and}\qquad \absolute{\eta}\leq2\deold{0601dePGeo2}.
\]

Since $x\in Z_4\subset Z_3\subset\psi^{-1}(Z_2)$, by the Pointwise Ergodic Theorem for the function $\chi_{\Kold{061KtempCpt1}}$ and flow $a_t^{(2)}$, there exists $s_0>0$ such that 
\[
a^{(2)}_{-s_0}\psi(x)\in\Kold{061KtempCpt1}\quad\text{and}\quad \absolute{e^{-2s_0}\tilde{t}}\leq\deold{0601dePGeo2}.
\]
Applying $a_{-s_0}^{(2)}$ from the left to the both sides of \eqref{eq:geolsemMod}, this together with the $\deold{0601dePGeo2}\leq\frac{\inj(\Kold{061KtempCpt1})}{100}$ gives
\[
\gamma=e,\qquad \tilde{c}=e,\qquad \eta=0,
\]
which prove the $u_1^{(1)}$-invariance of $c(x,\delta)$ and $f(x,\delta)$.
\end{proof}

We can now conclude the proof of Lemma~\ref{lem:Geolm1}. 
By Claim~\ref{claim:GeoSecondClaim} and ergodicity of $u_1^{(1)}$, there exists a full measure $u_1^{(1)}$-invariant set \[
Z\subset Z_4\cap Z_5
\]
such that for every $x\in Z$, 
\[
c(x,\delta)=c(\delta) \text{ and } f(x,\delta)=f(\delta).
\]
for some $c(\delta)\in \exp(\mathfrak w)$ and $\absolute{f(\delta)}< \deold{0601dePGeo2}$, establishing Lemma~\ref{lem:Geolm1}.\hfill$\Box$

\subsubsection{Proof of Proposition~\ref{prop:Geo}}\label{sec:Geolm2} 
\paragraph{\emph{Sketch of the proof: }} By Lemma~\ref{lem:Geolm1},  
\begin{equation}\label{eq:matching in 2}
u^{(2)}_{t} u_{w(x,\delta)}^{(2)}.\psi(x) \quad \text{and} \quad u^{(2)}_{e^{2f(\delta)}t }.\psi(a_{\delta}^{(1)}.x)
\end{equation}
stay very close for all $t$.
On the other hand, 
\begin{equation}\label{eq:matching in 1}
u^{(1)}_{t} .x \quad \text{and} \quad u^{(1)}_{e^{2\delta} t} a_{\delta}^{(1)}. x
\end{equation}
also stay very close for all $t$.
Using Lusin and the fact that $\psi$ is a Kakutani equivalence, \eqref{eq:matching in 1} implies that
\begin{equation}\label{eq:matching in 2'}
u^{(2)}_{\tau(x,t)} .\psi(x) \quad \text{and} \quad u^{(2)}_{\tau(a_{\delta}^{(1)}.x,e^{
2\delta}t) }.\psi(a_{\delta}^{(1)}.x)
\end{equation}
stay close for most $t$, where $\tau$ is the time change cocycle corresponding to $\psi$. Since $\psi$ is an even Kakutani equivalence, $\tau(x,t)/t \to 1$ a.s. Now employing Lemma~\ref{lem:longHamming}, we see that the only way to reconcile \eqref{eq:matching in 2} and \eqref{eq:matching in 2'} is if $f(\delta)=\delta$.
We now proceed with the details of the proof.
\medskip

\noindent Let $\delta\in(0,\deold{0602dePGeo1})$ be fixed as in item~\ref{item:deltaFixed} on~p.~\pageref{item:deltaFixed}. Then Lemma~\ref{lem:Geolm1}~gives
\begin{itemize}
    \item a full measure set $Z=Z(\delta)\subset\mathbf{G}_1/\Gamma_1$,
    \item a measurable function $w(x,\delta):Z\to\R$,
    \item $c(\delta)\in B_{\deold{0601dePGeo2}}^{\mathbf{G}_2}\cap C_{\mathbf{G}_2}(\mathbf{U}_2)$,
    \item $\absolute{f(\delta)}\leq\deold{0601dePGeo2}$,
\end{itemize} 
such that for every $x\in Z$
\begin{equation}\label{eq:geoKakeq}
\psi(a_{\delta}^{(1)}.x)= c(\delta).a_{f(\delta)}^{(2)}.u_{w(x,\delta)}^{(2)}.\psi(x).
\end{equation}

\medskip

We fix a $x$ and an $\epsilon$ as following 
\[
x\in K_2\cap Z,\qquad \epsilon\in(0,10^{-10}\deold{0612dePGeo}),
\]
where $K_2\subset \mathbf{G}_1/\Gamma_1$ is defined in item~\ref{item:geoConstant6} on p.~\pageref{item:geoConstant6}.

\medskip

Some additional constants are needed for this proof:
\begin{enumerate}[label=\textup{(\roman*)}]
    \item Since $\psi$ is an even Kakutani equivalence, if we set 
    \[
    \int_0^{\rho(x,t)}\alpha(u_s^{(1)}.x)ds=t,\qquad \int_0^{t}\alpha(u_s^{(1)}.x)ds=\tau(x,t),
    \]
    with $\alpha$ the time change function corresponding to $\psi$, 
    we obtain from the Pointwise Ergodic Theorem that there exists $N_2=N_2(\epsilon,x,\delta)>N_1$ such that for every $t\geq N_2$ 
\begin{equation}\label{eq:geotime}
\begin{aligned}
     \absolute{\rho(x,t)-t}\leq\epsilon\absolute{t}, \qquad \absolute{\tau(a_{\delta}^{(1)}.x,t)-t}\leq\epsilon\absolute{t},
\end{aligned}
\end{equation}
where $N_1$ is defined in item~\ref{item:ergodicUnipotentGoodSetConstant} on p.~\pageref{item:ergodicUnipotentGoodSetConstant}.

\item Let $N_3=N_3(\epsilon,x,\delta)$ be a large constant defined as
\begin{equation}\label{eq:geoChoiceR}
    N_3=100\left(N_2+\frac{100\absolute{w(x,\delta)}+100}{\epsilon}\right).
\end{equation}

\end{enumerate}

\begin{proof}[Proof of Proposition~\ref{prop:Geo}]
    On the one hand, for every $t \in \R$, we have the points~$u^{(1)}_{t}.x$ and $u^{(1)}_{e^{2\delta}t}a^{(1)}_{\delta}.x$ are $2\delta$-close. Lemma~\ref{lem:matchingPrese} gives for sufficiently large $R$, there exists $A_R\subset[-R,R]$ with $0\in A_R$ and $l(A_R)>(1-10^{-97})2R$ such that
    \[
    \emph{for every $t\in A_R$, $\psi (u^{(1)}_{t}.x)$ and $\psi(u^{(1)}_{e^{2\delta}t}a^{(1)}_{\delta}.x)$ are $\deold{0612dePGeo}$-close.}
    \]
    On the other hand, it follows from \eqref{eq:geoKakeq} that 
    \[
    \emph{for every $t\in\R$, $u_{e^{-2f(\delta)}t+w(x,\delta)}^{(2)}.\psi(x)$ and  $u_{t}^{(2)}.\psi(a_{\delta}^{(1)}.x)$ are $2\deold{0612dePGeo}$-close.}
    \]
    Note that the definitions of $\rho$ and $\tau$ give
\[\psi(u_{\rho(x,t)}^{(1)}x)=u_{t}^{(2)}\psi(x), \qquad \psi(u_{t}^{(1)}x)=u_{\tau(x,t)}^{(2)}\psi(x).
\]
    Then for sufficiently large $R$, $\psi(x)$ and $u_{w(x,\delta)}^{(2)}.\psi(x)$ are $(4\deold{0612dePGeo},10^{-97},R)$-two sides matchable with matching function $h$:
    \begin{equation}\label{eq:hDef}
    h(t)=e^{-2f(\delta)}\tau(a_{\delta}^{(1)}.x,e^{2\delta}\rho(x,t)).
    \end{equation}

    \medskip

    Recall $\psi(x)\in\Kold{036Klh1}$ and let $\widetilde{\psi(x)}\in\mathbf{G}_2$ be a lift of $\psi(x)$. By  Lemma~\ref{lem:longHamming}, there exist $s_R\in[-R,R]$, $L_R\geq3R/2$ and  $\gamma_R\in\Gamma_2$ such that for $p=s_R$ and $s_R+L_R$
\begin{equation}\label{eq:geolong}
    \begin{aligned}
    u_{h(p)+w(x,\delta)}^{(2)}.\widetilde{\psi(x)}&\in\Kak(R,4\deold{0601dePGeo2})u_{p}^{(2)}.\widetilde{\psi(x)}\gamma_R.
    \end{aligned}
\end{equation}
Recall the choice of $x$ implies for some $R\geq N_3$
\begin{equation}\label{eq:geoNormalCompact}
a_{-\frac{1}{2}\log(R/\sqrt{\deold{0601dePGeo2}})}^{(2)}.\psi(x)\in\Kold{061KtempCpt1}.
\end{equation}
Apply $a_{-\frac{1}{2}\log(R/\sqrt{\deold{0601dePGeo2}})}^{(2)}$ from the right to both sides of~\eqref{eq:geolong}, which together with Lemma~\ref{lem:diangonalConjugateKak} and the choice of $\deold{0601dePGeo2}$ give 
\begin{equation*}\label{eq:rtri}
    (a_{-\frac{1}{2}\log(R/\sqrt{\deold{0601dePGeo2}})}^{(2)}.\widetilde{\psi(x)}).\gamma_R^{-1}.(a_{-\frac{1}{2}\log(R/\sqrt{\deold{0601dePGeo2}})}^{(2)}.\widetilde{\psi(x)})^{-1}\in B_{\inj(\Kold{061KtempCpt1})/5)}^{\mathbf{G}_2}.
\end{equation*}
Combining this inclusion with \eqref{eq:geoNormalCompact} and definition of $\inj(\Kold{061KtempCpt1})$, we must have that
$\gamma_R=e$.

\medskip

Then this together with \eqref{eq:geolong} gives that there is a $C>0$ depending only on $\mathbf{G}_2$ so that
\begin{equation}\label{eq:srlr inequality}
\begin{aligned}
    &\absolute{h(s_R)-s_R+w(x,\delta)}<C\deold{0601dePGeo2},\\
    &\absolute{h(L_R+s_R)-L_R-s_R+w(x,\delta)}<C\deold{0601dePGeo2}.
\end{aligned}
\end{equation}
Since $L_R\geq3R/2$ and $R\geq N_3$,
\begin{equation*}
\begin{aligned}
    \max(\absolute{s_R},\absolute{s_R+L_R})\geq &\max\left(N_2,\frac{100\absolute{w(x,\delta)}+100}{\epsilon}\right).
    \end{aligned}
\end{equation*}
Together with $h(0)=0$, \eqref{eq:geotime}, the definition of $h$ in \eqref{eq:hDef} and \eqref{eq:srlr inequality},
 \[
 \absolute{-f(\delta)+\delta}<100\epsilon.
 \]
Since  $\epsilon>0$ is arbitrary this completes the proof.
\end{proof}

\subsection{Controlling deviation in the centralizer of \texorpdfstring{$\mathbf{H}_2$}{H2}}\label{sec:trivialGeo}
In this section, we provide an improved version of Proposition~\ref{prop:Geo}, by showing that the deviation $c(\delta)$  in the direction $C_{\mathbf{G}_2}(\mathbf{U}_2)$ between $\psi(x)$ and $\psi(a_\delta^{(1)}.x)$ is in fact in $C_{\mathbf{G}_2}(\mathbf{U}_2)\cap \mathbf{G}_2^{+}$.
 This will allow us to cancel the term $c(\delta)$ altogether by replacing $\psi$ with $c\psi$ for a suitable $c \in C_{\mathbf{G}_2}(\mathbf{U}_2)\cap \mathbf{G}_2^{+}$.
 
\begin{Proposition}\label{prop:GeoTri}
The group element $c(\delta)$ in Proposition~\ref{prop:Geo} satisfies  
\[
\log(c(\delta))\in \mathfrak{v}^+,
\]
where $\mathfrak{v}^+$ is defined in \eqref{eq:Cgn}.
\end{Proposition}

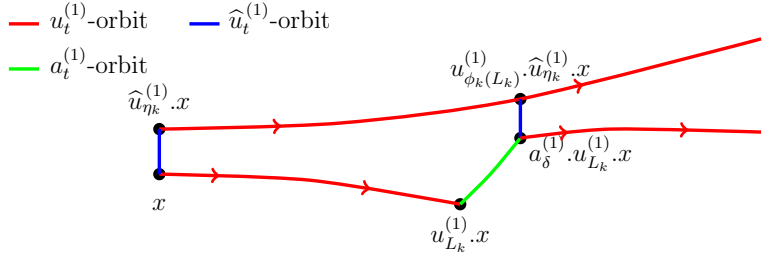
\begin{figure}[H]
    \scalebox{0.8}
  {
  \begin{tikzpicture}[scale=5]
  \tikzstyle{vertex}=[circle,minimum size=20pt,inner sep=0pt]
 	 \tikzstyle{selected vertex} = [vertex, fill=red!24]
 	 \begin{scope}[thick,decoration={
    markings,
    mark=at position 0.2 with {\arrow{>}},
    mark=at position 0.7 with {\arrow{>}}}
] 

      \draw[ultra thick,red,-] plot[smooth] coordinates{(-0.5,0.75)(-0.4,0.75)};
      \node[vertex] (v10) at (-0.2,0.77) {$u^{(1)}_t$-orbit};

      \draw[ultra thick,blue,-] plot[smooth] coordinates{(0.1,0.75)(0.2,0.75)};
      \node[vertex] (v20) at (0.4,0.77) {$\oua_t$-orbit};
      
      \draw[ultra thick,green,-] plot[smooth] coordinates{(-0.5,0.6)(-0.4,0.6)};
      \node[vertex] (v30) at (-0.2,0.62) {$a^{(1)}_t$-orbit};
      
      \node[vertex] (v00) at (0,0.15) {$x$};
      \draw[fill=black] (0,0.25) circle (0.5pt);
      
      \node[vertex] (v01) at (0,0.5) {$\oua_{\eta_k}.x$};
      \draw[fill=black] (0,0.4) circle (0.5pt);
      
      \node[vertex] (v02) at (1,0.05) {$u_{L_k}^{(1)}.x$};
      \draw[fill=black] (1,0.15) circle (0.5pt);
      
      \node[vertex] (v03) at (1.2,0.6) {$u_{\phi_k(L_k)}^{(1)}.\oua_{\eta_k}.x$};
      \draw[fill=black] (1.2,0.5) circle (0.5pt);

      \draw[ultra thick,postaction={decorate},red,-] plot[smooth] coordinates{(0,0.25)(0.5,0.23)(1,0.15)};

      \draw[ultra thick,postaction={decorate},red,-] plot[smooth] coordinates{(0,0.4)(0.6,0.42)(1.2,0.5)(2,0.7)};

      %\node[vertex] (v04) at (0.2,0.25) {$v^{(1)}$-orbit};
      \draw[ultra thick,blue,-] plot[smooth] coordinates{(0,0.25)(0,0.4)};

      \node[vertex] (v05) at (1.4,0.32) {$a_{\delta}^{(1)}.u_{L_k}^{(1)}.x$};
      \draw[fill=black] (1.2,0.37) circle (0.5pt);
      \draw[ultra thick,green,-] plot[smooth] coordinates{(1,0.15)(1.1,0.25)(1.2,0.37)};
      \draw[ultra thick,blue,-] plot[smooth] coordinates{(1.2,0.37)(1.2,0.5)};

      \draw[ultra thick,postaction={decorate},red,-] plot[smooth] coordinates{(1.2,0.37)(1.5,0.4)(2,0.39)};
   \end{scope}    
  \end{tikzpicture}
  }
    \caption{Sketch of proof of Proposition~\ref{prop:GeoTri}}
    \label{fig:proofIdeas}
\end{figure}
Roughly speaking, the proof of Proposition~\ref{prop:GeoTri} can be explained through Figure~\ref{fig:proofIdeas}. Recall that $\oua_t=\exp(t\bou_1)$, cf.~\S\ref{sec:sl2basis}.
\begin{enumerate}[label=\textup{(\roman*)}] 
    
    \item Let $\eta_k\to 0$  be a given sequence. Using this sequence, we construct a sequence $L_k\to\infty$ and a sequence of ``good'' matching functions $\phi_k : \R \to \R$. For suitable choice of such sequences, we can find for every $k$ a ``good'' point\footnote{$x$ implicitly depends on $k$.} $x$ such that 
    \smallskip
    \begin{itemize}
        \item the following points are also ``good''
    \[
    u_{L_k}^{(1)}.x,\quad a_{\delta}^{(1)}u_{L_k}^{(1)}.x,\quad  \oua_{\eta_k}.x,\quad u_{\phi_k(L_k)}^{(1)}\oua_{\eta_k}.x,
    \]
    \item there exists $c>0$ such that
    \begin{equation}
    \label{eq: at L_k}
u_{\phi_k(L_k)}^{(1)}\oua_{\eta_k}.x=\oua_{c\eta_k}a_{\delta}^{(1)}u_{L_k}^{(1)}.x.
    \end{equation}
    \end{itemize}
    
    \item Since $\psi$ is uniformly continuous on  ``good'' points, we obtain
    \[
    d_{\mathbf{G}_2/\Gamma_2}(\psi(\oua_{\eta_k}.y),\psi(y))\leq\beta_k,
    \]
    for some sequence $\beta_k\to 0$ (depending in an ineffective way on~$\eta_k$)\footnote{In the proof we actually first chose $\beta_k$, then take $\eta_k$ to be compatible with that choice.}.

    \item\label{item:KakLarge2} Our choice of parameters imply that  $u_t^{(1)}.x$ and $u_{\phi_k(t)}^{(1)}.\oua_{\eta_k}.x$ stay $2\delta$ close for every $t\in[0,L_k]$, hence
    \[
    x\in\Kak(L_k,2\delta,\oua_{\eta_k}.x).
    \]
    Using the above property of Kakutani isomorphism, and as long as $\delta$ is small enough (depending on $\epsilon>0$) this implies that
    \begin{equation}
    \label{eq: long Kakutani}
\psi(x)\in\Kak\left(L_k,\epsilon,\psi(\oua_{\eta_k}.x)\right).
    \end{equation}
    
    \item\label{item:auxLemma2} In view of \eqref{eq: at L_k}, we can apply the above also to the two points $a_{\delta}^{(1)}.u_{L_k}^{(1)}.x$ and  $u_{\phi_k(L_k)}^{(1)}.\oua_{\eta_k}.x$
    \begin{gather*}
\psi(a_{\delta}^{(1)}u_{L_k}^{(1)}.x)\in\Kak\left(L_k,\epsilon,\psi(u_{\phi_k(L_k)}^{(1)}\oua_{\eta_k}.x)\right),\\
  d_{\mathbf{G}_2/\Gamma_2}\left(\psi(a_{\delta}^{(1)}u_{L_k}^{(1)}.x), \ \ \psi(u_{\phi_k(L_k)}^{(1)}\oua_{\eta_k}.x)\right)\leq \beta_k.
    \end{gather*}

    \item Equation \eqref{eq: long Kakutani} together with \ref{item:auxLemma2} imply that the difference in the $\mathbf{Z}_2=C_{\mathbf{G}_2}(\mathbf{H}_2)$ direction between $\psi(u^{(1)}_{L_k}.x)$ and $\psi(u_{\phi_k(L_k)}^{(1)}.\oua_{\eta_k}.x)$ is of size $O (\beta_k)$.
    Applying Proposition~\ref{prop:Geo} to  $\psi(a_{\delta}^{(1)}.u_{L_k}^{(1)}.x)$ and $\psi(u_{L_k}^{(1)}.x)$, and using also the estimate from \ref{item:auxLemma2}, shows that the $\mathbf{Z}_2$ component of $c(\delta)$ is of size bounded by $O(\beta_k)$.

\end{enumerate}

\medskip
\noindent
By letting $k\to\infty$, we complete the proof.

\medskip

We start with the choices of constants and sets (cf.~\S\ref{sec:nonTrivialConsSet} and \S\ref{sec:kxmatching}). The proof of Proposition~\ref{prop:GeoTri} will be presented in \S\ref{sec:proofGeoTri}. 

\subsubsection{Notation}\label{sec:nonTrivialConsSet}
Here is a table of notations from previous sections. 

    \begin{table}[H]
    \centering
    \begin{tabular}{c|c}
    \hline
      Notation   & Origin and definition \\
    \hline
      $\deold{0601dePGeo2}$, $\deold{0602dePGeo1}$, $\Rold{061Rwbound}$, $S(\cdot)$&  Proposition~\ref{prop:Geo}  \\
    \hline
    $\Kold{035Klm1}$, $\Rold{035Rlm1}$, $\deold{035delm1}$, $\eold{035eMainDel}(\cdot)$  & Lemma~\ref{lem:main} with $\eta=10^{-8}$\\
    \hline
    $\Kold{061KtempCpt1}$ & item \ref{item:KtempCpt1} on p.~\pageref{item:KtempCpt1}\\
    \hline
    \end{tabular}
    \caption{\label{table:SmallCenter}}
    \end{table}

    In addition we will also need the following:
    \begin{enumerate}[label=\textup{(n:\roman*)}] 
        \item Fixed positive constants $\delta$ and $\zeta$ with
        \[
        \delta\in(0,\deold{0602dePGeo1}) \quad \text{ and }\quad \zeta=1-\exp(-\delta).
        \]
        \item\label{n:beta-k} For every $k\in\N$, define $\beta_k$ by 
        \[
        \beta_k=\deold{0601dePGeo2}/(k\Cold{039infiniteMatch}).
        \]
        \item\label{n:n-k} Let $n_k$ be the smallest positive integer such that 
        \[
        2\sqrt{\delta/n_k}<\min(\eold{035eMainDel}(\beta_k),\delta).
        \]
        \item\label{item:sigmak} Let $\eta_k$, $\zeta_k$, $\sigma_k$, $L_k$ and $L'_k$ be defined as
        \begin{alignat*}{2}
           \eta_k=\delta/n_k,\qquad  &\sigma_k=\eta_k(1-\zeta), \qquad  &&L_k=\zeta/\eta_k.
        \end{alignat*}
        \item\label{item:WM1} Let $w(x,\delta)$ be defined as in Proposition~\ref{prop:Geo} with above fixed $\delta$. Given $M>0$, let $W(M)\subset\mathbf{G}_1/\Gamma_1$ be defined as
        \[
        W(M)=\{x\in \Zold{060Kpn1}:\absolute{w(x,\delta)}<M\}.
        \]
        Since $w(x,\delta)$ is measurable, there exists $M_1>0$ such that 
        \[
        m_1(W(M_1))>0.9.
        \]
        \item \label{eq:LkW2} For $k\in\N$, let $K_1,S_k\subset \mathbf{G}_1/\Gamma_1$ and $K_2\subset \mathbf{G}_2/\Gamma_2$ be defined as
\begin{gather*}
K_1=\Kold{035Klm1}\cap W(M_1),\qquad 
K_2=a_{\frac{1}{2}\log(L_k/\sqrt{\deold{0601dePGeo2}})}^{(2)}\Kold{061KtempCpt1},\\
 S_k=K_1\cap \oua_{-\eta_k}(K_1)\cap u_{-L_k}^{(1)}.a_{-\delta}^{(1)}\left(K_1\cap \oua_{-\sigma_k}(K_1) \right)\cap\psi^{-1}(K_2).
\end{gather*}

 \item  Let $N_1$ be a positive constant defined as 
        \[
        N_1=10^{1000}M_1^2\Rold{035Rlm1}^2.
        \]
\end{enumerate}

\medskip
    
Here are some simple observations from our choices of notation:
\begin{enumerate}[label=\textup{(o:\roman*)}] 
    \item The choice of $n_k$ guarantees that: 
    \begin{equation*}
        \text{if $x,y\in \Kold{017KmP1}$ and $d_{\mathbf{G}_1/\Gamma_1}(x,y)<2\sqrt{\eta_k}$, then $
d_{\mathbf{G}_2/\Gamma_2}(\psi(x),\psi(y))<\beta_k$.}
    \end{equation*}
    \item The choices of $\delta$ and $\zeta$ guarantee that 
    \[
    0<\zeta\leq\delta\leq\frac{\eold{014EMa}}{2},
    \]
    where $\eold{014EMa}$ is a constant from Lemma~\ref{lem:matchingFunction}.

\end{enumerate}

\subsubsection{Selection of $k$, $x$ and matching function}\label{sec:kxmatching}
We fix $k$ and $x$ as following.
\begin{itemize}
    \item First fix  $k\geq N_1$.
    \item It follows from the choice of $N_1$ and \ref{eq:LkW2} that 
$    m_1(S_k)>0$.
Fix $x$ such that
\begin{equation}\label{eq:choiceX}
    x\in S_k.
\end{equation} 
\end{itemize}

\medskip

 For every $\eta\in(0,\eta_k]$, let~$\phi_{\eta}$ be defined as 
\begin{equation}\label{eq:phiEta}
\phi_{\eta}(t)=\frac{t}{1-t\eta}. 
\end{equation}
The function $\phi_\eta(t)$ is the best matching function (in the sense explicated in Lemma~\ref{lem:matchingFunction}) between $x$ and $\oua_{\eta}.x$. Then
Lemma~\ref{lem:matchingFunction} gives that for every $\absolute{t}\in[0,\eold{014EMa}\eta^{-1}]$
\begin{equation}\label{eq:pvm1}
    \begin{aligned}
    u_{\phi_{\eta}(t)}^{(1)}\oua_{\eta}u_{-t}^{(1)}&=\oua_{s_t}a_{p_t}^{(1)},
    \end{aligned}
\end{equation}
where 
\begin{equation}\label{eq:pvm2}
    \begin{aligned}
    &s_t=\eta(1-t\eta), \qquad p_t=-\log(1-t\eta).
    \end{aligned}
\end{equation}
Moreover, for every $\epsilon\in(0,\eold{014EMa}]$, Lemma~\ref{lem:matchingFunction} gives that 
\begin{equation}\label{eq:pvm3}
\begin{aligned}
    &\absolute{\phi'_{\eta}(t)-1}<\sqrt{\epsilon}\text{ for } \absolute{t}\leq\epsilon\eta^{-1}.
    \end{aligned}
\end{equation}

\subsubsection{Proof of Proposition~\ref{prop:GeoTri}}\label{sec:proofGeoTri}
Our choice of $x$ in particular guarantees that $x,\oua_{\eta_k}.x \in \Kold{035Klm1}$. Applying Lemma~\ref{lem:main} for $x$ and $\oua_{\eta_k}.x$ with $R=L_k$ there is $g_1\in\mathbf{G}_2$ satisfying 
\[
\psi(\oua_{\eta_k}.x)=g_1.\psi(x)
\]
so that
\[
    g_1 \in \Kak(L_k,\deold{0601dePGeo2}).
\]
Moreover, since $\psi$ is uniformly continuous on $\Kold{035Klm1}$, by choice of $\eta_k$,
\begin{equation}\label{eq:centsmall2}
\begin{gathered}
\absolute{\vartheta_{0,j}(g_1)}<\beta_k\text{ for }1\leq j\leq n^{(2)}
\end{gathered}
\end{equation}
(other components of $g_1$ are also smaller than $\beta_k$, but we need this only for $\vartheta_{0,j}(g_1)$).

Apply \eqref{eq:pvm1} and \eqref{eq:pvm2} with $t=L_k$ and $\eta=\eta_k$,
\[
\oua_{\sigma_k}a_{\delta}^{(1)}u_{L_k}^{(1)}.x= u_{\phi_{\eta_k}(L_k)}^{(1)}\oua_{\eta_k}.x.
\]
Here $\sigma_k$ is defined in item~\ref{item:sigmak} on p.~\pageref{item:sigmak}. Apply Lemma~\ref{lem:main} for  $a_{\delta}^{(1)}u_{L_k}^{(1)}.x$, $u_{\phi_{\eta_k}(L_k)}^{(1)}\oua_{\eta_k}.x$ and $R=L_k$, there exists $g_2\in\mathbf{G}_2$ satisfying 
\[
\psi(u_{\phi_{\eta_k}(L_k)}^{(1)}\oua_{\eta_k}.x)=g_2.\psi(a_{\delta}^{(1)}u_{L_k}^{(1)}.x)
\]
so that
\begin{equation}\label{eq:centsmall4}
\begin{gathered}
g_2\in\Kak(L_k,\deold{0601dePGeo2}), \\ \absolute{\vartheta_{0,j}(g_2)}<\beta_k\text{ for }1\leq j\leq n^{(2)}.
\end{gathered}
\end{equation}

\medskip

Since $\psi(\mathbf{U}_1.x)=\mathbf{U}_2.\psi(x)$ (cf.~\eqref{eq:controlKakCon}), Proposition~\ref{prop:Geo} gives
\begin{equation*}
    \begin{aligned}
    \psi(a_{\delta}^{(1)}u_{L_k}^{(1)}.x)&=\psi(u_{e^{2\delta}L_k}^{(1)}a_{\delta}^{(1)}.x)=u^{(2)}_{\tau(a_{\delta}^{(1)}.x,e^{2\delta}L_k)}.\psi(a_{\delta}^{(1)}.x)\\
    &=u^{(2)}_{\tau(a_{\delta}^{(1)}.x,e^{2\delta}L_k)}c(\delta)a_{\delta}^{(2)}u_{w(x,\delta)}^{(2)}.\psi(x),\\
    \psi(u_{\phi_{\eta_k}(L_k)}^{(1)}\oua_{\eta_k}.x)&=u^{(2)}_{\tau(\oua_{\eta_k}.x,\phi_{\eta_k}(L_k))}.\psi(\oua_{\eta_k}.x).
    \end{aligned}
\end{equation*}
It follows from the above equations and the definitions of $g_1,g_2$ that for some $\gamma_3\in\Gamma_2$
\begin{equation}\label{eq:geoTriFin}
    \begin{aligned}
u_{s_1}^{(2)}g_1.&\widetilde{\psi(x)}=g_2u_{s_2}^{(2)}c(\delta)a_{\delta}^{(2)}.\widetilde{\psi(x)}\gamma_3,
    \end{aligned}
\end{equation}
where $\widetilde{\psi(x)}$ is a lift of $\psi(x)$ to $\mathbf G_2$, and $s_1$, $s_2$ are defined as
\begin{gather*}
    s_1=\tau(\oua_{\eta_k}.x,\phi_1(L_k)),\quad
    s_2=\tau(a_{\delta}^{(1)}.x,e^{2\delta}L_k)+e^{2\delta}w(x,\delta).
\end{gather*}
Combining \eqref{eq:controlKakCon}, the definition of $W(M_1)$ (cf.~\ref{item:WM1} on p.~\pageref{item:WM1}) and the choice of $\delta$ and $N_1$, we have that
\begin{multline}\label{eq:timeControlPsi}
\absolute{\tau(a_{\delta}^{(1)}.x,e^{2\delta}L_k)+e^{2\delta}w(x,\delta)}<(1+10^{-98})\absolute{\tau(\oua_{\eta_k}.x,\phi_{\eta_k}(L_k))}.
\end{multline}
Apply $a_{-\frac{1}{2}\log(L_k/\sqrt{\deold{0601dePGeo2}})}^{(2)}$ from left to the both sides of \eqref{eq:geoTriFin}, then since $x$ is in the ``good'' set $S_k$, and using $g_1,g_2\in\Kak(L_k,\deold{0601dePGeo2})$, Lemma~\ref{lem:diangonalConjugateKak} and~\eqref{eq:timeControlPsi} gives $\gamma_3=e$, which together with \eqref{eq:geoTriFin} guarantee
\begin{equation}\label{eq:trivalLa}
u_{-s_2}^{(2)}g_2^{-1}u_{s_1}^{(2)}g_1a_{-\delta}^{(2)}=c(\delta).
\end{equation}

\medskip

Recall that $c(\delta)$ can be uniquely written as
\begin{equation*}
   c(\delta)=\prod_{j=1}^{n^{(2)}}\exp(c_j(\delta)\bx_2^{0,j}).
\end{equation*}
Conjugate both sides of \eqref{eq:trivalLa} by $a_{-\frac{1}{2}\log(L_k/\sqrt{\deold{0601dePGeo2}})}^{(2)}$, then by \eqref{eq:centsmall2},  \eqref{eq:centsmall4}, the fact that $g_1,g_2\in\Kak(L_k,\deold{0601dePGeo2})$ and Lemma~\ref{lem:diangonalConjugateKak}
\begin{equation}\label{eq:tempCtrivail}
    u_{-\bar{s}_2}^{(2)}\bar{g}_2^{-1}u_{\bar{s}_1}^{(2)}\bar{g}_1a_{-\delta}^{(2)}=\prod_{j=1}^{n^{(2)}}\exp(L_k^{-q_j/2}c_j(\delta)\bx_2^{0,j}),
\end{equation}
where $\absolute{\bar{s}_1},\absolute{\bar{s}_2}\leq 2$,
\begin{equation}\label{eq:many bounds on g1g2}
\begin{gathered}
   \bar{g}_1^{\mathfrak{z}}\in B_{\Cold{038Clh1}\beta_k\dim\mathfrak{g}_2}^{\mathbf{G}_2}, \quad \bar{g}_1^{\mathfrak{h}}\in B_{\sqrt{\deold{0601dePGeo2}}\dim\mathfrak{g}_2}^{\mathbf{G}_2},\quad 
\bar{g}_1^{\tr}\in B_{\sqrt{\deold{0601dePGeo2}}\dim\mathfrak{g}_2/\sqrt{L_k}}^{\mathbf{G}_2},\\
\bar{g}_2^{\mathfrak{z}}\in B_{\Cold{038Clh1}\beta_k\dim\mathfrak{g}_2}^{\mathbf{G}_2}, \quad \bar{g}_2^{\mathfrak{h}}\in B_{\sqrt{\deold{0601dePGeo2}}\dim\mathfrak{g}_2}^{\mathbf{G}_2},\quad 
\bar{g}_2^{\tr}\in B_{\sqrt{\deold{0601dePGeo2}}\dim\mathfrak{g}_2/\sqrt{L_k}}^{\mathbf{G}_2},
\end{gathered}
\end{equation}
and moreover by \eqref{eq:timeControlPsi}, \ $\absolute{\bar{s}_1-\bar{s}_2}\leq 10^{-98}\sqrt{\deold{0601dePGeo2}}$.
It follows from \eqref{eq:tempCtrivail} and~\eqref{eq:many bounds on g1g2} that there are $g_3\in \mathbf{H}_2$ and $g_4\in\mathbf{G}_2$ so that
\[
u_{-\bar{s}_2+\bar{s}_1}^{(2)}g_3(\bar{g}_2^{\mathfrak{z}})^{-1}\bar{g}_1^{\mathfrak{z}}g_4=\prod_{j=1}^{n^{(2)}}\exp(L_k^{-q_j/2}c_j(\delta)\bx_2^{0,j})
\]
with
\begin{gather*}
    \norm{g_3-\id}\leq100\sqrt{\deold{0601dePGeo2}}(\dim\mathfrak{g}_2)^2,\quad
    \norm{g_4-\id}\leq100\sqrt{\deold{0601dePGeo2}}(\dim\mathfrak{g}_2)^2/\sqrt{L_k}.
\end{gather*}
Then it follows from Lemma~\ref{lem:groupDecom},
\[
\absolute{c_j(\delta)}<C\beta_k+\frac{C}{\sqrt{L_k}} \text{ for every }j\in I_c,
\]%\todo{double check whether $L_k^{-1}$ is needed or not, seems to me not now..}
where $I_c=\{1\leq j\leq n^{(2)}:q_j=0\}$ and $C>0$ is a constant that depends only on $\mathbf{G}_2$. 

Since $k$ is arbitrary, $\beta_k\to0$ and $L_k\to+\infty$ as $k\to+\infty$, we obtain that $c_j(\delta)=0$ for every $j\in I_c$, finishing the proof of Proposition~\ref{prop:GeoTri}.\hfill$\Box$

\subsection{Proof of Proposition~\ref{prop:CompatibleGeoMain}}\label{sec:CompatibleGeoMain}
Using Proposition~\ref{prop:Geo} and Proposition~\ref{prop:GeoTri} it is now easy to complete the proof of Proposition~\ref{prop:CompatibleGeoMain}.

Combining Proposition~\ref{prop:Geo} and Proposition~\ref{prop:GeoTri} we see that for $\delta$ small enough there is a a full measure subset  $\Zold{060Kpn1}$ and 
 \[c= c(\delta)\in \mathbf V^+
 \]
 so that
\[
\psi(a^{(1)}_{\delta}.x)\in  c a_{\delta}^{(2)}\mathbf{U}_2.\psi(x). 
\]
Consider the map 
\[
\mathcal T: v \to a_{-\delta}^{(2)}vca_{\delta}^{(2)}.
\]
Since the metric on $\mathbf{V^+}$ is right invariant, this is a contraction on $\mathbf{V^+}$, and given $c$ one can find a bounded open set $\mathcal U \subset \mathbf V^{+}$ so that $\mathcal T$ maps $ \mathcal U$  to a subset of itself (indeed, take any neighbourhood of identity $\tilde{\mathcal U} \subset \mathbf V^{+}$ that is mapped under conjugation by $a_{-\delta}^{(2)}$ to a subset of itself, and take $\mathcal U$ to be $\tilde{\mathcal U}$ conjugated by $a_t^{(2)}$ for $t$ large).

It follows that $\mathcal T$ has a fixed point in $ \overline{\mathcal U}$, say $\bar v$. Then
\[
\bar{v} c=a_{\delta}^{(2)}\bar{v}a_{-\delta}^{(2)}
\]
and hence if $\tilde\psi(x) = \bar v \psi(x)$, 
\begin{align*}
\tilde\psi(a^{(1)}_{\delta}.x) &= \bar v \psi(a^{(1)}_{\delta}.x) \in \mathbf{U}_2 \bar v c a^{(2)}_{\delta}. \psi(x)\\
&\qquad= \mathbf{U}_2 a_{\delta}^{(2)}\bar v.\psi(x) =\mathbf{U}_2 a_{\delta}^{(2)}.\tilde{\psi}(x). 
\end{align*}
It follows that for any $n \in \N$
\[
\tilde\psi(a^{(1)}_{n\delta}.x)\in a_{n\delta}^{(2)}\mathbf{U}_2.\tilde\psi(x).
\]
Since we may as well take $\delta$ so that $\delta^{-1}$ is an integer this completes the proof of Proposition~\ref{prop:CompatibleGeoMain}.

\begin{Remark}\label{rmk:CPsiKak} In view of Proposition~\ref{prop:CompatibleGeoMain}, we may as well (as we will do in the remainder of this paper) assume $\psi$ itself satisfies 
\begin{equation*}
\psi(a^{(1)}_{n\delta}.x)\in a_{n\delta}^{(2)}\mathbf{U}_2.\psi(x).
\end{equation*}
\end{Remark}

\section{Renormalization}\label{sec:renormalization}
In this section, we study the limit behavior of $\psi_{n}(x)=a_{-n}^{(2)}\psi(a_{n}^{(1)}.x)$ as $n\to+\infty$. We show that for $m_1$-a.e. $x\in \mathbf{G}_1/\Gamma_1$, there exists a subsequence $\{n_i(x)\}_{i\in\N}$ such that $\phi(x)=\lim_{i\to\infty}\psi_{n_i(x)}(x)$ exists. Moreover, we show that $\phi:\mathbf{G}_1/\Gamma_1\to \mathbf{G}_2/\Gamma_2$ is an invertible measurable isomorphism between the $u^{(1)}_t$-flow  on $\mathbf{G}_1/\Gamma_1$ and the $u^{(2)}_t$-flow  on $\mathbf{G}_2/\Gamma_2$. Such isomorphisms are well understood by \cites{ratner1982rigidity, witte1985rigidity,witte1989rigidity,ratner1990measure}, and in particular are induced by algebraic isomorphisms between the corresponding groups.  

This strategy of renormalization was used by Ratner in \cite{ratner1986rigidity} when she studied the time change rigidity for horocycle flows \emph{with a H\"older regularity assumption}. Without this regularity assumption establishing convergence of the $\psi_n$ is more delicate. To overcome this we will employ an $L^2$-ergodic theorem from Appendix \ref{sec:sl2ergodic} to find ``good'' points to control the behaviour of $\psi(a^{(1)}_n.x)$ (or at least $\psi(a^{(1)}_n.x')$ for $x'$ close to $x$) for all $n$. The main part of the proof is carried out in~\S\ref{sec:pfCau}. We begin that subsection by given an overview of the argument.

\medskip

Recall that $\tau(x,s)$ is the time change function defined by 
$\psi(u_t^{(1)}.x)=u_{\tau(x,t)}^{(2)}.\psi(x)$, cf.~\S\ref{sec:timeChagneGoodSet}. The map $\psi_n(x)$ is also an even Kakutani equivalence between the corresponding unipotent flows on $\mathbf{G}_1/\Gamma_1$ and $\mathbf{G}_2/\Gamma_2$ and it has a corresponding time change function. Explicitly, let
\[
\tau_n(x,t)=e^{-2n}\tau(a_{n}^{(1)}.x,e^{2n}t),
\]
then for $m_1$-a.e. $x\in\mathbf{G}_1/\Gamma_1$ and for all $t\in\R$
\begin{equation}\label{eq:psinTimeChange}
    \psi_n(u_t^{(1)}.x)=u_{\tau_{n}(x,t)}^{(2)}.\psi_{n}(x).
\end{equation}

Let $\{n_i\}_{i\in\N}$ be a sequence of natural numbers. For any $n\in\N$, we define
\[
d_n(\{n_i\}_{i\in\N})=\operatorname{Card}(\{n_i\}_{i\in\N}\cap[1,n]).
\]

\begin{Definition}[Full density]
Given a sequence of integer valued functions $\{n_i(x)\}_{i\in\N}$, with $n_i(x)$ strictly increasing in $i$ for every $x$, we say it is of \textbf{full density} if for $m_1$-a.e. $x\in\mathbf{G}_1/\Gamma_1$
\[
\lim_{k\to+\infty}\frac{1}{k}d_{k}(\{n_i(x)\}_{i\in\N})=1.
\]
\end{Definition}

The following is the main theorem of this section:
\begin{Theorem}\label{thm:reno}
For $m_1$-a.e. $x\in \mathbf{G}_1/\Gamma_1$, there exists a subsequence $\{n_i(x)\}_{i\in\N}$ of natural numbers with full density such that 
\[
\phi(x)=\lim_{i\to+\infty} \psi_{n_i(x)}(x)=\lim_{i\to+\infty}a_{-n_i(x)}^{(2)}\psi(a_{n_i(x)}^{(1)}.x)
\]
exists and the following hold
\begin{enumerate}[label=(\alph*)]
    \item\label{item:measurable} $\phi:\mathbf{G}_1/\Gamma_1\to \mathbf{G}_2/\Gamma_2$ is a measurable map,
    \item\label{item:stayUorbit} $\phi(x)\in \mathbf{U}_2.\psi(x)$,
    \item\label{item:uCommute} $\phi(u_t^{(1)}.x)=u_t^{(2)}.\phi(x)$ for $m_1$-a.e. $x\in \mathbf{G}_1/\Gamma_1$ and every $t\in\mathbb{Z}$.
\end{enumerate}
\end{Theorem}

In the next subsection \S\ref{sec:thmreno} we state Lemma~\ref{lem:fullDensitySeq}, which is essentially identical to Theorem~\ref{thm:reno} but without explicitly stating \ref{item:measurable} above. As we explain in that subsection, the measurability of $\phi$ follows from a standard argument. 

\subsection{Measurability of $\phi$}\label{sec:thmreno}

The main result in the remainder of \S\ref{sec:renormalization} is the following lemma:

\begin{Lemma}\label{lem:fullDensitySeq}
There exist a full measure Borel set  $\Znew\label{073KsCauF}\subset \mathbf{G}_1/\Gamma_1$\index{$\Zold{073KsCauF}$, Lemma~\ref{lem:fullDensitySeq}} and a collection of measurable integer values functions $\{n_i(x)\}_{i\in\N}$ on $\Zold{073KsCauF}$ such that
\begin{enumerate}[label=\textup{(\alph*)}] 
    \item\label{item:ptlimit} for every $x\in\Zold{073KsCauF}$, $\phi(x)=\lim_{i\to+\infty}\psi_{n_i(x)}(x)$ exists and lies on the $\mathbf{U}_2$-orbit of $\psi(x)$,
    \item\label{item:interU} for every $t\in\Z$ and every  $x\in \Zold{073KsCauF}$, $\phi(u_t^{(1)}.x)=u_t^{(2)}.\phi(x)$,
    \item\label{item:ptupperlimit} $\lim_{k\to+\infty}\frac{1}{k}d_{k}(\{n_i(x)\}_{i\in\N})=1$ for every $x\in \Zold{073KsCauF}$.
\end{enumerate}
\end{Lemma}

In this section we explain how it implies Theorem~\ref{thm:reno}. The only issue that needs to be explained is the measurability of $\phi$, i.e. \ref{item:measurable} of that theorem.

\medskip

\begin{proof}[Proof of Theorem~\ref{thm:reno} assuming Lemma~\ref{lem:fullDensitySeq}]

Let $\Zold{073KsCauF}$, $\phi$ and $\{n_i(x)\}_{i\in\N}$ be defined as in Lemma~\ref{lem:fullDensitySeq}. The \ref{item:stayUorbit} and \ref{item:uCommute} of Theorem~\ref{thm:reno} follows directly from \ref{item:ptlimit} and \ref{item:interU} of Lemma~\ref{lem:fullDensitySeq}.

Let $\mathcal{M}(\mathbf{G}_2/\Gamma_2)$ be the space of Borel probability measures on $\mathbf{G}_2/\Gamma_2$ and $\delta_x$ be the Dirac measure on point $x\in\mathbf{G}_2/\Gamma_2$. We then define $\zeta_N:\mathbf{G}_1/\Gamma_1\to\mathcal{M}(\mathbf{G}_2/\Gamma_2)$ as
\begin{equation*}
    \zeta_N(x)=\frac{1}{N}\sum_{i=0}^{N-1}\delta_{\psi_{n_i(x)}(x)}.
\end{equation*}

Recall by \ref{item:ptlimit} of Lemma~\ref{lem:fullDensitySeq}, for every $x\in \Zold{073KsCauF}$,
\begin{gather*}
    \phi(x)=\lim_{i\to+\infty}\psi_{n_i(x)}(x),\qquad 
    \lim_{k\to+\infty}\frac{1}{k}d_{k}(\{n_i(x)\}_{i\in\N})=1.
\end{gather*}
Thus we obtain that for any $x \in \Zold{073KsCauF}$, the sequence of probability measures $\zeta_{k}(x)$ on $\mathbf{G}_2/\Gamma_2$ converges to $\delta_{\phi(x)}$ in the weak$^{*}$ topology. 
It follows that \[
x\mapsto\delta_{\phi(x)}
\]
is Borel measurable on the full measure Borel set $\Zold{073KsCauF}$, hence so is \[
x\mapsto\phi(x).
\]
This proves \ref{item:measurable} of Theorem~\ref{thm:reno}.
\end{proof}

\subsection{Proof of Lemma~\ref{lem:fullDensitySeq}}\label{sec:pfFullDensitySeq}
The heart of the Lemma~\ref{lem:fullDensitySeq} is Proposition~\ref{prop:cauchySeq}, which establishes the convergence of $\psi_{n_i(x)}(x)$ on a large measure subset of $\mathbf{G}_1/\Gamma_1$ for a large density sequence $\{n_i(x)\}_{i\in\N}$. The proof of Proposition~\ref{prop:cauchySeq} is provided in \S\ref{sec:pfCau}. Once we have Proposition~\ref{prop:cauchySeq}, we will be able to improve Proposition~\ref{prop:cauchySeq} to a full measure subset of $\mathbf{G}_1/\Gamma_1$ with a full density subsequence, and thus complete the proof of Lemma~\ref{lem:fullDensitySeq}.

The structure of this subsection is following: In \S\ref{sec:fullPrep}, we define the notion of a well behaved triplet (Definition~\ref{def:wellBehaved}), state Proposition~\ref{prop:cauchySeq}, and prove two preparatory lemmas, namely Lemma~\ref{lem:sequenceGluing} and Lemma~\ref{lem:invarianceUni}. In \S\ref{sec:fullProof}, we complete the proof of Lemma~\ref{lem:fullDensitySeq} assuming Proposition~\ref{prop:cauchySeq}.

\subsubsection{Preparation for the proof of Lemma~\ref{lem:fullDensitySeq}}\label{sec:fullPrep}

\begin{Definition}\label{def:wellBehaved}
Let $W\subset \mathbf{G}_1/\Gamma_1$ be a measurable subset, $\{n_i\}_{i\in\N}$ a sequence of measurable functions from $W$ to $\N$, and  $\epsilon\in[0,1)$. We say the triplet
\[
(W,\{n_i(x)\}_{i\in\N},\epsilon)
\]
is \textbf{well behaved} if the following hold:
\begin{enumerate}[label=(\alph*)]
    \item\label{wellBehaved:Measurability} For every $x \in W$ the sequence $n_i(x)$ is monotone increasing.
    \item\label{wellBehaved:uniformDensity} For every $x\in W$,
    \[
    \liminf_{k\to+\infty}\frac{1}{k}d_{k}(\{n_i(x)\}_{i\in\N})\geq 1-\epsilon.
    \]
    
    \item\label{wellBehaved:uniformLimit} The sequence $\{\psi_{n_i(x)}(x)\}_{i\in\N}$ converges, uniformly in $x\in W$. Moreover, for any $x \in W$ it holds that $\lim_{i\to\infty}\psi_{n_i(x)}(x)$ is on the $\mathbf{U}_2$-orbit of $\psi(x)$.
\end{enumerate}
\end{Definition}

With the notion of well behaved triplets, we are able to state Proposition~\ref{prop:cauchySeq}:

\begin{Proposition}\label{prop:cauchySeq}

For every $\epsilon\in(0,1)$, we can find a compact set $\Knew\label{075KsCau}=\Kold{075KsCau}(\epsilon)\subset \mathbf{G}_1/\Gamma_1$\index{$\Kold{075KsCau}$, Proposition~\ref{prop:cauchySeq}} with $m_1(\Kold{075KsCau})\geq1-\epsilon$ and a collection of measurable functions $n_i:\Kold{075KsCau}\to\N$ such that the triplet
\[
(\Kold{075KsCau},\{n_i(x)\}_{i\in\N},\epsilon)
\]
is well behaved.
\end{Proposition}

The proof of Proposition~\ref{prop:cauchySeq} is postponed to \S\ref{sec:pfCau}.

\begin{Lemma}\label{lem:sequenceGluing}
Suppose that the triplets $(W^{(k)},\{n_i^{(k)}(x)\}_{i\in\N},\delta_k)$ 
are well behaved for every $k\in\N$ with $\delta_k\to0$. Given $\epsilon>0$, there exists a compact set $\Knew\label{081KIntersection}\subset\cap_{k=1}^{\infty}W^{(k)}$\index{$\Kold{081KIntersection}$, Lemma~\ref{lem:sequenceGluing}} with $m_1(\Kold{081KIntersection})>m_1(\cap_{k=1}^{\infty}W^{(k)})-\epsilon$ and a collection of measurable functions $n_i:\Kold{081KIntersection}\to\N$ so that the triplet
\[
(\Kold{081KIntersection},\{n_i(x)\}_{i\in\N},0)
\]
is well behaved.
\end{Lemma}

\begin{proof}[Proof of Lemma~\ref{lem:sequenceGluing}]
Without loss of generality, $\delta_k<1/4$ for all $k$.
Since for every $k\in\N$, \ $x\in W^{(k)}$ we have that $\frac{1}{\ell}d_{\ell}(\{n_i^{(k)}(x)\}_{i\in\N}) > 1-\delta_k$ for $\ell$ large enough there exists a compact set $K_k\subset W^{(k)}$ with 
\[
m_1(K_k)\geq m_1(W^{(k)})-\frac{\epsilon}{2^{k+2}}
\]
and $\ell_k$ such that $\frac{1}{\ell}d_{\ell}(\{n_i^{(k)}(x)\}_{i\in\N}) > 1-2\delta_k$ for all $x \in X_k$ and $\ell>\ell_k$.

Let
\[
\Kold{081KIntersection}=\cap_{k=1}^{\infty}K_k.
\]
It is clear that $\Kold{081KIntersection}\subset\cap_{k=1}^{\infty}W^{(k)}$ is a compact set whose measure is greater equal than $m_1(\cap_{k=1}^{\infty}W^{(k)})-\epsilon$.

\medskip

Since the triplet
\[
(W^{(k)},\{n_i^{(k)}(x)\}_{i\in\N},\delta_k)
\]
is well behaved for every $k\in\N$, there exists $\phi_k(x)\in \mathbf{G}_2/\Gamma_2$ such that for every $k\in\N$ and every $x\in\cap_{k=1}^{\infty}W^{(k)}$
\begin{equation*}
    \phi_k(x)=\lim_{i\to\infty}\psi_{n_i^{(k)}(x)}(x).
\end{equation*}
Since for every $x\in\cap_{k=1}^{\infty}W^{(k)}$ and every $k$ the sequence $n_i^{(k)}(x)$ has density $>1/2$, there is a (measurable) function $\phi:\cap_{k=1}^{\infty}W^{(k)} \to \mathbf{G}_2/\Gamma_2$ so that $\phi_k(x)=\phi(x)$ for all $k$.

This implies that there exists a sequence of positive integers $N_k$ from $k=1$ such that for every $x\in\Kold{081KIntersection}$
\begin{enumerate}[label=(\roman*)]
    \item\label{item:Nk1} if $n_i^{(j)}(x)\geq N_k$ for $j=1,\ldots,k$, then there exists 
    $\absolute{t_{i,j}}\leq 2^{-k}$
    so that 
    \[\psi_{n_i^{(j)}(x)}(x)=u_{t_{i,j}}^{(2)}.\phi(x).
    \]
    
    \item\label{item:Nk2} $N_{k+1}\geq  10^kN_{k}$.
    \item\label{item:Nk3} $d_{\ell}(\{n_i^{(k)}(x)\}_{i\in\N})\geq (1-2\delta_{k})\ell$ for every $\ell\geq N_{k}$.
\end{enumerate}

\medskip

Let $n_i(x)$ be the monotone increasing sequence obtained by concatenating $\{n_i^{(1)}(x)\}_{i\in\N}\cap [1,N_{2})$ and the finite sequences $\{n_i^{(k)}(x)\}_{i\in\N}\cap [N_k,N_{k+1})$ for $k=2,3,\ldots$. It follows from \ref{item:Nk2} and \ref{item:Nk3} that for every $x\in\Kold{081KIntersection}$,
\begin{equation*}
    \lim_{k\to\infty}\frac{1}{k}d_{k}(\{n_i(x)\}_{i\in\N})=1,
\end{equation*}
and from \ref{item:Nk1} that $\psi_{n_i(x)}(x) \to \phi(x)$ uniformly on $\Kold{081KIntersection}$.
\end{proof}

\medskip

If the triplet $(W,\{n_i(x)\}_{i\in\N},0)$ is well behaved, we are able to find a new triplet $(W',\{\tilde{n}_j(x)\}_{j\in\N},\epsilon)$ such that the limiting map $\phi=\lim\psi_{\tilde{n}_j(x)}$ intertwines the action of $u_t^{(1)}$ with $u_t^{(2)}$.

\begin{Lemma}\label{lem:invarianceUni}
Let $\epsilon\in(0,1)$ and $(\Knew\label{0821KextendFull},\{n_i(x)\}_{i\in\N},0)$\index{$\Kold{0821KextendFull}$, Lemma~\ref{lem:invarianceUni}} be a well behaved triplet. Then there exist a measurable subset $\Knew\label{0822KextendFullM}\subset\Kold{0821KextendFull}$\index{$\Kold{0822KextendFullM}$. Lemma~\ref{lem:invarianceUni}} with $m_1(\Kold{0822KextendFullM})\geq m_1(\Kold{0821KextendFull})-\epsilon$ and a collection of measurable integer valued functions $\{\tilde n_j(x)\}_{j\in\N}$ on $\Kold{0822KextendFullM}$ so that the following holds:
\begin{enumerate}[label=(\alph*)]
    \item\label{item:invUwellBehave} The triplet $(\Kold{0822KextendFullM},\{\tilde{n}_{j}(x)\}_{j\in\N},\epsilon)$ is well behaved.
    \item\label{item:invUextendSeq} For every $t\in\R$, we have that $\lim_{j\to\infty}\psi_{\tilde{n}_{j}(x)}(u_t^{(1)}.x)$ converges uniformly in $x\in\Kold{0822KextendFullM}$ to  $u_t^{(2)}.\lim_{j\to\infty}\psi_{\tilde{n}_{j}(x)}(x)$.
\end{enumerate}
\end{Lemma}

\begin{proof}[Proof of Lemma~\ref{lem:invarianceUni}]

Let $K_1\subset \mathbf{G}_1/\Gamma_1$ be a compact subset with $m_1( K_1)>1-\epsilon/4$ so that $\lim_{s\to+\infty}\frac{1}{s}\tau(x,s)=1$ uniformly in  $x\in K_1$ where $\tau(x,s)$ is the time change function defined by 
\[
\psi(u_{t}^{(1)}.x)=u_{\tau(x,t)}^{(2)}.\psi(x);
\]
cf.~\S\ref{sec:timeChagneGoodSet}. Applying Egorov's Theorem and the Pointwise Ergodic Theorem for $\chi_{ K_1}$ and $a_1^{(1)}$, there exists $K_2\subset \Kold{0821KextendFull}\cap K_1$ with $m_1(K_2)>m_1(\Kold{0821KextendFull})-\epsilon/2$ such that 
\[
\frac{1}{n}\sum_{k=1}^{n}\chi_{ K_1}(a_k^{(1)}.x)\to m_1( K_1) \qquad\text{as $n \to\infty$,}
\]
uniformly for $x\in K_2$.

\medskip

The definition of well-behaved triplets implies that $\{\frac{1}{\ell}d_{\ell}(\{n_i(x)\}_{i\in\N})\}_{\ell\in\N}$ is a sequence of measurable functions from $\Kold{0821KextendFull}$ to $\N$ that converges to $1$ as $\ell\to\infty$. By Egorov's Theorem, there exists a compact set $\Kold{0822KextendFullM}\subset K_2$ with 
\[
m_1(\Kold{0822KextendFullM})\geq m_1(\Kold{0821KextendFull})-\epsilon
\]
such that 
\[
\text{$\frac{1}{\ell}d_{\ell}(\{n_i(x)\}_{i\in\N})$ converges uniformly for every $x\in\Kold{0822KextendFullM}$ as $\ell\to\infty$.}
\]

\medskip

Recall that $\tau_n(x,t)=e^{-2n}\tau(a_{n}^{(1)}.x,e^{2n}t)
$. Equation \eqref{eq:psinTimeChange} implies for every~$t\in\R$ and every $x\in \Kold{0822KextendFullM}$
\begin{equation}\label{eq:psinut}
\begin{aligned}
    \psi_{n_i(x)}(u_t^{(1)}.x)=u^{(2)}_{\tau_{n_i(x)}(x,t)}.\psi_{n_i(x)}(x).
\end{aligned}
\end{equation}

Let $J(x)$ be the set
\begin{equation}\label{eq:defJx}
    J(x)=\{i\in\N:a_{i}^{(1)}.x\in  K_1\}.
\end{equation}
The choice of $\Kold{0822KextendFullM}$ implies that for $N_1$ large,
\begin{equation}\label{eq:adergodic}
    \frac{1}{n}\operatorname{Card}(J(x)\cap[1,n])>1 - 0.75\epsilon \qquad\text{for all $x\in\Kold{0822KextendFullM}$.}
\end{equation}
Since $\{n_i(x)\}_{i\in\N}$ has full density, uniformly on $x \in \Kold{0822KextendFullM}$,
there exists $N_{2}>0$ such that if $n\geq N_{2}$
\begin{equation}\label{eq:uniNewSeq}
    \frac{1}{n}\operatorname{Card}(J(x)\cap\{n_i(x)\}_{i\in\N}\cap[1,n])\geq 1 - \epsilon.
\end{equation}

\medskip

For any $x\in \Kold{0822KextendFullM}$, we define
\begin{equation*}
  \{\tilde{n}_{j}(x)\}_{j\in\N}=J(x)\cap\{n_i(x)\}_{i\in\N}.  
\end{equation*}
It is clear that  $\tilde{n}_j:\Kold{0822KextendFullM}\to\N$ is measurable for every $j\in\N$.

\medskip

Notice that \eqref{eq:uniNewSeq} gives \ref{wellBehaved:uniformDensity} of Definition~\ref{def:wellBehaved}. Parts \ref{wellBehaved:Measurability} and \ref{wellBehaved:uniformLimit} of Definition~\ref{def:wellBehaved} are also immediate, hence the triplet
\[
(\Kold{0822KextendFullM},\{\tilde{n}_{j}(x)\}_{j\in\N},\epsilon)
\]
is well behaved. This proves \ref{item:invUwellBehave} of Lemma~\ref{lem:invarianceUni}.

\medskip

If $t=0$, then item \ref{item:invUextendSeq} of Lemma~\ref{lem:invarianceUni} is obvious and thus we suppose $t\neq0$. Fix a $t\neq0$. Since $a_{\tilde{n}_{j}(x)}^{(1)}.x\in  K_1$ (cf.~\eqref{eq:defJx}) for every $x\in \Kold{0822KextendFullM}$ and $\lim_{s\to+\infty}\frac{\tau(x,s)}{s}=1$ uniformly in $x\in K_1$, 
\begin{equation}\label{eq:convergenceSpeedUni}
 \absolute{ \tau_{\tilde{n}_j(x)}(x,t)-t}\to0\text{ uniformly in $x\in \Kold{0822KextendFullM}$ as }j\to+\infty.
\end{equation}
This together with \eqref{eq:psinut} proves \ref{item:invUextendSeq} of Lemma~\ref{lem:invarianceUni}.
\end{proof}

\subsubsection{Proof of Lemma~\ref{lem:fullDensitySeq} assuming Proposition~\ref{prop:cauchySeq}}\label{sec:fullProof}
The proof of Lemma~\ref{lem:fullDensitySeq} is a combination of Proposition~\ref{prop:cauchySeq}, Lemma~\ref{lem:sequenceGluing} and Lemma~\ref{lem:invarianceUni}.

\medskip

By Proposition~\ref{prop:cauchySeq}, for every $j,k\in\N$, there exist a family of compact sets $K_1^{(j,k)}\subset \mathbf{G}_1/\Gamma_1$ with $m_1(K_1^{(j,k)})\geq 1-10^{-100-k-j}$ and a collection of measurable functions $n_i^{(j,k)}:K_1^{(j,k)}\to\N$ so that the triplets
\begin{equation}\label{eq:doubleSeqWellbehaved}
(K_1^{(j,k)},\{n_i^{(j,k)}(x)\}_{i\in\N},10^{-100-k-j})
\end{equation}
are well behaved (cf.~Definition \ref{def:wellBehaved}).  

Applying Lemma~\ref{lem:sequenceGluing} for the triplets in \eqref{eq:doubleSeqWellbehaved} with fixed $j\in\N$, then there exist a compact set $K_2^{(j)}\subset\cap_{k=1}^{\infty}K_1^{(j,k)}$ with 
\[
m_1(K_2^{(j)})\geq m_1(\cap_{k=1}^{\infty}K_1^{(j,k)})-10^{-50-j}
\]
and a collection of measurable functions $n_i^{(j)}:K_2^{(j)}\to\N$ such that the triplets
\begin{equation*}
    (K_2^{(j)},\{n_i^{(j)}(x)\}_{i\in\N},0)
\end{equation*}
are well behaved (cf.~Definition \ref{def:wellBehaved}).

\medskip

 By Lemma~\ref{lem:invarianceUni}, for every $j\in\N$, there exist  $K_3^{(j)}\subset K_2^{(j)}$ with $m_1(K_3^{(j)})\geq m_1(K_2^{(j)})-10^{-100-j}$ and a collection of measurable integer values functions $\{\tilde{n}_{i}^{(j)}(x)\}$ on $K_3^{(j)}$ such that
 \begin{enumerate}[label=(\Alph*)]
     \item the triplet $(K_3^{(j)},\{\tilde{n}_{i}^{(j)}(x)\}_{i\in\N},10^{-100-j})$ is well behaved,
     
     \item\label{eq:transferUconver} given $t\in\R$, then 
 $\psi_{\tilde{n}_i^{(j)}(x)}(u_t^{(1)}.x)$ converges to  $u_t^{(2)}.\lim_{i\to\infty}\psi_{\tilde{n}_i^{(j)}(x)}(x)$ uniformly in $x\in K_3^{(j)}$ as $i\to\infty$.
 \end{enumerate}

\medskip

Combining Lemma~\ref{lem:sequenceGluing} and \ref{eq:transferUconver}, there exists a compact set $K_4\subset \cap_{j=1}^{\infty}K_3^{(j)}$ with $m_1(K_4)\geq 0.99$ guarantees that there exists a collection of measurable functions $n_i:K_4\to\N$ so that 
\begin{enumerate}[label=(\Roman*)]
    \item\label{item:tempGood1} for every $x\in K_4$,
    \[
    \lim_{k\to+\infty}\frac{1}{k}d_{k}(\{n_i(x)\}_{i\in\N})=1;
    \]
    \item\label{item:tempGood4} $\phi(x)=\lim_{i\to+\infty}\psi_{n_i(x)}(x)$ is well defined on $K_4$ and $\phi(x)$ stays on $\mathbf{U}_2$-orbit of $\psi(x)$;
    \item\label{item:tempGood3} for every $x\in K_4$ and every $t\in\R$, 
    \[\lim_{i\to\infty}\psi_{n_{i}(x)}(u_t^{(1)}.x)=u_t^{(2)}.\phi(x).\] 
\end{enumerate}

\medskip

Notice that the ergodicity of $(u_1^{(1)},m_1,\mathbf{G}_1/\Gamma_1)$ implies that 
\[
\Zold{073KsCauF}=\bigcup_{k=-\infty}^{\infty}u_k^{(1)}(K_4)
\]
is a full measure Borel subset of $\mathbf{G}_1/\Gamma_1$. For every $x\in \Zold{073KsCauF}$, let $t_x$ be defined as
\begin{equation*}
    t_x=\min\{k\in\Z_{\geq0}:x\in u_k^{(1)}(K_4)\}.
\end{equation*}
Here $t_x$ is finite for every $x\in \Zold{073KsCauF}$. We then extend $n_i(x)$ as following, 
\begin{gather*}
    n_{i}(x)=n_{i}(u_{-t_x}^{(1)}.x)\qquad\text{for every  $x\in\Zold{073KsCauF}$.}
\end{gather*}
By \ref{item:tempGood4} and \ref{item:tempGood3}, we obtain that for every $x\in\Zold{073KsCauF}$
\[
\phi(x):=\lim_{i\to+\infty}\psi_{n_i(x)}(x) \text{ exists},\quad \text{$\phi(x)$ stays on $\mathbf{U}_2$-orbit of $\psi(x)$.}
\]
Our extension also guarantees that \ref{item:tempGood1} holds for every $x\in\Zold{073KsCauF}$. These together give \ref{item:ptlimit} and \ref{item:ptupperlimit} of Lemma~\ref{lem:fullDensitySeq}.

\medskip

By \ref{item:tempGood1} and \ref{item:tempGood3} above imply that for every $x\in K_4$ and every $t\in\Z$
\[
\phi(u_t^{(1)}.x)=u_t^{(2)}.\phi(x).
\]
Recall that for every $x\in \Zold{073KsCauF}$, we have $x_0=u_{-t_x}^{(1)}.x\in K_4$. Thus by the way we extended $\phi$ to $\Zold{073KsCauF}$, for every $x\in \Zold{073KsCauF}$ and  $t\in\Z$
\[
\phi(u_t^{(1)}.x)=\phi(u_{t+t_x}^{(1)}.x_0)=u_{t+t_x}^{(2)}.\phi(x_0)=u_t^{(2)}.\phi(x).
\]
This proves \ref{item:interU} of Lemma~\ref{lem:fullDensitySeq}.\qed

\subsection{Proof of Proposition~\ref{prop:cauchySeq}}\label{sec:pfCau}
Recall the definition of well behaved triples given in Definition~\ref{def:wellBehaved}.
\begin{prop:cauchySeq}
For every $\epsilon\in(0,1)$, we can find a compact set $\Kold{075KsCau}\subset \mathbf{G}_1/\Gamma_1$ with $m_1(\Kold{075KsCau})\geq1-\epsilon$ and a collection of measurable functions $n_i:\Kold{075KsCau} \to \mathbb N$ such that the triplet
\[
(\Kold{075KsCau},\{n_i(x)\}_{i\in\N},\epsilon)
\]
is well behaved.
\end{prop:cauchySeq}

Loosely, the proof of Proposition~\ref{prop:cauchySeq} involves comparing for each $n$ the difference between $\psi_n(x)$ and $\psi_{n+1}(x)$. In contrast to the setting in \cite{ratner1986rigidity}, we do not have any apriori regularity for the time change function and thus do not have the control of $\mathbf{U}_2$-direction in Proposition~\ref{prop:CompatibleGeoMain}. 

In fact, Proposition~\ref{prop:CompatibleGeoMain} shows that for a.e.\ $x$ there is a $t_x\in\R$ so that 
    \[
    \psi(a_1^{(1)}.x)=a_1^{(2)}.u_{t_x}^{(2)}.\psi(x).
    \]
    It follows that for any $\epsilon>0$ we can find a ``good'' set $W$ of measure $>1-\epsilon/10$ and $M>0$ so that if $x\in W$, we have that $\absolute{t_x}\leq M$.
    Iterating one gets that for any $n$
    \[
    \psi(a_n^{(1)}.x)=a_n^{(2)}.u_{t_{n,x}}^{(2)}.\psi(x)
    \]
    with $t_{n,x}=\sum_{j=0}^{n-1} e^{-2j} \,t_{a^{(1)}_j\!.x}$.

    The ergodic theorem for $a_1^{(1)}$ only guarantees that $a_{n}^{(1)}.x\in W$ for \emph{most} of $n\in\N$. 
    Whenever $a_n^{(1)}.x\notin W$, we do not have a direct bound on $t_{a^{(1)}_n.x}$, hence we lose any control on the size of $t_{n',x}$ for every $n' > n$. 
    
    The key idea in the proof of Proposition~\ref{prop:cauchySeq} is to overcome this serious problem with the help of an appropriate $L^2$-ergodic theorem (Corollary~\ref{cor:sl2ergodic}), that allows us to find some $h\in B_{\epsilon}^{\mathbf{H}_1}$ so that $a_n^{(1)}h.x\in W$. 
    Our proof strategy can be summarized as follows:

    
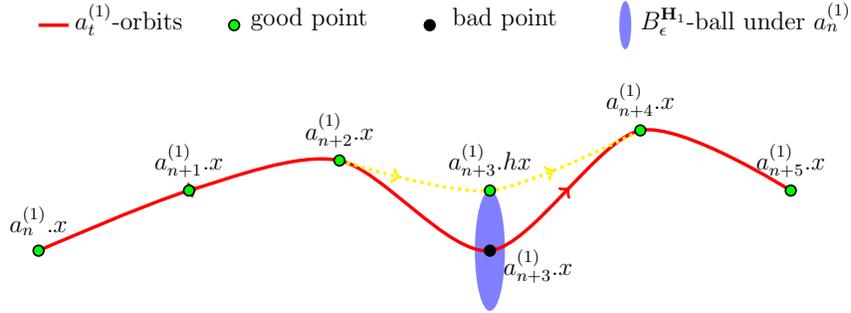
\begin{figure}[H]
    \scalebox{0.8}
  {
  \begin{tikzpicture}[scale=5]
  \tikzstyle{vertex}=[circle, minimum size=20pt,inner sep=0pt]
 	 \tikzstyle{selected vertex} = [vertex, fill=red!24]
 	 \begin{scope}[thick,decoration={
    markings,
    mark=at position 0.2 with {\arrow{>}},
    mark=at position 0.7 with {\arrow{>}}}
] 

      \draw[ultra thick,red,-] plot[smooth] coordinates{(-0.5,0.65)(-0.4,0.65)};
      \node[vertex] (v10) at (-0.2,0.67) {$a_t^{(1)}$-orbits};

      \draw[fill=green]  (0.15,0.65) circle (0.5pt);
      \node[vertex] (v20) at (0.4,0.67) {good point};
      
      \draw[fill=black]  (0.8,0.65) circle (0.5pt);
      \node[vertex] (v30) at (1.05,0.67) {bad point};

      \fill[blue!50] (1.45,0.65) ellipse (0.02 and 0.08);
      \node[vertex] (v30) at (1.85,0.67) {$B_{\epsilon}^{\mathbf{H}_1}$-ball under $a_n^{(1)}$};

      \fill[blue!50] (1,-0.1) ellipse (0.05 and 0.2);
      \draw[ultra thick,postaction={decorate},red,-] plot[smooth] coordinates{(-0.5,-0.1)(0,0.1)(0.5,0.2)(1,-0.1)(1.5,0.3)(2,0.1)};

      \draw[ultra thick,postaction={decorate},yellow,dotted] plot[smooth] coordinates{(0.5,0.2)(1,0.1)(1.5,0.3)};

      \node[vertex] (v01) at (-0.5,0) {$a_{n}^{(1)}.x$};
      \draw[fill=green] (-0.5,-0.1) circle (0.5pt);

      \node[vertex] (v02) at (0,0.2) {$a_{n+1}^{(1)}.x$};
      \draw[fill=green] (0,0.1) circle (0.5pt);

      \node[vertex] (v03) at (0.5,0.3) {$a_{n+2}^{(1)}.x$};
      \draw[fill=green] (0.5,0.2) circle (0.5pt);

      \node[vertex] (v04) at (1.16,-0.15) {$a_{n+3}^{(1)}.x$};
      \draw[fill=black] (1,-0.1) circle (0.5pt);

      \node[vertex] (v05) at (1,0.2) {$a_{n+3}^{(1)}.hx$};
      \draw[fill=green] (1,0.1) circle (0.5pt);

      \node[vertex] (v06) at (1.5,0.4) {$a_{n+4}^{(1)}.x$};
      \draw[fill=green] (1.5,0.3) circle (0.5pt);

      \node[vertex] (v07) at (2,0.2) {$a_{n+5}^{(1)}.x$};
      \draw[fill=green] (2,0.1) circle (0.5pt);

   \end{scope}    
  \end{tikzpicture}
  }
\end{figure}
\begin{enumerate}[label=(\roman*)]
    \item\label{prop-outline-step1} Let $W'\subset \mathbf{G}_1/\Gamma_1$ be a compact set with large measure such that for any $x\in W'$, the sequence of return times $n_i(x)\in\N$ for which $a^{(1)}_{n_i(x)}.x\in W$
    has large density.  

    \medskip
    
    We then use Lemma~\ref{lem:main} and Corollary~\ref{cor:sl2ergodic} to estimate the error terms between the pair $\psi_{n_i(x)}(x)$ and $\psi_{n_{i+1}(x)}(x)$. More precisely, if $x\in W'$ and $i$ is sufficiently large,
    \[
    \psi_{n_{i+1}(x)}(x)=h_{(i)}c_{(i)}tr_{(i)}.\psi_{n_i(x)}(x),
    \]
    where
    \[
    \quad h_{(i)}\in B_{10(n_{i+1}(x)-n_i(x))\delta}^{\mathbf{H}_2},\quad c_{(i)}\in B^{\mathbf Z_2}_{2(n_{i+1}(x)-n_i(x))\delta},\quad tr_{(i)}\in \Btr_{e^{-0.9n_i(x)}}.
    \]
    Here $\mathbf{H}_2$ is the connected Lie subgroup of $\mathbf{G}_2$ whose Lie algebra equals to $\operatorname{span}_{\R}\{\bu_2,\ba_2,\bou_2\}$, $\mathbf{Z}_2=C_{\mathbf{G}_2}(\mathbf{H}_2)$ and $\Btr_{\epsilon}$ is the $\epsilon$-ball in the transversal direction of $\mathbf{H}_2\oplus\mathbf{Z}_2$.

\item\label{prop-outline-step3}

Take $\Kold{075KsCau}\subset W'$ to be a compact set with large measure so that for every $x\in \Kold{075KsCau}$, most of its iterations under the transformation $a_{1}^{(1)}$ will stay in $W'$. 

Fix $x\in \Kold{075KsCau}$ and $n_i(x)$ sufficiently large. By Proposition~\ref{prop:CompatibleGeoMain}, there exists $t_i=t_i(x)\in\R$ so that
\begin{equation}\label{eq:UoriginConnect}
\psi_{n_{i+1}(x)}(x)=u_{t_i}^{(2)}.\psi_{n_i(x)}(x).
\end{equation}
We want to show that $\absolute{t_i}\leq e^{-n_i(x)}$, hence $\sum_{i\geq 1} t_i$ converges absolutely, which is enough to establish what is claimed in Proposition~\ref{prop:cauchySeq}. Suppose in contradiction that $\absolute{t_i}> e^{-n_i(x)}$; we will show this is not possible. 

\medskip

The key idea to obtain a contradiction is to change the ``observation base point''. Indeed, the choice of $x$ and $n_i(x)$ guarantee that there exists $k\in\Z$ such that 
    \begin{equation}\label{eq:intermediateStep}
    a_{-k}^{(1)}.x\in W',\quad k\geq -2n_i(x)/3,\quad e^{\eta n_i(x)}\leq\absolute{t_ie^{-2k}}\leq e^{10\eta(n_i(x)+k)},
    \end{equation}
where $\eta$ is a small constant independent of $x$ and $i$. According to the definition of the return time function $n_i({\bullet})$, there exists $j\in\N$ such that
\[
n_j(a_{-k}^{(1)}.x)=n_i(x)+k,\qquad n_{j+1}(a_{-k}^{(1)}.x)=n_{i+1}(x)+k.
\]
Then we can rewrite \eqref{eq:UoriginConnect} as
\begin{equation}\label{eq:-kMommentDiff}
\psi_{n_{i+1}(x)+k}(a_{-k}^{(1)}.x)=u_{t_ie^{-2k}}^{(2)}.\psi_{n_i(x)+k}(a_{-k}^{(1)}.x).
\end{equation}

\medskip

However, in step \ref{prop-outline-step1} we got a different relation between the points
\[
\psi_{n_{i+1}(x)+k}(a_{-k}^{(1)}.x),\qquad \psi_{n_i(x)+k}(a_{-k}^{(1)}.x).
\]
Since $a_{-k}^{(1)}.x$ is in the good set $W'$ we can use~Corollary~\ref{cor:globalRepEst}
to combine the information from these two different sources of information about the relative position of the two points above to
obtain
\[
\absolute{t_ie^{-2k}}\leq e^{10^{-6}\eta^2 n_i(x)}.
\]
However, this contradicts \eqref{eq:intermediateStep}. Hence the assumption in contradiction was false, i.e. 
\[
\absolute{t_i}\leq e^{-n_i(x)}.
\]

\end{enumerate}

\noindent
 The structure of this subsection is following: 
\begin{enumerate}[label=(\Roman*)]
    \item In \S\ref{sec:plem}, we provide preliminaries for Proposition~\ref{prop:cauchySeq}. 
    \item In \S\ref{sec:goodsets}, we specify some constants and sets that will be used in the proof of Proposition~\ref{prop:cauchySeq}.
    \item In \S\ref{sec:inclusion1}, we carry out step~\ref{prop-outline-step1} in the above outline. The main result will be Lemma~\ref{lem:mostControlWeak}, which is obtained by using $L^2$-ergodic theorem for $\SL_2(\R)$ (Corollary~\ref{cor:earlySl2Ergodic}) and Main lemma (Lemma~\ref{lem:main}) to get a ``rough'' estimates on the difference in $\mathbf{U}_2$ direction between $\psi_{n_i}(x)$ and $\psi_{n_{i+1}}(x)$
    \item In \S\ref{sec:expDecayUConnect}, we realize step~\ref{prop-outline-step3} in the above outline. The main result will be Lemma~\ref{lem:convergenceAlongU0Pre}, which is obtained by renormalization compatible proposition (Proposition~\ref{prop:CompatibleGeoMain}) and Lemma~\ref{lem:mostControlWeak}. 
    \item In \S\ref{sec:wrapUpProCauchy}, we finish the proof of Proposition~\ref{prop:cauchySeq}.
\end{enumerate}

\subsubsection{Preliminaries for Proposition~\ref{prop:cauchySeq}}\label{sec:plem}
Here we provide some preliminaries that are needed for the proof of Proposition~\ref{prop:cauchySeq}.

The following lemma describes how $\psi$ changes under small perturbations in $\mathbf{H}_1$. 
\begin{Lemma}\label{lem:seqeunceXin}
Given $\epsilon\in(0,10^{-3})$, there exist $\denew\label{081desx},\Rnew\label{081Rsx}>0$\index{$\deold{081desx}$, Lemma~\ref{lem:seqeunceXin}}\index{$\Rold{081Rsx}$, Lemma~\ref{lem:seqeunceXin}} and a compact set $\Knew\label{088Ksx}\subset \mathbf{G}_1/\Gamma_1$\index{$\Kold{088Ksx}$, Lemma~\ref{lem:seqeunceXin}} with $m_1(\Kold{088Ksx})\geq 1-\epsilon$ such that the following holds. Let $\delta\in(0,\deold{081desx})$, then there exist $\enew\label{081eSequenceXin}=\eold{081eSequenceXin}(\delta)\in(0,\delta)$ \index{$\eold{081eSequenceXin}(\cdot)$, Lemma~\ref{lem:seqeunceXin}}so that if $h_1,h_2\in B_{\eold{081eSequenceXin}}^{\mathbf{H}_1}$,  $x\in \mathbf{G}_1/\Gamma_1$ and $t\geq\Rold{081Rsx}$ satisfying 
\[
a_{t}^{(1)}h_1.x,\quad a_{t}^{(1)}h_2.x\in \Kold{088Ksx},
\]
then for some $\absolute{s}<2\eold{081eSequenceXin} e^{2t}$,
\begin{equation*}
\psi(a_{t}^{(1)}.h_2.x)\in\Kak\left(e^{2t},\delta,u_{s}^{(2)}.\psi(a_{t}^{(1)}h_1.x)\right).
\end{equation*}

\end{Lemma}
\begin{proof}
Fix an $\eta\in(0,10^{-3})$. Let $\Kold{035Klm1}$, $\Rold{035Rlm1}$,  $\deold{035delm1}$ and $\eold{035eMainDel}(\cdot)$ be defined as in Lemma~\ref{lem:main} for this $\eta$. Set
\[
\deold{081desx}=\deold{035delm1}/100, \qquad \Rold{081Rsx}=(\log\Rold{035Rlm1})/2,\qquad \Kold{088Ksx}=\Kold{035Klm1}.
\]
Let $\delta\in(0,\deold{081desx})$ be fixed, then denote $\eold{081eSequenceXin}(\delta)=\eold{035eMainDel}(\delta)/100$. 

\medskip

Since $h_1,h_2\in B_{\epsilon}^{\mathbf{H}_1}$, there exist $\absolute{t_1},\absolute{t_2},\absolute{t_3}<2\epsilon$ such that
\begin{equation*}
    \begin{aligned}
    h_1h_2^{-1}=u_{t_1}^{(1)}a_{t_2}^{(1)}\oua_{t_3},
    \end{aligned}
\end{equation*}
hence
\begin{equation*}
    \begin{aligned}
    a_{t}^{(1)}h_1h_2^{-1}a_{-t}^{(1)}=u_{\bar{t}_1}^{(1)}a_{t_2}^{(1)}\oua_{\bar{t}_3},
    \end{aligned}
\end{equation*}
where $\bar{t}_1=e^{2t}t_1$ and $\bar{t}_{3}=e^{-2t}t_3$. This in particular gives 
\begin{equation*}
\begin{aligned}
    a_{t}^{(1)}h_1.x=u_{\bar{t}_1}^{(1)}a_{t_2}^{(1)}\oua_{\bar{t}_3}a_{t}^{(1)}h_2.x,
\end{aligned}
\end{equation*}
hence by the definition of Kakutani-Bowen balls (cf.~Definition~\ref{def:Kakball}),
\begin{equation}\label{eq:uslidesest}
    a_{t}^{(1)}h_2.x\in\Kak\left(e^{2t},\epsilon,u_{-\bar{t}_1}^{(1)}a_{t}^{(1)}h_1.x\right).
\end{equation}
Since $a_{t}^{(1)}h_1.x,a_{t}^{(1)}h_2.x\in \Kold{088Ksx}$ and $\absolute{\bar{t}_1}\leq2\eold{081eSequenceXin}(\delta) e^{2t}$, Lemma~\ref{lem:seqeunceXin} follows by applying Lemma~\ref{lem:main} to \eqref{eq:uslidesest}.
\end{proof}

In \S\ref{sec:sl2Ergodiccorollary} we establish the following:

\begin{cor:sl2ergodic}\label{cor:earlySl2Ergodic}
Let $\mathbf{G}_2,\Gamma_2,m_2, \mathbf{H}_2,a^{(2)}_t$ be as in \S\ref{subsec:assupmtions}.
Given $\epsilon\in(0,1)$, then for any $f\in L^2(\mathbf{G}_2/\Gamma_2,m_2)$ and $m_2$-a.e. $x\in \mathbf{G}_2/\Gamma_2$, we have
\[
\lim_{n\to+\infty}\frac{1}{m_{\mathbf{H}_2}(B_{\epsilon}^{\mathbf{H}_2,\norm{\cdot}})}\int_{B_{\epsilon}^{\mathbf{H}_2,\norm{\cdot}}}f(a^{(2)}_{n}h.x)dm_{\mathbf{H}_2}(h)=\int_{\mathbf{G}_2/\Gamma_2}fdm_2.
\]
\end{cor:sl2ergodic}
This corollary is crucial for our construction of the good sets in the next section, cf. item~\ref{item:L2uniformErg} on p.~\pageref{item:L2uniformErg} for more details.

\subsubsection{The choices of constants and sets}\label{sec:goodsets}
In this subsection, we specify some notations that will be used in the proof of Proposition~\ref{prop:cauchySeq}. We start with some algebraic constructions, and choices of good sets and constants.
\begin{enumerate}[label=(\Roman*)]
\item\label{item:rhoInvariantVector1} Recall that 
\begin{gather*}
    \qquad\quad \mathbb{L}_2(\R)=\{g\in\mathbb{G}_2(\R):\Ad(g)\mathfrak{h}_2=\mathfrak{h}_2\},\qquad
    \mathbf{L}_2=\mathbb{L}_2(\R)\cap\mathbf{G}_2.
\end{gather*}
Let $\mathbf{N}_2$ be the normal core of $\mathbf{L}_2$, that is $\mathbf{N}_2=\bigcap_{g\in\mathbf{G}_2}g\mathbf{L}_2g^{-1}$. By  Chevalley's Theorem (cf.\ Theorem~\ref{thm:Chevalley}), there exist an algebraic representation $\rho:\mathbb{G}_2\to\operatorname{SL}(V)$ over $\R$ and an $\R$-vector $v_{\mathbf{L}_2}\in V$ so that
\[
\mathbf{L}_2=\{g\in\mathbf{G}_2:\rho(g).v_{\mathbf{L}_2}=v_{\mathbf{L}_2}\}.
\]
\item\label{item:091phiConstant} Recall that $\tilde{\phi}$ is the morphism of algebraic groups from $\SL_2$ to the algebraic group $\mathbb{G}_2$ that sends $\SL_2(\R)$ to $\mathbf{H}_2$ defined in \S\ref{sec:sl2basis}. Let $\Cnew\label{091phiConstant}>1$\index{$\Cold{091phiConstant}$} be a constant depending only on $\tilde{\phi}$ such that for every $t>0$ and $T$ sufficiently large, 
\[
\norm{u_t^{(2)}}<e^{T}\text{ implies that } \absolute{t}\leq e^{\Cold{091phiConstant}T}.
\]
    \item\label{eq:eta0} Let $\deold{027dexpSmallnew}$, $\Cold{027CcoeSmall}$ be as in Corollary~\ref{cor:globalRepEst} with $\delta=1$ and $\deold{0301Clattice}$ the lower bound on the growth exponent for the lattice $\Gamma_2$ as in~\eqref{eq:latticeCount}. Fix $\neta$ to be small (depending on $\deold{027dexpSmallnew}$ and  $\deold{0301Clattice}$). Specifically we may take
\begin{equation*}
    \neta=\deold{027dexpSmallnew}^2\deold{0301Clattice}/C,
\end{equation*}
with $C$ a sufficiently big constant depending only  on $\mathbf{G}_2$ and $\rho$, for instance we may take 
\[
C=(9\dim(\mathfrak{g}_2)\deg(\rho))^{99},
\]
where  $\deg(\rho)=\max_{1\leq i,j\leq \dim(V)}(\deg(\rho_{ij}))$.
    \item\label{item:epsilonFixed} 
    Let $\enew\label{089esmallChoice}$\index{$\eold{089esmallChoice}$, Proposition~\ref{prop:cauchySeq}} be defined as
    \[
    \eold{089esmallChoice}=c\neta^2.
    \]
    with $c$ a sufficiently small constant depending only on $\mathbf{G}_2$ (for instance we may take $c=\deold{007mnC1}\,\eold{008eComExp}\,(9\Cold{091phiConstant}\dim(\mathfrak{g}_2))^{-99}$ with $\deold{007mnC1}$ as in Lemma~\ref{lem:matrixNorm} and $\eold{008eComExp}$ as in Lemma~\ref{lem:comExp}). 
    
    Without loss of generality, we can assume $\epsilon$ is as small as we wish as the statement of Proposition~\ref{prop:cauchySeq} is less restrictive the smaller $\epsilon$ is. From now on, we fix an $\epsilon$ such that
    \[
    \epsilon\in(0,\eold{089esmallChoice}).
    \]
    \item\label{item:deltaFixed10} Let $\deold{081desx}$, $\Rold{081Rsx}$, $\eold{081eSequenceXin}(\cdot)$ and $\Kold{088Ksx}$ be defined as in Lemma~\ref{lem:seqeunceXin} applied with $\epsilon'=10^{-100}\epsilon$. Let $\kappa_1\in(0,1)$ be a constant sufficiently small such that for every $r\in(0,\kappa_1)$
    \[
    m_{\mathbf{H}_1}\left(B_{r}^{\mathbf{H}_1}\cap B_{r}^{\mathbf{H}_1,\norm{\cdot}}\right)\geq (1-10^{-9})m_{\mathbf{H}_1}\left(B_{r}^{\mathbf{H}_1}\right).
    \]
    We then define $\denew\label{085deltaFinal}$\index{$\deold{085deltaFinal}$, \ref{item:deltaFixed10}} as
    \begin{equation*}
        \deold{085deltaFinal}=10^{-100}\deold{081desx}\epsilon\kappa_1.
    \end{equation*}

\item\label{item:KtildeCpt} Let $\Knew\label{089KpreRepCpt}\subset\mathbf{G}_2/\Gamma_2$\index{$\Kold{089KpreRepCpt}$} with $m_2(\Kold{089KpreRepCpt})>1-10^{-100}\epsilon$ and $\widetilde{\Kold{089KpreRepCpt}}\subset\mathbf{G}_2$ be compact sets, so that the projection of $\widetilde{\Kold{089KpreRepCpt}}$ on $\mathbf{G}_2/\Gamma_2$ contains~$\Kold{089KpreRepCpt}$.

\item\label{item:repreEstFinal} For any $\sigma>0$, let $V_{\sigma}$ be the set of $\bar{z}\in\widetilde{\Kold{089KpreRepCpt}}$ such that
\[
\norm{\rho(\gamma\bar{z}^{-1}).v_{\mathbf{L}_2}-\rho(\bar{z}^{-1}).v_{\mathbf{L}_2}}_V<\sigma^{1/\deold{027dexpSmallnew}}
\]
for some $\gamma\in\Gamma_2$ satisfying 
\begin{gather}
    \min_{g\in \mathbf{N}_2}\norm{\gamma-g}\geq\Cold{027CcoeSmall}\sigma \quad \text{ and }\quad
    \norm{\gamma}\leq  \sigma^{-\deold{027dexpSmallnew}\deold{0301Clattice}/2}.\nonumber
\end{gather}
(Cf.~\ref{eq:eta0} for definition of $\Cold{027CcoeSmall}, \deold{027dexpSmallnew},\deold{0301Clattice}$.)

 \item\label{eq:CoeN} Recall that $\neta$ is defined in item \ref{eq:eta0}. We define $\Cnew\label{092CoeN}$\index{$\Cold{092CoeN}$, \ref{eq:CoeN}} as
\begin{equation*}
    \Cold{092CoeN}=10^3\neta(\dim\mathfrak{g}_2)^2/(\deold{027dexpSmallnew}\deold{0301Clattice}).
\end{equation*}
For every $n\in\N$, let $E_n$ be defined as
\[
E_n=\pi_2(V_{e^{-n\Cold{092CoeN}}}).
\]
Notice Corollary~\ref{cor:globalRepEst} (with $\delta=1$) guarantees that for $n$ large
\[
m_2(E_n)\leq e^{-n\deold{0302CglobalCom1}\Cold{092CoeN}},
\]
where $\deold{0302CglobalCom1}$ is as in Corollary~\ref{cor:globalRepEst}. Then there is a $\Rnew\label{083RMeasEstExcepReturn}\in\N$\index{$\Rold{083RMeasEstExcepReturn}$} so that
\[
\qquad\quad \Rold{083RMeasEstExcepReturn}\geq\max(\Rold{081Rsx},-\log(\deold{085deltaFinal}))\quad \text{and} \quad  \sum_{n=\Rold{083RMeasEstExcepReturn}}^{+\infty}m_2\left(E_n\right)<10^{-9}\epsilon.
\]
As a result, we can find a compact $\Knew\label{0900KMeasEstExcepReturn}\subset \mathbf{G}_2/\Gamma_2$\index{$\Kold{0900KMeasEstExcepReturn}$} satisfying 
\begin{equation*}
\Kold{0900KMeasEstExcepReturn}\subset\Kold{089KpreRepCpt}\setminus\left(\bigcup_{n=\Rold{083RMeasEstExcepReturn}}^{+\infty}E_n\right)\quad\text{and}\quad m_2(\Kold{0900KMeasEstExcepReturn})\geq1-10^{-8}\epsilon.
\end{equation*}

\end{enumerate}

\medskip

\noindent We proceed to construct a few additional good sets that will be used later.
\begin{enumerate}[label=(\roman*)]
    \item\label{item:boundedConjugation} By Proposition~\ref{prop:CompatibleGeoMain}, there exist $\Rnew\label{0841Rboundset}>\Rold{083RMeasEstExcepReturn}$\index{$\Rold{0841Rboundset}$, \ref{item:boundedConjugation}} and $\Knew\label{0901Kboundset}\subset \mathbf{G}_1/\Gamma_1$\index{$\Kold{0901Kboundset}$, \ref{item:boundedConjugation}} with $m_1(\Kold{0901Kboundset})>1-10^{-100}\epsilon$ such that for every $x\in \Kold{0901Kboundset}$
    \begin{equation*}
        \begin{aligned}
        \psi(a_{1}^{(1)}.x)=a_{1}^{(2)}u_{t_x}^{(2)}.\psi(x),\qquad\absolute{t_x}<\Rold{0841Rboundset};
        \end{aligned}
    \end{equation*}
    (in the outline at the beginning of \S\ref{sec:pfCau}, we used $M$ to instead of $\Rold{0841Rboundset}$).
    \item\label{item:L2uniformErg} Recall that $\eold{081eSequenceXin}(\cdot)$ and $\Kold{088Ksx}$ are defined as in Lemma~\ref{lem:seqeunceXin} with $\epsilon'=10^{-100}\epsilon$, \ and $\deold{085deltaFinal}$ is fixed in item~\ref{item:deltaFixed10}. Define $\enew\label{093eLusin}$\index{$\eold{093eLusin}$, \ref{item:L2uniformErg}} and $\Knew\label{0902Kerg1}$\index{$\Kold{0902Kerg1}$, \ref{item:L2uniformErg}} as
    \[
    \eold{093eLusin}=\eold{081eSequenceXin}(10^{-100}\deold{085deltaFinal}), \qquad \Kold{0902Kerg1}=\Kold{088Ksx}\cap \Kold{0901Kboundset}.
    \]
    Let
   \begin{gather*}
      s_n(x)=\frac{1}{m_{\mathbf{H}_1}\left(B_{\eold{093eLusin}}^{\mathbf{H}_1,\norm{\cdot}}\right)}\int_{B_{\eold{093eLusin}}^{\mathbf{H}_1,\norm{\cdot}}}\chi_{\Kold{0902Kerg1}}\left(a^{(1)}_{n}h.x\right)dm_{\mathbf{H}_1}(h).
    \end{gather*}
    By Corollary~\ref{cor:sl2ergodic} (cf.~p.~\pageref{cor:earlySl2Ergodic}), there exists a compact set $\Knew\label{0903Kserg1}\subset \Kold{0902Kerg1}$\index{$\Kold{0903Kserg1}$, \ref{item:L2uniformErg}} with $m_1(\Kold{0903Kserg1})\geq 1-10^{-10}\epsilon$ such that for every $x\in\Kold{0903Kserg1}$, 
    \[
    \text{$s_n(x)$ converges uniformly to $m_1(\Kold{0902Kerg1})$ as $n\to\infty$.}
    \]
    \item\label{item:stayCompactErgodic} 
    
   Applying the Pointwise Ergodic Theorem for the map~$a_{1}^{(2)}$ and function $\chi_{\Kold{0900KMeasEstExcepReturn}}$, there exist a compact set $\Knew\label{0910KcomUniformErg}\subset\mathbf{G}_2/\Gamma_2$\index{$\Kold{0910KcomUniformErg}$, \ref{item:stayCompactErgodic}} with $m_2(\Kold{0910KcomUniformErg})\geq1-10^{-9}\epsilon$ and an integer $\Rnew\label{0842RN_1}\geq \Rold{0841Rboundset}$\index{$\Rold{0842RN_1}$, \ref{item:stayCompactErgodic}} such that for every $x\in \Kold{0910KcomUniformErg}$ and $n\geq \Rold{0842RN_1}$
    \[
    \absolute{\left\{k\in[-n,n]:a_k^{(2)}.x\in\Kold{0900KMeasEstExcepReturn}\right\}}\geq(1-10^{-7}\epsilon)(2n+1).
    \]
    \item\label{item:uniformConvergeSeqErg} 
  
    Applying the Pointwise Ergodic Theorem for the map~$a_1^{(1)}$ and function $\chi_{\Kold{0903Kserg1}\cap \psi^{-1}(\Kold{0910KcomUniformErg})}$, there exists a compact subset $\Knew\label{0911KdenErg1}\subset \Kold{0903Kserg1}\cap\psi^{-1}(\Kold{0910KcomUniformErg})$\index{$\Kold{0911KdenErg1}$, \ref{item:uniformConvergeSeqErg}} with $m_1(\Kold{0911KdenErg1})>1-10^{-8}\epsilon$ such that for every~$x\in\Kold{0911KdenErg1}$, the expression
    \[
    \frac{1}{2n+1}\absolute{\left\{k\in[-n,n]:a_k^{(1)}.x\in \Kold{0903Kserg1}\cap \psi^{-1}(\Kold{0910KcomUniformErg})\right\}}
    \]
    converges uniformly to $m_1(\Kold{0903Kserg1}\cap \psi^{-1}(\Kold{0910KcomUniformErg}))$ as $n\to\infty$.

    \item\label{item:stayGoodErgodic} 
    
   Applying the Pointwise Ergodic Theorem for the map $a_{1}^{(1)}$ and function $\chi_{\Kold{0911KdenErg1}}$, there exist a compact set  $\Knew\label{093KTarPropDecay}\subset\Kold{0911KdenErg1}$\index{$\Kold{093KTarPropDecay}$} with $m_1(\Kold{093KTarPropDecay})>1-\epsilon/2$ and an integer $\Rnew\label{0843RN_211}\geq\Rold{0842RN_1}$\index{$\Rold{0843RN_211}$, \ref{item:stayGoodErgodic}} such that for every $x\in \Kold{093KTarPropDecay}$ and $n\geq \Rold{0843RN_211}$
    \[
    \absolute{\left\{k\in[-n,n]:a_{k}^{(1)}.x\in \Kold{0911KdenErg1}\right\}}\geq (1-10^{-5}\epsilon)(2n+1).
    \]
   
\end{enumerate}

\medskip

Given a $x\in\Kold{0911KdenErg1}$, let $\{n_i(x)\}_{i\in\N}$ be the sequence of positive integers for which 
\begin{equation}\label{eq:goodSeq1}
    a_{n_i(x)}^{(1)}.x \in\Kold{0903Kserg1}\cap \psi^{-1}(\Kold{0910KcomUniformErg}), \qquad \forall i\in\N.
    \end{equation}
Note that by \ref{item:uniformConvergeSeqErg} we know the density of this sequence exists and is greater than $1-10^{-8}\epsilon$.

\begin{Remark}\label{rmk:simplifyNotation}
    In order to simplify the notation, we will write $\{n_i\}_{i\in\N}$ instead of $\{n_i(x)\}_{i\in\N}$ unless the dependence on $x$ needs to be specified. For notational simplicity, in this subsection we shall also use the following notations: 
    \[
    \delta=\deold{085deltaFinal}.
    \]
Here $\deold{085deltaFinal}$ is fixed as in \ref{item:deltaFixed10} on p.~\pageref{item:deltaFixed10}.
\end{Remark}

\subsubsection{Relation between successive renormalizations}\label{sec:inclusion1}

In this section, we show that if $x$ stays in a ``good'' set and $n_i=n_i(x)$ is sufficiently large, then we have a mild control for the deviation between $\psi_{n_i}(x)$ and $\psi_{n_{i+1}}(x)$ by our $L^2$-ergodic theorem (Corollary~\ref{cor:sl2ergodic}).  

Recall that $\mathbf{H}_2$ is the connected Lie subgroup of $\mathbf{G}_2$ whose Lie algebra equals to $\operatorname{span}_{\R}\{\bu_2,\ba_2,\bou_2\}$, $\mathbf{Z}_2=C_{\mathbf{G}_2}(\mathbf{H}_2)$ and 
\[
\Btr_{\epsilon}=\{g\in B_{\epsilon}^{\mathbf{G}_2}:g^{\mathfrak{h}}=g^{\mathfrak{z}}=e\},
\]
where $g^{\mathfrak{h}}$ and $g^{\mathfrak{z}}$ are defined in \eqref{eq:gDecom}.

\begin{Lemma}\label{lem:mostControlWeak}
There exist $\Rnew\label{RinductionLemma1}>0$\index{$\Rold{RinductionLemma1}$, Lemma~\ref{lem:mostControlWeak}} such that for every $x\in\Kold{0911KdenErg1}$ and $i$ such that $n_i\geq \Rold{RinductionLemma1}$, we have 
\begin{equation*}
    \begin{aligned}
    \psi_{n_{i+1}}(x)=h_{(i)}c_{(i)}tr_{(i)}.\psi_{n_i}(x),
    \end{aligned}
\end{equation*}
where
    \begin{alignat*}{3}
    &h_{(i)}\in B_{10(n_{i+1}-n_i)\delta}^{\mathbf{H}_2},\qquad 
   &c_{(i)}\in B^{\mathbf Z_2}_{2(n_{i+1}-n_i)\delta},\qquad 
   &&tr_{(i)}\in \Btr_{e^{-0.9n_i}}.
    \end{alignat*}
\end{Lemma}

\begin{proof}
We first specify the choice of $\Rold{RinductionLemma1}$ as in the following:
\begin{enumerate}[label=(\alph*)]
   
    \item As $\Kold{0911KdenErg1}\subset \Kold{0903Kserg1}$, item \ref{item:L2uniformErg} on p.~\pageref{item:L2uniformErg} implies that there exists  $N_{1}>\Rold{0843RN_211}$ such that for every $x\in\Kold{0911KdenErg1}$ and every $n\geq N_{1}$
     \[
     s_{n}(x)\geq 0.99.
     \]

    \item\label{eq:boundedGrowth1} By item \ref{item:uniformConvergeSeqErg} on p.~\pageref{item:uniformConvergeSeqErg}, there exists $N_{2}\in\N$ such that for every~$x\in\Kold{0911KdenErg1}$ and every~$n_i(x)\geq N_{2}$
    \begin{equation*}
     n_{i+1}(x)\leq 2n_i(x).
    \end{equation*}
    \item\label{item:constantChoice95} We define $\Rold{RinductionLemma1}$ as
    \begin{equation*}
        \begin{aligned}
\Rold{RinductionLemma1}=10^9\max(N_{1},N_{2}).   
        \end{aligned}
    \end{equation*}
    
\end{enumerate}

Let $S_n(x)$ be defined as following:
\[
S_n(x)=\{h\in B_{\eold{093eLusin}}^{\mathbf{H}_1}:a^{(1)}_{n}h.x\in\Kold{0902Kerg1}\}\subset\bigcup_{\absolute{s}\leq e^{2n}\eold{093eLusin}}\Kak\left(e^{2 n},\eold{093eLusin}, u_s^{(1)}.a^{(1)}_{n}.x\right).
\]
Here $\eold{093eLusin}$ is defined as in item \ref{item:L2uniformErg} on p.~\pageref{item:L2uniformErg}. Recall that if $n\geq \Rold{RinductionLemma1}/2$, then $s_n(x)\geq0.99$. Thus the definition of $s_n(x)$ and \ref{item:deltaFixed10} on p.~\pageref{item:deltaFixed10} give
\begin{equation*}
    m_{\mathbf{H}_1}\left(S_n(x)\right)\geq 0.95 m_{\mathbf{H}_1}\left(B_{\eold{093eLusin}}^{\mathbf{H}_1}\right).
\end{equation*}
This gives that
\[
S_{n+1}(x)\cap S_n(x)\neq\emptyset;
\]
let 
\[
h_n\in S_{n+1}(x)\cap S_n(x).
\]

\medskip

For any $n\geq \Rold{RinductionLemma1}$, we then define
\begin{equation*}
    \begin{aligned}
    y_n=a^{(1)}_{n}h_{n-1}.x,\qquad
    z_n=a^{(1)}_{n} h_n.x,
    \end{aligned}
\end{equation*} 
so $y_{n+1}=a_{1}^{(1)}z_{n}$. Since both $h_{n-1},h_n\in S_{n}(x)$, Lemma~\ref{lem:seqeunceXin} gives that
\begin{equation}\label{eq:CauchyInd11}
    \psi(z_{n})\in\Kak\left(e^{2 n},\delta,u_{t_{n,1}}^{(2)}.\psi(y_{n})\right),
\end{equation}
for some $\absolute{t_{n,1}}<2e^{2 n}\eold{093eLusin}$.

On the other hand, as $z_{n}\in\Kold{0902Kerg1}\subset \Kold{0901Kboundset}$ (cf.~\ref{item:L2uniformErg} on p.~\pageref{item:L2uniformErg}), there exists $\absolute{t_{n,2}}<\Rold{0841Rboundset}$ so that
\begin{equation}\label{eq:CauchyInd21}
    \psi(y_{n+1})=a_{1}^{(2)}u_{t_{n,2}}^{(2)}.\psi(z_{n}).
\end{equation}

Combining \eqref{eq:CauchyInd11} and \eqref{eq:CauchyInd21}, we obtain that
\[
\psi(y_{n+1})\in a_{1}^{(2)}u_{t_{n,2}}^{(2)}\Kak\left(e^{2n},\delta,u_{t_{n,1}}^{(2)}.\psi(y_{n})\right),
\]
hence applying Lemma~\ref{lem:diangonalConjugateKak} on $\Kak\left(e^{2n},\delta,u_{t_{n,1}}^{(2)}.\psi(y_{n})\right)$, we have
\begin{equation}\label{eq:iterationKak1}
a^{(2)}_{-(n+1)}.\psi(y_{n+1})=\sigma_n a^{(2)}_{- n}.\psi(y_{n}).
\end{equation}
with $\sigma_n \in B^{\mathfrak l \oplus \tr}_{5\delta,\:\delta e^{-0.99 n}}$. Here $B^{\mathfrak l \oplus \tr}_{r_1,r_2}$ is defined as in \eqref{eq:bch}, where $r_1$ specifies the size of this ball in the directions of $\mathbf {L}_2=\mathbf{H}_2 \cdot \mathbf{Z}_2$ while $r_2$ bounds the size in all the transverse directions.

\medskip

It follows that
\begin{equation*}
a^{(2)}_{-n_{i+1} }\psi(y_{n_{i+1}})=\left(\prod_{k=n_i}^{n_{i+1}-1}\sigma_k\right)a^{(2)}_{-n_{i} }.\psi(y_{n_{i}})
\end{equation*}
where $\sigma_k$ is defined as in \eqref{eq:iterationKak1}  for $k=n_i,\ldots,n_{i+1}-1$.
By equation~\eqref{eq:goodSeq1} for both $j=i,i+1$
\[
a_{n_j}^{(1)}.x\in \Kold{0903Kserg1}\cap \psi^{-1}(\Kold{0910KcomUniformErg})\subset \Kold{088Ksx}.
\]
As $y_{n_i} \in \Kak\left(e^{2n_{i}},\eold{093eLusin}, a_{ n_i}^{(1)}.x\right)$ and $y_{n_{i+1}} \in \Kak \left(e^{2n_{i+1}},\eold{093eLusin},a_{ n_{i+1}}^{(1)}.x\right)$,
applying Lemma~\ref{lem:seqeunceXin} twice we may now conclude that
\begin{equation}\label{eq:inductionij11}
    \begin{aligned}
    \psi_{n_{i+1}}(x) &= a^{(2)}_{-n_{i+1}}.\psi(a^{(2)}_{n_{i+1}}.x)\\
    &= \tilde{\sigma}_{n_{i+1}} a^{(2)}_{-n_{i+1}}. \psi(y_{n_{i+1}})\\
    &=\tilde{\sigma}_{n_{i+1}} \left(\prod_{k=n_i}^{n_{i+1}-1}\sigma_k\right) a^{(2)}_{- n_{i} }.\psi(y_{n_{i}})\\
    &= \tilde{\sigma}_{n_{i+1}} \left(\prod_{k=n_i}^{n_{i+1}-1}\sigma_k\right)\tilde{\sigma}_{n_i}.\psi_{n_i}(x),
    \end{aligned}
\end{equation}
where $\tilde{\sigma}_{n_{i+1}},\tilde{\sigma}_{n_i}\in B^{\mathfrak l \oplus \tr}_{5\delta,\:\delta e^{-0.99n_i}}$. As $\delta<10^{-5}$ and $n_{i+1}\leq 2n_i$, by Lemma~\ref{lem:comExp} we complete the proof.
\end{proof}

\subsubsection{Exponential decay of difference between successive renormalizations}\label{sec:expDecayUConnect}

By Proposition~\ref{prop:CompatibleGeoMain}, for a.e.\ $x$ and every $i$ there is a $t_i=t_i(x) \in \R$ such that $\psi_{n_{i+1}}(x)=u_{t_i}^{(2)}.\psi_{n_i}(x)$. By combining Lemma~\ref{lem:mostControlWeak} with Corollary~\ref{cor:globalRepEst}, we will show in Lemma~\ref{lem:convergenceAlongU0Pre} below that for $x$ in a good set of arbitrarily large measure we bound $t_i$ from above, showing it decays exponentially. This lemma will be an important step to complete the proof of Proposition~\ref{prop:cauchySeq}.
\begin{Lemma}\label{lem:convergenceAlongU0Pre}
There exist constants $\Rnew\label{RnewComparePre}>0$\index{$\Rold{RnewComparePre}$, Lemma~\ref{lem:convergenceAlongU0Pre}} such that for every $n_i\geq \Rold{RnewComparePre}$ and $x\in\Kold{093KTarPropDecay}$ the following hold:
    \begin{equation}\label{eq:convergenceAlongU0Pre}
         \psi_{n_{i+1}}(x)=u_{t_i}^{(2)}.\psi_{n_{i}}(x)\quad\text{with}\quad \absolute{t_i}\leq e^{-n_i}.   
    \end{equation}
\end{Lemma}
In \eqref{eq:convergenceAlongU0Pre} and below in the proof we shall sometimes write $n_i$ and $t_i$ (without explicit mention of a point) for $n_i(x)$ or $t_i(x)$ respectively.

\begin{proof}
    Fix a $x\in\Kold{093KTarPropDecay}$, here $\Kold{093KTarPropDecay}$ is the good set obtained in \ref{item:stayGoodErgodic} on p.~\pageref{item:stayGoodErgodic}.
    \begin{itemize}
        \item Let $\Rold{RinductionLemma1}$ be defined as in Lemma~\ref{lem:mostControlWeak}.
        \item Recall that $\Kold{0911KdenErg1}$ is defined in \ref{item:uniformConvergeSeqErg} on p.~\pageref{item:uniformConvergeSeqErg}. By items~\ref{item:stayCompactErgodic} and \ref{item:uniformConvergeSeqErg} on p.~\pageref{item:stayCompactErgodic}  and the construction of $\{n_i(x)\}_{i\in\N}$ (cf.~\eqref{eq:goodSeq1}), there exists $N_{1}\in\N$ such that for every $x\in\Kold{0911KdenErg1}$ and every $n_i(x)\geq N_{1}$, there exists $\ell=\ell_i(x)\in[0,\eta n_i(x)]$ such that
    \begin{equation*}
    a_{\ell}^{(2)}.\psi_{n_i(x)}(x)\in \Kold{0900KMeasEstExcepReturn}.
    \end{equation*}
    Here $\Kold{0900KMeasEstExcepReturn}$ is defined in \ref{eq:CoeN} on p.~\pageref{eq:CoeN}.
    \item We then define
\begin{equation}\label{eq:RinterBound}
\Rold{RnewComparePre}=10^9\eta^{-1}\max(\Rold{RinductionLemma1},N_1).
\end{equation}
From now on we assume $n_i(x)\geq \Rold{RnewComparePre}$.
    \end{itemize}

    \medskip
    
    By Proposition~\ref{prop:CompatibleGeoMain}, there exists $t_i=t_{i}(x)\in\R$ such that
\begin{equation}\label{eq:Uconnect0}
    \psi_{n_{i+1}}(x)=u_{t_{i}}^{(2)}.\psi_{n_i}(x).
\end{equation}

We proceed proof by contradiction. Suppose that
\begin{equation}\label{eq:contradictionTarget}
    \absolute{t_i}> e^{-n_i(x)}.
\end{equation}

\medskip

Since $1\gg\eta\gg\epsilon$ (cf.~\ref{item:epsilonFixed} on p.~\pageref{item:deltaFixed10}), by definition of $\Kold{093KTarPropDecay}$ there exists $k=k(x)\in\Z$ so that
    \begin{equation}\label{eq:tiInitialSmall}
     a_{-k}^{(1)}.x\in \Kold{0911KdenErg1},\quad k\geq -2n_i(x)/3,\quad  e^{\eta n_i(x)}\leq\absolute{t_ie^{-2k}}\leq e^{10\eta (n_i(x)+k)},
    \end{equation}
    where $\neta>0$ is defined in item \ref{eq:eta0} on p.~\pageref{eq:eta0}. Notice that the definition of $n_i(x)$ in \eqref{eq:goodSeq1} implies that there exists $j\in\N$ such that
\begin{equation*}
    n_j(a_{-k}^{(1)}.x)=n_i(x)+k,\qquad n_{j+1}(a_{-k}^{(1)}.x)=n_{i+1}(x)+k.
\end{equation*}
In particular, the choices of $k$ and $n_i(x)$ guarantee that 
\[
n_j(a_{k}^{(1)}.x)=n_i(x)+k\geq n_i(x)/3 \geq\Rold{RnewComparePre}/3.
\]

\medskip

    By item~\ref{eq:boundedGrowth1} on p.~\pageref{eq:boundedGrowth1} we have $n_{i+1}(x)<2n_i(x)$. Applying Lemma~\ref{lem:mostControlWeak} for $a_{-k}^{(1)}.x$ and $n_{j}(a_{-k}^{(1)}.x)$, then 
    \begin{equation}\label{eq:inductiveControl1}
        \psi_{n_{i+1}+k}(a_{-k}^{(1)}.x)= h_{(j)}c_{(j)}tr_{(j)}.\psi_{n_{i}+k}(a_{-k}^{(1)}.x),
\end{equation}
where
\begin{equation}\label{eq:coeSmallShrinking1}
    \begin{alignedat}{3}
    h_{(j)}\in B^{\mathbf{H}_2}_{10n_i(x)\delta},\qquad c_{(j)}\in B_{2n_i(x)\delta}^{\mathbf{Z}_2},\qquad 
   tr_{(j)}\in \Btr_{e^{-0.9(n_i(x)+k)}}. 
    \end{alignedat}
\end{equation}
Equation \eqref{eq:Uconnect0} implies (directly from the definition of $\psi_n$) that
\[
\psi_{n_{i+1}+k}(a_{-k}^{(1)}.x)= u_{t_{i}e^{-2k}}^{(2)}.\psi_{n_{i}+k}(a_{-k}^{(1)}.x).
\]
Comparing this with \eqref{eq:inductiveControl1} we get that
\begin{equation}\label{eq:firstShrinkingEquation11}
    u_{t_{i}e^{-2k}}^{(2)}.\psi_{n_{i}+k}(a_{-k}^{(1)}.x)=h_{(j)}c_{(j)}tr_{(j)}.\psi_{n_{i}+k}(a_{-k}^{(1)}.x).
\end{equation}

\medskip
We want to conclude from \eqref{eq:firstShrinkingEquation11} and \eqref{eq:coeSmallShrinking1} a bound for the size of $t_{i}e^{-2k}$, namely  $O(\exp(10n_i(x)\delta))$. To do this, we first need two additional claims.
 
\begin{Claim}\label{claim:convergenceAlongU11}
Let $x$, $n_i(x)$, $j$, $k$, $t_i$, $h_{(j)}$, $c_{(j)}$ and $tr_{(j)}$ be as in \eqref{eq:firstShrinkingEquation11}, then there exist 
\[\bar{y}\in \widetilde{\Kold{089KpreRepCpt}},\qquad \ell=\ell_j(a_{-k}^{(1)}.x)\in[0,\eta (n_i(x)+k)], \qquad\gamma\in\Gamma_2\]
such that 
\begin{equation*}
    u_{t_ie^{-2k}}^{(2)}a_{-\ell}^{(2)}\bar{y}\gamma=h_{(j)}c_{(j)}tr_{(j)}a_{-\ell}^{(2)}\bar{y}
\end{equation*}
and moreover
\[
\min_{g\in \mathbf{N}_2}\norm{\gamma-g}< e^{-(n_i(x)+k)\Cold{092CoeN}/2}.
\]
\end{Claim}

\medskip
\noindent
Recall that $\widetilde{\Kold{089KpreRepCpt}}$ is a compact subset of $\mathbf{G}_2$  fixed in \ref{item:KtildeCpt} on  p.~\pageref{089KpreRepCpt}.

\medskip

\begin{proof}[Proof of Claim~\ref{claim:convergenceAlongU11}]

Since $x\in\Kold{093KTarPropDecay}$, we obtain that 
\[
\psi(a_{n_{j}(a_{-k}^{(1)}.x)-k}^{(1)}.x)=\psi(a_{n_i}^{(1)}.x)\in \Kold{0910KcomUniformErg}.
\]
This together with $n_i(x)+k\geq \Rold{RnewComparePre}/3$, choice of $\Rold{RnewComparePre}$ (cf.~\eqref{eq:RinterBound}) and definition of $\Kold{0910KcomUniformErg}$ (cf.~\ref{item:stayCompactErgodic} on p.~\pageref{item:stayCompactErgodic}) implies that
 there exists   
\[
\ell\in[0,\eta(n_i(x)+k)]
\]
such that
    \begin{equation*}
    a_{\ell}^{(2)}.\psi_{n_i+k}(a_{-k}^{(1)}.x)=a_{\ell-n_i-k}^{(2)}.\psi(a_{n_i}^{(1)}.x)\in \Kold{0900KMeasEstExcepReturn}.
    \end{equation*}
Since $\Kold{0900KMeasEstExcepReturn}\subset\Kold{089KpreRepCpt}\subset\pi_2(\widetilde{\Kold{089KpreRepCpt}})$, there exists $\bar{y}\in \widetilde{\Kold{089KpreRepCpt}}$ such that
\[\pi_{2}(\bar{y})=a_{\ell}^{(2)}.\psi_{n_i+k}(a_{-k}^{(1)}.x),\] 
and there exists $\gamma\in\Gamma_2$ such that 
\begin{equation}\label{eq:inductionNi}
u_{t_ie^{-2k}}^{(2)}a_{-\ell}^{(2)}\bar{y}\gamma=h_{(j)}c_{(j)}tr_{(j)}a_{-\ell}^{(2)}\bar{y}.
\end{equation}
The above equation gives that
\begin{equation}\label{eq:gammaEst1}
    \norm{\gamma}\leq e^{10^2(\dim\mathfrak{g}_2)^2\eta (n_i(x)+k)}.
\end{equation}
Since $v_{\mathbf{L}_2}$ is fixed by $\rho(\mathbf{L}_2)$
(cf.~\ref{item:rhoInvariantVector1} on p.~\pageref{item:rhoInvariantVector1})
\begin{equation}\label{eq:inducRep01}
    \begin{aligned}
        \rho(\gamma \bar{y}^{-1}a_{\ell}^{(2)}&tr_{(j)}^{-1}c_{(j)}^{-1}h_{(j)}^{-1}).v_{\mathbf{L}_2}\\
        =&\rho(\gamma \bar{y}^{-1}a_{\ell}^{(2)}tr_{(j)}^{-1}).v_{\mathbf{L}_2}=\rho(\gamma \bar{y}^{-1}a_{\ell}^{(2)}tr_{(j)}^{-1}a_{-\ell}^{(2)}).v_{\mathbf{L}_2}\\
        =&\rho(\gamma \bar{y}^{-1}).(v_{\mathbf{L}_2}+v_1),
    \end{aligned}
\end{equation}
where (in view of our choice of $\neta$, cf.~\ref{eq:eta0} on p.~\pageref{eq:eta0}) the vector $v_1$ satisfies $\norm{v_1}_{V}\leq e^{-0.8(n_i(x)+k)}$.
Equations~\eqref{eq:gammaEst1} and \eqref{eq:inducRep01} imply
\begin{equation}\label{eq:inducRep11}
    \begin{aligned}
        \rho(\gamma \bar{y}^{-1}a_{\ell}^{(2)}tr_{(j)}^{-1}c_{(j)}^{-1}h_{(j)}^{-1}).v_{\mathbf{L}_2}=\rho(\gamma \bar{y}^{-1}).v_{\mathbf{L}_2}+v_2,
    \end{aligned}
\end{equation}
where $v_2=\rho(\gamma \bar{y}^{-1}).v_1$ satisfies $\norm{v_2}_V\leq e^{-0.75(n_i(x)+k)}$.
Note also that
\begin{equation}\label{eq:inducRep21}
    \begin{aligned}
&\rho(\bar{y}^{-1}a_{\ell}^{(2)}u_{-t_ie^{-2k}}^{(2)}).v_{\mathbf{L}_2}=\rho(\bar{y}^{-1}).v_{\mathbf{L}_2}.
    \end{aligned}
\end{equation}

Combining \eqref{eq:inducRep11} and \eqref{eq:inducRep21}, we obtain that
\[
\norm{\rho(\gamma \bar{y}^{-1}).v_{\mathbf{L}_2}-\rho(\bar{y}^{-1}).v_{\mathbf{L}_2}}\leq e^{-0.7(n_i(x)+k)}.
\]
This together with \eqref{eq:gammaEst1}, the fact that $\pi_2(\bar{y})\in\Kold{0900KMeasEstExcepReturn}$, and items~\ref{item:repreEstFinal} and \ref{eq:CoeN} on p.~\pageref{item:repreEstFinal} (essentially, by Corollary~\ref{cor:globalRepEst}) guarantees that 
\[
\min_{g\in \mathbf{N}_2}\norm{\gamma-g}< \Cold{027CcoeSmall}e^{-(n_i(x)+k)\Cold{092CoeN}},
\]
where $\Cold{092CoeN}$ is defined in \ref{eq:CoeN} on p.~\pageref{eq:CoeN}. This completes the proof of Claim~\ref{claim:convergenceAlongU11}.
\end{proof}

\medskip

Recall that $\deold{007mnC1}$ is defined as in Lemma~\ref{lem:matrixNorm} with $\mathbf{G}=\mathbf{G}_2$. Then we have the following claim:
\begin{Claim}\label{claim:convergenceAlongU21}
For $t_i$ and $k$ defined as in Claim~\ref{claim:convergenceAlongU11}, we have
\begin{equation*}
\absolute{e^{-2k}t_i}\leq e^{30n_i(x)\delta/\deold{007mnC1}}.
\end{equation*}
\end{Claim}
\begin{proof}[Proof of Claim~\ref{claim:convergenceAlongU21}]
By Claim~\ref{claim:convergenceAlongU11}, there exist  
\[
\ell\in[0,\eta (n_i(x)+k)],\qquad \bar{y}\in \widetilde{\Kold{089KpreRepCpt}},\qquad \gamma\in\Gamma_2
\]
such that for $h_{(j)}$, $c_{(j)}$ and $tr_{(j)}$ as in \eqref{eq:firstShrinkingEquation11}, we have that
\begin{equation*}
    u_{t_ie^{-2k}}^{(2)}a_{-\ell}^{(2)}\bar{y}\gamma=h_{(j)}c_{(j)}tr_{(j)}a_{-\ell}^{(2)}\bar{y}
\end{equation*}
and moreover
\begin{equation}\label{eq:gammaNik0}
    \min_{g\in \mathbf{N}_2}\norm{\gamma-g}< e^{-(n_i(x)+k)\Cold{092CoeN}/2}.
\end{equation}

Recall that $\mathbf{N}_2$ is a normal subgroup of $\mathbf{G}_2$ (cf.~\ref{item:rhoInvariantVector1} on p.~\pageref{item:rhoInvariantVector1}). Since $\bar{y}\in\widetilde{\Kold{089KpreRepCpt}}$, equation \eqref{eq:gammaNik0} implies that there is a constant $\kappa_1>0$, independent of $i$ and $x$, so that
\[
\min_{g\in \mathbf{N}_2}\norm{a_{\ell}^{(2)} u_{-t_ie^{-2k}}^{(2)}h_{(j)}c_{(j)}tr_{(j)}a_{-\ell}^{(2)}-g}\leq\kappa_1e^{-(n_i(x)+k)\Cold{092CoeN}/2}.
\]
By choice of $\Cold{092CoeN}$ (cf.~\ref{eq:CoeN} on p.~\pageref{eq:CoeN}) and the bound we have on $\ell$ we may conclude that
\[
\min_{g\in \mathbf{N}_2}\norm{u_{-t_ie^{-2k}}^{(2)}h_{(j)}c_{(j)}tr_{(j)}-g}\leq\kappa_1e^{-(n_i(x)+k)\Cold{092CoeN}/4}.
\]
This together with Lemma~\ref{cor:metricRel} implies that there is a $g \in \mathbf{N}$ so that 
\[
d_{\mathbf{G}_2}(u_{-t_ie^{-2k}}^{(2)}h_{(j)}c_{(j)}tr_{(j)},g)\leq e^{-\kappa_2(n_i(x)+k)}
\]
with $\kappa_2\in(0,1/2)$ some constant (independent of $i$ and $x$). Since $d_{\mathbf{G}_2}$ is right invariant, it follows that
there is some $g_1\in B^{\mathbf{G}_2}_{e^{-\kappa_2(n_i(x)+k)}}$ so that 
\[
g_1u_{-t_ie^{-2k}}^{(2)}h_{(j)}c_{(j)}tr_{(j)}\in \mathbf{N}_2.
\]

Since $\mathbf{N}_2$ is a normal subgroup of~$\mathbf{G}_2$, we obtain that
\begin{equation*}
\begin{aligned}
    u_{-t_ie^{-2k}}^{(2)}h_{(j)}c_{(j)}tr_{(j)}g_1\in\mathbf{N}_2.
\end{aligned}
\end{equation*}
Note that $tr_{(j)}g_1\in B_{e^{-\kappa_2(n_i(x)+k)/2}}^{\mathbf{G}_2}$.

\medskip

By Lemma~\ref{lem:HNint} and again using the bounds on the respective sizes from \eqref{eq:coeSmallShrinking1}, there exist $\kappa_3>0$, which is independent of $i$ and $x$,  and  $h_0$ in the center of~$\mathbf{H}_2$, so that
\begin{equation*}
    u_{-t_ie^{-2k}}^{(2)}h_{(j)}\in B_{e^{-\kappa_3(n_i(x)+k)}}^{\mathbf{H}_2}(h_0).
\end{equation*}
This together with $h_{(j)}\in B^{\mathbf{H}_2}_{10n_i(x)\delta}$ gives 
\begin{equation}\label{eq:wReducedInclusion1}
    u_{-t_ie^{-2k}}^{(2)}\in B_{20n_i(x)\delta}^{\mathbf{H}_2}.
\end{equation}

Equation \eqref{eq:wReducedInclusion1} and Lemma~\ref{lem:matrixNorm} imply that
\[
\norm{u_{-t_ie^{-2k}}^{(2)}}\leq e^{30n_i(x)\delta/\deold{007mnC1}}.
\]
It follows from \ref{item:091phiConstant} on p.~\pageref{item:091phiConstant} that 
\[
\absolute{t_ie^{-2k}}\leq e^{30n_i(x)\Cold{091phiConstant}\delta/\deold{007mnC1}},
\]
which finishes the proof of Claim~\ref{claim:convergenceAlongU21}.
\end{proof}

\medskip

We now continue the proof of Lemma~\ref{lem:convergenceAlongU0Pre}. On the one hand, notice by Claim~\ref{claim:convergenceAlongU21}
\[
\absolute{t_ie^{-2k}}\leq e^{30n_i(x)\Cold{091phiConstant}\delta/\deold{007mnC1}}.
\]
Notice that $\eta^2\geq10^9\Cold{091phiConstant}\delta/\deold{007mnC1}$ (cf.~\ref{item:epsilonFixed} and \ref{item:deltaFixed10} on p.~\pageref{item:epsilonFixed} and Remark~\ref{rmk:simplifyNotation}), then $30n_i(x)\Cold{091phiConstant}\delta/\deold{007mnC1}<10^{-6}\eta^2n_i(x)$ and thus
\begin{equation}\label{eq:contradictionTar1}
    \absolute{t_ie^{-2k}}\leq e^{10^{-6}\eta^2n_i(x) }.
\end{equation}

On the other hand, recall that \eqref{eq:tiInitialSmall} gives that
\begin{equation*}
    \absolute{t_ie^{-2k}}\geq e^{\eta n_i(x)}.
\end{equation*}
However, this contradicts \eqref{eq:contradictionTar1} and hence \eqref{eq:contradictionTarget} does not hold, which in particular gives
\[
\absolute{t_i}\leq e^{-n_i(x)}.
\]
This finishes the proof of Lemma~\ref{lem:convergenceAlongU0Pre}.
\end{proof}

\subsubsection{Wrapping up the proof}\label{sec:wrapUpProCauchy}
In this subsection, we use Lemma~\ref{lem:convergenceAlongU0Pre} prove Proposition~\ref{prop:cauchySeq}. 

\begin{proof}[Proof of Proposition~\ref{prop:cauchySeq}]
  Let $\epsilon$ be fixed as in item~\ref{item:epsilonFixed} on p.~\pageref{item:epsilonFixed} and take $\Kold{075KsCau}=\Kold{093KTarPropDecay}$. By Lemma~\ref{lem:convergenceAlongU0Pre},  
 for every $x\in\Kold{075KsCau}$ and every $n_j>n_i\geq \Rold{RnewComparePre}$, we have
    \[
    \psi_{n_j}(x)=u_{t_{i,j}}^{(2)}.\psi_{n_i}(x),\qquad \absolute{t_{i,j}}\leq \sum_{k=i}^{j-1}e^{-n_k}\leq \sum_{k=i}^{j-1}e^{-k}.
    \]
    Then for any $\kappa>0$, there exists $N_{\kappa}\in\N$ so that for every $n_j>n_i\geq N_{\kappa}$, we can guarantee that $\absolute{t_{i,j}}\leq \kappa$ and $\lim_{i\to\infty}\psi_{n_i}(x)$ stays on the $\mathbf{U}_2$-orbit of $\psi(x)$, which gives  item \ref{wellBehaved:uniformLimit} of Definition~\ref{def:wellBehaved}.  As for items~\ref{wellBehaved:uniformDensity} and \ref{wellBehaved:Measurability} of Definition~\ref{def:wellBehaved}, they follow from \eqref{eq:goodSeq1}. These together finish the proof of Proposition~\ref{prop:cauchySeq}.
\end{proof}

\section{Proof of Theorem~\ref{thm:main}}\label{sec:proofofMain}
By Lemma~\ref{lem:chainLK}, $u_t^{(1)}$ is loosely Kronecker if and only if $\mathfrak{l}_1=\mathfrak{g}_1$. Suppose that $u_t^{(1)}$ is not loosely Kronecker, then $\mathfrak{l}_1\neq\mathfrak{g}_1$ and by Theorem~\ref{thm:reno},  for $m_1$-a.e. $x\in \mathbf{G}_1/\Gamma_1$, there exists a subsequence $\{n_i(x)\}_{i\in\N}$ of natural numbers of full density such that 
\[
\phi(x)=\lim_{i\to+\infty} \psi_{n_i(x)}(x)=\lim_{i\to+\infty}a_{-n_i(x)}^{(2)}\psi(a_{n_i(x)}^{(1)}.x)
\]
exists and the following hold
\begin{enumerate}[label=(\roman*)]
    \item\label{item:mpFinal} $\phi:\mathbf{G}_1/\Gamma_1\to \mathbf{G}_2/\Gamma_2$ is a measurable map,
    \item\label{item:stayUFinal} $\phi(x)\in \mathbf{U}_2.\psi(x)$,
    \item\label{item:mUInvFinal} $\phi(u_t^{(1)}.x)=u_t^{(2)}.\phi(x)$ for $m_1$-a.e. $x\in \mathbf{G}_1/\Gamma_1$ and every $t\in\mathbb{Z}$.
\end{enumerate}

Recall that $m_1$ is the volume measure on $\mathbf{G}_1/\Gamma_1$ induced from Haar measure on $\mathbf{G}_1$. Then items \ref{item:mpFinal} and \ref{item:mUInvFinal} imply that $\phi_*m_2$ is an $u_t^{(1)}$-invariant measure on $\mathbf{G}_2/\Gamma_2$. By Ratner's measure classification theorem \cite{ratner1991measure}*{Theorem 1}, $\phi_*m_2$ is the volume measure induced from Haar measure of a subgroup $\mathbf{G}'_1<\mathbf{G}_1$. By \cite{ratner1990measure}*{Corollary 6}, $\mathbf{G}_2$ is isomorphic to a subgroup of $\mathbf{G}_1$. Since $\psi$ is invertible, we can repeat the above for $\psi^{-1}$ and then obtain that $\mathbf{G}_1$ is isomorphic to a subgroup of $\mathbf{G}_2$. As a result, we obtain that
\[
\phi_*m_2=m_1.
\]

Since $\phi:(\mathbf{G}_1/\Gamma_1,m_1)\to(\mathbf{G}_2/\Gamma_2,m_2)$ is a measurable isomorphism between the transformations $u_1^{(1)}$ and $u_1^{(2)}$, it follows from \cite{ratner1990measure}*{Corollary 6} that there exist $c\in\mathbf{G}_2$ and a group isomorphism~$\varphi:\mathbf{G}_1\to\mathbf{G}_2$ such that
\begin{gather}
    \varphi(\Gamma_1)=c\Gamma_2c^{-1},\nonumber\\
    \phi(g\Gamma_1)=\varphi(g)c\Gamma_2, \qquad \text{for $m_1$-a.e. $g\Gamma_1\in\mathbf{G}_1/\Gamma_1$.}\label{eq:algebraicMap}
\end{gather}

By \ref{item:stayUFinal}, there exists a measurable function $\sigma:\mathbf{G}_1/\Gamma_1\to\R$ such that 
\[
\phi(x)=u_{\sigma(x)}^{(2)}.\psi(x)\qquad \text{for $m_1$-a.e. $x\in\mathbf{G}_1/\Gamma_1$.}
\]
This together with \eqref{eq:algebraicMap} completes the proof.\hfill\qed

\appendix
\section{Smooth matching function lemma}\label{sec:appSmoothM}
The proof of Lemma~\ref{lem:smoothMatching} depends on the following abstract lemma:
\begin{Lemma}\label{lem:goodCover}
Fix $\epsilon>0$ sufficiently small and $R>10$. Let $\{O_i\}_{i\in I}$ be  a collection of disjoint open subsets of $[0,R]$ such that $\sum_{i\in I}l_i<20\epsilon R$, where $l_i\geq l(O_i)$ for every $i\in I$. Then there exists another collection $\{Q_j\}_{j\in J}$ of disjoint open subsets and a function $f:I\to J$ such that
\begin{enumerate}
    \item $O_{i}\subset Q_{f(i)}$;
    \item $l(Q_j)=\sum_{f(i)=j}\frac{l_i}{\sqrt{\epsilon}}$.
\end{enumerate}
\end{Lemma}
\begin{proof}
Since $\{O_i\}_{i\in I}$ is a collection of disjoint open sets, we have either $I=\N$ or $I=\{1,\ldots,n\}$ for some $n\in\N$. We start our construction in both cases from $O_1$ by choosing any open interval $Q_1^{(1)}$ such that $O_1\subset Q_1^{(1)}$ and $l(Q_1^{(1)})=\frac{l_1}{\sqrt{\epsilon}}$, which is possible since $\frac{l_1}{\sqrt{\epsilon}}<\sum_{i\in I} \frac{l_i}{\sqrt{\epsilon}}<20\sqrt{\epsilon}R$. 

Assume the family of disjoint open sets $\{Q_j^{(i-1)}\}_{j\in J_{i-1}}$ have already been chosen for the sets $O_1,\ldots, O_{i-1}$ such that (1) and (2) hold. Let $\tilde{J}_{i-1}=\{j\in J_{i-1}:Q_j^{(i-1)}\cap O_i\neq\emptyset\}$, then the interval topology implies that $\operatorname{Card}(\tilde{J}_{i-1})\leq 2$. Denote $Q=O_i\cup\bigcup_{j\in\tilde{J}_{i-1}}Q^{(i-1)}_j$ and let $\bar{Q}$ be a new interval such that $\bar{Q}$ covers $Q$ and has length $\frac{l_i}{\sqrt{\epsilon}}+\sum_{j\in\tilde{J}_{i-1}}l(Q_j^{(i-1)})$. If $\bar{Q}$ overlaps with any other elements in $\{Q_j^{(i-1)}\}_{j\in J_{i-1}}$, then we combine them together and repeat the construction of $\bar{Q}$ for their union. Since there are only finite many intervals and $\sum_{i\in I}l_i<20\epsilon R$, this process will stop in finite time and does not exhaust the interval $[0,R]$. Let $\{Q_j^{(i)}\}_{j\in J_i}$ be the cover constructed by the this process, it is clear that  this cover satisfies (1) and (2) for $O_1,\ldots,O_i$.

If $I\neq\N$, the above argument has already completed the proof. If $I=\N$, recall $Q_j^{(i)}\subset Q_k^{(i+1)}$ for a unique value $k=k(j)$, then for $O_1$, we chose the sequence $Q_{j_1}^{(1)}\subset Q_{j_2}^{(2)}\subset\ldots$ such that $Q_{j_k}^{k}$ is the unique open set in $\{Q_j^{(k)}\}_{j\in J_k}$ containing $O_1$. Then we define $Q_1=\cup_{k=1}^{\infty}Q_{j_k}^{(k)}$ and $A_1=\{i\in\N:O_i\in Q_1\}$. Repeat this construction for $O_{i_2}$ to obtain a new interval $Q_2$ and $A_2$, where $i_2=\inf\{\N\setminus A_1\}$. Notice that the construction implies that $Q_2$ is disjoint with $Q_1$ and $A_2$ is also disjoint with $A_1$. Repeat this process to define set $Q_j$ and $A_j$. It is also clear that that $Q_j$ with function $f(i)=j$ if  $i\in A_j$ is our desired function and hence we complete the proof of the Lemma~\ref{lem:goodCover}.
\end{proof}

We now back to the proof of Lemma~\ref{lem:smoothMatching}. Since $x,y$ are $(\delta,\epsilon,R)$-two sides matchable, there exist $A,A'\subset[-R,R]$ with $l(A),l(A')\geq(1-\epsilon)2R$ and an increasing  absolutely continuous function $h:A\to A'$ such that
\begin{enumerate}
    \item $d(T_{h(t)}x,T_{t}y)<\delta$;
    \item $0\in A$ and $h(0)=0$;
    \item $\absolute{h'(t)-1}<\epsilon$ for $t\in A$.
\end{enumerate}

Let $h_+=h|_{[0,R]\cap A}$ and $\bar{A}_+$ be a closed subset of $A\cap[0,R]$, whose measure is $(1-4\epsilon)R$. Then the set $E=[0,R]\setminus\bar{A}_+$ is an open set with $l(E)\leq 4\epsilon R$. Moreover, since $E$ is open, there exists a family of disjoint countable union of open intervals such that $E=\cup_{i\in I}O_i$, where $I\subset\N$ and $O_i$ is an open interval of $[0,R]$. Since 
\[
l(A'\cap[0,R])\geq(1-2\epsilon)R,\quad  l(\bar{A}_+)\geq(1-4\epsilon)R, \quad h_+(\bar{A}_+)\subset A'\cap[0,R]
\]
and $\absolute{h'_+(t)-1}<\epsilon$ for $t\in A$, then
\[
\sum_{i\in I}(\sup_{t\in O_i}h_+(t)-\inf_{t\in O_i}h_+(t))\leq 2\epsilon R+(1+\epsilon)4\epsilon R\leq 10\epsilon R.
\]
Combining above estimates with $l(E)\leq 4\epsilon R$, we can apply Lemma~\ref{lem:goodCover} with $l_i=\max\left(l(O_i),\sup_{t\in O_i}h(t)-\inf_{t\in O_i}h(t)\right)$ to obtain a family of open intervals $\{Q_j\}_J$ and $f:I\to J$. Let $\bar{h}_+:[0,R]\to[0,R]$ be:
\[
\bar{h}_+(t)=\left\{
  \begin{array}{ll}
    h(t),& \hbox{if $t\notin\bigcup_{j\in J}Q_j$;} \\
    \text{linear function maps $Q_j$ to } & \hbox{if $t\in Q_j$.}\\
    \text{$[\inf_{t\in Q_j}h(t),\sup_{t\in Q_j}h(t)]$, }& \\
  \end{array}
\right.
\]
Then we have the following estimates for the $\bar{h}_+$'s derivative on $Q_j$, i.e. for every $t\in Q_j$:
\begin{equation}\label{eq:derivativeEst}
    \begin{aligned}
    &\frac{(1-\epsilon)(l(Q_j)-l(\cup_{f(i)=j}O_i))}{l(Q_j)}\leq \bar{h}'_+(t)\\
    &\frac{(1+\epsilon)(l(Q_i)-l(\cup_{f(i)=j}O_i))+\sum_{f(i)=j}(\sup_{t\in O_i}h_+(t)-\inf_{t\in O_i}h_+(t))}{l(Q_j)}.
    \end{aligned}
\end{equation}
Since 
\begin{equation*}
    \begin{aligned}
    &\frac{(1+\epsilon)(l(Q_j)-l(\cup_{f(i)=j}O_i))+\sum_{f(i)=j}(\sup_{t\in O_i}h_+(t)-\inf_{t\in O_i}h_+(t))}{l(Q_j)}\\
    \leq&\frac{(1+\epsilon)l(Q_j)+\sum_{f(i)=j}(\sup_{t\in O_i}h_+(t)-\inf_{t\in O_i}h_+(t))}{l(Q_j)}\leq1+\epsilon+\sqrt{\epsilon},
    \end{aligned}
\end{equation*}
and
\[
\frac{(1-\epsilon)(l(Q_j)-l(\cup_{f(i)=j}O_i))}{l(Q_j)}\geq(1-\epsilon)(1-\sqrt{\epsilon}),
\]
hence \eqref{eq:derivativeEst} and $\absolute{\bar{h}'_+(t)-1}<\epsilon$ for $t\in \bar{A}_+$ give that 
\begin{equation}\label{eq:extendMatching}
    \begin{gathered}
    \absolute{\bar{h}'_+(t)-1}<\epsilon+\sqrt{\epsilon}\text{ for a.e. }t\in[0,R],\\
    d(T_{\bar{h}_+(t)}x,T_{t}y)<\delta\text{ for }t\in[0,R]\setminus(\bigcup_{j\in J}Q_j),\\
    l(\bigcup_{j\in J}Q_j)\leq10\sqrt{\epsilon}R.
    \end{gathered}
\end{equation}

As $\bar{h}_+$ is an increasing absolutely continuous function from $[0,R]$ to $[0,R]$, we obtain that $\bar{h}_+\in W^{1,1}([0,R])$, where $W^{1,1}$ is Sobolev space. Then by Meyers-Serrin theorem, we know that $C^{\infty}([0,R])\cap W^{1,1}([0,R])$ is dense in $W^{1,1}([0,R])$ and hence there exists 
\[
\tilde{h}_+\in C^{\infty}([0,R])\cap W^{1,1}([0,R])
\]
such that
\begin{equation}\label{eq:MSDense}
\absolute{\tilde{h}_+-\bar{h}_+}_{W^{1,1}([0,R])}\leq \min(\epsilon,\delta).
\end{equation}
Let $\tilde{A}_+=\bar{A}_+\setminus(\bigcup_{j\in J}Q_j)$, then \eqref{eq:extendMatching} and \eqref{eq:MSDense} give that
\begin{gather*}
    l(\tilde{A}_+)\geq(1-20\sqrt{\epsilon})R,\\
    d(T_{\tilde{h}_+(t)}x,T_{t}y)<2\delta, \text{ for }t\in\tilde{A}_+,\\
    \absolute{\tilde{h}'_+(t)-1}<3\sqrt{\epsilon}, \text{ for }t\in[0,R] \text{ and }\absolute{\tilde{h}'_+(0)}\leq\epsilon.
\end{gather*}

Repeat above process for $h_-=h|_{[-R,0]\cap A}$, we obtain $\tilde{h}_-\in C^{\infty}([0,R])$ and $\tilde{A}_-\subset[-R,0]$ such that
\begin{gather*}
    l(\tilde{A}_-)\geq(1-20\sqrt{\epsilon})R,\\
    d(T_{\tilde{h}_-(t)}x,T_{t}y)<2\delta \text{ for }t\in\tilde{A}_-,\\
    \absolute{\tilde{h}'_-(t)-1}<3\sqrt{\epsilon} \text{ for }t\in[-R,0] \text{ and }\absolute{\tilde{h}'_-(0)}\leq\epsilon.
\end{gather*}

Let $\delta(t)=\left\{
  \begin{array}{ll}
    e^{-1/(1-(\frac{t}{10\epsilon})^2)},& \hbox{$\absolute{t}<10\epsilon$;} \\
    0, & \hbox{other.}
  \end{array}
\right.$ and $F(t)=\frac{\int_{-\infty}^t\delta(s)ds}{\int_{-\infty}^{\infty}\delta(s)ds}$, then $F(t)$ is a smooth function that equals $0$ for $t\leq-10\epsilon$ and $1$ for $t\geq 10\epsilon$. We define $\tilde{h}:[-R,R]\to[-R,R]$ as following:
\[
\tilde{h}(t)=(1-f_1(t))\tilde{h}_-(t)+(f_1(t)-f_2(t))t+f_2(t)\tilde{h}_+(t),
\]
where $f_1(t)=F(t+20\epsilon)$ and $f_2(t)=F(t-20\epsilon)$. Let 
\[
\tilde{A}=(\tilde{A}_-\cap[-R,-20\epsilon])\cup\{0\}\cup(\tilde{A}_+\cap[20\epsilon,R]),
\]
since $R\geq100$, we have
\begin{gather*}
    l(\tilde{A})\geq(1-50\sqrt{\epsilon})2R,\qquad 0\in\tilde{A}\text{ and } \tilde{h}(0)=0,\\
    d(T_{\tilde{h}(t)}x,T_{t}y)<2\delta \text{ for }t\in\tilde{A},\\
    \absolute{\tilde{h}'(t)-1}<9\sqrt{\epsilon}(\kappa_1+1) \text{ for }t\in[-R,R].
\end{gather*}
By setting $C=\max(100,\frac{10e^{-1}}{\int_{-\infty}^{\infty}\delta(s)ds}+9)$,  we complete the proof of Lemma~\ref{lem:smoothMatching}.

\section{Improving behavior of even Kakutani equivalences}\label{sec:appGoodTime}
In this section we prove Lemma~\ref{lem:goodTimeChange}. While this result is well known, we provide the details for completeness.

\medskip

We recall the flow under a function construction of flows (i.e. $\R$-actions) from $\Z$ actions. Let $\phi$ be an ergodic (invertible) transformation on $(X,\mathcal{B},\mu)$, $f\in L_+^1(X,\mathcal{B},\mu)$ and $\lambda$ the Lebesgue measure on $\R$. The special flow of $\phi$ over $f$ is defined as $\phi_t^f$ and acts on space $(X^f,\mathcal{B}^f,\mu^f)$, where 
\[
X^f=\{(x,s):x\in X,0\leq s<f(x)\},\quad  \mathcal{B}^f=\mathcal{B}\otimes\mathcal{B}(\R),\quad \mu^f=\mu\times\lambda.
\]
The action of $\phi_t^f$ moves point in $X^f$ vertically with unit speed identifying $(x,f(x))$ with $(\phi(x),0)$.

\medskip

Given an ergodic measure preserving flow $T_t$, by Rokhlin's special flow representation theorem \cite{rokhlin1967lectures}, there exists an ergodic automorphism $\phi$ on $(X_0,\mathcal{B}_0,\mu_0)$ and $\tau\in L_+^1(X_0,\mathcal{B}_0,\mu_0)$ such that $\phi_t^{\tau}$ is isomorphic to $T_t$, i.e. there exists an invertible (modulo null sets) map  $\psi_{\tau}: X_0^{\tau}\to X$ such that $(\psi_{\tau})_*\mu_0^{\tau}=\mu$ and 
\[
\psi_{\tau}\circ \phi_t^{\tau} (\bar{x})= T_t\circ \psi_{\tau}(\bar{x}) \text{ for $\mu_0^{\tau}$-a.e. $\bar{x}\in X_0^{\tau}$ and all $t\in\R$.}
\]
 Moreover, by restricting to a subset of $X_0$, we may assume that 
\begin{equation}\label{eq:taularge}
    \int_{X_0}\tau d\mu_0\geq 10.
\end{equation} 
We may (and will) assume from now on that $X$ itself is of this form.

Now suppose that $S_t$ is another measure preserving flow evenly Kakutani equivalent to $T_t$. Then by definition this means that there exist a $\alpha\in L_+^1(X,\mathcal{B},\mu)$ with $\int_X\alpha d\mu=1$ and an invertible map $\psi:X\to Y$ such that $(\psi)_*\mu^{\alpha}=\nu$ and 
\begin{equation}\label{eq:R2Kakutani}
\psi\circ T_{\rho(x,t)}(x)=S_t\circ \psi(x) \text{ for $\mu$-a.e. $x\in X$ and all $t\in\R$,}
\end{equation}
 where $\rho(x,t)$ satisfies
\begin{equation}\label{eq:R2u}
\int_0^{\rho(x,t)}\alpha(T_sx)ds=t.
\end{equation}

Let $\xi^{\tau}$ be a countably generated $\sigma$-algebra of subsets of $X=X_0^{\tau}$ so that for $\mu_0^{\tau}$-almost every $(z,t)$ ($z\in X_0$, $0\leq t <\tau(z)$) the atom of $[(z,t)]_{\xi^{\tau}}$ is of the form $c_z^{\tau}=\{z\}\times[0,\tau(z)]$. Then \eqref{eq:R2Kakutani} implies that for almost every $z$ the segment $c_z^{\tau}$ is mapped under $\psi$ to a segment of a trajectory of $S_t$ in $Y$ and that the segment $\psi(c_{\phi(z)}^{\tau})$ immediately follows the segment $\psi(c_z^{\tau})$ with respect to the $S_t$-flow. Define $\sigma(z) : X_0 \to \R^+$ by 
\begin{equation}\label{eq:newHeight}
\rho(\psi_{\tau}(z,0),\sigma(z))=\tau(z);
\end{equation}
thus $\sigma(z)$ is the length of segment (in the flow $S_t$) of  $\psi (c_z^{\tau})$.

Let $\psnew\label{0962psA3}:X_0^{\tau}\to X_0^{\sigma}$\index{$\psold{0962psA3}$} be  $\psold{0962psA3}(z,s)=(z,\theta(z,s))$, where $\theta(z,s)$ satisfies
\begin{equation}\label{eq:isometryCocy}
\rho(\psi_{\tau}(z,0),\theta(z,s))=s \text{ for all $s\in[0,\tau(z))$.}
\end{equation}
Then $\psi_1 \circ \psi^{-1}$ gives a measurable isomorphism between $(Y,S_t,\nu)$ and $(X_0^{\sigma},\phi_t^{\sigma},\mu_0^{\sigma})$. 

\medskip

We now construct some ``good'' sets:
\begin{enumerate}[label=(\roman*)]
    \item Let $Z_1\subset X_0$ be the set of points such that the pointwise ergodic theorem holds at $(z,t)$ for the flow $(X,T_t,\mu)$ and function $\alpha$ for a.e.\ $t\in [0,\tau(z))$.
    \item Let $Z_2$ be the set of points such that the ergodic theorem holds for transformation $(X_0,\phi,\mu_0)$, and functions $\tau$ and $\sigma$. 
    \item Let $Z_3\subset X_0$ be the set of points satisfying \eqref{eq:newHeight}.
    \item Let  $Z_4\subset X_0$ be the set of points such that $\tau$ and $\sigma$ is finite. 
    \item Define
\[
Z_5=\bigcap_{i=-\infty}^{+\infty}\phi^{-i}\left(Z_1\cap Z_2\cap Z_3\cap Z_4\right).
\]
It follows that $\mu_0(Z_5)=1$. Notice for any $z_0\in Z_5$,  there exists $s_0\in[0,\tau(z_0))$ such that $(z_0,s_0)\in \psi_{\tau}^{-1}(Z_1)$. 
\end{enumerate}

\medskip

Combining \eqref{eq:R2u}, \eqref{eq:newHeight} and \eqref{eq:isometryCocy}, then for $z_0\in Z_5$
\[
\int_{s_0}^{\tau(z_0)}\alpha(T_s(\psi_{\tau}(z_0,s_0)))ds=\sigma(z_0)-\theta(z_0,s_0).
\]
This together with $\phi^iz_0 \in Z_1\cap Z_2\cap Z_3\cap Z_4$ gives that
\begin{equation}\label{eq:ergodicTarget}
\int_{s_0}^{\sum_{i=0}^n\tau(\phi^iz_0)}\alpha(T_s(\psi_{\tau}(z_0,s_0)))ds=-\theta(z_0,s_0)+\sum_{i=0}^n\sigma(\phi^iz_0).
\end{equation}
Since $z_0\in Z_5$ and $z_0\in Z_1$, we have
\begin{equation*}
    \begin{gathered}
    \lim_{n\to\infty}\frac{1}{n}\left(-s_0+\sum_{i=0}^n\tau(\phi^iz_0)\right)=\int_{X_0}\tau d\mu_0,\\
    \lim_{n\to\infty}\frac{1}{n}\left(-\theta(z_0,s_0)+\sum_{i=0}^n\sigma(\phi^iz_0)\right)=\int_{X_0}\sigma d\mu_0,\\
    \lim_{n\to\infty}\frac{\int_{s_0}^{\sum_{i=0}^n\tau(\phi^iz_0)}\alpha(T_s(\psi_{\tau}(z_0,s_0)))ds}{-s_0+\sum_{i=0}^n\tau(\phi^iz_0)}=\int_X\alpha d\mu,
    \end{gathered}
\end{equation*}
which together with \eqref{eq:ergodicTarget} and $\int_X\alpha d\mu=1$ imply that 
\[
\int_{X_0}\tau d\mu_0=\int_{X_0}\sigma d\mu_0.
\]

\medskip

By Rokhlin lemma and pointwise ergodic theorem, there exist a set $A\subset Z_5$ with $\mu_0(A)>0$ and an positive integer $N>10$ such that if we set 
\[r_A(z)=\min\{i\in\Z_+:\phi^i(z)\in A\}
\]
then for every $z\in A$, $r_A(z)> N$, and moreover for every $z\in A$ and $n\geq N$
    \begin{equation*}
        \begin{gathered}
            \absolute{\frac{\sum_{i=0}^{n-1}\tau(\phi^iz)}{\sum_{i=0}^{n-1}\sigma(\phi^iz)}-1}<0.01\epsilon,\\ \sum_{i=0}^{n-1}\tau(\phi^iz)\geq 10.
        \end{gathered}
    \end{equation*}

Let $\tau_1(z)=\sum_{i=0}^{r_A(z)-1}\tau(\phi^iz)$ and $\sigma_1(z)=\sum_{i=0}^{r_A(z)-1}\sigma(\phi^iz)$. Then 
\begin{enumerate}[label=\textit{\roman*}.]
    \item the flow $\left((\phi_A)_t^{\tau_1},A^{\tau_1},(\mu_0|_A)^{\tau_1}\right)$ is isomorphic to $(\phi_t^{\tau},X,\mu_0^{\tau})$ through the measurable isomorphism $\psi_{\tau_1}:A^{\tau_1}\to X$ satisfying:
    \[
    \psi_{\tau_1}(z,0)=(z,0),\qquad \psi_{\tau_1}(z,\tau_1(z))=(\phi_Az,0);
    \]
    \item the flow $\left((\phi_A)_t^{\sigma_1},A^{\sigma_1},(\mu_0|_A)^{\sigma_1}\right)$ is isomorphic to $(\phi_t^{\sigma},X_0^{\sigma},\mu_0^{\sigma})$ through the measurable isomorphism $\psi_{\sigma_1}:A^{\sigma_1}\to X_0^{\sigma}$ satisfying:
    \[
    \psi_{\sigma_1}(z,0)=(z,0),\qquad \psi_{\sigma_1}(z,\sigma_1(z))=(\phi_Az,0);
    \]
    \item for every $z\in A$ 
    \begin{alignat*}{3}
        &\psold{0962psA3}\circ\psi_{\tau_1}(z,0)=\psi_{\sigma_1}(z,0),\qquad &&\psold{0962psA3}\circ\psi_{\tau_1}(z,\tau_1(z))=\psi_{\sigma_1}(z,\sigma_1(z));
    \end{alignat*}
    
    \item for every $z\in A$, 
    \[
    \absolute{\tau_1(z)/\sigma_1(z)-1}<0.02\epsilon \qquad\text{and} \qquad\tau_1(z)\geq 10.
    \]
\end{enumerate}

\medskip

We then construct a new Kakutani equivalence between \[
\left((\phi_A)_t^{\tau_1},A^{\tau_1},(\mu_0|_A)^{\tau_1}\right)\text{ and }\left((\phi_A)_t^{\sigma_1},A^{\sigma_1},(\mu_0|_A)^{\sigma_1}\right).
\]
Let 
\begin{equation*}
    \begin{aligned}
        \delta(x)&=\left\{
  \begin{array}{ll}
    e^{-1/(1-(x/\epsilon^2)^2)},& \hbox{$\absolute{x}<\epsilon^2$,} \\
    0, & \hbox{otherwise;}
  \end{array}\right.\\
  F(t)&=\frac{\int_{-\infty}^t\delta(y)dy}{\int_{-\infty}^{\infty}\delta(y)dy}.
    \end{aligned}
\end{equation*}
Then $F(t)$ is a smooth function that equals $0$ for $t\leq-\epsilon^2$ and $1$ for $t\geq \epsilon^2$. Let $\kappa_1=\frac{e^{-1}}{\int_{-\infty}^{\infty}\delta(y)dy}$, 
then we define $g_z(x)$ as following: 
\[
g_z(t)=\left\{
  \begin{array}{ll}
    t(1-f_1(t))+(f_1(t)-f_2(t))\frac{\sigma_1(z)}{\tau_1(z)}t& \hbox{$z\in A$,} \\
     \ \ \ \ \ \ +(t-\tau_1(z)+\sigma_1(z))f_2(t),&\\
    0, & \hbox{otherwise,}
  \end{array}
\right.
\]
where $f_1(t)$ and $f_2(t)$ are defined as
\[
f_1(t)=F(t-2\epsilon^2),\qquad f_2(t)=F(t-\tau_1(z)+2\epsilon^2).
\]
These together with $\epsilon\in(0,10^{-3})$ and $\tau_1(z)\geq 10$ give for every $z\in A$
\begin{enumerate}[label=(\Roman*)]
    \item $g_z\in C^{\infty}([0,\tau_1(z)])$, 
    \item $g'_z(0)=g'_z(\tau_1(z))=1$, $g_z(0)=0$ and $g_z(\tau_1(z))=\sigma_1(z)$,
    \item $\absolute{g'_z(s)-1}<0.3\epsilon$ for $z\in A$ and $s\in[0,\tau_1(z)]$. 
\end{enumerate}
Then let $\psnew\label{098psA4}:A^{\tau_1}\to A^{\sigma_1}$\index{$\psold{098psA4}$} be defined as
\[
\psold{098psA4}(z,s)=(z,g_z(s))\text{ for }(z,s)\in A^{\tau_1}.
\]
Let $\bar{\alpha}(z,s)=g'_z(s)$ and $\bar{\rho}((z,s),t)$ satisfy 
\begin{equation*}
    \int_0^{\bar{\rho}((z,s),t)}\bar{\alpha}((\phi_A)_l^{\tau_1}(z,s))dl=t.
\end{equation*} 
It follows that $(\psold{098psA4})_*((\mu_0|_A)^{\tau_1})^{\bar{\alpha}}=(\mu_0|_A)^{\sigma_1}$ and 
\begin{equation*}
    \begin{aligned}
    \psold{098psA4}\circ (\phi_A)^{\tau_1}_{\bar{\rho}((z,s),t)}(z,s)&=(\phi_A)_t^{\sigma_1}\circ  \psold{098psA4}(z,s)
    \end{aligned}
\end{equation*}
for $(\mu_0|_A)^{\tau_1}$-a.e. $(z,s)\in A^{\tau_1}$ and $t\in\R$. Hence $\psold{098psA4}$ is an $\epsilon$-well-behaved Kakutani equivalence. Since $(\phi_t^{\tau_1},X_0^{\tau_1},\mu_0^{\tau_1})$ is isomorphic to $(T_t,X,\mu)$ and $(\phi_t^{\sigma_1},X_0^{\sigma_1},\mu_0^{\sigma_1})$ is isomorphic to $(S_t,Y,\nu)$, we obtain that $\psold{098psA4}$ and $\bar{\alpha}$ induce an $\epsilon$-well-behaved Kakutani equivalence map $\tilde{\psi}:X\to Y$ between $T_t$ and $S_t$.

\medskip

It follows from the definition of $\tau_1$, $\sigma_1$ and special flow that Kakutani equivalence $\psi_{\sigma_1}^{-1}\circ\psold{0962psA3}\circ\psi_{\tau_1}$ is cohomologous to $\psold{098psA4}$. Recall that $\psold{0962psA3}$ is induced by Kakutani equivalence $\psi$, hence $\psi$ is cohomologous to $\tilde{\psi}$. We complete our proof by setting  $Z=\psi_{\tau}\circ\psi_{\tau}(A^{\tau_1})$.

\section{Spectral gap and an \texorpdfstring{$\SL_2(\mathbb{R})$}{SL2(R)}-square function estimate}\label{sec:sl2ergodic}
In this appendix we prove a square function estimate for $\SL_2(\R)$-actions. This result is used in \S\ref{sec:sl2Ergodiccorollary} to prove a certain ergodic result for $\SL_2(\R)$-actions. 

Let $\mathbf{H}=\SL_2(\R)$ and $a_t=\exp(t\ba)$ with $\ba=\left(\begin{smallmatrix}
    1 & 0\\ 0 & -1\\ 
\end{smallmatrix}\right)$. For $\epsilon>0$, define
\begin{equation*}
B_{\epsilon}^{\mathbf{H},\norm{\cdot}}=\{h\in \mathbf{H}:\norm{h-\id}<\epsilon\},
\end{equation*}
where $\norm{\cdot}$ is the Hilbert-Schmidt norm of matrix. We denote by $m_{\mathbf{H}}$ the Haar measure on $\mathbf{H}$.

\medskip

We define $L^2$-Sobolev norms in various settings:
\begin{itemize}
    \item Recall that the $L^2$-Sobolev space $H_s(\R/\Z)$ with $s$-derivatives is defined as the completion of $C^{\infty}(\R/\Z)$ with respect to the norm defined~by
    \[
    \norm{f}^2_{H_s}\doteq\sum_{n\in\Z}(1+\absolute{n}^{2s})\absolute{c_n}^2\text{  for $f=\sum_{n\in\Z}c_ne^{2\pi in\theta}$.}
    \]
    \item More generally, given a unitary action $\rho$ of $\R/\Z$ on a Hilbert space $\mathcal{H}_{\rho}$ we say that 
$f \in\mathcal{H}_{\rho}$ is in $L^2$-Sobolev space $H_{s}(\mathcal{H}_{\rho})$ with respect to this action if when we decompose $f$ into $f=\sum_{n\in\Z} f_n$ with 
\[
\rho(\theta)f_n = e^{2\pi in\theta}f_n,
\]
then 
\[
\norm{f}_{H_{s}}^2 \doteq\sum_{n\in\Z} (1+\absolute{n}^{2s}) \norm{f_n}_2^2 < \infty.
\]
\item If $(\rho, \mathcal{H}_{\rho})$ is a unitary representation of $\mathbf{H}$ we say that 
\[
f \in H_s (\mathcal{H}_{\rho}; k_{\theta})
\]
if it in $H_{s}(\mathcal{H}_{\rho})$ considered as a $\R/\Z$ representation via the action of $k_{\theta}$. This in particular applies to $\mathcal{H}_{\rho}= L^2(\mathbf{H})$.
\end{itemize}

\medskip

The main result of this section is the following theorem:
\begin{Theorem}\label{thm:NewNewsl2ergodic}
   Let $\rho$ be a unitary representation of $\mathbf{H}=\SL_2(\R)$ on a Hilbert space $\mathcal H_{\rho}$ with no fixed vectors and a spectral gap. Let $\chi : \mathbf{H} \to \R$ be of compact support in Sobolev space $H_{s}(L^2(\mathbf{H});k_{\theta})$ for some $s>0$, $\chi\geq 0$ and $\int \chi(h) dm_{\mathbf{H}}(h)=1$. Then there is a $C(\rho,\chi)>0$ depending on the spectral gap and $\chi$ so that for any $v \in \mathcal{H}_{\rho}$
\begin{equation}
  \sum_{n=1}^{\infty}\norm{\int_{\mathbf{H}}\chi(h)\rho( ha_{-n})v \,dm_{\mathbf{H}}(h)}_2^{2}<C(\rho,\chi)\norm{v}_2^2. \label{appendix:eq:1}  
\end{equation}
\end{Theorem}

\medskip

We recall the following standard fact:
\begin{Lemma}
    If $\rho$ a unitary action of $\R/\Z$ on a Hilbert space $\mathcal{H}$ and $\norm{\rho(\theta)f -f }_2\leq C\absolute{\theta}^s$ for every  $\theta\in(0,1]$, then for every $s'<s$ 
\[
\norm{f}_{H_{s'}} \ll C(s-s')^{-1} +\norm{f}_2.
\]
\end{Lemma}
\begin{proof}

    Suppose that $\norm{\rho(\theta)f -f }_2\leq C \absolute{\theta}^s$ for every  $\theta\in(0,1]$. 
    Let $f=\sum_{n\in\Z}f_n$ with $\rho(\theta)f_n = e^{2\pi in\theta}f_n$, then for every $s'\in(0,s)$
    \begin{equation*}
        \begin{aligned}
            \int_{\R/\Z}\frac{\norm{\rho(\theta)f-f}^2_2}{\absolute{\theta}^{1+2s'}}d\theta =&\sum_{n\in\Z\setminus\{0\}}\left(\int_{\R/\Z}\frac{\absolute{e^{ 2\pi in\theta}-1}^2}{\absolute{\theta}^{1+2s'}}d\theta\right)\norm{f_n}_2^2\\
            =&\sum_{n\in\Z\setminus\{0\}}\left(\int_{\R/\Z}\frac{\absolute{e^{ 2\pi in\theta}-1}^2}{\absolute{n}^{2s'}\absolute{\theta}^{1+2s'}}d\theta\right)\absolute{n}^{2s'}\norm{f_n}_2^2.
        \end{aligned}
    \end{equation*}
    Notice that $
    \int_{\R/\Z}\absolute{e^{2\pi  in\theta}-1}^2\absolute{n}^{-2s'}\absolute{\theta}^{-1-2s'}d\theta\geq1$,     
    hence 
    \[
    \sum_{n\in\Z}\absolute{n}^{2s'}\norm{f_n}_2^2\leq \int_{\R/\Z}\frac{\norm{\rho(\theta)f-f}_2^2}{\absolute{\theta}^{1+2s'}}d\theta.
    \]
    This together with our assumption gives
    \[
     \sum_{n\in\Z}\absolute{n}^{2s'}\norm{f_n}_2^2\leq \int_{\R/\Z}\frac{C^2}{\absolute{\theta}^{1+2(s'-s)}}d\theta,
    \]
    and the lemma follows.
\end{proof}

In particular if $\chi = \frac 1{m_{\mathbf{H}}(B)}1_B$ for an open $B \subset \SL_2(\R)$  with 
\[
m_{\mathbf{H}}(B \bigtriangleup k_{\theta}B) \ll \absolute{\theta},
\]
a condition that holds for nice ball like sets, then 
\[
\norm{\rho(\theta)\chi -\chi }_2\ll \absolute{\theta}^{1/2}.
\]
Hence Theorem \ref{thm:NewNewsl2ergodic} can be applied for any $s'<1/2$. This holds in particular for the balls $B_{\epsilon}^{\mathbf{H},\norm{\cdot}}$ defined above.

Let $\rho$ be a unitary representation of $\mathbf{H}=\SL_2(\R)$ on a Hilbert space $\mathcal H_{\rho}$. Set for any function $\chi:\mathbf{H}\to\R$
\[
{\rho}(\chi)v = \int_{\mathbf{H}} \chi(h) \rho(h)v \,dm_{\mathbf{H}}(h).
\]
\begin{Lemma}\label{lem:rhoWellDef}
     If $\chi \in H_s(L^2(\mathbf{H}); k_{\theta})$ and $v \in \mathcal H_{\rho}$, then 
     \[
   {\rho}(\chi)v \in H_s(\mathcal{H}_{\rho}; k_{\theta}),\qquad \norm{ {\rho}(\chi)v}_{H_s} \ll_{\chi} \norm{v}_2.
     \]
\end{Lemma}

\begin{proof}
Let $\Omega \subset \mathbf H$ be a compact $k_\theta$-invariant set supporting $\chi$, and write $\chi=\sum_{n\in\Z}\chi_n$ with $\chi_n(k_{\theta}^{-1}h)=e^{2\pi in\theta}\chi_n(h)$ for every $h\in\mathbf{H}$. Note that $\rho(k_{\theta}) {\rho}(\chi_n)v = e^{2\pi in\theta} {\rho}(\chi_n)v$ hence
    \begin{equation}\label{eq:eigenDecomRhoChi}
        \begin{aligned}
           \norm{{\rho}(\chi)v}_{H_s}^2=\sum_{n\in\Z}(1+\absolute{n}^{2s})\norm{ {\rho}(\chi_n)v}_2^2.
        \end{aligned}
    \end{equation}
    By the triangle inequality
    \begin{equation*}
    \begin{aligned}
    \norm{ {\rho}(\chi_n)v}_2 &= \norm{\int_{\mathbf{H}} \chi_n (h) \rho(h)v\,  dm_{\mathbf{H}}(h)}_2\\
    &\leq \norm{\chi_n}_1\norm{v}_2 \leq m_{\mathbf{H}}(\Omega)^{1/2}\norm{\chi_n}_2\norm{v}_2,
    \end{aligned}
    \end{equation*}
    which together with \eqref{eq:eigenDecomRhoChi} completes the proof of Lemma~\ref{lem:rhoWellDef}.
\end{proof}

For the proof of Theorem~\ref{thm:NewNewsl2ergodic},
it is enough to consider the case of irreducible $\rho$ (as long as the constant depends only on how far $\rho$ is from the trivial representation). 

\begin{Lemma}\label{appendix:lem:1}
 There are $C'(\rho,\chi), \lambda'(\rho)>0$ such that for arbitrary $p,w \in \mathcal H_{\rho}$ if we set $p'=\rho(\chi^{\sim}) p$ and $w'=\rho(\chi^{\sim}) w$ then 
\[
\absolute{\innproduct{\rho(a_n)p', \rho(a_m)w'}}\leq C'(\rho,\chi) e^{-\lambda'(\rho) |n-m|}\norm{p}_2\norm{\vphantom{p}w}_2.
\]
Here $\chi^{\sim}(h)=\chi(h^{-1})$ for $h\in\mathbf{H}$.
\end{Lemma}

\begin{proof}[Proof of Lemma~\ref{appendix:lem:1}]
Write
\[      {\rho}(\chi^{\sim}) p =p'= \sum_{n\in\Z} p'_n \text{ with } \rho(k_{\theta})p'_n=e^{2\pi in\theta}p'_n.
\]
Let $p'_{(N)}=\sum_{\absolute{n}\leq N}p'_n$; then by definition $\|p'-p'_{(N)}\|\leq \|p'\|_{H_s} N^{-s}$. It follows from  Lemma~\ref{lem:rhoWellDef} that
\[
        \norm {p'-p'_{(N)}}_2\ll_\chi N^{-s}\norm{p}_2.
\]

From a result by Cowling, Haagerup and Howe \cite{Cowling1988Almost}*{Applications}, there is a positive constant $\lambda(\rho)$ depending only on the distance of $\rho$ from the trivial representation so that
\begin{equation*}
\begin{aligned}
\absolute{\innproduct{\rho(a_{t})p'_{(N)},w'_{(N)}}} &\ll N e^{-\lambda(\rho) t}\norm{p'_{(N)}}_2\norm{w'_{(N)}}_2 \\
&\ll_\chi N e^{-\lambda(\rho) t}\norm{p}_2\norm{w}_2.
\end{aligned}
\end{equation*}

Hence for some $C''(\rho,\chi)>0$
\[
\absolute{\innproduct{\rho(a_t)p', w'}} \leq C''(\rho,\chi)\left (N e^{-\lambda(\rho) t} + N^{-s}\right)\norm{p}_2\norm {\vphantom p w}_2.
\]
Optimizing by choosing $N=e^{\lambda(\rho) t/(s+1)}$ we get
\[
\absolute{\innproduct{\rho(a_t)p', w'}}\leq 2C''(\rho,\chi) e^{-\lambda(\rho) s t /(1+s)}\norm{p}_2\norm {\vphantom p w}_2,
\]
establishing the lemma for $\lambda'(\rho)=\lambda(\rho) s / (s+1)$. 
\end{proof}

\begin{Lemma}\label{lem:newTT*lemma}
    Let $f$, $g_i$ be elements in a Hilbert space $\mathcal H$. Then there is some sequence $\{c_i\}_{i\in\N}\in \ell^2(\N)$ with $\sum_{i=0}^{\infty} |c_i|^2=1$ so that
$$
\sum_{i=0}^{\infty} \absolute{\innproduct{f,g_i}}^2 \leq \norm{f}_2^2 \left(\sum_{i=0}^{\infty} \sum_{j=0}^{\infty} c_i \overline
 {c_j} \innproduct{g_i,g_j}\right).
$$
\end{Lemma}
This type of estimate is known as the $TT^*$-method (also the method of almost orthogonality). For proof see~\cite{SteinHarmonicAnalysis}*{\S VII} or \cite{TerrenceTaoTT*method}*{Proposition~2}.

\medskip

\begin{proof}[Proof of Theorem~\ref{thm:NewNewsl2ergodic}]
Let 
\[
g_n = \frac{\rho(a_{n}){\rho}(\chi^{\sim}){\rho}(\chi)\rho(a_{-n})v}{\norm{{\rho}(\chi)\rho(a_{-n})v}_2},\qquad \chi^{\sim}(h)=\chi(h^{-1}),
\]
so that the left hand side of \eqref{appendix:eq:1} can be rewritten as 
\begin{equation}\label{eq:newnewSL2firstConvert}
    \sum_{n=0}^{\infty}\norm{ {\rho}(\chi)\rho(a_{-n})v}_2^2=\sum_{n=0}^{\infty}\absolute{\innproduct{v,g_n}}^2.
\end{equation}
By Lemma~\ref{lem:newTT*lemma}, there exists $\{c_i\}_{i\in\N}\in\ell^2(\N)$ with $\sum_{i=0}^{\infty}\absolute{c_i}^2=1$ such that
\begin{equation}\label{eq:newTT*lemmaApp}
   \sum_{n=0}^{\infty}\absolute{\innproduct{v,g_n}}^2\leq\norm{v}^2_2\left(\sum_{i=0}^{\infty}\sum_{j=0}^{\infty}c_i\bar{c}_j\innproduct{g_i,g_j}\right).
\end{equation}
The $g_n$ in \eqref{eq:newnewSL2firstConvert} satisfy
\[
g_n=\rho(a_n){\rho}(\chi^{\sim})v_n\text{ with }v_n=\frac{{\rho}(\chi)\rho(a_{-n})v}{\norm{{\rho}(\chi)\rho(a_{-n})v}_2}.
\]
Combining \eqref{eq:newnewSL2firstConvert}, \eqref{eq:newTT*lemmaApp} and Lemma~\ref{appendix:lem:1} (with $p=v_i$ and $w=v_j$)
\begin{equation*}
    \begin{aligned}
        \sum_{n=0}^{\infty}\norm{\bar{\rho}(\chi)\rho(a_{-n})v}_2^2& \leq \norm{v}_2^2\left(\sum_{i=0}^{\infty}\sum_{j=0}^{\infty}c_i\bar{c}_j\innproduct{g_i,g_j}\right)\\
        & \leq C'(\rho,\chi)\norm{v}_2^2\left(\sum_{i=0}^{\infty}\sum_{j=0}^{\infty}e^{-\lambda'(\rho)\absolute{i-j}}c_i\bar{c}_j\norm{v_i}_2\norm{v_j}_2\right)\\
        &\leq C'(\rho,\chi)\norm{v}_2^2\left(\sum_{i=0}^{\infty}\absolute{c_i}^2+2\sum_{i=1}^{\infty}\sum_{j=0}^{\infty}e^{-\lambda'(\rho)i}\absolute{c_j}\absolute{c_{j+i}}\right)\\
        &\leq C(\rho,\chi)\norm{v}_2^2,
    \end{aligned}
\end{equation*}
where in the last inequality we used the Cauchy-Schwartz inequality and the fact that $\sum_{i=0}^{\infty}\absolute{c_i}^2=1$.
Note that the constant $C(\rho,\chi)>0$ does indeed depend only on $\chi$ and the spectral gap. This completes the proof of Theorem~\ref{thm:NewNewsl2ergodic}.
\end{proof}

\bibliographystyle{plain} 
\bibliography{refs}

\end{document}